\documentclass[a4paper,11pt]{report}
\usepackage{latexsym}
\usepackage{amsfonts}
\usepackage{anysize}
\usepackage{graphicx,amssymb,amsfonts,amsmath,amsthm,xypic}
\usepackage{newlfont,color,marvosym}
\usepackage{epsfig}
\usepackage[latin1]{inputenc}
\usepackage[T1]{fontenc}
\usepackage[all]{xy}
\usepackage{amscd}
\usepackage{palatino, url, multicol}
\usepackage{makeidx}
\makeindex
\newcommand{\singlespaced}{\renewcommand{\baselinestretch}{1}\normalfont}

\theoremstyle{plain}
\newtheorem{thm}{Theorem}[chapter]
\newtheorem{cor}[thm]{Corollary}
\newtheorem{lem}[thm]{Lemma}
\newtheorem{prop}[thm]{Proposition}

\newtheorem{defn}[thm]{Definition}
\newtheorem{defn/lemme}[thm]{Definition/Lemma}
\newtheorem{defns}[thm]{Definitions}
\newtheorem{conj}[thm]{Conjecture}

\newtheorem{ex}[thm]{Example}
\newtheorem{exes}[thm]{Examples}
\newtheorem{claim}[thm]{Claim}

\newtheorem{rem}[thm]{Remark}

\newtheorem*{thm*}{Theorem}
\newtheorem*{rmk*}{Remark}
\newtheorem*{ques*}{Question}
\newtheorem*{conj*}{Conjecture}
\newtheorem{example}[thm]{Example}

\numberwithin{equation}{section}

\font \rus= wncyr10
\font\tencyr=wncyr10

\newcommand{\piy}{\pi_{\mathcal{Y}}}
\newcommand{\PSL}{\mathbb{PSL}}
\newcommand{\Mn}{{\mathfrak{M}}_{0,n}}
\newcommand\nz{\mathfrak{nz}}
\newcommand\Q {\mathbb{Q}}
\newcommand\R{\mathbb{R}}
\newcommand\C{\mathbb{C}}
\newcommand\Z{\mathbb{Z}}

\newcommand\nfz{\mathfrak{nfz}}
\newcommand\ds{\mathfrak{ds}}
\newcommand\grt{\mathfrak{grt}}
\newcommand\Lxy{\mathbb{L}[x,y]}
\newcommand\Qxy{{\mathbb{Q}}\langle \langle x,y \rangle \rangle}

\newcommand\Qyi{{\mathbb{Q}}\langle \langle y_i \rangle \rangle}
\newcommand\I{{\mathbb{I}}}
\newcommand{\Pro}{\mathbb{P}}

\newcommand{\sha}{\hbox{\rus x}}
\newcommand{\Res}{\mathrm{Res}}
\newcommand{\Mod}{{\mathfrak{M}}}
\newcommand{\M}{\overline{\mathfrak{M}}}
\newcommand{\Sym}{\mathfrak{S}}
\newcommand\Spec{\mathrm{Spec}}
\newcommand\sh{\hbox{\tencyr{x}}}
\newcommand\Cxy{{\mathbb{C}}\langle \langle x,y \rangle\rangle}
\newcommand\z{\mathfrak{z}}
\newcommand\Lyi{{\mathbb{L}}[y_i]}

\newcommand{\Hom}{\hbox{Hom}}

\newcommand{\To}{\longrightarrow}

\newcommand{\Modf}{\mathfrak{M}^\delta}

\newcommand{\N}{\mathbb{N}}

\newcommand{\ord}{\mathrm{ord}}
\newcommand{\A}{\mathbb{A}}

\newcommand{\gr}{\mathrm{gr}}

\newcommand{\MT}{\mathcal{MT}}
\newcommand{\FMT}{\mathcal{M}}

\newcommand{\MZV}{\mathcal{Z}}

\newcommand{\F}{\mathfrak{F}}
\newcommand{\cfl}{\langle}
\newcommand{\cfr}{\rangle}
\newcommand{\LIE}{\mathfrak{L}}
\newcommand{\FZ}{\mathcal{FZ}}

\newcommand{\spc}{\, \, ; \,\,}
\newcommand{\HRule}{\rule{\linewidth}{0.5mm}}

\def\thetitle{Multizeta values: Lie algebras and periods on $\Mn$}
\def\theauthor{Sarah Carr}   
\def\theadvisor{Leila Schneps, Advisor}

\newenvironment{changemargin}[2]{%
  \begin{list}{}{%
    \setlength{\topsep}{0pt}%
    \setlength{\leftmargin}{#1}%
    \setlength{\rightmargin}{#2}%
    \setlength{\listparindent}{\parindent}%
    \setlength{\itemindent}{\parindent}%
    \setlength{\parsep}{\parskip}%
  }%
  \item[]}{\end{list}}

\begin{document}

\pagenumbering{roman}

\thispagestyle{empty}
\begin{center}
{\textsc{Universit\'e Pierre et Marie Curie}}

\vspace{2.5cm}
{\bf Th\`ese de doctorat}

\vspace{.6cm}
Sp\'ecialit\'e : Math\'ematiques

\vspace{1cm}
pr\'esent\'ee par

\vspace{.6cm}
Sarah \textsc{Carr}

\vspace{1cm}
et soutenue \`a Paris le 27 juin 2008

\vspace{1cm}
Pour obtenir le grade de\\
\textsc{Docteur} de l'\textsc{Universit\'e Pierre et Marie Curie}
\vspace{1cm}

\singlespaced

\HRule \\[0.4cm]
\Large{{\bf \thetitle}}\\
-\\
\Large{{\bf Valeurs multiz\^eta : alg\`ebres de Lie et p\'eriodes sur $\Mn$}}\\
\HRule

\end{center}
\singlespaced

\vspace{3cm}
\begin{multicols}{2}{
\noindent{\bf Directrice de th\`ese }\\
\noindent Leila \textsc{Schneps} (CNRS)\\

\vspace{.6cm}
\noindent{\bf Rapporteurs}\\
\noindent Don \textsc{Zagier} (Coll\`ege de France, MPIM)\\
\noindent Masanobu \textsc{Kaneko} (Kyushu Univ.)

\noindent{\bf Jury}\\
\noindent Daniel \textsc{Bertrand} (Univ. Paris VI)\\
\noindent Jacky \textsc{Cresson} (Univ. Pau)\\
\noindent Hidekazu \textsc{Furusho} (Nagoya Univ.)\\
\noindent Pierre \textsc{Lochak} (CNRS)\\
\noindent \textsc{Hoang Ngoc Minh} (Univ. Lille)\\
\noindent Leila \textsc{Schneps} (CNRS)\\
\noindent Don \textsc{Zagier} (Coll\`ege de France, MPIM)
}
\end{multicols}

\vspace*{\fill}


\newpage
\singlespaced
\begin{center}
  ABSTRACT\\
\thetitle\\
\vspace{.5in}
  \theauthor\\
  \theadvisor
\end{center}

\noindent
This thesis is a study of algebraic and geometric relations
between multizeta values.  There are many such known sets of relations,
coming from different theories,
which are conjecturally equivalent to each other and which
conjecturally describe all relations on multizeta values.
This thesis was inspired by the conjectures of equivalence of these relations.

To study the algebraic relations, we begin by looking at the double
shuffle Lie algebra associated to multizeta values, which encodes the
double shuffle relations.  In chapter 2 of this
thesis, we prove a result which gives the dimension of the associated
depth-graded pieces of the double shuffle Lie algebra in depths 1 and
2, thus verifying the conjecture that the double shuffle Lie algebra
is isomorphic to the Grothendieck-Teichm\"uller Lie algebra in small
depths.

Another conjecturally equivalent set of relations between
multizeta values comes from their expression as periods on $\Mn$,
stemming originally from the work of Cartier and Kontsevich
(among others).  In chapters 3
and 4, we study these geometric relations.  The results obtained from this study provide some evidence
toward the conjecture that the associated formal period algebra is
isomorphic to the formal zeta value algebra.  The key ingredient in
this study is the top dimensional de Rham cohomology of partially compactified
moduli spaces of genus 0 curves with $n$ marked points,
$H^{n-3}(\Mn^\delta)$.  In order to encode multizeta values in a
formal period algebra, we give an explicit expression for a
basis of $H^{n-3}(\Mn^\delta)$.  The techniques used in this construction
are generalized in chapter 4, in which we explicitly describe the bases of
the cohomology of other partially compactified moduli spaces. 
This thesis concludes with a result which gives a new presentation of $Pic(\M_{0,n})$.
\pagebreak
\begin{center}
  R\'ESUM\'E\\
Valeurs multiz\^eta : alg\`ebres de Lie et p\'eriodes sur $\Mn$\\
\vspace{.5in}
  \theauthor\\
  Leila Schneps, directrice de th\`ese
\end{center}

\noindent
Cette th\`ese est une \'etude des relations alg\'ebriques et g\'eom\'etriques entre valeurs multiz\^eta.  Il y a de nombreux ensembles de telles relations, provenant de th\'eories diff\'erentes. Conjecturalement, ces ensembles sont \'equivalents et d\'ecrivent de plus toutes les relations entre valeurs multiz\^eta.  Cette th\`ese s'inspire de ces conjectures portant sur l'\'equivalence de ces relations.

Afin d'\'etudier les relations alg\'ebriques, on commence par regarder l'alg\`ebre de Lie, $\ds$, qui encode les relations de double m\'elange.  Dans le chapitre 2, on d\'emontre un r\'esultat qui donne la dimension des parties gradu\'ees de $\ds$ associ\'ees \`a sa filtration par profondeur en profondeurs 1 et 2. On d\'emontre donc que $\ds$ est isomorphe a l'alg\`ebre de Lie $\grt$ dans les petites profondeurs.

Un autre ensemble de relations entre multiz\^etas, conjecturalement \'equivalent au syst\`eme de double m\'elange, d\'ecoule de leur expression comme p\'eriodes sur $\Mn$, suivant les m\'ethodes de Cartier et Kontsevich (parmi d'autres).  Dans les chapitres 3 et 4, on \'etudie ces relations g\'eom\'etriques.  Les r\'esultats obtenus sont en accord avec
la conjecture affirmant que l'alg\`ebre formelle des p\'eriodes est isomorphe \`a l'alg\`ebre formelle des multiz\^etas.  L'ingr\'edient principal dans cette \'etude est la cohomologie de de Rham des espaces de modules de courbes en genre 0 avec $n$ points marqu\'es partiellement compactifi\'es, $H^{n-3}(\Mn^\delta)$.  Afin d'encoder les valeurs multiz\^eta dans l'alg\`ebre formelle des p\'eriodes, on donne une expression explicite pour une base de $H^{n-3}(\Mn^\delta)$.  Ces techniques sont g\'en\'eralis\'ees dans le chapitre 4, dans lequel on d\'ecrit explicitement les bases de la cohomologie d'autres espaces de modules partiellement compactifi\'es.  Dans la derni\`ere partie, on fournit une nouvelle pr\'esentation de $Pic(\M_{0,n})$.
\vspace*{\fill}
\newpage

\vspace*{18cm}
{\it This thesis is dedicated to my father, Bill Higgins, because he understands.}

\newpage

\tableofcontents

\newpage
\pagenumbering{arabic}

\chapter{Introduction}
My research focuses on the study of multizeta values, real
numbers defined by the iterated sums, $$\zeta(k_1,...,k_d) \label{zetadef}=
\sum_{n_1> n_2>\cdots >n_d>0} \frac{1}{n_1^{k_1}n_2^{k_2}\cdots
  n_d^{k_d}}\ \ k_i\in \Z, \ (k_1\geq 2).$$\index{$\zeta(k_1,...,k_d)$}  Multizeta values are objects
meriting much attention of late, and the multizeta value has acquired
many nicknames in the process.  We will refer to a multizeta value as
a multiple zeta value, a multizeta, an $MZV$, a zeta value, or simply a zeta. 

There is a myriad of conjectures and important recent results about
multizeta values in various fields.  One well-known number theoretic
question is   ``Are multizeta values transcendental numbers?''
Euler proved that $\zeta(2n)$ is a rational multiple of $\pi^{2n}$ and more
recently R. Ap\'ery and T. Rivoal showed that certain $\zeta(2n+1)$
are irrational.  The conjecture underlying
my research interests however is an even larger question, the
understanding of which would prove the transcendence conjecture.
The conjecture arises from number theoretic and geometric identities on
multizeta values.  

If one multiplies two multizeta values, one obtains a sum of
multizeta values according to the double shuffle multiplication laws,
shuffle and stuffle. The shuffle multiplication law comes from
multizetas viewed as periods on the moduli space of genus 0 curves,
while stuffle multiplication (already known to Euler) comes from the
number theoretic expression of
multizetas.  We may then endow the vector space over $\Q$ of multizetas
with multiplication given by shuffle, and therefore multizetas form an
algebra with a set of quadratic relations given
by stuffle.  
We denote by $\MZV$\index{$\MZV$}, the algebra generated by $\Q$ and multizeta values,
and call the set of  
multiplication relations the double shuffle (there is universal
convention to consider $1=\zeta(\emptyset)$, so that $\Q\subset \MZV$).

\begin{defn}\label{depth} The {\bf depth} of $\zeta(k_1,...,k_d)$ is $d$ and
  its {\bf weight} 
  is $\sum_{i=1}^d k_i$. \end{defn}
Both the stuffle and
shuffle relations preserve the weight of an expression for
multizetas.  I emphasize ``expression'' since two expressions for
a multizeta value may give the same number.  Although the depth of an
expression for a multizeta is easy to understand, it is not an
invariant of a multizeta number.  One example of this was known
already to
Euler, who proved that $\zeta(3)=\zeta(2,1)$.  The study of the weight
and depth of multizetas leads to the main algebraic conjecture on
multizetas.
\begin{conj}\label{mainconj}
A generating system of relations over $\Q$ between multizeta values is
essentially given by the
shuffle and
stuffle relations (for complete detail see definition \ref{FZ}).  In
particular, there are no linear relations between
multizetas of different weight, hence $\MZV$ forms a graded algebra.
\end{conj}

This thesis is not an attempt to make progress toward this conjecture,
which is extremely
difficult because of the analytic nature of the transcendence problem.
Rather, this thesis is an attempt to better understand its
implications, in particular the combinatorial identities that arise
from the known relations on multizeta values.
Hence, we may define
graded algebras that satisfy major families relations on multizetas
and see what we can learn about multizetas from these algebras.  This
thesis is a study and comparison of two such algebras, the double
shuffle Lie
algebra and the period algebra of formal cell numbers.

In chapter 1, I give the main objects and state well-known theorems
on which this study of multizetas is based.  In this introduction,
I will define the algebras and Lie algebras associated to multizeta
values and explain how they are related and what the main conjectures
are.  These conjectures are there to 
provide the reader with a flavor of
the questions that inspired this thesis.  The conjectures presented in
chapter 1
may be summarized as saying that all of the maps between the algebras (in upper case calligraphic font) and the maps between Lie algebras (in lower case 
Fraktur font) in the following
commutative diagram are isomorphisms.  Those shown are known to exist,
except for the dotted arrows, whose conjectural definitions are known,
but which are not proved to be well-defined, much less isomorphisms. 
\[\xymatrix{
 & \mathfrak{nfz} \ar[r] & \widetilde{\nfz} \ar@{=}[d]\\
{\mathcal FC} \ar@{->>}[d] \ar@{<.>}[r] & \FZ \ar@{->>}[u] \ar@{->>}[d]
& \ds^{\vee} 
   \ar@{<.>}[r] \ar@{->>}[d] & \grt^{\vee} \ar@{->>}[dl] \\
{\mathcal C} \ar@{=}[r] & \MZV \ar@{->>}[r] & \nz.
}\]
Chapter 2 presents some evidence toward the well-known conjecture
that the double shuffle Lie algebra, $\ds$, is isomorphic to the
Grothendieck-Teichm\"uller Lie algebra, $\grt$.  In chapter 2, I
calculate the dimensions of the first two associated depth-graded parts of
$\ds$, confirming the conjecture that $\ds\simeq \grt$ in small depths
since the analogous
dimensions were computed by Ihara \cite{Ih2} for $\grt$.  (I note
here that I found these dimensions in 2005.  The result I present has
since been published in \cite{IKZ} using different methods.)

Chapters 3 and 4
deal with the
algebra of periods and the study of the cohomology of $\Mn^\delta$,
the partially 
compactified moduli space of genus 0 curves with $n$ marked points
consisting of $\Mn$, with the boundary divisors that bound the
standard associahedron adjoined.
Chapter 3 is an intact article which is
joint work with F. Brown and L. Schneps.  The main result presented in
this article is the presentation of an explicit basis for the top dimensional
de Rham cohomology of the partially compactified moduli space,
$H^{n-3}(\Mn^\delta)$.  This presentation allows us to compute
the dimension of the cohomology using a recursive formula.  In this
article, we construct the algebra of periods on $\Mn^\delta$,
denoted $\mathcal{C}$, which 
is isomorphic to the algebra of multizeta values, $\MZV$ \cite{Br}.
 This
leads us to define an algebra
of formal cell numbers, $\mathcal{FC}$, which encodes the known combinatorial
relations coming from geometry on certain special generating periods
called cell numbers.  Since multizeta values are cell numbers, there
is reason to believe that these geometric combinatorial relations
describe all relations on multizeta values.

In chapter 4, I generalize the results of chapter 3 to calculate the top dimensional de
Rham cohomology of some more general partial compactifications of $\Mn$ and
give explicit dimensions for each cohomology.  The investigation into
the description of the cohomology also led to finding a new presentation of
the Picard group, $Pic(\M_{0,n})$.

Although the chapters have disparate
titles, they are intimately linked by the search for the connection
between sets of relations and properties coming from the
different geometric and number theoretic expressions for multizeta
values.  I outline these here because the goal  
of this thesis is to present results that emphasize the common properties
of these different points of view. 
The double shuffle Lie algebra is an
object which encodes both the number theoretic expression of multizeta
values and the geometric expression of multizeta values as periods.
In fact, it was recently shown with a clever manipulation of Cartier
\cite{Br},
that the number theoretic identity, stuffle, may be seen as a period identity
on multizeta values.  This observation (among others) led
naturally from our study of the Lie algebra of multizetas to the
period algebra and we believe
that all of the identities on multizeta values may be encoded as
the identities which we derived from the geometry of moduli spaces.
In this way, the study of
the double shuffle Lie algebra is closely related to that of the
period algebra.

\section{Properties of multizeta values}

In this section, I give the basic properties of multizeta values and
the definitions from which this thesis is built.

To a sequence of positive integers, $(k_1,...,k_d)$, we associate
a sequence in the noncommutative variables, $x$ and $y$, by associating
every $k_i$ to the monomial $x^{k_i-1}y$.  By concatenation,
we then associate the sequence to the monomial,
\begin{equation}\label{inttoxy}
(k_1,...,k_d) \sim x^{k_1-1}y\cdots x^{k_d-1}y
\end{equation}
whose degree is the same as the weight
of the sequence of integers.  Then we denote $\zeta(k_1,...k_d)$ by
$\zeta(x^{k_1-1}y\cdots x^{k_d-1}y)$.  

\vspace{5mm}
\begin{defn}\label{convergentdef}\index{Convergent sequence} Let $\underline{k}= (k_1,...,k_d)$ be
  a sequence of positive integers and
  $\underline{\epsilon}=(\epsilon_1,...,\epsilon_n)$ be corresponding
  a monomial in
  $x$ and $y$.  The sequences are {\bf convergent} if $k_1\geq 2$,
  $\epsilon_1=0$ and $\epsilon_n=1$. \end{defn}

The $x,y$ notation for a multizeta comes from its expression as 
an iterated integral.  The
following proposition is due to Kontsevich and is found in many texts about
multiple zeta values, for example \cite{Dr} and \cite{IKZ}.

\begin{prop} Let $(k_1,k_2,...,k_r)$ be a sequence of positive
  integers and $\underline{\omega}$ the corresponding word in $x$
  and $y$.  We associate to $\underline{\omega}$ a tuple of 0's and 1's,
  $\underline{\epsilon}$, by replacing each $x$ and $y$ in
  $\underline{\omega}$ by 0 and 1 respectively to obtain the sequence,
$$(\epsilon_n,\ldots,\epsilon_1)=(0,\ldots,0,1,0,...,0,1),$$
so that $r$ is the number of $1$'s in the tuple, and $\epsilon_1=1$.

Then for an indeterminate $z$, we have
\begin{equation}\label{zetaper}\sum_{n_1>\cdots>n_r>0}
  {{z^{n_1}}\over{n_1^{k_1}\cdots n_r^{k_r}}}=
(-1)^r\int_0^z {{dt_n}\over{t_n-\epsilon_n}}
\int_0^{t_n}{{dt_{n-1}}\over{t_{n-1}-\epsilon_{n-1}}}\cdots
\int_0^{t_2}{{dt_1}\over{t_1-\epsilon_1}}.\end{equation}

When $k_1>1$, by setting $z=1$, we have
\begin{align*} \zeta(k_1,\ldots,k_r)& =(-1)^r\int_0^1
  {{dt_n}\over{t_n-\epsilon_n}}
\int_0^{t_n}{{dt_{n-1}}\over{t_{n-1}-\epsilon_{n-1}}}\cdots
\int_0^{t_2}{{dt_1}\over{t_1-\epsilon_1}}\\
&= (-1)^r\int_{0<t_1<t_2<\cdots <t_n<1} \frac{dt_1dt_2\cdots dt_n}{
(t_n-\epsilon_n)\cdots (t_1-\epsilon_1)}.\end{align*}
\end{prop}

\begin{proof}
 
We prove \eqref{zetaper} by induction on $n$.  For the base case $n=1$,
$k_1=k_r=1$, so $\epsilon_1=1$.  We have then that 
\begin{align*}
\int_0^z \frac{dt_1}{(1-t_1)} &= \int_0^z (\sum_{n=0}^\infty
t_1^n)dt_1\\ &= \sum_{n=1}^\infty \frac{z^{n}}{n}.\end{align*}

We now check the two base cases where $n=2$, 
namely the tuple $(0,1)$ and the tuple $(1,1)$, by repeated use of
the series expansion $1/(1-t)=\sum_{i\ge 0} t^i$.  For the case
$(1,1)$ we have
\begin{equation*} \int_0^z
  {{dt_2} \over{1-t_2}} \int_0^{t_2} {{dt_1}\over{1-t_1}} =
  \sum_{n_1>n_2 \ge 1}{{z^{n_1}}\over{n_1n_2}}. \end{equation*}
And for (0,1) we have,
\begin{equation*}
\int_0^z \frac{dt_2}{t_2} \int_0^{t_2}\frac{dt_1}{1-t_1} =
\sum_{n=1}^\infty \frac{z^n}{n^2}.
\end{equation*}

Now assume \eqref{zetaper} true for tuples of length $n-1$ and consider a tuple
$(\epsilon_{n},\ldots,\epsilon_1)$.  Assume first that $\epsilon_n=0$.
Then by the induction hypothesis the right hand side of \eqref{zetaper} becomes
\begin{align*} \int_0^z {{dt_n}\over{t_n}}\sum_{n_1>\cdots>n_r>0}
  {{t_n^{n_1}}\over
{n_1^{k_1-1}\cdots n_r^{k_r}}}
&=\sum_{n_1>\cdots>n_r>0}{{1}\over{n_1^{k_1-1}\cdots n_r^{k_r}}}
\int_0^z t_n^{n_1-1} dt_n\cr
&=\sum_{n_1>\cdots>n_r>0}{{z^{n_1}}\over{n_1^{k_1}\cdots
    n_r^{k_r}}}. \end{align*} 
To finish we only need to deal with the case where $\epsilon_n=1$.
\begin{align*}\int_0^z {{dt_n}\over{1-t_n}}\sum_{n_2>\cdots>n_r>0} {{t_n^{n_2}}
\over {n_2^{k_2}\cdots n_r^{k_r}}}
&=\sum_{n_2>\cdots>n_r>0}{{1}\over{n_2^{k_2}\cdots n_r^{k_r}}}
\int_0^z \sum_{i\ge 0}t_n^{i+n_2} dt_n\cr
&=\sum_{i\ge 0}\sum_{n_2>\cdots>n_r>0}{{1}\over{n_2^{k_2}\cdots n_r^{k_r}}}
\int_0^z t_n^{i+n_2} dt_n\cr
&=\sum_{i\ge 0}\sum_{n_2>\cdots>n_r>0}{{1}\over{n_2^{k_2}\cdots n_r^{k_r}}}
{{z^{i+n_2+1}}\over{(i+n_2+1)}} \cr
&=\sum_{n_1>n_2>\cdots>n_r>0}{{z^{n_1}}\over{n_1^{k_1}\cdots
    n_r^{k_r}}}, \end{align*}
where in the last line we set $n_1=i+n_2+1$ and by the hypothesis that
$\epsilon_1=1$, $k_1=1$.
 
\end{proof}

\section{Quadratic relations on multizeta values}

If one multiplies two multizeta values, one obtains the sum of
multizeta values, but this expression is not unique.  One expression
was known already to Euler.  In order to give these identities, we
shall first define
the shuffle and the stuffle products on sequences.
We let $\cdot$\label{cdot}\index{Concatenation product, $\cdot$} denote the concatenation product of sequences.

\vspace{.5cm}
\begin{defn}\label{stuffledefn}\index{$*$, $st(\underline{a},\underline{b})$} For any two sequences of positive integers,
      $\underline{a}$, $\underline{b}$ the {\bf stuffle product}
      of $\underline{a}$ and $\underline{b}$, denoted
      $st(\underline{a},\underline{b})$ or $\underline{a} *
      \underline{b}$,
      is the formal sum obtained by the recursion:
\begin{enumerate}
\item $st(\underline{a},\emptyset)= st(\emptyset,\underline{a})=
  \underline{a}$,
\item $st(a_0\cdot \underline{a},b_0 \cdot \underline{b}) = a_0\cdot
  st(\underline{a}, b_0\cdot \underline{b}) + b_0 \cdot st(a_0
  \cdot \underline{a},\underline{b}) + (a_0 + b_0) \cdot
  st(\underline{a} , \underline{b})$.
\end{enumerate}

\end{defn}

Morally, the stuffle product is obtained by taking permutations of
$\underline{a}\cdot\underline{b}$ such that the orders
of both sequences are preserved and then adding 
adjacent pairs of elements, one from $\underline{a}$ and one from
$\underline{b}$, in all
possible ways, in other words ``stuffing'' the elements of $\underline{a}$ and
$\underline{b}$ into the same slot.

\begin{ex}
The stuffle product $(2,1)*(3)= (2,1,3) +(2,3,1) + (3,2,1)+(2,4)+(5,1)$.
\end{ex}

\begin{defn}\label{shdefintro}\index{$\sha$, $sh(\omega_1,\omega_2)$}  Let
  $\underline{\alpha}=(\alpha_1,\dots ,\alpha_k)$ and $\underline{\beta} = 
    (\beta_1,\dots, \beta_l)$ be two sequences.  The {\bf shuffle product}
    of $\underline{\alpha}$ 
    and $\underline{\beta}$, denoted by
    $sh(\underline{\alpha},\underline{\beta})$, or 
    $\underline{\alpha}\sh \underline{\beta}$, is the formal sum
    obtained by the recursive procedure:
\begin{enumerate}
\item $sh(\underline{\alpha},\emptyset) = sh(\emptyset,
  \underline{\alpha}) = \underline{\alpha}$,
\item $sh(a_0\cdot \alpha, b_0\cdot \beta) =a_0\cdot sh(\alpha,
  b_0\cdot \beta) + b_0\cdot sh(a_0\cdot \alpha, \beta)$.
\end{enumerate}
\end{defn}

We will often rely on an equivalent definition of the shuffle product,
$$sh(\alpha, \beta) = \sum_{\sigma} \sigma(\alpha \cdot \beta),$$
where $\sigma \in \Sym_{k+l}$ runs over all permutations which preserve the
orders of $\alpha$ and $\beta$.  For ease of notation, we
write $\underline{\gamma}\in sh(\underline{\alpha},
\underline{\beta})$ to mean that $\underline{\gamma}$ is a term in
the sum $ sh(\underline{\alpha},
\underline{\beta})$.

\begin{ex}
$(0,1)\sh (0,1)= 2(0,1,0,1) + 4(0,0,1,1).$
\end{ex}

Both of the combinatorial products above are commutative.  They were
 defined here in order to present the
 following two classical expressions for the product of multizetas.

\begin{prop}[Euler]\label{stuffle} Let $\underline{a}_1$ and
  $\underline{a}_2$ be two convergent sequences of positive integers.  Then,
\begin{equation*}\zeta(\underline{a}_1) \zeta( \underline{a}_2)
  =\sum_{\underline{a}\in st(
    \underline{a}_1,\underline{a}_2)
    }\zeta(\underline{a})=\zeta(\underline{a}_1 *
    \underline{a}_2). \end{equation*} \end{prop}

The iterated integral expression in proposition \ref{zetaper} for multizetas
written in the $x,y$ notation
leads to the following alternative expression for the product of
multizetas, also attributed to Kontsevich.

\begin{prop}\label{shuffle}
Let $\underline{\epsilon}_1$ and $\underline{\epsilon}_2$ be two
convergent 
sequences in the variables $x$ and $y$.  Then,
$$\zeta(\underline{\epsilon}_1)
\zeta(\underline{\epsilon}_2) = \sum_{\underline{\epsilon} \in
  \underline{\epsilon}_1\sha \underline{\epsilon}_2}
\zeta(\underline{\epsilon})= \zeta(\underline{\epsilon}_1\sha
\underline{\epsilon}_2) .$$ 
\end{prop}

The shuffle product on multizetas endows the vector space over $\Q$
generated by multizeta values
with the structure of a $\Q$ algebra, while the stuffle product gives
this algebra a set of relations.  The system of relations on multizeta
values given by the shuffle and the stuffle
products is known as the system of double shuffle
relations\label{doubleshufrels}\index{Double shuffle}.  If we restrict
ourselves to double shuffle relations on multizeta values, we do not
obtain what is
conjectured to be a complete set of relations on multizetas.  

An important system of
relations on multizetas comes from regularization of nonconvergent
zeta values, a technique
from physics to define a notion of cancelling divergences. 
Based on early work of Ecalle and Zagier (see \cite{IKZ}), one may extend the double shuffle
relations by allowing identities to be obtained from applying the
double shuffle relations to
divergent sums.  The following important relation coming from regularization, known as {\bf Hoffman's Relation}, is conjectured to complete the system of generating relations, along with double shuffle, on multizeta values.

\begin{prop}[HO]\label{Hoffman}\index{Hoffman's relation}  Let $\underline{k}$ be a convergent
  sequence of positive 
  integers and let $\underline{\omega}$ be its corresponding sequence
  in $x$ and $y$ by the association \eqref{inttoxy}.   Then,
  $$\zeta\bigl( \sum_{\underline{l}\in 
  (1)*\underline{k}} \underline{l} - \sum_{\underline{\lambda}\in
  (y)\sh\underline{\omega}} \underline{\lambda} \bigr) = 0.$$ 
\end{prop}

Note that although Hoffman's relation comes from regularization, it is a relation only on
convergent zeta values, since each sum in
the expression has only one non-convergent term, $(1,\underline{k})$
and $(y,\underline{\omega})$, but these terms are equal and disappear in the
difference.

The propositions \ref{stuffle}, \ref{shuffle} and \ref{Hoffman} are
conjectured to be a generating set of relations on the algebra of
multizeta values, $\MZV$.  This conjecture is one of the precise
formulations of conjecture \ref{mainconj}.  As before, the analytic nature of
this conjecture renders it out of reach with present techniques.
Furthermore, even the algebraic structure of these three sets of
relations is not fully understood.  In order to study these relations,
while avoiding the transcendence problem, we define
the algebra of formal multizetas consisting of symbolic multizeta
values and satisfying only these three sets of relations by definition.

\begin{defn}\label{FZ}\index{$\FZ$, $\zeta^F(w)$} Let $\FZ$ be the formal algebra generated by the symbols,
  $$W=\{ \zeta^{F} (\omega);\hbox{ where }\omega\hbox{ is a convergent word
  in }x,y\},$$
  and containing the symbols $$T=\{ \zeta^{F} (a);\hbox{ where }a\hbox{ is a
  convergent sequence of positive integers}\},$$
with the following relations,

\begin{enumerate}
\item For every $\zeta^{F}(\omega)\in W$ there is a unique
  $\zeta^{F}(a)\in T$ such that $\zeta^{F}(\omega)=\zeta^{F}(a)$ by
  the correspondence \eqref{inttoxy},
\item For all $\zeta^{F}(\omega_1),\zeta^{F}(\omega_2)\in W$,
  $\zeta^{F}(\omega_1)\zeta^{F}(\omega_2) =
  \zeta^{F}(\omega_1\sha \omega_2)$,
\item For all $\zeta^{F}(a),\zeta^{F}(b)\in T$,
  $\zeta^{F}(a)\zeta^{F}(b) =
  \zeta^{F}(a*b)$,
\item For all $\omega\in W$ and $a\in T$ such that
  $\zeta^F(\omega)=\zeta^F(a)$ by relation 1,
$$\zeta^{F} (\sum_{l\in (1)*(a)} l - \sum_{\lambda \in (y)\sha (w)}
  \lambda) =0.$$
\end{enumerate}
The formal multizeta algebra is graded by weight, so we have
$$\FZ=\bigoplus_{n=0}^\infty \FZ_n,$$
where we set $\FZ_0=\Q$, $\FZ_1=0$.  We also write $\FZ_{>0} = \bigoplus_{n=1}^\infty \FZ_n$.
\end{defn}

\begin{defn}\label{NFZdef}\index{$\nfz$, $\mathfrak{z}(w)$}
We let $\nfz$ denote the vector space obtained by quotienting $\FZ$ by products, $\zeta(2)$ and $\Q$:
\begin{equation*}
\nfz:= \FZ/\langle \FZ_{>0}^2 \oplus \FZ_2 \oplus \FZ_0 \rangle.
\end{equation*}
We denote the elements of $\nfz$ by $\mathfrak{z}(w)$ where $w$ is a convergent word, or by $\mathfrak{z}(a)$ where $a$ is a convergent sequence of integers.
\end{defn}

In the following section, we introduce the double shuffle Lie algebra, $\ds$,
and relate it to a Lie coalgebra, $\widetilde{\nfz}$, which is conjecturally isomorphic to
$\nfz$.

\section{The Lie algebras, $\ds$ and $\grt$}

The motivation for studying the double shuffle Lie algebra is its
close relationship to multizeta values, which will be outlined in
detail in the following section.  I will summarize this relationship
to motivate the results that are presented in this section.

\begin{defn}\label{nznumbersdef}\index{$\nz$, $\overline{\zeta}(w)$}
Let $\MZV^2$ be the ideal in $\MZV$ generated by products of multizeta values.
We define the $\Q$-vector space of {\bf new zeta values} to be $\nz = \MZV/\langle \MZV^2 \oplus \Q\cdot \zeta(2) \oplus \Q\rangle. $  We denote an element of $\nz$ by $\overline{\zeta}(k_1,...,k_d)$ or by $\overline{\zeta}(x^{k_1-1}y\cdots x^{k_d-1}y)$ where 
$k_1,...,k_d$ is a convergent sequence of integers.\end{defn}

In G. Racinet's thesis, he constructs a subspace of the power series
algebra, $\Q\langle \langle x,y \rangle \rangle$, that is conjecturally
isomorphic to dual space, ${\nz}^{\vee}$, and proves that this
subspace, $\ds$, is
a Lie algebra for the Poisson bracket.  We call this Lie algebra {\bf the
double shuffle Lie algebra}.  This section is
dedicated to defining the double shuffle Lie algebra.

In this chapter we work in the two noncommutative power series algebras,
$\Qxy$ and $\Q\langle \langle y_i;
1\leq i <\infty \rangle \rangle$\label{qxyqyi}\index{$\Qxy,\ \Qyi$}.  For $f$ a power
series in one of these 
algebras, we denote by $(f|w)$ the
coefficient of the word $w$ in $f$.

We associate an element in 
$\Qxy$ to $\Qyi$ via the linear map, $\pi_{\mathcal{Y}}$, following
\cite{Ec} and \cite{Ra}.  It is closely
linked to the alternative notation for a multizeta in the association
\eqref{inttoxy}. 

\begin{defn}\label{piydefintro}\index{$\piy$}
\begin{align*}
\pi_{\mathcal{Y}} :  \Qxy &\rightarrow \Qyi \\
\widetilde{\pi}_{\mathcal{Y}} (x^{k_1-1}yx^{k_2-1}y \cdots x^{k_n-1}y
x^{k_{n+1}}) & = \begin{cases} 0 & k_{n+1} \neq 0 \\
y_{k_1}y_{k_2}\cdots y_{k_n} & k_{n+1} = 0 \end{cases} \\
\pi_{\mathcal{Y}} (f) & = \widetilde{\pi}_{\mathcal{Y}}(f) +
\sum_{n=2} (f| x^{n-1}y) \frac{(-1)^{n-1}}{n} y_1^n.
\end{align*}
\end{defn}

The polynomial algebras $\Qxy$ and $\Qyi$ may be equipped with the
following coproducts defined on the generators, $x,y,y_i$, and
extended multiplicatively: 
\begin{defn}\label{coprods}\index{$\Delta_{\sh},\ \Delta_*$}
\begin{align*}
\Delta_{\sh}: \Qxy & \rightarrow \Qxy\otimes_{\Q} \Qxy \\
x& \mapsto x\otimes 1 + 1\otimes x \\
y& \mapsto y\otimes 1 + 1\otimes y \\
\Delta_{*}: \Qyi &\rightarrow \Qyi \otimes_{\Q} \Qyi \\
y_i & \mapsto \sum_{n+m=i} y_n\otimes y_m
\end{align*}
\end{defn}

\begin{defn}\label{dsdefintro}\index{$\ds$}
The vector subspace, $\ds\subset \Qxy$, is generated by elements, $f$,
that are primitive for $\Delta_{\sha}$, and such that $\piy(f)$ is
primitive\index{Primitive} for $\Delta_*$: $$\Delta_{\sha}(f) =
f\otimes 1 + 1\otimes f,\ \ 
\Delta_*(\piy(f)) = \pi_{\mathcal{Y}}(f) \otimes 1 + 1\otimes \pi_{\mathcal{Y}}(f).$$
\end{defn}

\begin{defn}\label{Pbdef}\index{Poisson bracket, $\{f,g\}$}
The Poisson bracket on elements of $\Qxy$ is the Lie bracket given by
$$\{f,g\} = [f,g] +D_f(g) - D_g(f)$$ where $[f,g]=fg-gf$ and the
$D_f$ are derivations
defined recursively by $D_f(x) = 0$, $D_f(y)=[y,f]$ and such that
$D_f(gh) =D_f(g)h + gD_f(h)$.
\end{defn}

The following result is one of the key ingredients in the
understanding of the formal multiple zeta algebra (see section 4 for details).

\begin{thm}\cite{Ra}
The double shuffle vector space, $\ds$, forms a Lie algebra for the
Poisson bracket.
\end{thm}

In the next section, we define a vector space $\widetilde{\nfz}$ and prove that
it is isomorphic to the dual space $\ds^\vee$ of $\ds$ (thereby proving in
particular that $\widetilde{\nfz}$ is a Lie coalgebra).  The importance of
the double shuffle Lie algebra in relation to multiple zeta values lies in the
fact that the surjection from $\widetilde{\nfz}$ (identified with $\ds^\vee$)
to $\nz$ given in the following proposition is conjectured to be an isomorphism.

\begin{prop}\label{newzetathm}
We have a surjective, $\Q$-linear map from $\ds^{\vee}$
to $\nz$, $$\ds^{\vee} \twoheadrightarrow \nz.$$
\end{prop}

The proof will be given in section \ref{grr}.

The relationship between $\ds$, multizetas, $\grt$ and mixed Tate
motives, which will be outlined in the remainder of this introduction,
is what led to our interest in studying multizeta values.

The double shuffle Lie algebra is graded by weight and each graded
piece can be endowed with a
filtration by depth.  By definition \ref{depth}, the
depth of a monomial in $x$ and $y$ is the number of times $y$ appears.
We denote the depth
filtration in the weight $n$ part by $$\ds= \label{Fndef}F_n^1\ds \supset F_n^2\ds
\supset 
\cdots \supset F_n^{n-1}\ds\supset F_n^n(\ds)=0,$$ where $F_n^i$ are generated by weight $n$
polynomials whose terms all have depth greater than or equal to $i$.
The Lie algebra $\ds$ is not graded by
depth, since
stuffle multiplication does not preserve depth.  However, we may
define an associated depth-graded object,
$$\bigoplus_{i\geq 1} F^i_n\ds /F^{i+1}_n\ds.$$
The dimensions of the vector spaces, $F^i_n\ds /F^{i+1}_n\ds$, are an
essential feature of the structure of $\ds$.
This leads us
to the main result in chapter 2 of this thesis:

\begin{thm}\label{dep12}
The dimensions of the associated $i$th depth-graded parts of $\ds$ for
$i=0,1$ are 
\begin{equation}\begin{split}
dim(F^1_n\ds / F^2_n\ds) &=\begin{cases} 1 & n\hbox{ odd}\\
0 & n\hbox{
  even}\end{cases}\\
dim(F^2_n\ds / F^3_n\ds) &=\begin{cases} 0 & n\hbox{ odd}\\
\lfloor \frac{n-2}{6} \rfloor & n\
  even. \end{cases}\end{split}\end{equation}
\end{thm}

(This calculation was done before we knew that this result has been known by Zagier who published it in 1993 \cite{Za1} and it was restated in \cite{IKZ} and \cite{GKZ}.)

Here, I will explain the motivation for theorem \ref{dep12}.
Y. Ihara defined the Lie algebra, $\grt$, which is related to the Lie
algebra of
the braid group on 5 strands, $\mathfrak{P}_{5}$ \cite{Ih1}.

\begin{defn}
Let $\mathfrak{P}_5$ be the Lie algebra with generators $x_{i}, 1\leq
i\leq 5$ with the following relations:
\begin{enumerate}
\item $[x_i, x_j] =0$ whenever $1<|i-j|<4$,
\item $[x_1,x_2] +[x_2,x_3] +[x_3,x_4] +[x_4,x_5] +[x_5,x_1]=0$.
\end{enumerate}
 \end{defn}

\begin{defn}\label{grtdef}\index{$\grt$}
The Grothendieck-Teichm\"uller Lie algebra, $\grt$, is the subspace of
polynomials, $\Q\oplus \langle f\in [\Lxy,\Lxy]\rangle$, such that
the generators, $f$, satisfy
the following 3 sets of
relations:
\begin{enumerate}
\item $f(x,y) + f(y,x) =0$,
\item $f(x,y) +f(y,z) +f(z,x) =0$, where $z=-x-y$,
\item $\sum_{i\in \Z/5} f(x_i,x_{i+1}) =0$, for $x_i\in \mathfrak{P}_5$.
\end{enumerate}
This subspace is a Lie algebra for the Poisson bracket. 
\end{defn}

One conjectures, and computations have verified in low weight, that $\ds
\simeq \grt$ and that the isomorphism is given simply by
\begin{align}\label{grtdsmap}
f:\ds & \rightarrow \grt \notag\\
f(x,y) &\mapsto f(x,-y).
\end{align}
This conjecture remains remarkably elusive, although Ecalle claims to
have shown that elements of $\ds$ satisfy the first relation of definition
\ref{grtdef} and an unpublished and incomplete preprint of Deligne and
Terasoma claims to have proven that map \eqref{grtdsmap} gives an
injection $\grt \hookrightarrow \ds$. 

The algebras, $\ds$ and $\grt$, encode two distinct, yet conjecturally
equivalent, sets of relations on
multizeta values.  We have following theorem, due
to Furusho \cite{Fu}, based on properties of the Drinfel'd associator,
$\Phi_{KZ}$ (see chapter 2).  This theorem, analogous to proposition
\ref{newzetathm}, underlines the relationship
between $\grt$ and $\ds$.

\begin{thm}\label{Dr}
Let $\grt^{\vee}$ be the dual vector space to
$\grt$.  Then there exists a canonical, surjective $\Q$ linear map,
$$\Psi_{DR}: \grt^\vee \twoheadrightarrow \nz.$$
\end{thm}

In \cite{Ih2}, Y. Ihara finds the following dimensions of the
depth-graded pieces of $\grt$:

\begin{thm}\label{grtdims}
The dimensions of the $i$th depth-graded parts of $\grt$ for $i=0,1$ are
\begin{equation}\begin{split}
dim(F^1_n\grt / F^2_n\grt) &=\begin{cases} 1 & n\hbox{ odd}\\
0 & n\hbox{
  even}\end{cases}\\
dim(F^2_n\grt / F^3_n\grt) &=\begin{cases} 0 & n\hbox{ odd}\\
\lfloor \frac{n-2}{6} \rfloor & n\
  even. \end{cases}\end{split}\end{equation}
\end{thm}

{\it Remark}:  Theorem \ref{dep12} was proved in 2005.  Since then, a
more general result has been published in \cite{IKZ}, where it is also
shown that $dim(F^i_n\ds / F^{i+1}_n\ds)=0$ whenever $n$ and $i$ have
different parity.  They also state without proof that for odd $n$,
$$dim(F_n^3\ds /F_n^4\ds)\geq \lfloor \frac{(n-3)^2-1}{48} \rfloor,$$
and they conjecture that $\geq$ is an equality, a conjecture which has
apparently been proven by Goncharov \cite{Go1}.
Our interpretation
of this result is as follows.  One conjectures that $\ds$ is a free
Lie algebra with one depth 1 generator in each odd weight \cite{Ih2}. 
We then believe that $F^3_n\ds / F^4_n\ds$
is generated by images of 
$\{\{f_i,f_j\},f_k\}$ mod $F^4_{n}\ds$, where $f_i,f_j,f_k$ are depth
one elements, $f_i, f_j, f_k\in \ds$ are homogeneous of odd weight
and 
$i,j,k\geq 3$.   Furthermore, we believe that the only
linear relations between these depth 3 elements 
come from bracketing a relation in depth 2 by a depth 1 element.
Indeed, by counting the number of partitions, $n=i+j+k$, and
subtracting off the number of depth 2 relations from theorem
\ref{dep12}, we do indeed obtain the expected dimension $\lfloor
\frac{(n-3)^2-1}{48}\rfloor$.

The situation in depth 4 (and higher) is much more complicated due to
the presence of other generators that are not just Poisson brackets of
depth 1 elements.  For example, it is known by computation that the element,
$\{f_3,f_9\} - 3\{f_5,f_7\}$ is an element of depth 4, where $f_i$ for
$i=3,5,7,9$ is the unique depth 1 element in weight $i$ such that
$(f_i | x^{i-1}y)=1$.

\section{The double shuffle Lie algebra and multizetas}\label{grr}

In this section, we define a vector space, $\widetilde{\nfz}$ which by work of Zagier, Ecalle, Le and Murakami, and finally Furusho, is known to surject onto $\nz$.  We prove that $\nfz$ is isomorphic to the dual of $\ds$ defined in the previous section and conclude that $\widetilde{\nfz}$ is a Lie coalgebra.  The cobracket on $\widetilde{\nfz}$ has been explicitly computed by Goncharov \cite{Go}.  Although none of the results in this section are new, they provide a framework for our work on $\ds$ in chapter 2.

\begin{defn}\label{nfzdef}\index{$\widetilde{\nfz}$, $\z^{\sh}(w),\ \z^{*}(v)$} The Lie coalgebra, $\widetilde{\nfz}$, is
  the $\Q$ vector 
  space generated by symbols $\z^{\sha}(w)$ for all monomials $w$
  in $\Qxy$ and symbols $\z^{*}(v)$ for all monomials $v\in \Qxy y$ (power
  series whose terms end in $y$) modulo
  the following relations:  
\begin{enumerate}
\item $\mathfrak{z}^{\sha}(w_1\sha
  w_2) =0$,
\item $\mathfrak{z}^{*}(v_1 *
  v_2) =0$,
\item $\z^{\sha}(1) = \z^{\sha}(y)=\z^{\sha}(x)=\z^{\sha}(xy)=0$,
\item If $w=v$ is a word ending in $y$ but not a power of $y$, then
  $\mathfrak{z}^*(v)=\mathfrak{z}^{\sha}(w)$, and
$\z^*(y^n)=\frac{(-1)^{n-1}}{n}
  \z^{\sha}(x^{n-1}y)$.  
\end{enumerate}
\end{defn}
\begin{prop}\label{dsvee}
$$\widetilde{\nfz}\simeq \ds^{\vee}.$$
\end{prop}

\begin{proof}
Let $f\in \ds\subset \Qxy$.  The relations in $\ds^\vee$ are given
by the duals of the relations in $\ds$, which are given by
$\ds_1=\ds_2=0$ and relations (\ref{thing1}) and (\ref{thing2}) below,
with 
\begin{equation}\label{corr}\piy(f) =\sum_{v\in \Qxy\cdot y} (f|v)v + \sum_{n\geq 1}
\frac{(-1)^{n-1} }{n} (f|x^{n-1}y) y^n.\end{equation}

Let an element $\z^{\sha}(w)\in \ds^{\vee}$, be
identified with the linear map, $\ds \rightarrow \Q$ given by
$\z^{\sha}(w)(f) = (f|w)$ for all $f\in \ds$. 
Let $\z^{*}(v)\in \ds^{\vee}$ be
identified with the linear map, $\ds \rightarrow \Q$ given by
$\z^{*}(v)(f) = (\piy(f)|v)$ for all $f\in \ds$. Note that in this proof,
the symbols $\z^{\sha}$ and $\z^*$ refer to elements of $\ds^\vee$, not
to elements of $\widetilde{\nfz}$, and the proof shows that they are equal.
As usual, $w$ always
stands for an arbitrary word in $x$ and $y$, and $v$ for a word
ending in $y$.  

We will show that these linear maps, $\z^{\sha}(w)
\in \ds^\vee$, satisfy the defining relations 1-4 of $\widetilde{\nfz}$
and no others.  Indeed, the relations between the linear maps 
$\z^{\sha}(w)$ are exactly the duals of the relations in $\ds$.
Let us compute the dual relations of each of the four 
relations in $\ds$.
   
First, we know that all $f\in \ds$ are primitive for
$\Delta_{\sha}$ which is equivalent to the condition:
\begin{equation}\label{thing1}\sum_{w\in w_1\sha w_2} (f|w) =
0,\end{equation} and hence $\z^{\sha}(w_1\sha w_2) =0$.  This is
defining relation 1 of $\widetilde{\nfz}$.

Next, we know that since for all $f\in \ds$, $\piy(f)$ is primitive for
$\Delta_{*}$, which is equivalent to the condition:
\begin{equation}\label{thing2}\sum_{v\in v_1* v_2} (\piy(f)|v) =
0,\end{equation} and hence $\z^{*}(v_1* v_2) =0$.  This is defining
relation 2 of $\widetilde{\nfz}$.  

Finally, we know that $\ds_1=\ds_2=0$.  This
immediately implies relation 3 of the definition of $\widetilde{\nfz}$.
Notice that $ny^n=(y)\sha (y^{n-1})$, so that for
$f\in \ds$, $(f|y^n)=0$.  It follows immediately
that $\z^{\sha}(y^n)=0$ for all $n\geq 1$. Similarly, $\z^{\sha}(x^n)=0$
for all $n\ge 1$.

The last relation in $\ds$ is the defining formula (\ref{corr}).
Therefore the coefficients of
any word must be the same on both sides.  If $v$ is a word ending in $x$,
this coefficient is $0$ on both sides.  If $v$ is a word ending in $y$ but
not a power of $y$, the equality of the coefficients implies that
$\z^*(v)=\z^{\sha}(v)$, which is the first part of defining relation 4
of $\widetilde{\nfz}$.  Finally,  if $v$ is a power of $y$, the equality
of the coefficients shows that $$\z^*(y^n)=\z^{\sha}(y^n)+
\frac{(-1)^{n-1} }{n} \z^{\sha}(x^{n-1}y)$$ since $\z^{\sha}(y^n)=0$
as we showed above.  

We have shown that the set of relations of $\ds^{\vee}$ is equal to
the set of relations from
definition \ref{nfzdef}, thus we have an isomorphism 
$\ds^{\vee}\simeq \widetilde{\nfz}$.

\end{proof}

Now we can prove proposition \ref{newzetathm} from the previous section.

\vspace{.2cm}
{\it Sketch of proof of proposition \ref{newzetathm}}.
The proof of this proposition relies on proposition \ref{dsvee} proving that
$\ds^\vee$ is isomorphic to the Lie coalgebra $\widetilde{\nfz}$ defined in
\ref{nfzdef}.  The conclusion then follows from the regularization
formula given by Furusho \cite{Fu} (proposition 3.2.3), expressing non-convergent
symbols $\widetilde{\nfz}$ as explicit linear combinations of convergent
symbols, thus giving an obvious surjection from $\widetilde{\nfz}$ to $\nz$
by mapping convergent symbols to the corresponding zeta values.\hfill{$\square$}

Another proof of this proposition can be obtained by directly adapting
Drinfel'd's and Furusho's proof of theorem \ref{Dr} \cite{Fu}.

Now we are in a position to translate theorem \ref{dep12} into the
language of multizetas.  As explained above, 
$\widetilde{\nfz}$ surjects onto
$\nz$.  Theorem \ref{dep12} implies that every
depth 2 new zeta value in $\nz$ of odd weight, $\overline{\zeta}(a,b) $
where $a+b$ is odd,
is equal to a rational multiple of the depth 1 new zeta value,
$\overline{\zeta}(a+b)$.

The proof of theorem \ref{dep12} yields as a corollary the
following formula for the coefficient of $\overline{\zeta}(i,j)$ in terms of
$\overline{\zeta}(i+j)$ in $\nz$, which is actually the simplification of a result known to Euler, who gave the complete expression for $\zeta(i,j)$ in $\MZV$.

\begin{cor}\label{zetaab} 
Assume that $i+j$ is odd, $i,j\ge 2$.  Then,
$$\overline{\zeta}(i,j)= \frac{(-1)^{j-1}{i+j \choose j} -1}{2}\overline{\zeta}(i+j).$$
\end{cor}

\section{The moduli space of genus 0 curves, $\Mn$}

Chapters 3 and 4 of this thesis are a study of multizeta values as
periods on moduli space via the top dimensional de Rham cohomology,
$H^{n-3}(\Mn)$\label{cohomnot} (we drop the subscript, $_{DR}$ for
``de Rham'', 
for the rest of the
text).  We begin this section by recalling the 
useful notations and properties of moduli space.

\vspace{.5cm}
\begin{defn}\label{Mndef}\index{$\Mn$} The moduli space of genus 0 curves over $\C$, $\Mn$,
  is the space whose points are isomorphism classes of Riemann spheres
  with $n$ distinct, ordered 
  marked points modulo the action of $\PSL_2(\C)$ on the points.
\end{defn}

The action of $\PSL_2$ is triply transitive, so we may denote a point
in $\Mn$ by $\overline{(z_1,...,z_n)}$ or by a well-chosen
representative in its equivalence class, $(0, t_1,...,t_\ell,
1,\infty)$, $\ell=n-3$\label{elldef}\index{$\ell$}.  In this way, we have the isomorphism,
\begin{equation}\label{MninP} \Mn \simeq (\Pro^1\setminus
  \{0,1,\infty\})^{\ell} \setminus \Delta,\end{equation}
where $\Delta$\label{fatdelta}\index{Fat diagonal, $\Delta$} denotes the ``fat'' diagonal, $\Delta =
\{ t_{i_1} = t_{i_j} ; \hbox{ for all distinct } i_k, 1\leq k\leq j\}$.

The moduli space, $\Mn$ is not compact.  A stable compactification,
$\M_{0,n}$\label{Mnbar}\index{$\M_{0,n}$} was
defined by Deligne and Mumford \cite{DM}.  Adding boundary components
to $\Mn$ corresponds to adding stable curves to $\Mn$.  These are genus 0 Riemann surfaces with nodes, such that each component has at least 3 marked or singular points.  A visual interpretation of a point on the boundary
$\M_{0,n} \setminus \Mn$ is given in figure 1, where the simple
closed loop on the left has been pinched to a geodesic of length 0 on
the right.

\begin{center}
 \scalebox{.8}{\input{pinchintro.pstex_t}}
\label{pinch}
\end{center}

The boundary divisors of $\M_{0,n}$ are closed, irreducible
codimension 1 subvarieties in $\M_{0,n}\setminus \Mn$.  (Many authors use the term boundary divisor to denote $\M_{0,n} \setminus \Mn$, whereas we use the term for the irreducible components of $\M_{0,n} \setminus \Mn$.)  In the
association given in \eqref{MninP}, they correspond to blowups of the
regions in $\Delta$.  We sometimes denote a boundary divisor by an
equation, $t_{i_1} =\cdots = t_{i_j}$, which is understood to be the
blowup in $\M_{0,n}$ of that region in $(\Pro^1)^{n-3}$.

The boundary divisors may be combinatorially enumerated
by specifying 
a partition of $S=\{0,t_1,...,t_\ell,1,\infty\}$ into two subsets, $A$
and $S\setminus A$, with
$2\leq |A| \leq n-2$. This is because any simple closed loop on the sphere
with $n$ marked points partitions the points of $S$ into two subsets
as in figure 1.  We may alternatively denote
by $d_A$\label{d_Adef}\index{$d_A$}, the boundary divisor in which the simple
closed loop pinches the subset 
$A\subset S$, hence $d_A=d_{S\setminus A}$.  During sections of this
thesis where no confusion may arise, we may simply denote the boundary
divisor by the set $A$.

\begin{defn}\label{Mnreal} We denote by $\Mn(\R)$ the space of points,
  $\{(0,t_1,...,t_\ell , 1,\infty) ; t_i \in \R\}$.
\end{defn}

While $\Mn$ is a connected manifold, $\Mn(\R)$ is not
connected.  Each connected component in $\Mn(\R)$ can be completely described by the real ordering of its marked points,
$t_{i_1}< \cdots <0< \cdots <1< \cdots <t_{i_{n-3}}$. 

\begin{defn}\label{celldef}\index{Cell, $(z_{i_1},...,z_{i_n})$} A connected component of $\Mn(\R)$ is called a
  {\bf cell}.  The cells in $\Mn(\R)$ are
  also called associahedra.
 We denote a cell by the cyclic ordering corresponding to the real ordering of its marked points, where $ (s_1,
 s_2,...,s_n)$ denotes to the cell
 $s_1 < s_2 <
  \cdots < s_n$ such that $\{s_i, 1\leq i \leq n\} = \{0,t_1,...,t_\ell,
  1, \infty\}$.
\end{defn} 

\begin{ex}
Figure \ref{m05} depicts $\Mod_{0,5}(\R)$, where the lines are
absent from the space, and the cells are the regions
between the lines.

\begin{figure}
\begin{center}
 \scalebox{.8}{\input{m05.pstex_t}}
\caption{$\Mod_{0,5}(\R)$}
\label{m05}
\end{center}
\end{figure}

\end{ex}

\section{Periods on $\Mn$ and the algebra, $\mathcal{C}$}

The inspiration for chapters 3 and 4 of this thesis is a recent
theorem of Francis Brown \cite{Br} in which he proves that every
period on $\Mn$ is a $\Q$ linear combination of multiple zeta values.
This led naturally to the question of whether the structure of the
multiple zeta value algebra might not be more transparent or more
symmetric by taking all periods as generators, and relations coming
from the geometry of moduli spaces.

\vspace{.5cm}
\begin{defn}\label{perioddefintro}\index{Period}
We define a {\bf period} on $\Mn$ to be a convergent integral,
$\int_\gamma \omega$, where $\gamma$ is a cell in $\Mn(\R)$ and 
$\omega$ is a differential $(n-3)$-form which is holomorphic on $\Mn$ and
which has at 
most simple poles along the boundary divisors.  We denote by $\mathcal{C}$\index{$\mathcal{C}$} the $\Q$ algebra
generated by periods on $\Mn$.
\end{defn}
Up to a variable change corresponding to permuting the marked points,
all periods may be written as integrals over the standard cell,
$\delta := 0<t_1<...<t_{n-3}<1$\label{standardcelldelta}\index{$\delta$, standard cell}.

One of the main points of chapter 3 is that the combinatorial properties of periods can be expressed by using polygons.  Let us now explain how polygons can be used to encode cells on $\Mn(\R)$, and also to encode certain differential forms on $\Mn$ called {\it cell forms}. 

We may identify an oriented $n$-gon, $\gamma$, to a cell in $\Mod_{0,n}(\R)$ by
labelling the sides of the $n$-gon with the marked points.  This
$n$-gon is
associated to the
cell given by the clockwise cyclic ordering of the
labelled edges of the polygon as in figure 3. Let
$Z=\{s_1,...,s_n\}= \{0,1,\infty, 
t_1,...,t_\ell\}$\label{Zdef}\index{$Z$} and let $\gamma$ be a polygon
decorated by $Z$, such that $s_i$ is followed by $s_{i+1}$ in
the clockwise labelling of the edges and where $i$ is taken modulo $n$.
Then we denote $\gamma$ by $(s_1,...,s_n)$ and we have that $\gamma =
\sigma(s_1,...,s_n)$ where $\sigma$ is any cyclic permutation in $\Sym_n$. 

Each component of the boundary of $\gamma$ lies in some boundary divisor $d_A
\subset \M_{0,n}\setminus \Mn$
such that $A=\{s_i, s_{i+1}, ..., s_{i+j}\}$ is a successive
block in the cyclically ordered tuple, $(s_1,...,s_n)$.

\begin{ex} A polygon cyclically labelled
$(t_1,0,t_3,1,t_2,\infty)=\gamma$ is identified with the cell
$t_1<0<t_3<1<t_2<\infty$ in $\Mod_{0,6}(\R)$ as in figure 3.
\begin{center}
 \scalebox{.8}{\input{cellgon.pstex_t}}
\label{cellgon}
\end{center}

\end{ex}

For each cell in $\Mn(\R)$, there exists a unique differential $\ell$-form
up to scalar multiple that is holomorphic on the interior and has
simple poles on all of the divisors on the boundary of that cell.  We
call such a form associated to the pole divisors of a cell a {\it cell form}.  
\begin{defn}\label{cellformdefintro}\index{Cell form, $[z_{i_1},...,z_{i_n}]$}
Let $\gamma$ be the cell, $\gamma= (s_1,s_2,...,s_n)$.
The {\bf cell
form}, $\omega_{\gamma}$, associated to $\gamma$ is defined as
$$\omega_{\gamma} = \frac{dt_1\wedge ... \wedge dt_{n-3}}{\Pi
(s_i-s_{i-1})},$$ where the $s_i$ are the cyclically labelled sides of
the polygon and where the side labelled $\infty$ is left out of the
product.  This form is holomorphic on $\Mn$ and has simple poles along 
exactly those boundary divisors bounding $\gamma$ and nowhere else on $\M_{0,n}$.  We denote a cell
form $\omega_\gamma$ by the cyclic ordering $[s_1,...,s_n]$.
\end{defn}

\begin{ex} The polygon cyclically
labelled $[0,1,t_1,t_3,\infty,t_2]$ corresponds to the cell form in figure
\ref{cellformintro}, 
$\frac{dt_1dt_2dt_3}
{(-t_2)(t_3-t_1)(t_1-1)}$.
\begin{figure}
 \scalebox{.8}{\input{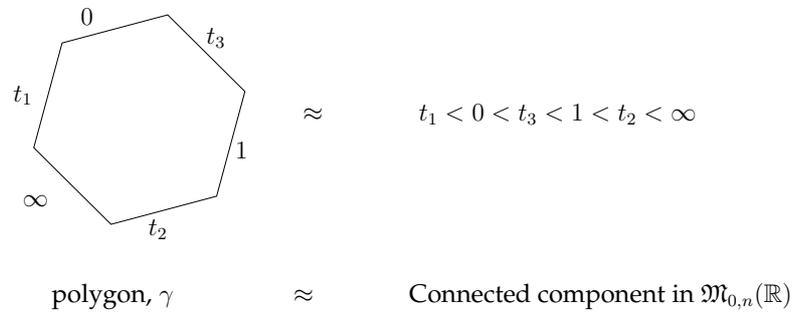}}
\caption{The polygon representation of a cell form}
\label{cellformintro}
\end{figure}
\end{ex}

We prove in chapter 3 that cell forms generate
the de Rham cohomology
group, $H^{\ell}(\Mod_{0,n})$, so that every differential $\ell$-form can be written
as a linear combination of these.  We also explicitly determine a basis for the subspace $H^{\ell}(\Mod_{0,n}^\delta)$  of differential $\ell$-forms converging on the
boundary divisors which bound the standard associahedron, $\delta$.
To do this, we associate the integral
$\int_{\gamma}\omega_\beta$ to the polygon pair $(\gamma, \beta)$ (even if this integral diverges).
Because some linear combinations of cell forms, which individually diverge on $\gamma$, may actually converge on $\gamma$,
the above results show that every convergent integral over $\gamma$ can be expressed as a
linear combination of pairs of polygons.

Using this association and Brown's theorem, we have defined (in a joint paper with F. Brown and
L. Schneps, included as chapter 3) a formal
algebra of periods, which is generated by polygon pairs, with
relations coming from geometric properties 
of moduli spaces.  The formal polygon pair algebra, $\mathcal{FC}$,
generalizes the formal multizeta algebra and 
allows us to prove some
results about periods and the cohomology of $\Mn$.  This gives a new approach to 
some conjectures about multizeta values and formal multizeta values.
To begin with, in the theorem below, we use polygons to give a new basis
for the top dimensional de Rham cohomology group, $H^{\ell}(\Mn)$, different
from Arnol'd's well-known basis and more useful for the study of periods.

Following a theorem of Arnol'd (which is more precise, see chapter 4), 
we have the following characterization of the top dimensional de Rham cohomology group, $H^{\ell}(\Mn)$ (top dimensional in the sense that $H^{m}(\Mn)=0$ for all $m> \ell$).
\begin{claim}\label{cohomclaim}\index{$H^{n-3}(\Mn)$}
$H^{\ell}(\Mn)$ is isomorphic
to the
vector space over $\Q$ of differential forms which are holomorphic on
$\Mn$ and which
have at most simple poles along the boundary divisors, $\M_{0,n}\setminus \Mn$.
\end{claim}

\begin{defns}\label{PnIn}\index{$\mathcal{P}_Z$}\index{$I_Z$}
Let ${\cal P}_Z$ be the $\Q$ vector space generated by oriented $n$-gons
decorated by the marked points in $\Mod_{0,n}$.

Let $I_Z\subset {\cal
  P}_Z$ be the vector subspace generated by shuffle sums with respect
to $\infty$, in other words polygon sums of the form $$\sum_{W\in A\sha
  B} [W,\infty],$$
where $A$, $B$ is a partition of $\{ 0,t_1,...,t_{n-3},1\}$.
\end{defns}

\begin{defn}
Let a cell form corresponding to a polygon in which 0 appears just to the left of 1 be called a 01-cell form.
\end{defn}

\begin{thm}\label{pnmodin}
${\cal P}_Z/I_Z$ is isomorphic to 
  $H^{\ell}(\Mod_{0,n})$ and a
  basis for $H^{\ell}(\Mn)$ is given by the set of 01-cell forms,
  $\{[0,1, \sigma(\infty,t_1,...,t_{\ell})],\ \sigma \in \Sym_{n-2}\}$. 
\end{thm}

Thus, each cohomology class contains a representative 01-cell form.

\begin{defn}\label{partcompdef}\index{Partial compactification, $\Mn^\gamma$}  Let $Z$ be the set denoting marked points on $\Mn$, $Z=\{z_1,...,z_n\}$.  Let $\rho$ be the set of partitions of $Z$, in which each set in the partition has cardinality greater than or equal to 2.  We denote by $D$ the disjoint union, $\sqcup_{i\in \rho} D_i$ where each $D_i$ is the (irreducible) boundary divisor in $\M_{0,n}\setminus \Mn$ defined by the partition $i$.  Likewise, if $\gamma\sqcup \gamma^c$ is a partition of $\rho$, we denote by $D_{\gamma^c}:= \sqcup_{i\in \gamma^c} D_i$.  We denote by $\Mn^\gamma:=\M_{0,n}\setminus D_{\gamma^c}$ and call $\Mn^\gamma$ a {\rm partial compactification} of $\Mn$.
\end{defn}

So we have $\Mn\subset \Mn^\gamma \subset \M_{0,n}$.  When no ambiguity can occur, to lighten the notation we may note $D_\delta$ by $\delta$.


Based on a theorem of Grothendieck \cite{Gr1}, we use in chapter 3 and
completely prove in chapter 4 that any period on $\Mod_{0,n}$ may be 
written as the integral of a linear combination of 01-forms which
converges on $D_\delta$, the set of
divisors each of which contains a face of the boundary of the standard associahedron, $\delta$, and such forms span the top dimensional
de Rham cohomology of the partially compactified moduli space,
$H^{\ell}(\Mn^\delta)$.  

Some 01-forms naturally converge on $D_\delta$. We define a {\it chord} on a
cell form, $\omega$, to be a set of marked points of a consecutive subsequence
on $\omega$ of the length between $2$ and $\lfloor
\frac{n}{2} \rfloor$.  The 01-forms which do not have any chords in
common with the polygon $\delta$ converge on the cell defined by the
cyclic ordering $\delta$.

However, there are also some linear combinations of nonconvergent 01-forms 
which converge on $D_\delta$; a basis for the space of these is the set of  
{\it insertion forms} defined in chapter 3.

With the above definitions, we can state one of the most important theorems
in the article contained in chapter 3, which is a key ingredient in the
definition of the algebra of periods (see next section).  It gives a 
combinatorial construction of an explicit basis of
$H^{\ell}(\Mn^\delta)$ and allows us to give a
recursive formula for the dimension of this cohomology group.

\vspace{3mm}
\begin{thm}\label{main}
The insertion forms and the convergent 01-cell forms form a basis for
$H^{\ell}(\Mod_{0,n}^\delta)$.\end{thm}

\vspace{3mm}
The proof of this theorem is the heart of our recent work and is given
in chapter 3.  The goal is to attain an explicit combinatorial
description of an algebra generated by ``formal periods'' in analogy
with the formal multizeta value algebra, $\FZ$.

\section{The algebra of formal periods, $\mathcal{FC}$}

The period algebra, $\mathcal{C}$,
has three known sets of relations coming from the following three
important geometric properties of moduli spaces:
\begin{enumerate}
\item Invariance under the symmetric group action corresponding to a
  variable change,
\item Forms given by shuffles with respect to one point are
  identically 0,
\item Product map relations coming from the pullback of maps on
moduli spaces (these are outlined in \cite{BCS} and \cite{Br}).
\end{enumerate}

In the style of [KZ], who conjecture that only algebraic relations of
certain geometric types exist between periods, we conjecture that these 
are the only relations on the periods on $\Mod_{0,n}$.  This is why
our strategic approach to understanding the implications of this
conjecture is to
define a formal algebra on polygon pairs satisfying these and only
these relations.

\begin{defn}\label{FCDEF}\index{$\mathcal{FC}$} The formal cell number algebra, ${\cal FC}$,
  is defined as 
  the algebra generated by pairs of polygons, ${\cal P}_Z\otimes_{\Q} {\cal
  P}_Z$, decorated by the marked points in $Z$ with the following sets of
  relations: 
\begin{enumerate}
\item $(\gamma, \omega) = (\sigma(\gamma), \sigma(\omega))\ \forall
  \sigma \in \Sym_n$,
\item For any
  $e\in Z$ and for any partition $A, B$ of $Z\setminus \{e\}$,$$((e,
  A\sha B),\omega) = (\gamma, (e, A\sha B))=0,$$

\item For any partition, $A, B$ of $Z\setminus \{0,1,\infty\}$, and
  for any four polygons, $\gamma_1$ and $\omega_1$ decorated by 
  $A\cup \{0,1,\infty\}$, $\gamma_2$ and $\omega_2$ decorated by
  $B\cup \{0,1,\infty\}$, we have the product map relation,
  $$(\gamma_1, \omega_1)(\gamma_2, 
  \omega_2) = (\gamma_1\sha \gamma_2, \omega_1\sha \omega_2).$$ 
\end{enumerate}
\end{defn}

The first relation on $\mathcal FC$ comes from variable changes on periods,
the second relation from theorem \ref{pnmodin}, and the third from product
maps on moduli spaces.

Using the definition of the period algebra and Brown's theorem, one shows
easily that the algebra of periods, $\mathcal{C}$, is isomorphic to the
algebra of multizeta values, $\MZV$ \cite{Br}. This key remark is our main
motivation for the definition of $\mathcal{FC}$, and leads naturally
to the conjecture that ${\mathcal FC}$
is isomorphic to the formal multizeta value algebra, $\FZ$. 
This conjecture seems likely because the algebra of formal cell numbers has
shuffle multiplication and ${\mathcal FC}$ encodes multizeta values, so it
should also have stuffle. We have not yet been able to prove this,
but computer calculations do support the hypothesis that ${\mathcal
  FC}$ has the
stuffle relation.  Such calculations are given at the end of chapter
3.  The relation between the different algebras 
is depicted in the following commutative diagram, where $f$ is
conjecturally an isomorphism:
\[
\xymatrix{
 {{\mathcal FC}}\ar@{->>}[d] \ar@{<.>}[r]^f & \FZ \ar@{->>}[d] \\
{{\mathcal C}} \ar[r]^{\sim} & {\mathcal{Z}}.
}\]

\section{Cohomology of partially compactified moduli spaces}

Chapter 4 of this thesis extends the methods of chapter 3 to
calculating the cohomology of $\Mn^\gamma$ for certain sets of
divisors, $\gamma$, such that $\Mn^\gamma$ is
affine.  The search for criteria for affineness led to a new,
combinatorial description of the Picard group, $Pic(\Mn)$, with a
basis given by polygons.

In the first two sections of chapter 4, we recall, in a self-contained way,
the proof of the following proposition, using the Leray theorem
of spectral 
sequences and a theorem of Grothendieck on algebraic de Rham complexes.

\begin{prop}
If $\Mn^\gamma$ is an affine variety, then the top dimensional de Rham cohomology,
$H^{\ell}(\Mn^\gamma)$ is isomorphic
to the subspace of $H^{\ell}(\Mn)$ of the classes of differential forms which have a representative that is
holomorphic on $\Mn^\gamma$ and has at most logarithmic singularities along
$\M_{0,n}\setminus \Mn^\gamma$.
\end{prop}

The third section is dedicated to defining certain criteria for $\Mn^\gamma$
to be affine.  The key observation of this section
is that if $\gamma$ is a subset of divisors that bound an
associahedron, then $\Mn^\gamma$ is affine.  In particular,
using the geometry of $\Mn$, we obtain the following proposition
as a corollary to this observation. 

\begin{prop}\label{somediv} If $\gamma$ is a set containing
only one divisor $\{d_A\}$ , two divisors $\{d_A, d_B\}$, 
or the set of three divisors $\{d_A, d_B, d_{A\cup B}\}$,
then $\Mn^\gamma$ is affine.
\end{prop}

In section 4 of chapter 4, we generalize the results of chapter 3 to find an
explicit basis of polygons for $H^{\ell}(\Mn^\gamma)$ in the cases
given in proposition \ref{somediv}. 
To do this, we exploit the residue
map on polygons and cell forms.

As before, we denote by ${\mathcal P}_{S_i}$ the $\Q$-vector space generated by
polygons decorated by the marked points in a set $S_i$, and by
$I_{S_i}\subset {\mathcal P}_{S_i}$ the subspace of shuffles with respect to
one element as in definition \ref{PnIn}.  We associate a divisor,
$d_{S_i}$ to a chord\label{chorddefCh1}\index{Chord} on a
polygon, $\omega^p= [s_{i_1}, ..., s_{i_n}] \in {\mathcal{P}}_Z$ to be
a partition of $\omega^p$ into consecutive blocks,
$[s_{i_j},...,s_{i_{j+k}}]$ and $[s_{i_{j+k+1}},...,s_{i_{j-1}}]$ such that $k\geq 1$ as
in the left-hand object in figure 5.

\begin{center}
 \scalebox{.8}{\input{residuechord.pstex_t}}
\label{residuechord}
\end{center}

Now, for every partition $Z$ given by $Z=S_1\cup S_2$,
we define a residue map on polygons with respect to the divisor,
$d_{S_1}=d_{S_2} = d$:
$$\Res^p_{d}:{\cal P}_Z \rightarrow {\cal P}_{S_1\cup\{d\}} \otimes_\Q {\cal
  P}_{S_2\cup\{d\}},$$
which is simply the tensor product of the two polygons formed by cutting along the divisor $d$.

\begin{defn}
Let $\omega^p$ be a polygon in ${\cal P}_Z$.  If the partition $S_1, S_2$
corresponds  to a chord of $\omega^p$, then it cuts $\omega^p$ into
two subpolygons 
$\omega^p_i$ ($i=1,2$) whose edges are indexed by the set $S_i$ and an edge
labelled $d$ corresponding to the chord $d$.  We set\label{polyresDEF}\index{Residue map, $\Res_d^p$}

\begin{equation}
\Res^p_d(\omega^p)=
\begin{cases}
\omega^p_1\otimes \omega^p_2&\hbox{if $d$ is a chord of
  $\omega^p$}\\
0&\hbox{if $d$ is not a chord of $\omega^p$}.
\end{cases}
\end{equation}

\end{defn}

Let $\pi:{\mathcal P}_Z \rightarrow H^\ell(\Mn)$ be the map from
polygons to cell forms as in definition \ref{cellformdefintro}.  In
chapter 4, the following 
theorem is proved.

\begin{thm}\label{allcohomsCh1} Let
$\gamma = \{\gamma_1,...,\gamma_k\}$ be a set of boundary divisors of
 $\M_{0,n}$ such
that $\Mn^\gamma$ is affine.  Then, the $\Q$ vector space, $H^\ell(\Mod_{0,n}^\gamma)$
  coincides with the differential forms in the intersection of vector
  spaces, $$ \bigcap_{i=1}^k \pi((\Res_{\gamma_i}^{p})^ {-1} (
I_{\gamma_i \cup \{d\}} \otimes {\cal P}_{Z\setminus \gamma_i \cup \{
  d\}})).$$

Furthermore, a basis for $H^\ell(\Mod_{0,n}^\gamma)$ can easily be
deduced from a Lyndon basis of the polygons in $I_{\gamma_i \cup
  \{d\}} \otimes {\cal P}_{Z\setminus \gamma_i \cup \{d\}}$ using
insertion forms.
\end{thm}

As a corollory to this theorem we display explicit bases for
$H^{\ell}(\Mn^\gamma)$ for sets $\gamma = \{d_A\}, \{d_A, d_B\}$
    and $\{d_A, d_B, d_{A\cup B}\}$ and give a closed formula for the
    dimensions.

The search for criteria on $\gamma$ such that $\Mn^\gamma$ is affine
led us to investigate $Pic(\M_{0,n})$.  If a divisor $\gamma^c$ is ample in the
Picard group, then $\Mn^\gamma$ is an affine space.  Although we didn't
succeed in proving that $\gamma^c$ was affine, for certain $\gamma$ that we were
interested in (such as the pole divisors of a multizeta form), this
search led to a new presentation of $Pic(\M_{0,n})$ with a basis of polygons.
The final section of this thesis is dedicated to the statement and
proof of this result.

\chapter{Comparison and combinatorics of the Lie algebras,
  $\ds$,
  $\grt$ and $\nfz$}

In this chapter we prove a dimension result on $\ds$ which provides
evidence toward the conjecture 
stated in the introduction that $\ds\simeq \grt$.  Some theorems and
definitions given in the introduction and used in the chapter are
restated for easy reference for the reader.

A multizeta value is a real number defined by the iterated sum,
$$\zeta(k_1,...,k_d) =
\sum_{n_1> n_2>\cdots >n_d>0} \frac{1}{n_1^{k_1}n_2^{k_2}\cdots
  n_d^{k_d}},$$ where $(k_1,...,k_d)$ is a sequence
of positive integers such that $k_1\geq 2$.   We may call a multizeta
value a multiple zeta value, a multizeta, an $MZV$, a zeta value or
simply a zeta. 
The {\bf depth}\label{depdivCh2}\index{Weight of $\zeta(\underline{k})$}\index{Depth of $\zeta(\underline{k})$} of $\zeta(k_1,...,k_d)$ is $d$ and
  its {\bf weight} 
  is $\sum_{i=1}^d k_i$.  Let $\MZV$ denote the algebra over $\Q$ generated by multizeta values, and let $\MZV_n$ denote the 
vector space over $\Q$ generated by multizeta values of weight $n$.

Although $\MZV$ is simple to define, there remain many open questions about this algebra.  The motivation for the results in this chapter stem from the following open problem about $\MZV$.  It is believed that all linear relations over $\Q$ on multizeta
values are generated by the double shuffle relations and Hoffman's
relation, relations which preserve the weight of elements in $\MZV$.  This in turn would imply the well-known ``direct sum conjecture'': 
\begin{conj*} The algebra, $\MZV$, is graded by weight and hence
 $\MZV:= \bigoplus_{n=0}^\infty \MZV_n$.
\end{conj*}

Note that this ambitious conjecture would imply the transcendence of every multizeta value, since the minimal polynomial of an algebraic multizeta value would yield a linear relation in different weights.

  Of particular interest to us are the
depth 1 generators of $\MZV$, $\zeta(n)$.  The depth
1 generators in even weight are well understood and have long been
known to be
transcendental. 

\begin{thm}[Euler] \begin{align*}
\zeta(2) &= \frac{\pi^2}{6}\\
\zeta(2n)& = \frac{2^{2n-1}|B_n| \pi^{2n}}{(2n)!}, \end{align*}
where $B_r$ is the Bernoulli number that is obtained by expanding the series,
\begin{equation*} \frac{y}{e^y-1}= \sum_{r=0}^{\infty} B_r
  \frac{y^r}{r!} .\end{equation*}
\end{thm}

However, the depth 1 generators in odd weight are less well
understood.  They are conjectured to be transcendental numbers.
R. Ap\'ery \cite{Ap} proved that $\zeta(3)$ was irrational and
T. Rivoal \cite{BR} recently proved that there are infinitely many
irrational $\zeta(2n+1)$.

This chapter is not an attempt to tackle the question of irrationality
of depth 1 zeta values, which seems very difficult
because of the analytic nature of the problem.  Yet, by working in the Lie
algebra, we obtain results relating to the conjecture that the double 
shuffle Lie algebra is isomorphic to the free Lie algebra with one
generator in each odd weight,\index{$\mathfrak{f}$}
$$\mathfrak{f}\label{freedef} = {\mathbb{L}}[x_{2n+1}:\ n\geq 1].$$

\section{The double shuffle Lie algebra, $\ds$}

In this chapter, we work in the two noncommutative power series algebras,
$\Qxy$ and $\Q\langle \langle y_i;
1\leq i <\infty \rangle \rangle$.  For $f$ a polynomial in one of these
algebras, we denote by $(f|w)$\label{f|wdef}\index{$(f \vert w)$} the
coefficient of the monomial $w$ in $f$.

\subsection{Shuffle on $\Lxy$}

The power series algebra, $\Qxy$, may be graded in two ways, by weight
and by depth according to the following definition.

\begin{defn}\label{depthQxy}
The algebra $\Qxy$ possesses a grading by the length of its
monomials, $\omega$, which we call the {\bf weight}\index{Weight in $\Qxy$} and we denote the
weight of $\omega$ by
$w(\omega)$.  Similarly,
we can define a grading on $\Qxy$ by the {\bf depth}\index{Depth in $\Qxy$} of the monomial,
which is the number of times $y$ appears and we denote the depth of
$\omega$ by $d(\omega)$.  The notation, $V_n$, where $V$ is any vector space of
polynomials, refers to its weight $n$ graded part.\index{$V_n$}
\end{defn}

The algebra, $\Qxy$, may be equipped with the
following coproduct to form a Hopf algebra,
\begin{align}
\Delta_{\sh}: \Qxy & \rightarrow \Qxy\otimes_{\Q} \Qxy \\
x& \mapsto x\otimes 1 + 1\otimes x \\
y& \mapsto y\otimes 1 + 1\otimes y.
\end{align}

\begin{defn}\label{primitive}
A element $f\in \Qxy$ is {\bf primitive} for the coproduct
$\Delta_{\sha}$, if $\Delta_{\sha}(f) = 1\otimes f
  + f\otimes 1$.\end{defn}

\begin{defn}
The Lie algebra, $\Lxy\subset \Qxy$\label{freeliedef}\index{$\Lxy$, bracket $[f,g]$}, is the subspace of
polynomials generated by successive bracketings of $x,y$ for the Lie
bracket, $[f,g]=fg-gf$. \end{defn}

The Lie algebra, $\Lxy$, possesses the grading by weight and depth
inherited from $\Qxy$.  We denote by $\mathbb{L}_n[x,y]$ the weight
$n$ graded part and by $\mathbb{L}_n^i[x,y]$\label{freeliegradeddef}\index{$\mathbb{L}_n^i[x,y]$}
the depth $i$, weight $n$ 
graded part, so that
\begin{align*}
\Lxy & = \oplus_n\mathbb{L}_n[x,y]\\
\mathbb{L}_n[x,y] &= \oplus_{1\leq i <n}
\mathbb{L}_n^i[x,y].\end{align*} 
The vector space, $\mathbb{L}_n[x,y]$, is of finite dimension for each $n$ (we will recall the dimension formula in section \ref{cdeds}).

Here we recall the definition of the shuffle product on monomials.

\begin{defn}\label{sh}  Let
  $\underline{\alpha}=(\alpha_1,\dots ,\alpha_k)$ and $\underline{\beta} = 
    (\beta_1,\dots, \beta_l)$ be two sequences.  The {\bf shuffle product}
    of $\underline{\alpha}$ 
    and $\underline{\beta}$, denoted by
    $sh(\underline{\alpha},\underline{\beta})$, or 
    $\underline{\alpha}\sh \underline{\beta}$, is the formal sum
    obtained by the recursive procedure:
\begin{enumerate}
\item $sh(\underline{\alpha},\emptyset) = sh(\emptyset,
  \underline{\alpha}) = \underline{\alpha}$,
\item $sh(a_0\cdot \alpha, b_0\cdot \beta) =a_0\cdot sh(\alpha,
  b_0\cdot \beta) + b_0\cdot sh(a_0\cdot \alpha, \beta)$.
\end{enumerate}
\end{defn}

The shuffle product on sequences
$\underline{\alpha}$ and $\underline{\beta}$ is the sum over all of
the permutations of $\underline{\alpha} \cdot \underline{\beta}$
that preserve the orders of both sequences.  For ease of notation, we
write $\underline{\gamma}\in sh(\underline{\alpha},
\underline{\beta})$ to mean that $\underline{\gamma}$ is a term in
the sum $ sh(\underline{\alpha},
\underline{\beta})$.

\begin{prop}\label{SeEc}\cite{Se}\cite{Re}
For $f\in \Qxy$ the following conditions are equivalent:
\begin{enumerate}
\item $f\in \Lxy$,
\item $\Delta_{\sha}(f)=f\otimes 1 + 1\otimes f$,  
\item For any $\underline{\omega}_1$, $\underline{\omega}_2$,
  non-empty sequences in $x$ and $y$,
\begin{equation*}
\sum_{\underline{\omega}\in
  sh(\underline{\omega}_1,\underline{\omega}_2)}
  (f|\underline{\omega}) = 0.
\end{equation*}
\end{enumerate}
\end{prop}

Let ${\mathcal{U}}(\Lxy)$ be the universal enveloping algebra of $\Lxy$ and we
denote by $\cdot$ its product. 
\begin{equation}{\mathcal{U}}(\Lxy) =\bigoplus_{n=o}^{\infty}
  T^{\otimes n}/<f\otimes g -
  g\otimes f - [f,g] \ | \ f,g \in \Lxy>,
\end{equation}
where $T^{\otimes n}=\bigotimes^n \Lxy$ is the $n$th tensor power of $\Lxy$.
The universal enveloping algebra naturally possesses a Hopf algebra structure
with coproduct, $\Delta_{\Lxy}$, which is the unique algebra morphism
which is primitive for the elements in $\Lxy$.  So we have the
following corollary:

\begin{cor}\cite{Se} We have the isomorphism of Hopf algebras,
\begin{align*}
(\Qxy,\Delta_{\sha},\cdot) & \simeq ({\mathcal{U}}(\Lxy),\Delta_{\Lxy},\cdot)\\
x&\mapsto x\\
y&\mapsto y.
\end{align*}
\end{cor}
\begin{proof} The universal enveloping algebra on the free Lie algebra
  on $n$ generators is isomorphic to the free polynomial algebra on
  $n$ variables.  By taking $n=2$ we have $(\Qxy,\cdot)\simeq
  ({\mathcal{U}}(\Lxy),\cdot)$.  By the theorem \ref{SeEc}, the
  primitive elements 
  for the coproduct, $\Delta_{\sha}$ are exactly those in the Lie
  algebra, $\Lxy$.
 \end{proof}

The weight grading that we give to $\Qxy$ is the same as one gives to
the grading in the universal enveloping algebra defined by the theorem
of Poincar{\'e}-Birkhoff-Witt.

\subsection{Stuffle on $\Lyi$}

Here, we make analogous definitions and statements for the algebra $\Qyi$.

The power series algebra, $\Qyi$\label{qyigradeddef}\index{Weight in $\Qyi$}\index{Depth in $\Qyi$}, possesses a
grading given by the sum of the 
indices of the monomial which we call the {\bf weight} of the
monomial, i.e. $w(y_{i_1}\cdots y_{i_r}) = \sum_{j=1}^r i_j$.
Similarly,
we can define a grading on $\Qyi$ by the {\bf depth} of the monomial,
which is the length of the monomial, i.e. $d(y_{i_1}\cdots y_{i_r}) =r$. 

The algebra, $\Qyi$, may be equipped with the
following coproduct to form a Hopf algebra,
\begin{align}
\Delta_{*}: \Qyi & \rightarrow \Qyi\otimes_{\Q} \Qyi \\
y_i& \mapsto \sum_{n+m=i} y_n\otimes y_m.
\end{align}

If $\Delta_{*}(f)=1\otimes f+f\otimes 1$, then $f$ is primitive for
$\Delta_{*}$ as in definition \ref{primitive}.

\begin{defn}\label{lyidef}\index{$\Lyi$}
The Lie algebra, $\Lyi \subset \Qyi$, is the subspace of
polynomials generated by successive bracketings of $y_i$ for the Lie
bracket, $[f,g]=fg-gf$. \end{defn}

The Lie algebra, $\Lyi$, possesses the grading by weight and depth
inherited from $\Qyi$.  

We define here the stuffle product of monomials in $\Qyi$, which is
analogous to the stuffle product on sequences of positive integers
given in the introduction.

\begin{defn} For any monomials in $\Qyi$,
      $\underline{a}$, $\underline{b}$ the {\bf stuffle product}
      of $\underline{a}$ and $\underline{b}$, denoted
      $st(\underline{a},\underline{b})$ or $\underline{a} *
      \underline{b}$,
      is the formal sum obtained by the recursion:
\begin{enumerate}
\item $st(\underline{a},\emptyset)= st(\emptyset,\underline{a})=
  \underline{a}$,
\item $st(y_i\cdot \underline{a},y_j\cdot \underline{b}) = y_i\cdot
  st(\underline{a}, y_j\cdot \underline{b}) + y_j \cdot st(y_i
  \cdot \underline{a},\underline{b}) + (y_{i+j}) \cdot
  st(\underline{a} , \underline{b})$.
\end{enumerate}
\end{defn}

The following proposition due to J. Ecalle, gives us an easy method of
determining whether $f\in\Qyi$ is in $\Lyi$ and also gives the link
between the stuffle relation and the coproduct, $\Delta_*$.

\begin{prop}\label{Ec}\cite{Ec}
For $f\in \Qyi$ the following conditions are equivalent:
\begin{enumerate}
\item $f\in \Lyi$.
\item $\Delta_{*}(f)=f\otimes 1 + 1\otimes f$.  
\item For all $\underline{\omega}_1$, $\underline{\omega}_2$,
  non-empty sequences in $\{y_i\}$,
\begin{equation*}
\sum_{\underline{\omega}\in
  st(\underline{\omega}_1,\underline{\omega}_2)}
  (f|\underline{\omega}) = 0.
\end{equation*}
\end{enumerate}
\end{prop}

We associate an element in 
$\Qxy$ to $\Qyi$ via the linear map, $\pi_{\mathcal{Y}}$,
the corrected projection onto $\Qyi$.  It is closely
linked to the alternative notation for a multizeta in the association,
\begin{align*}
 x^{k_1-1}y\cdots x^{k_d-1}y & \sim k_1 k_2\cdots k_d \\
  \zeta(x^{k_1-1}y\cdots x^{k_d-1}y) & = \zeta(k_1 k_2\cdots k_d). 
\end{align*}
\begin{defn}\label{piy}  Let $\piy$ be the $\Q$ linear map defined by:
\begin{align}
\pi_{\mathcal{Y}} :  \Qxy &\rightarrow \Qyi \\
\widetilde{\pi}_{\mathcal{Y}} (x^{k_1-1}yx^{k_2-1}y \cdots x^{k_n-1}y
x^{k_{n+1}}) & = \begin{cases} 0 & k_{n+1} \neq 0 \\
y_{k_1}y_{k_2}\cdots y_{k_n} & k_{n+1} = 0 \end{cases} \\
\pi_{\mathcal{Y}} (f) & = \widetilde{\pi}_{\mathcal{Y}}(f) +
\sum_{n=2} (f| x^{n-1}y) \frac{(-1)^{n-1}}{n} y_1^n.
\end{align}
\end{defn}

\begin{ex} Let $f=2 x^2y + x^3y +4 xy^2 -8yxy +4y^2x$.  Then,
\begin{align*} 
\widetilde{\piy}(f)&= 2y_3 +y_4 +4y_2 y_1 -8 y_1y_2,\\
\piy(f) &= 2y_3 +y_4 +4y_2 y_1 -8y_1y_2 +
  \frac{2}{3}y_1^3 - \frac{1}{4}y_1^4.
\end{align*} 
\end{ex}

Note that $\widetilde{\piy}$ preserves the depth and the weight, but $\piy$ only preserves the weight.

\subsection{The double shuffle Lie algebra, $\ds$}

\begin{defn}\label{dsdef}
The vector subspace, $\ds\subset \Lxy$, is generated by polynomials, $f$,
that satisfy the following sets of relations,
\begin{enumerate}
\item The weight of any term in $f$ is greater than or equal to 3,
\item $f$ is primitive for $\Delta_{\sha}$: $\Delta_{\sha}(f) =
f\otimes 1 + 1\otimes f$,
\item $\piy(f)$ is primitive for $\Delta_{*}$: $\Delta_*(\piy(f)) =
  \piy(f)\otimes 
  1 + 1\otimes \piy(f)$.
\end{enumerate}
\end{defn}

\begin{defn}\label{pbrakdef}
The Poisson bracket on elements of $\Lxy$ is the Lie bracket given by
$$\{f,g\} = [f,g] +D_f(g) - D_g(f)$$ where $[f,g]=fg-gf$ and the
$D_f$ are derivations
defined recursively by $D_f(x) = 0$, $D_f(y)=[y,f]$.
\end{defn}

\begin{thm}\cite{Ra}
The double shuffle elements, $\ds$, form a Lie algebra for the
Poisson bracket.
\end{thm}

The double shuffle Lie algebra is graded by weight because the double
shuffle relations preserve the weight, and we denote each weight $n$
graded part by $\ds_n$.  However, $\ds$ is not graded by depth because the
stuffle forces relations between words of different depth, such as the
classical relation, $\zeta(2)*\zeta(2) = 2\zeta(2,2) + \zeta(4)$.  In
the proof of the main theorem \ref{theoreme}, the relations between
depth one and depth
two elements given by stuffle are fully explained.

A useful way to calculate the action of the derivation, $D_f(g)$, is
given in \cite{Sc}.  
\begin{prop}\cite{Sc}
Let $f,g\in \Lxy$, such that the depth of $g$ is $d$.  Then $D_f(g)$
is given by the sum over the Lie elements, $\sum_{i=1}^d
g_i(x,y,[y,f])$ where each $g_i$ is gotten by substituting one $y$ in
$g$ by $[y,f]$. 
\end{prop}

\begin{ex}
Let $g=[[[x,y],[x,[x,y]]],[y,x]]$, so we have
\begin{equation*}\begin{split}
D_f(g)=[[[x,[y,f]],[x,[x,y]]],[y,x]] + [[[x,y],[x,[x,[y,f]]]],[y,x]] +
\\ 
[[[x,y],[x,[x,y]]],[[y,f],x]].\end{split}\end{equation*} \end{ex}

\section{Lyndon-Lie words}\label{cdeds}

\begin{defn}\label{Lyndonworddef}\index{Lyndon word} A Lyndon word is a monomial, $\omega\in
  \Qxy$, such that all of the right factors of $\omega$ are greater
  than $\omega$ for the
  lexicographic ordering.
In other words, if $\omega=a_1\cdots a_n,\ a_i\in
  \{x,y\}$, then $a_1\cdots a_n< a_i\cdots a_n\ \forall
  i>1$. \end{defn}
The simplest example of a Lyndon word is given in depth 1, where the
  only Lyndon word is $x^ny$.

Given a Lyndon word, $\omega$, we can construct an element of $\Lxy$, denoted
$[\omega]$,\index{$[\omega]$}
by recursively bracketing in the following manner.  
Let $\omega$ be written as
$\omega=u\cdot v$ such that $v$ is the smallest, non-trivial right
factor.  
Then we bracket $[u,v]$.  We can repeat this procedure recursively on
$u$ and $v$, since $u$ and $v$ are Lyndon words.
If $v$ is the smallest right factor, it is smaller than all of
its right factors.
Furthermore, $u$ is smaller than all of its right factors since
if $u=u_1\cdot u_2$ where $u_2<u$,
then $u_2\cdot v<u\cdot v=\omega$ which is impossible since we
supposed that $\omega< u_2\cdot v$.  So we can recursively bracket in
the same way as the base step until we obtain an element of $\Lxy$.

\begin{defn}\label{lyndonlieworddef}\index{Lyndon-Lie word} A {\bf Lyndon-Lie word} (or
  Lyndon-Lie monomial) is an 
  element of $\Lxy$ obtained by a
  bracketing a Lyndon word in the above recursive procedure.
\end{defn}

\begin{thm}\cite{Re}\label{baseLL}  Lyndon-Lie words form a basis for
  the $\Q$ vector space $\Lxy$.  We call this basis the Lyndon-Lie basis.
\end{thm}

\begin{thm}[Witt dimension formula]\cite{Se}  Let
  $\mathbb{L}[x_1,...,x_r]$ be the free Lie algebra on $r$
  generators.  The dimension of the $n$th graded piece is given by
\begin{equation*} dim(\mathbb{L}_n[x_1,...,x_r])=\frac{1}{n}
  \sum_{d|n} \mu(d)r^{n/d},
\end{equation*}
where the M\"obius function, $\mu$, is defined by 
\begin{equation*}\mu(d)\begin{cases} 1 & d=1\\ (-1)^k & d=p_1\dots
    p_k\ (p_i \text{ distinct primes})\\
0 & d\ {\text{has a square factor}}. \end{cases} \end{equation*}
In particular,\begin{equation*}
dim(\mathbb{L}_n[x,y])=\frac{1}{n}
  \sum_{d|n} \mu(d)2^{n/d}.\end{equation*}\end{thm}

\begin{lem}\cite{Re}\label{bw} Let $f\in\mathbb{L}[x,y]$
  and let $\overleftarrow{\omega}$ be the word $\omega$ written
  backwards.
  Then, $(f|\omega)=(-1)^{n-1}(f|\overleftarrow{\omega})$.\end{lem}

From theorem \ref{baseLL}, we obtain the following corollary.
\begin{cor}
For any $n$, $\mathbb{L}_n^1[x,y]$ has dimension 1 and its Lyndon-Lie
basis is $\{[x^{n-1}y]\}$. \end{cor} 

This corollary is immediate, since $x^{n-1}y$ is the only Lyndon word
in depth 1.

In this thesis, we use the notation $C_a^b$\label{Cni}\index{$C_n^i$} to denote the binomial coefficient, ${a \choose b}$.

\begin{lem} We have the following expression for the
    depth one basis element as a polynomial in $\Qxy$,
  $$[x^{n-1}y]=\sum_{i=0}^{n-1}(-1)^i 
  C_{n-1}^i x^{n-1-i}y x^i.$$ 
\end{lem}

\begin{proof} We reason by induction.  The smallest right factor of
  $x^{n-1}y$ is $x^{n-2}y$.  By applying the recursive procedure,
  $[x^{n-1}y] = [x,[x,\dots,[x,y]]] = ad(x)^{n-1}(y)$.  The lemma is true
  for $n=1$ and we suppose it's true for $n$.  Then,
\begin{align*}ad(x)^{n}(y) &= ad(x)(\sum_{i=0}^{n-1} (-1) ^i C_{n-1}^i
  x^{n-1-i}yx^i) \\
&= \sum_{i=0}^{n-1} (-1)^i C_{n-1}^i x^{n-i}yx^{i} - \sum_{i=0}^{n-1}
(-1)^i C_{n-1}^i x^{n-1-i}yx^{i+1} \\
&= \sum_{i=0}^{n-1} (-1)^i C_{n-1}^i x^{n-i}yx^{i} + \sum_{i=1}^n
(-1)^i C_{n-1}^{i-1} x^{n-i}y x^i\\
&= C_{n-1}^0 x^{n}y +\sum_{i=1}^{n-1}(-1)^i (C_{n-1}^{i-1} +
C_{n-1}^i) x^{n-i}yx^i + (-1)^n
yx^n\\
&= \sum_{i=0}^n (-1)^i C_n^i x^{n-i}y x^i.
\end{align*}\end{proof}

We obtain a similar corollary for the weight $n$, depth 2 graded parts.

\begin{cor}\label{dep2basis} The Lyndon-Lie basis for
  $\mathbb{L}_n^2[x,y]$ is given by
$\{[x^ryx^sy]\ |\ r>s,\ r+s=n-2 \}$ and its dimension is
  $\lfloor \frac{n-1}{2} \rfloor$.\end{cor}

\begin{proof}  A Lyndon word must end in a $y$, otherwise the right
  factor $x$ would be smaller than the word.  Also if $r\leq s$ then
  the right factor $x^{s}y$ would be smaller than the word.  Then the
  dimension is just the number of ways to distibute $y$ into the
  sequence $x^{n-2}$ such that the number of $x$ on the left is
  greater that the number of $x$ on the right. \end{proof}

\section{Coefficients on monomials in $\ds$}

Let $f\in \ds$.  By the shuffle relation, we know that $f\in \Lxy$.
 To study the behavior in depths 1 and 2, we
 write $f$ in two ways, one in terms of the Lyndon-Lie basis and one
 in terms of the basis of monomials in $\Qxy$.  Since being in $\Lxy$
 is equivalent to satisfying the shuffle relation, we only need to
 study how the stuffle relation behaves in $\Lxy$.  The stuffle
 relation is seen by projecting onto terms ending in $y$.  For this
 reason, we only label coefficients on monomials that end in $y$ and
 to study depths 1 and 2, we only label coefficients in those depths.
 
\begin{equation}\label{f}
\begin{split}
f &=A[x^{n-1}y] + \sum_{s=0}^{\lfloor \frac{n-3}{2}\rfloor} a_s
 [x^ryx^sy] +\dots\\
 &= Ax^{n-1}y + b_0 x^{n-2}y^2 +b_1x^{n-3}yxy +\cdots
 +b_{n-2}yx^{n-2}y +\dots ,\end{split}\end{equation}
where the subscript on the $b_i$ coefficients equals the number of $x$
 between the two $y$, $x^{n-2-i}yx^iy$.

\begin{lem}\label{b_i} The coefficients of $f$ satisfy the following
 relation: 
$$b_i=\sum_{j=0}^{i} (-1)^{i-j} a_j\Bigl( C_{j+1}^{i-j} +
C_{j}^{i-j-1}\Bigr),$$ following the convention that $C_c^d=0$
whenever $c<d$ or $d<0$
  and taking $a_j=0$ whenever $j>{n-3\over2}$. \end{lem}

\begin{proof}
For each $j<\lfloor \frac{n-1}{2} \rfloor$ let
$L_j^n=[x^{n-2-j}yx^jy]$ be the basis element from 
corollary \ref{dep2basis}.  So we have by the recursive procedure for
bracketing Lyndon words,

\begin{align} L_j^n&=[x,\cdots[[x^{j+1}y],[x^jy]]\cdots]\\
&=[x,\cdots[\sum_{k=0}^{j+1}(-1)^kC_{j+1}^kx^{j+1-k}yx^k ,\sum_{l=0}^j
(-1)^l C_{j}^lx^{j-l}yx^l]\cdots].\end{align}

To isolate the coefficients on $x^{n-2-i}yx^iy$, thus determining their
contribution to $b_i$, we only need to consider those terms coming
from the inner most bracket product. (This follows from the fact that
in the Lie word
$[x^{n-1}y]$, there is only one term not ending in an $x$
and its coefficient is 1.)  There are two terms, $x^{2j+1-i}yx^iy$ in this
product, one coming from $l=0,\ k=i-j$, the other one coming from
$k=0,\ l=i-j-1$.  These two contributions give a coefficient of
$$(-1)^{i-j} C_{j+1}^{i-j} - (-1)^{i-j-1} C_j^{i-j-1}= (-1)^{i-j}
(C_{j+1}^{i-j} + C_j^{i-j-1}).$$
\begin{claim} These are the only terms in $L_j^n$ that contribute to
  $b_i$.\end{claim}
Since the terms we are looking for must end in a $y$, either
  $l=0$ or $k=0$.  If $l=0$, then $k=i-j$ because
  $x^kx^{j-l}=x^kx^j=x^i$ and if
  $k=0$, then $l=i-j-1$ because $x^lx^{j+1-k}=x^lx^{j+1}=x^i$.

To find the complete contribution to $b_i$ from all of the
  $L_j^n$, we only need to sum from $j=0$ to $i$.  This is seen by
  reapplying the argument from the above claim.  Namely, if $l=0$ then
  there are at least $j\ x$'s in front of the final $y$, so we won't
  find any terms, $x^{2j+1-i}yx^iy$ when $j>i$.  Similarly, if $k=0$,
  then we have at least $j+1\ x$'s in front of the final $y$, so the
  coefficients on these terms do not contribute to $b_i$ when $j>i-1$.

In light of the
  above analysis, the complete expression for $b_i$ in terms of the
  coefficients on the Lyndon Lie basis polynomials is
  $$b_i=\sum_{j=0}^i
  (-1)^{i-j}a_j(C_{j+1}^{i-j}+C_j^{i-j-1}).$$
\end{proof}

\begin{cor}\label{b_0} The coefficient $b_0$ is given by $$b_0=
  \frac{n-1}{2} (f|[x^{n-1}y])= \frac{n-1}{2} A.$$
\end{cor}

\begin{proof} We use the convention that $C_a^b=0$ whenever $a<b$ so that by lemma \ref{b_i}, we may write
\begin{align*}
\sum_{i=0}^{n-2} b_i & = \sum_{i=0}^{n-2} \sum_{j=0}^{n-2} a_j (-1)^{i-j} (C_{j+1}^{i-j} + C_j^{i-j-1}) \\
&= \sum_{j=0}^{n-2} a_j \bigl(  \sum_{i=0}^{n-2} (-1)^{i-j} C_{j+1}^{i-j} + \sum_{i=0}^{n-2} (-1)^{i-j} C_j^{i-j-1}\ bigr)\\
&= -a_0, 
\end{align*}
where the last equality is gotten from the identity, $\sum_{k=0}^{j} (-1)^kC_j^k = 0$, which leaves a coefficient 0 on all terms except $a_0$.

For $n$ odd, we have the $\frac{n-1}{2}$ stuffle relations, $b_{i-1} + b_{n-i-1} = -A$.  By summing over all such relations, we have $\sum_{i=0}^{n-2} b_i = -\frac{n-1}{2} A$.  To finish, note that $b_0=a_0$ by lemma \ref{b_i}.  So we have,
$$b_0 = a_0 = \frac{n-1}{2} A.$$

For $n$ even, we have the $\frac{n}{2}$ stuffle relations, $b_{i-1} + b_{n-i-1} = -A$.  The last of the stuffle relations gives $2b_{n/2-1} = -A,\ b_{n/2-1}=-\frac{1}{2}A$.  By summing over all of the stuffle relations we obtain, $\sum_{i=0}^{n-2} b_i = -\frac{n}{2} A - b_{n/2 -1} = -\frac{n-1}{2}A$.  So we also have that $$b_0 = a_0 = \frac{n-1}{2} A.$$

\end{proof}

\section{Statement of main theorem and generating polynomials}\label{PQTDA}

The purpose of this section is to shed some light, and give some
evidence toward the conjecture that $\grt\simeq \ds$.  In subsequent
sections, we use combinatorial methods based on Lyndon-Lie theory from
\cite{Re} to prove a result parallel to work by Ihara on $\grt$ \cite{Ih2}:

\begin{thm}
The dimensions of the $i$th depth-graded parts of $\grt$ for $i=0,1$ are
\begin{equation}\begin{split}
dim(F^1_n\grt / F^2_n\grt) &=\begin{cases} 1 & n\hbox{ odd}\\
0 & n\hbox{
  even}\end{cases}\\
dim(F^2_n\grt / F^3_n\grt) &=\begin{cases} 0 & n\hbox{ odd}\\
\lfloor \frac{n-2}{6} \rfloor & n\
  even. \end{cases}\end{split}\end{equation}
\end{thm}

First, we establish some key properties about the structure of the
double shuffle Lie algebra.

We have a depth filtration of $\ds$, as a vector subspace of $\Lxy$,
\begin{equation*}
\ds=F^1\ds \supseteq F^2\ds \supseteq  F^3\ds \supseteq\dots,
\end{equation*}
where each depth-filtered part is defined as\index{$F^i_n(\ds)$}
\begin{align}\label{Fids}\begin{split}
F^i\ds &= \{ f(x,y)\in\ds \ |\ (f|t)=0,\ \forall\ \mbox{terms } t  \mbox{
  of depth} < i\} \\ &=\oplus_{n=1}^\infty F_n^i\ds\subset
  \oplus_{n=1}^\infty
F_n^i\Lxy.
\end{split}\end{align}

In each weight $n$ piece $F_n^i\Lxy$ is finite dimensional so that
$F_n^i\ds$ is finite dimensional.

We are interested in the depth grading given by 
$F^i_n\ds/F^{i+1}_n\ds=gr_n^i(\ds)$, which is of course also finite
dimensional.

\begin{thm}\label{theoreme}  The dimensions of the depth-graded parts
  for depths 1 and 2 of $\ds$ are:
\begin{equation*}
(i)\qquad dim(F_n^1\ds/F_n^2\ds) =
\begin{cases}
1&n\geq3,\mbox{ odd}\\
0&n=1\ \mbox{or}\ n\mbox{ even,} \end{cases}\\
\end{equation*}
in other words, there exists at least one element
$f_n\in\ds_n$ for all odd $n\geq3$ such that $(f_n|x^{n-1}y)=1$.

Furthermore,
\begin{equation*}
(ii)\qquad dim(F_n^2\ds/F_n^3\ds) = \begin{cases}
0 &n \mbox{ odd}\\
\lfloor \frac{n-2}{6}\rfloor & n\mbox{
  even,}\end{cases}\end{equation*}
and this space is generated by the images of $\{f_i,f_j\}$ with
$i+j=n$, $i,j\geq3$ are odd, and where the $f_i$ are the unique depth
1 elements from
case (i).
\end{thm}

The proof of theorem \ref{theoreme} is done in the sections
\ref{parti} and \ref{partii}.  Before proceeding to the proof, we give some
preliminary combinatorial lemmas necessary in the proof.

\begin{lem}\label{Biglem}\index{$P$}  Let $P$ be the rational function in the ring of power series on three commutative variables, $\Q[[X,Y,T]]$, given by
\begin{equation}
P=\frac{1+YT^2}{(1-(XT+YT^2))(1-(X+Y))}.\end{equation}

The coefficients of the monomials, $X^{n-i}Y^iT^{j}$, for the $n, i, j$ listed below, in the series
expansion of $P$ are given by
\begin{equation} (P| X^{n-i}Y^iT^{j}) =
\begin{cases}C_{n}^{i} & i+1 \leq j \leq n\ or \ j=0\\
C_{n}^{i}-1 & j=n+1,\ i\ even \\
C_{n}^{i}+1 & j=n+1,\ i\ odd\\
C_n^i +1 & j=i,\ i>0\ even\\
C_n^i -1 & j=i,\ i\ odd\\
0 & i=n,\ j\mbox{ odd and } j\leq n \\
2 & i=n,\ j\mbox{ even and } 0< j\leq n.
\end{cases}
\end{equation}
\end{lem}

\begin{proof}
Let $R$ denote the ring $\Q\langle \langle x,y\rangle \rangle [[T]]$ of power series in which $x$ and $y$ do not commute with each other, but $T$ commutes with both $x$ and $y$.  We define the power series in $R$:
\begin{align*}
F(x,y,T) &:= \frac{1}{1-(xT+yT^2)} = \sum_{r=0}^{\infty} (xT+yT^2)^r , \\
G(x,y)& := \frac{1}{1-(x+y)} = \sum_{s=0}^{\infty} (x+y)^s.
\end{align*}

Set $H=FG$.
For any $J\in R$, let $J^{ab}$ be its image in the commutative power series ring, $\Q[[X,Y, T]]$ by the map which sends $x$ to $X$ and $y$ to $Y$.  In particular, the power series expansion of $P$ is equal to $(H +yT^2H)^{ab}$. 

{\it Case 1, i=0}:
For $0\leq j\leq n$, we have $(P|X^nT^j)= C_{n}^{i}=1$ because
$(F^{ab}|X^jT^j)=1$ and $(G^{ab}|X^{n-j})=1$.  For $j=n+1$, $X^nT^{n+1}$ 
has coefficient $0=C_{n}^i-1$ in $P$
because the factor of $T$ must come from
$F^{ab}$, and in this expansion, since the number of $T$ in a term in $F^{ab}$ is equal 
to the sum of the powers of $X$ and twice the powers of $Y$.

{\it Case 2, j=0}:  The coefficients, $X^{n-i}Y^i$, come uniquely from
$G^{ab}$, and by the binomial expansion, this coefficient is $C_n^i$.

{\it Case 3, $i<n$, $i\leq j \leq n+1$}:

 We say that a monomial in $R$ is of ``type $(n-i,i)$'' if it is of degree $n-i$ in $x$ and degree $i$ in $y$.  Let
\begin{equation}
     \epsilon = \begin{cases} -1 & j=i, \ i \hbox{ odd or } j=n+1,\ i \hbox{ even}\\
                    +1 & j=i, \ i \hbox{ even or } j=n+1,\ i \hbox{ odd}\\
0 & i+1\leq j \leq n.
                   \end{cases}
\end{equation}
We say that a monomial in $R$ is ``divisible at $j$'' if it can be written $V\cdot WT^j$ where if $\alpha$ is the degree in $y$ in $V$ and $\beta$ the degree in $x$ in $V$, then $j=2\alpha + \beta$.  We say in this case that $V$ is of weight $j$ in $T$.  A monomial, $V\cdot WT^j$, is divisible at $j$ if and only if $VT^j$ is a monomial appearing in $F$.

The coefficient of every monomial in $H$ is 1, hence the coefficient $(P|X^{n-i}Y^{i}T^{j})$ is equal to the number of monomials of type $(n-i,i)$ in $H$ which are divisible at $j$ plus the number of monomials of type $(n-i,i-1)$ in $H$ which are divisible at $j-2$.
We will partition the entire contribution to $(P|X^{n-i}Y^{i}T^{j})$ of monomials from $H+yT^2H$ into the three following sets: 
\begin{align*}
&\# \{\hbox{type }(n-i,i)\hbox{, beginning with }x\hbox{ and divisible at }j\} + \\ & \qquad \#\{\hbox{type }(n-i,i)\hbox{, beginning with }y\hbox{ and divisible at }j\} +\\ & \qquad\#\{\hbox{type }(n-i,i-1)\hbox{ and divisible at }j-2\}.\end{align*}

But multiplying the elements of the third set on the left by $yT^2$ yields a bijection between the second and third sets so that
$$(P|X^{n-i}Y^{i}T^j)= \#\{xV\hbox{ divisible at }j\} + 2\#\{yV\hbox{ divisible at }j\}.$$
Note that the total number of monomials of type $(n-i,i)$, with $0<i<n$ and $i\leq j\leq n+1$ is just $C_n^i$, so we have,
$$C_n^i = \#\{xV\hbox{ divisible at }j\} + \#\{yV\hbox{ divisible at }j\} + \#\{V\hbox{ not divisible at }j\}.$$

Thus to show that $(P|X^{n-i}Y^iT^j) = C_n^i + \epsilon$, we only need to show that the cardinalities of the two following sets of words of type $(n-i,i)$ with $0<i<n$ and $i\leq j\leq n+1$ satisfy:
\begin{equation}\label{leilaind} \#\{yV\hbox{ divisible at }j\} = \#\{V\hbox{ not divisible at }j\} + \epsilon.\end{equation}


{\it Base case, $n=2$, $(i,j) = (1,1), (1,2), (1,3)$}:  The words can be counted by hand.  In the first case, there are no words starting with $y$, divisible at 1, and one word, $yx$, not divisible at 1.  In the second case, $\#\{y|x\} = 1 = \#\{xy\}$.  And in the third case, there is one word, $yx$, starting with $y$ and divisible at 3 and no words which are not divisible at 3.\\

{\it Induction Case}: Now we assume the induction hypothesis that $\#\{yV|W, \hbox{ type}(n-i,i-1),\hbox{ divisible at } j-1 \}+\epsilon = \#\{V|yW,\hbox{ type}(n-i,i-1),\hbox{ divisible at }j-2\}$
(where the $\epsilon$ is on left hand side since the parity of the number $y$ is different in this induction hypothesis).

  We now partition the left hand set of \eqref{leilaind} into the two subsets of monomials of the form:
\begin{equation}\label{leilalem5} \{yV\hbox{ divisible at }j\} = \{yV|xW\} \sqcup \{yV|yW\}, \end{equation}

Furthermore, the set of words that is not divisible at $j$ is equal to the set of words divisible at $j-1$ of type $(n-i,i)$ of the form:
\begin{align}\label{leilalem3} \{V\hbox{ not divisible at }j\} & = \{V|yW, \hbox{ divisible at } j-1\}\notag \\
& = \{Vx|yW\} \sqcup \{Vy|yW\}.\end{align}

The first set of \eqref{leilalem5} is in bijection with the first set of \eqref{leilalem3} by permutation of the terms:
\begin{equation*}
\{yV|xW,\hbox{ divisible at }j\} \leftrightarrow \{Vx|yW,\hbox{ divisible at }j-1\}.
\end{equation*}

Furthermore, we have the following equalities relating the second sets of \eqref{leilalem5} and \eqref{leilalem3}:
\begin{align*}\#\{yV|yW\}&= \#\{V|yW,\hbox{ type }(n-i,i-1),\hbox{ divisible at }j-2\}\ \hbox{ (removal of leading }y)\\
& =\#\{yV|W,\hbox{ type }(n-i,i-1),\hbox{ divisible at }j-1\} + \epsilon\qquad \ \ \hbox{ (induction)}\\
&=\#\{Vy|yW\} +\epsilon \qquad \ \ \hbox{ (by permutation of the terms and removal of $y$)}.
\end{align*}

This proves the lemma for $i<n,\ i\leq j \leq n+1$.


{\it Case 4, $i=n$, $0<j\leq n$}:  If $j$ odd,
there is no way to cut a word, $y^\alpha|y^\beta T^{j}$ in such a way
that $2\alpha=j$.  So, $(P|Y^iT^{j})=0$.  Likewise, if $j$ is even, all words are divisible, therefore $P$ has a coefficient of 1 coming from $H^{ab}$ and a coefficient of 1 coming from $ (yT^2H)^{ab}$.
\end{proof}

\begin{rmk*}
I would like to thank the reporter D. Zagier, for suggesting an alternative proof of this lemma, which is shorter and provides all of the coefficients of $P$.
\end{rmk*}

\begin{cor}\label{a_n}\index{$Q$}
The coefficient of $X^{n-i}Y^{i}T^{j}$ in the rational function, 
\begin{equation}\label{Q}
Q(X,Y,T)=\frac{(1+YT^2)(1+X)}{(1-XT-YT^2)(1-X-Y)}
\end{equation} is equal to \begin{equation}
\begin{cases}C_{n}^{i} + C_{n-1}^{i} & i+1 \leq j\leq n-1 \\
C_{n}^{i} + C_{n-1}^{i} -1 & j=n,\ i\ even,\ i<n\\
C_{n}^{i} + C_{n-1}^{i} +1 & j=n,\ i\ odd,\ i<n\\
2 & i=n,\ j \mbox{ even and }\leq n\\
0 & i=n,\ j\mbox{ odd and }\leq n 
 \end {cases}
\end{equation}
\end{cor}

\begin{proof}
The rational function, $Q(X,Y,T)=P(X,Y,T)+XP(X,Y,T)$ where $P$ is the
same as in lemma \ref{Biglem}.  The coefficient, $(Q| X^{n-i}Y^iT^j)$
is equal to
$(P|X^{n-i}Y^iT^{i+1})+(P|X^{n-i-1}Y^iT^{i+1})$.\end{proof}

\begin{defn}\label{TDEF}\index{$\Lambda$, $\Lambda_D,\ \Lambda_A$}
Let $\Lambda$ be the Pascal triangle obtained by the recurrence relation, $\Lambda_0^0=1$,
$\Lambda_1^0=2$, $\Lambda_1^1 =1$, $\Lambda_n^k=0$ for $i<0$ or $k<0$, and
$\Lambda_n^k=\Lambda_{n-1}^{k-1} + \Lambda_{n-1}^k$. \end{defn}

We note here that $\Lambda$ is the sum of two Pascal triangles with binomial
coefficients interposed on top of one another, one with its tip on
the first column of the second row of the other.
We have then that $\Lambda_n^k=C_n^k +
C_{n-1}^k,\ 0\leq k \leq n $.  In this expression, $\Lambda_n^k$ is
$k$th column of the $n$th row of $\Lambda$:
\begin{center}
\makebox{\xymatrix @!0 @R=5mm @C=5mm{
&&&& 1&&&& \\
&&&2&&1&&& \\
&&2&&3&&1&& \\
&2&&5&&4&&1& \\
2&&7&&9&&5&&1. }}
\end{center}

Now we associate commutative monomials to terms of $\Lambda$ in two different ways.
Let $(\Lambda_D)_n^k = \Lambda_n^k X^kY^{n-k}T^{k+2(n-k)}$ and $(\Lambda_A)_n^k =
T_n^kX^{n-k}Y^k$: 
\begin{center}
\makebox{\xymatrix @!0 @R=5mm @C=6mm{
 & & & & & & & &1& & & & & & & & \\
\\
&&&&&&&2YT^2& &1XT&&&&&&& \\
&&&\ar[rrrrrrrrrrrrdddddd] &&&\ar[rrrrrrrrdddd] &&&&&&&&&&&&&&\\
\Lambda_D=& & & & &2Y^2T^4& & &3XYT^3& & &1X^2T^2&&&&& \\
&&&&&&&&&&&&&&&&&&\\
 & & &2Y^3T^6& & &5XY^2T^5& & &4X^2YT^4& & &1X^3T^3&& & && \\
&&&&&&&&&&&&&&D_3&&&&&\\
 & &2Y^4T^8&&&7XY^3T^7&&&9X^2Y^2T^6&&&5X^3YT^5&&&1X^4T^4&&\\
&&&&&&&&&&&&&&&D_4&&&\\
}}
\end{center}
The descending arrows, $D_i$ represent the vectors of monomials with $T^i$ formed from the corresponding descending diagonal of $\Lambda_A$.
\begin{center}
\makebox{\xymatrix @!0 @R=5mm @C=6mm{
 & & & & & & & &1& & & & & & & & \\
\\
&&&&&&&2X& &1Y&&&A_3&&&& \\
&&&&&& &&&&&&&A_4&&\\
\Lambda_D=& & & & &2X^2& & &3XY& & &1Y^2&&&&& \\
&&&&&&&&&&&&&&&&&&\\
 & & &2X^3& & &5X^2Y& & &4XY^2& & &1Y^3&& & && \\
&\ar[rrrrrrrrrrruuuuu] &&&&&&&&&&&&&&&&\\
 & &2X^4&&& 7X^3Y&&&9X^2Y^2&&&5XY^3&&&1Y^4&&\\
\ar[rrrrrrrrrrrrruuuuuu]
}}
\end{center}
The ascending arrows represent vectors, $A_i$, such that the power on $X$ plus twice the power on $Y$ equals the constant $i$.

We define the two rational functions, $Q_1$ and $Q_2$ such that $Q=Q_1Q_2$:
\begin{align}
Q_1&=\frac{1+YT^2}{1-(XT+YT^2)}=\sum_{i=0} (XT+YT^2)^i +
YT^2\sum_{i=0} (XT+YT^2)^i\\ Q_2 &= \frac{1+X}{1-(X+Y)} = \sum_{i=0} (X+Y)^i +
X\sum_{i=0} (X+Y)^i.
\end{align}

The expansion of $Q_1$
gives us $(Q_1 | X^{n-i}Y^iT^{n+i}) = C_{n}^{n-i}+ C_{n-1}^{n-i}=
\Lambda_n^{n-i}$, and so $Q_1$ is the infinite sum of all of the terms of $\Lambda_D$.
Likewise, $Q_2$ is the infinite sum of all terms in $\Lambda_A$.

The arrows in the triangles $\Lambda_D$ and $\Lambda_A$ represent descending and
ascending vectors, $D_j$ and $A_k$, such that their scalar product 
is a monomial in $Q_1Q_2$, by adding the extra condition 
that we must add some zeroes at the beginning of the vectors, $D_j$,
according to the following rule.

\begin{defn}\label{D_jz}\index{$D_j,\ D_{j,z}$}  Let $D_{j,z}$ be the vector of $z$ zeroes
  concatenated 
  by the descending diagonal vector $D_j$ in $\Lambda_D$ as in the above diagram.  A closed
  formula for $D_{j,z}$ is given by the formula:
\begin{equation}
\begin{split}
D_{j,z} &=T^j\Bigl( \overbrace{0,0,\cdots,0}^{z},\Lambda_{\lfloor \frac{j+1}{2}
  \rfloor}^{j\ \mathrm{mod} \ 2}
  Y^{\lfloor \frac{j}{2}
  \rfloor} X^{j\ \mathrm{mod}\ 2}, \cdots ,\Lambda_{j-2}^{j-4}Y^2X^{j-4},
  \Lambda_{j-1}^{j-2} YX^{j-2}
  , \Lambda_{j}^{j} X^j \Bigr) \\
&=T^j \Bigl( 0^{z} \cdot \Bigl( \Lambda_{\lfloor \frac{j+1}{2}
  \rfloor +l}^{ j-2\lfloor \frac{j}{2}
  \rfloor +2l} Y^{\lfloor \frac{j}{2} \rfloor -l} X^{j-2\lfloor
  \frac{j}{2}\rfloor +2l} \Bigr)_{l=0}^{\lfloor \frac{j}{2}\rfloor} \Bigr),
\end{split}\end{equation}
where $T^j$ is distributed to all of the terms in the vector.
\end{defn}

\begin{defn}\label{A_kdef}\index{$A_k$}  We define the vectors, $A_k$ associated to
  $\Lambda_A$ similarly:
\begin{equation}
\begin{split}
A_k &= \bigr( \Lambda_{k}^0 X^{k}, \Lambda_{k-1}^1 X^{k-2}Y,\cdots,
  \Lambda_{k-\lfloor\frac{k}{2} \rfloor}^{\lfloor\frac{k}{2} \rfloor} X^{k-2\lfloor
  \frac{k}{2} \rfloor} Y^{\lfloor\frac{k}{2} \rfloor}\bigl)\\
&= \Bigl( \Lambda_{k-m}^m Y^mX^{k-2m} \Bigr)_{m=0}^{\lfloor \frac{k}{2} \rfloor}
\end{split}
\end{equation}\end{defn}

\begin{defn}  The scalar product, $A_k\cdot D_{j,z}$,
is defined as 
\begin{multline}\label{right}
\sum _{m=z}^{\lfloor \frac{k}{2}\rfloor}
\Lambda_{k-m}^m Y^mX^{k-2m}\cdot \Lambda_{\lfloor\frac {j+1}{2} \rfloor + ( m-z)}^{j -
  2\lfloor \frac{j}{2} \rfloor +2(m-z) } Y^{\lfloor \frac{j}{2} \rfloor
  - (m-z)}
X^{j - 2\lfloor \frac{j}{2} \rfloor +2(m-z) }T^j = \\ \Bigl(  \sum _{m=z}^{
  \lfloor\frac{k}{2}\rfloor}
\Lambda_{k-m}^m \cdot \Lambda_{\lfloor\frac {j+1}{2} \rfloor + ( m-z)}^{j -
  2\lfloor \frac{j}{2} \rfloor +2(m-z) }   \Bigr) X^{k+j-2 \lfloor
  \frac{j}{2} \rfloor -2z}Y^{\lfloor \frac {j}{2} \rfloor +z}T^j.
\end{multline}
\end{defn}

\begin{lem}\label{AdotD}
The term $a_{n-i,i,j}X^{n-i}Y^iT^{j}$ ($i+1 \leq j \leq n$) in
$Q(X,Y,T)$ is equal to the scalar product $A_{n+i-j}\cdot
D_{j,i-\lfloor\frac{j}{2}\rfloor}$, where $a_{n-i,i,j}$ is given by corollary
\ref{a_n}.
\end{lem}

\begin{proof}
By the expression \eqref{right}, we know that $A_{n+i-j}\cdot
D_{j,i-\lfloor\frac{j}{2}\rfloor}$ is
a monomial of the right form.  Now, $A_{n+i-j}\cdot
D_{j,i-\lfloor\frac{j}{2}\rfloor}$
is exactly the term, $a_{n-i,i}X^{n-i}Y^iT^{j}$ 
in the product
$Q_1Q_2$.  But $Q=Q_1Q_2$, by \eqref{Q} and this proves the lemma.
\end{proof}

\section{$\grt\simeq \ds$ in depth 1}\label{parti}

The goal of this section is to prove that $F_n^1\ds$ is
non-empty for odd $n$ and that $F_n^1\ds=0$ for even $n$.  The proof of this
part, for odd $n$, relies on
sophisticated machinery due to Racinet, Furusho and their inspirations which
include, among others, Drinfel'd, Ihara, Le and Murakami.  The style of this
proof is very different from the
combinatorial nature of the rest of this chapter.  One reason for this
is that this result follows almost immediately from theirs, therefore
a long combinatorial construction was unnecessary.  We
prove the existence, without explicit construction, of depth 1
elements of $\ds$ which are conjectured to be 
the generators of $\ds$ as a Lie algebra, and are therefore
important objects in furthering the study of multizeta values.

\begin{proof}[Proof of theorem \ref{theoreme} (i)] 
Let $f\in \ds\subset\Lxy$ and we write $f$ in a Lyndon-Lie basis as in
expression \eqref{f}.  Since $F_n^1\ds \subset \Lxy$ and since $\Lxy$
has only one depth 1 element in each weight $n$, namely $[x^{n-1}y]$,
any two elements in $F_n^1\ds$ are
equivalent modulo $F_n^2\ds$.  Therefore, we may assume that $dim\
F_n^1\ds/F_n^2\ds=0$ or 1.
Without loss of generality, we may assume that $(f|[x^{n-1}y])=0$ or 1.

{\it Case 1, $n$ is even}:
The stuffle relation, $(n-1)*(1) = (n-1,1)+(1,n-1) +(n)$, gives
$$0=(f|x^{n-2}y^2) + (f|yx^{n-2}y) + 
(f|x^{n-1}y).$$  From lemma \ref{bw}, $(f|yx^{n-2}y)=0$.  We have then that
$$(f|x^{n-2}y^2)=-(f|x^{n-1}y).$$
Let's assume that $(f|x^{n-1}y)\neq 0$.  But by \ref{b_0},
$$(f|x^{n-2}y^2)=\frac{n-1}{2}(f|x^{n-1}y),$$ 
which is a contradiction since
$n\neq -1$.  Therefore, $(f|x^{n-1}y)=0$, and so $F_n^1\ds/F_n^{2}\ds=0$. 

{\it Case 2, $n$ is odd}:

In order to treat this case, we introduce Drinfel'd's associator.
\begin{defn}\label{Phikz}\index{$\Phi_{KZ}$} The Drinfel'd associator, $\Phi_{KZ}$, is defined as 
\begin{equation}\Phi_{KZ}(x,y) =\sum (-1)^{d(w)}\zeta^{sh} (w)w \in
  \Cxy,\end{equation}  
where the sum is a power series over all of the
monomials, $w$, in
$x$ and $y$.
The coefficients, $\zeta^{sh}(w)$, are real numbers, called
regularized zeta values, which have the following 
properties: \begin{enumerate}
\item If $w$ is a convergent word, then $\zeta^{sh}(w) =\zeta(w)$,
\item For all non-convergent words, $w$, the $\zeta^{sh}(w)$ are 
  linear combinations of convergent multizeta values that satisfy the
  property that
  $\zeta^{sh}(w_1) \zeta^{sh}(w_2) = \zeta^{sh}(w_1\sha w_2)$.
\end{enumerate}
An explicit expression for $\zeta^{sh}(w)$  was calculated by Furusho \cite{Fu}
and is based on work of Le and
Murakami \cite{LM}.
We set $\Phi_{\sha}$\label{Phisha}\index{$\Phi_{\sh}$} to be $\Phi_{KZ}(x,-y)$.
\end{defn}

The group, $DM$\label{DMDEF}\index{$DM,\ DM_{\gamma}$}, or ``double m\'elange'', which was
defined by Racinet and 
is related to $\ds$, plays a key role in this proof.  The group,
$DM\subset \Cxy$,
is graded by 
weight and its weight $n$ graded piece is denoted $DM^n$.  The group
law on $DM$ is denoted by $\circledast$.  For an explicit description
of generators and relations on $DM$, see \cite{Ra}.  

\begin{defn}
Let $DM_\lambda$ be set of power series, $f\in DM$, such that
$(f|xy)=\lambda$.
\end{defn}

The series, $\Phi_{\sha}(x,y)$ and
$\Phi_{\sha}(-x,-y)$, are both elements of $DM_{\zeta(2)}$.
Racinet proved that for all $F\in DM_0$, the weight $n$ part of the
power series, $(\ln F)^n$, satisfies the
double shuffle relations given in \ref{dsdef}.

\begin{thm}\cite{Ra}
The set, $DM_0$, is a group that acts transitively on $DM_{\zeta(2)}$.
\end{thm}

As a consequence, there exists
an element, $F\in DM_0$
such that
\begin{equation}\label{F}{\Phi}_{\sha}(x,y)\circledast F=
\Phi_{\sha}(-x,-y).\end{equation}  This $F$ provides us with a depth 1
double shuffle element according to the following construction. 

To see what happens in the Lie algebra, $(\ln DM_0)$, we take $\ln$ of both
sides of the equation to obtain,
\begin{equation}\label{ch}\ln{\Phi_{\sha}}(x,y)\underset{\!C\!-\!H\!}{\star}
    \ln F = \ln\Phi_{\sha}(-x,-y),
\end{equation}
where $\underset{\!C\!-\!H\!}{\star}$ is the multiplication law by the
Campbell-Baker-Hausdorff formula.  The goal is to show that
$(\ln F|x^{n-1}y )$ cannot be 0 for
odd $n$, in order to give a non-zero element in
$F^1_n\ds$ over $\Q$.  To do this, we write the expansion of
$\ln\Phi_{\sha}$ and
$\ln F$ as a power series in a Lyndon-Lie basis of
$\mathbb{L}_{\C}[x,y]$.  To lighten the notation, we will write $\ln F=f$.

It is known that for any
element in $F\in DM_0$ and for any $i$, $(F|x^i)=(F|y^i)=0$ and that
for any $i$, $(\Phi_{\sha}|x^i)=(\Phi_{\sha}|y^i) =0$ \cite{Ra}
\cite{Dr}. Let $\Phi_{\sha}(x,y)=1+\Phi_0(x,y)$.  By writing 
$$\ln\Phi_{\sha} = \ln(1+\Phi_0) = \Phi_0 - \frac{1}{2} \Phi_0^2
+\cdots$$
we see that the contribution to the monomials,
$x^{n-1}y$, in the Lie algebra
comes uniquely from the first term, $\Phi_0$.
And we have by a formula of Furusho \cite{Fu},
\begin{equation*}\begin{split}\ln \Phi_{\sha}(x,y) &=\zeta(2)[x,y]
    + \zeta(3) [x,[x,y]] + 
\zeta(4)[x,[x,[x,y]]] + \cdots \\
\ln \Phi_{\sha}(-x,-y)&=\zeta(2)[x,y] - \zeta(3) [x,[x,y]] +
\zeta(4)[x,[x,[x,y]]] + \cdots \end{split}\end{equation*}

By the formula of 
    Campbell-Baker-Hausdorff, \eqref{ch} becomes
\begin{multline}\label{verybeautifuleq}(\zeta(2)[x,y]+\zeta(3)[x,[x,y]]
    +\zeta(4)  
[x,[x,[x,y]]]+\cdots) + f +\frac{1}{2} \{\ln\Phi_{\sha}, f \} +\cdots=\\
\zeta(2)[x,y]-\zeta(3)[x,[x,y]] +\zeta(4)
[x,[x,[x,y]]]+\cdots.\end{multline}

We have the following important properties:
$(f|x)=(f|y)=0=(\ln\Phi_{\sha} |x) = (\ln\Phi_{\sha} | y)$.  Hence, both 
$\ln \Phi_{\sha}$ and $f$ are in
$[\mathbb{L}_{\C}[x,y],\mathbb{L}_{\C}[x,y]]$.  Because of this, the
terms of bracketings of $f$ and $\ln\Phi_{\sha}$ 
do not contribute to the coefficient of $[x^{n-1}y]$ since $[x^{n-1}y]\notin
[[\mathbb{L}_{\C}[x,y],\mathbb{L}_{\C}[x,y]],
[\mathbb{L}_{\C}[x,y],\mathbb{L}_{\C}[x,y]] ]$.  

By studying both sides of equation \eqref{verybeautifuleq}, we see that for even $n$,
$(f|x^{n-1}y)=0$ and for odd $n$, $(f|x^{n-1}y)=-2\zeta(n)$.  Let 
$ f_n$ be the homogeneous degree $n$ graded part of $f$.  In order to
find a non-zero element in $\ds$, we write $$f_n=\sum a_j L_n^j,\
a_i\in\C$$ where $\{L_n^j\}$ is a Lyndon basis of $\Lxy$.  Let $V$ be
the $\Q$-vector space generated by the complex numbers,
$\langle a_i \rangle$, so we may choose the basis for $V$ over $\Q$,
$\{\zeta_1=\zeta(n), \zeta_2,\dots, \zeta_r\}$.

We then write,
\begin{equation*} f_n = \sum_{i=1}^r p_i(x,y)\zeta_i
\end{equation*} where 
the $p_i(x,y)$ are homogeneous polynomials of degree $n$ with
coefficients in $\Q$.  Since $f_n$ satisfies double shuffle,
for any couple of words, $u,v$, we have
\begin{equation*}\begin{split}0 &= \sum_{w\in\ sh/st(u,v)}
    (f_n|w)
    \\ &=\sum_{w\in\ sh/st(u,v)} (\sum_{i=1}^r p_i(x,y)\zeta_i|w)\\
& =\sum_{i=1}^r\bigl(\sum_{w\in\ sh/st(u,v)}
(p_i(x,y)|w)\bigr)\zeta_i.
\end{split}\end{equation*}  The $\zeta_i$ are linearly independent by
    hypothesis, so we see that
each $p_i(x,y)$ satisfies double shuffle.
Therefore we have, $(p_1|x^{n-1}y)=-2$ and $p_1$ furnishes us with a
nonzero element in 
$F_n^1\ds$. 
\end{proof}

\noindent{\bf Remark.}
It is tempting to take the log of the element, $\Phi_{\sh}$, in order
to obtain a non-zero element of $F_n^1\ds$, because the
coefficients of $\Phi_{\sha}$ are multizeta values, and 
multizeta values satisfy the double shuffle relations.  There is still
some work to do in this construction, but I will not outline it, since
the proof fails at a certain point which I will explain here.  In
fact, $\ln \Phi_{\sha}\in \widehat{\bigoplus}_{n\geq 2} (\nz_n\otimes_{\Q}
\ds_n)$, where $\nz$ is the new 
zeta space from definition \ref{nznumbersdef}.  This proof then fails
in the last step of the correct proof above, because $\nz$ is modded
out by $\Q$, and since we do not know whether $\zeta(n)$ is
irrational, it might be identically 0 in $\nz$, and therefore we may
not take it as 
a basis element in $\nz$ over $\Q$.

\section{$\grt \simeq \ds$ in depth 2}\label{partii}

In this section, we study $F^2_n\ds / F^3_n\ds$ by looking at the
coefficients of polynomials in $\ds$ considered as a subspace of $\Lxy$.
In section \ref{PQTDA}, we showed some combinorial properties of
depth 2 words.  Now, we use these properties for, first of all,
showing that for $n$ odd, all of the coefficients on depth 2 monomials
only depend on the coefficient of $[x^{n-1}y]$ and deduce from that
that for $n$  odd, $F^2_n\ds/ F^3_n\ds=0 $.  Analogously, we prove
that the
dimension of $F^2_n\ds / F^3_n\ds$ for $n$ even is
$\lfloor \frac{n-2}{6}
\rfloor$, thereby making an important connection between the Lie algebras,
$\grt$ and $\ds$.

\begin{proof}[Proof of \ref{theoreme} (ii).] Recall that we may write
  $f\in \ds$ as in equation \eqref{f},
\begin{equation*}\begin{split}
f &=A[x^{n-1}y] + \sum_{s=0}^{\lfloor\frac{n-3}{2} \rfloor} a_s
 [x^ryx^sy] +\dots\\ 
 &= Ax^{n-1}y + b_0 x^{n-2}y^2 +b_1x^{n-3}yxy +\cdots
 +b_{n-2}yx^{n-2}y +\dots
\end{split}\end{equation*}
{\it Case 1, n is odd}: We have the following
$\frac{n-1}{2} $ relations on the coefficients given by
stuffle:
\begin{equation}\label{stuf2}
b_{i-1} + b_{n-1-i} + A = 0, \ 1\leq i \leq  \frac{n-1}{2}.
\end{equation}
By substituting the relation from lemma \ref{b_i} between the $b_i$
and the $a_i$
into \eqref{stuf2}, we have the following system of relations on the
coefficients, $a_i$, for each $i$, $0\leq i\leq
\frac{n-3}{2}$:  
\begin{multline}\label{ajsystem}
\sum_{j=0}^{i}
(-1)^{i-j}a_j\Bigl(C_{j+1}^{i-j}+C_{j}^{i-j-1}\Bigr) + \\
\sum_{k=0}^{n-2-i}
(-1)^{n-2-i-k} a_k\Bigl( C_{k+1}^{n-2-i-k} + C_{k}^{n-3-i-k} \Bigr) =
-A.\end{multline}  

The system given in \eqref{ajsystem} may be solved by finding a
solution to the matrix equation,
$M \cdot (a_0,a_1,\dots,a_{\frac{n-3}{2}}
)=-A(1,1,\dots,1)$, 
where the matrix $M$ is given by\index{$M$}
\begin{multline}\label{MDEF}
M(i,j) = (-1)^{i-j}\Bigl( C_{j}^{i-j} +
C_{j-1}^{i-j-1}\Bigr) + \\(-1)^{n-i-j}\Bigl( C_{j}^{n-i-j} +
C_{j-1}^{n-1-i-j} \Bigr),\ 1\leq i,j\leq \frac{n-1}{2}
,\end{multline} where a row $i$ of $M$ corresponds to the equation
$b_{i-1} + b_{n-1-i}=-A$, so the $j$th entry in that row is equal to
the coefficient of $a_{j-1}$ in $b_{i-1} + b_{n-i-1}$.

\vspace{.3cm}
\noindent{\bf Remark.} Let $M'$ be the block matrix in the upper
left-hand corner of $M$ for $i,j\leq 
\frac{n-1}{3} $.  We have in this block that
$M'(i,j)=(-1)^{i-j}\Lambda_j^{2j-i}$ and the
$i$th row of $M'$ 
is equal, up to sign, to the vector of coefficients of $D_i$, such
that the right-most term in the vector is
on the diagonal, and where the sign is given by $(-1)^{i-j}$.
For $i\leq \frac{n-1}{3},\ j>\frac{n-1}{3}$ we have
$M(i,j)= (-1)^{n-i-j}\Lambda_j^{2j-n+i}$.  The other terms of $M$
are the sum of these two types.  
This structure permits us to associate certain rows of $M$ to vectors
$D_{i,z}$ and certain columns to rows of $\Lambda$.

\begin{ex}
For $n=11$, the matrix, $M$, is given by
\[ \left( \begin{array}{rrrrr}
1& 0 & 0 & 0 & -2 \\
-2& 1& 0& 0& 9\\
0& -3& 1& 2 & -16\\
0 & 2& -4 & -6&14\\
0& 0 & 3 & 4 & -5 \end{array} \right) .\]
\end{ex}

Let $N$\label{NDEF}\index{$N$} be the $\frac{n-1}{2}$ square invertible matrix
such that the top of the $\bigl( \frac{n-1}{2} - k \bigr)$th column is equal to
$(\overbrace{0,...,0}^{2k})\cdot
A_{\frac{n-3}{2}-3k}, \ 0\leq k\leq \lfloor \frac{n-5}{6} \rfloor$ up
to sign, and let the sign be given by $sgn(N(i,j))=(-1)^{i+j}$. 
Let the other coefficients of $N$ be those of the identity matrix.

Now we show that $MN$ is a lower triangular matrix.  From the corollary 
\ref{AdotD}, above the diagonal, in the $k$th row and the $i$th
column, $i<\frac{n-1}{2}-k$, we have
\begin{align*} MN(i,\frac{n-1}{2} - k ) & = (-1)^{i+\frac{n-1}{2}-k}
  D_{i,\lfloor \frac{i-1}{2}\rfloor -2k} \cdot A_{\frac{n-3}{2}-3k} +
  M(i,\frac{n-1}{2}-k) \\
&=(-1)^{i+\frac{n-1}{2}-k} \bigl( Q|X^{\frac{n+1}{2}+k-i}Y^{i-1-2k}
  T^i \bigr)  \\
  & \qquad +(-1)^{\frac{n+1}{2}-i+k} \Bigl( C_{\frac{n-1}{2}-k}^{i-1-2k} 
  +C_{\frac{n-3}{2}-k}^{i-1-2k} \Bigr)\\
& = (-1)^{\frac{n-1}{2}-k+i} \Bigl( C_{\frac{n-1}{2}-k}^{i-1-2k}
  +C_{\frac{n-3}{2}-k}^{i-1-2k}\Bigr) \\ &
  \qquad+(-1)^{\frac{n+1}{2}-i+k} \Bigl(
  C_{\frac{n-1}{2}-k}^{i-1-2k} 
  +C_{\frac{n-3}{2}-k}^{i-1-2k} \Bigr) \\
& = 0.
\end{align*}
Therefore, we have that $MN$ is lower triangular.  We now need to show
that the terms on the diagonal of $MN$ are non-zero to show that $M$
is invertible.

If $k> \lfloor \frac{n-5}{6} \rfloor$, then the $\bigl(
\frac{n-1}{2}-k\bigr) $th
diagonal term of $MN$ is that of $M$, since $N$ is the identity in this
upper, right block.  These terms are -1 or 1.

If $0\leq k \leq \lfloor \frac{n-5}{6} \rfloor$ and
$\frac{n-3}{2}-3k$ is odd, the term on the diagonal,
$MN(\frac{n-1}{2}-k,\frac{n-1}{2}-k)$, for
$\frac{n-1}{2}-k \geq \lfloor \frac{n}{3} \rfloor$, is equal to 
\begin{equation*}\begin{split}
  MN(\frac{n-1}{2}-k,\frac{n-1}{2}-k)&=D_{\frac{n-1}{2}-k,\lfloor
  \frac{n-3-2k}{4}\rfloor -2k} \cdot A_{\frac{n-3}{2}-2k} +
  M(\frac{n-1}{2}-k, \frac{n-1}{2}-k)\\
&= (Q|X^{2k+1}Y^{\frac{n-3}{2}-3k} T^{\frac{n-1}{2}-k}) +
  C_{\frac{n-1}{2}-k}^{ 0} - C_{\frac{n-1}{2}-k}^{2k + 1}-
  C_{\frac{n-3}{2}-k}^{2k}\\
&= C_{\frac{n-1}{2}-k}^{\frac{n-3}{2}-3k}
  +C_{\frac{n-3}{2}-k}^{\frac{n-3}{2}-3k} +1 +1 -
  C_{\frac{n-1}{2}-k}^{\frac{n-3}{2}-3k} -
  C_{\frac{n-3}{2}-k}^{\frac{n-3}{2}-3k} \\
&=2.\end{split}\end{equation*}

When $\frac{n-3}{2}-3k$ is even, 
the length of the $\bigl( \frac{n-1}{2}-k\bigr)$th
column of $N$ ($0^{2k}A_{\frac{n-3}{2}-3k}$) is $\frac{n+1+2k}{4}$.  The
$\bigl( \frac{n-1}{2}-k\bigr) $th row of $M$ is equal to
$D_{\frac{n-1}{2}-k,\frac{ n-3-2k}{4}} -
  D_{\frac{n-1}{2}+k+1,\frac{n+1+2k}{4}}$. 
So the only term of
  $D_{\frac{n-1}{2}+k+1, \frac{n+1+2k}{4}}$ 
that plays a role in the scalar product is the first and it is equal
to -2 when
  $\frac{n-3}{2}-3k$ is even.  
Furthermore, this term contributes to the scalar product by
multiplication of the last term of $A_{\cdot}$ which is 1 in this case.

So we have,
\begin{equation*}\begin{split}
  MN(\frac{n-1}{2}-k,&\frac{n-1}{2}-k)\\ &=D_{\frac{n-1}{2}-k,\lfloor
  \frac{n-3-2k}{4}\rfloor -2k} \cdot A_{\frac{n-3}{2}-2k} +
  M(\frac{n-1}{2}-k, \frac{n-1}{2}-k)-2\\
&= (Q|X^{2k+1}Y^{\frac{n-3}{2}-3k} T^{\frac{n-1}{2}-k}) +
  C_{\frac{n-1}{2}-k}^{ 0} - C_{\frac{n-1}{2}-k}^{2k + 1}-
  C_{\frac{n-3}{2}-k}^{2k}-2\\
&= C_{\frac{n-1}{2}-k}^{\frac{n-3}{2}-3k}
  + C_{\frac{n-3}{2}-k}^{\frac{n-3}{2}-3k} -1 +1 -
  C_{\frac{n-1}{2}-k}^{\frac{n-3}{2}-3k} -
  C_{\frac{n-3}{2}-k}^{\frac{n-3}{2}-3k}-2. \\
&=-2.\end{split}\end{equation*}

So for $n$ odd, we have that the matrix $M$ is invertible, so that
there exists a unique solution to the matrix equation, namely $M^{-1}
\cdot
  (-A,-A,\dots,-A)=(a_0,a_1, \dots a_{\frac{n-3}{2}})$, 
and the Lyndon-Lie basis elements are $\Q$ linear combinations of
$A$, showing that
$$F_n^2\ds / F_n^3 \ds=0.$$
 
{\it Case 2, $n$ is even}:
We have the following
$\frac{n-2}{2} $ relations on the coefficients given by
stuffle:
\begin{equation}\label{stuf2even}
b_{i-1} + b_{n-1-i} + A = 0, \ 1\leq i \leq  \frac{n-2}{2},
\end{equation}
where we have now from part (i) that $A=0$.
By substituting the relation from lemma \ref{b_i} between the $b_i$
and the $a_i$
into \eqref{stuf2even}, we have the following system of relations on the
coefficients, $a_i$, for each $i$, $0\leq i\leq
\frac{n-4}{2}$:  
\begin{multline}\label{ajsystemeven}
\sum_{j=0}^{i}
(-1)^{i-j}a_j\Bigl(C_{j+1}^{i-j}+C_{j}^{i-j-1}\Bigr) + \\
\sum_{k=0}^{n-2-i}
(-1)^{n-2-i-k} a_k\Bigl( C_{k+1}^{n-2-i-k} + C_{k}^{n-3-i-k} \Bigr) =
0.\end{multline}  
         
The system given in \eqref{ajsystemeven} may be solved by finding 
solutions to the matrix equation,
$M \cdot (a_0,a_1,\dots,a_{\frac{n-4}{2}}
)=(0,0,...,0)$, in other words by finding the kernel of $M$.  We will
only find its dimension, the nullity of $M$.  

In the even case, the matrix $M$ is given by the same formula as in
the odd case,
\begin{multline}
M(i,j) = (-1)^{i-j}\Bigl( C_{j}^{i-j} +
C_{j-1}^{i-j-1}\Bigr) + \\(-1)^{n-i-j}\Bigl( C_{j}^{n-i-j} +
C_{j-1}^{n-1-i-j} \Bigr),\ 1\leq i,j\leq \frac{n-2}{2}
.\end{multline}

In the same way as the odd case, we construct a matrix $N$ such that
$MN$ is lower triangular.
The top of the
$\bigl( \frac{n-2}{2}-k \bigr)$th
column of $N$ is $(\overbrace{0,...,0}^{2k+1})\cdot
A_{\frac{n-6}{2}-3k}$ for
$0\leq k\leq \lfloor \frac{n-8}{6} \rfloor$ 
and equal to the identity matrix elsewhere up to sign, where
$sgn(N(i,j))=(-1)^{i+j-1}$, except on the diagonal, where the sign is
positive.  A similar calculation shows that $MN$ is a lower
triangular matrix.
 
To find the rank of this matrix, we calculate the terms on the
diagonal.  If $\frac{n-6}{2}-3k$ is odd, for $\frac{n-2}{2}-k >
\lfloor \frac{n}{3} \rfloor$, 
$MN(\frac{n-2}{2}-k,\frac{n-2}{2}-k)$ is equal to
\begin{equation*}\begin{split}
-D_{\frac{n-2}{2}-k,\lfloor \frac{n-4-2k}{4} \rfloor -2k-1}\cdot & 
A_{\frac{n-6}{2} - 3k} + M(\frac{n-2}{2}-k,\frac{n-2}{2}-k) \\
&= - (Q|X^{2+2k} Y^{\frac{n-6}{2}-3k} T^{\frac{n-2}{2}-k}) + 1 +
C_{\frac{n-2}{2}-k}^{\frac{n-6}{2}-3k} + C_{\frac{n-4}{2}-k}^{ \frac
  {n-6}{2} -3k}\\ &= - C_{\frac{n-2}{2}-k}^{\frac{n-6}{2}-3k} -
C_{\frac{n-4}{2}-k}^{ \frac{n-6}{2} -3k} -1 + 1 +
C_{\frac{n-2}{2}-k}^{\frac{n-6}{2}-3k} + C_{\frac{n-4}{2}-k}^{ \frac
  {n-6}{2} -3k}\\ &=0.
\end{split}\end{equation*}
Finally, for identical reasons as in the odd $n$ case,
if $\frac{n-6}{2}-3k$ is even
$MN(\frac{n-2}{2}-k,\frac{n-2}{2}-k)$ is equal to
\begin{equation*}\begin{split} 
-D_{\frac{n-2}{2}-k,\lfloor \frac{n-4-2k}{4} \rfloor -2k-1}\cdot & 
A_{\frac{n-6}{2} - 3k} -2 + M(\frac{n-2}{2}-k,\frac{n-2}{2}-k)\\
&= - (Q|X^{2+2k} Y^{\frac{n-6}{2}-3k} T^{\frac{n-2}{2}-k}) -2+ 1 +
C_{\frac{n-2}{2}-k}^{\frac{n-6}{2}-3k} + C_{\frac{n-4}{2}-k}^{ \frac
  {n-6}{2} -3k}\\ &= - C_{\frac{n-2}{2}-k}^{\frac{n-6}{2}-3k} -
C_{\frac{n-4}{2}-k}^{ \frac{n-6}{2} -3k} + 1 -2 + 1 +
C_{\frac{n-2}{2}-k}^{\frac{n-6}{2}-3k} + C_{\frac{n-4}{2}-k}^{ \frac
  {n-6}{2} -3k}\\ &=0.
\end{split}\end{equation*}

Now, because $0\leq k\leq \frac{n-8}{6}$ the nullity of $M$ is
equal to
$\lfloor \frac{n-2}{6}\rfloor$.  In other words, $F_n^2 \ds / F_n^3\ds
\leq \lfloor \frac{n-2}{6}\rfloor$, since there may be relations
between the generators that come from other systems of equations
besides the equations \ref{stuf2}.  In fact there are not any other
relations, and we use the combinatorial properties of the Poisson
bracket to justify this.

We will now verify that $\lfloor \frac{n-2}{6}\rfloor$ is a lower
bound for the dimension.
From theorem 
\ref{theoreme}
(i), let 
$S=\{ \{f_{2i+1},f_{n-2i-1}\},\ 1\leq i \leq \lfloor
\frac{n-4}{4} \rfloor\} $,
be the set of Poisson brackets of weight $n$ generators of 
$F_1^{2i+1}\ds$.   Let $D$ be the vector space generated by $S$. 

We consider the image of $D$, $\overline{D}$ in 
$F_2^n \ds / F_3^n\ds$.  By work of Zagier, Ihara and Takao (unpublished, see
\cite{Sc}), we know 
that the nullity of this system of equations is equal to a number
which turns out to be exactly the dimension
of the space of period polynomials, which is itself equal to the dimension of
the space of cusp forms of weight $n$ on ${\mathrm{SL}}_2(\Z)$
(denoted $S_n({\mathrm{SL}}_2(\Z))$) \cite{Sc}.  Therefore, we have that
\begin{equation*}\begin{split}
dim(\overline{D}) & = |S| -
dim(S_n(\mathrm{SL}_2(\Z)))\\  & = \lfloor \frac{n-4}{4} \rfloor -
\begin{cases}\lfloor n/12 \rfloor - 1 & n\equiv 2\ \mathrm{mod}\ 12\\
\lfloor n/12 \rfloor & \mathrm{otherwise}\end{cases}\\
&= \lfloor \frac{n-2}{6} \rfloor. \end{split}\end{equation*} 
Since $\overline{D} \subset F_2^n \ds / F_3^n \ds$, 
\begin{equation*} \lfloor \frac{n-2}{6} \rfloor \leq dim( F_2^n
  \ds / F_3^n \ds )\leq \lfloor \frac{n-2}{6} \rfloor,\end{equation*}
and hence the theorem is proved.
\end{proof}
 
Recall the definition \ref{nznumbersdef} of the new zeta value
algebra, $\nz$, which is the 
quotient of the algebra of 
multizeta values by  products.  We showed in chapter 1 that
$\ds^{\vee}$ surjects onto $\nz$.  The proof of theorem \ref{dep12}
yields the following corollary which gives the expression of
depth 2 multizeta value as a rational multiple of a depth 1 multizeta modulo products, thus
recovering a (weaker version of a) result well-known to Euler.

\begin{cor} Let $\overline{\zeta}(i,j)$ be a new zeta value of depth 2, and odd weight $n$ ($i+j=n$ is odd).\label{littlenzdef}
We have the following expression for $\overline{\zeta}(i,j)$ in terms of $\overline{\zeta}(i+j)$:
\begin{align*}
\overline{\zeta}(x^{i-1}yx^{j-1}y) &= \frac{(-1)^{j-1}C_n^j -1}{2}\overline{\zeta}(x^{n-1}y)\\
\overline{\zeta}(i,j) &= \frac{(-1)^{j-1}C_n^j -1}{2}\overline{\zeta}(n).
\end{align*}

\end{cor}
\begin{proof}
For $n$ odd, since the matrix $M$ is invertible, there exists a
unique solution to the equation,
 $M\cdot
  (a_0,..., a_{\frac{n-3}{2}})=-A(1,\dots,1)$.  We propose $a_i =A
  \frac{(-1)^i}{2} C_{n-i-1}^{i+1} $ and show that this is the
  solution.  In this case we have,
\begin{align}\label{cutesoln} 
M\cdot (a_0,...,a_{\frac{n-3}{2}}) & =\Biggl(
A\sum_{j=\lfloor \frac{i+1}{2} \rfloor}^i \frac{(-1)^{n-j}}{2} C_{n-j}^j
(-1)^{i-j} \biggl( C_j^{i-j} + C_{j-1}^{i-j-1} \biggr)+\\
&\quad A\sum_{k=\frac{n-1}{2}-\lfloor\frac{i-1}{2}
  \rfloor}^{\frac{n-1}{2}} \frac{(-1)^{n-k}}{2} C_{n-k}^k (-1)^{n-i-k} \biggl(
C_k^{n-i-k} + C_{k-1}^{n-i-k-1} \biggr) \Biggr)_{i=1}^{\frac{n-1}{2}}.\notag
\end{align}

The power series, $P$, from lemma \ref{Biglem}, has the expression,
$$P= YT^2\Bigl(\sum_{j=0} (XT+YT^2)^j\Bigr) \Bigl(\sum_{k=0}
(X+Y)^k\Bigr) + \Bigl(\sum_{j=0} (XT+YT^2)^j\Bigr) \Bigl(\sum_{k=0}
(X+Y)^k\Bigr).$$
Hence, the first term in equation \eqref{cutesoln} gives exactly
$A\frac{(-1)^{i-1}}{2}$ times the
coefficient of $X^{n-i}Y^iT^i$ in the expansion of $P$.  By lemma
\ref{Biglem}, this term is
equal to
$A\frac{-C_n^i -1}{2}$ for even $i$ and equal to $A\frac{C_n^i -1}{2}$
for odd $i$.  Furthermore, this term is exactly the expression of $b_{i-1}$, by
the constructions \eqref{ajsystem} and \eqref{MDEF} of $M$.
The second term gives
$A\frac{(-1)^{i}}{2}$ 
times the coefficient of $X^i Y^{n-i} T^{n-i}$ 
in the expansion of $P$ and is equal to
$A\frac{-C_n^i -1}{2}$ for odd $i$ and
equal to $A\frac{C_n^i -1}{2}$ for even $i$.  This second term is
equal to $b_{n-i-1}$. 
In both even and odd $i$ cases, this sum is equal to 
$-A$, so the expression $a_i=A\frac{(-1)^i}{2} C_{n-i-1}^{i+1}$ is
indeed the unique solution to
the system.  

Since the first term in equation \eqref{cutesoln} is equal to
$b_{i-1}$, we have 
\begin{equation*}
b_{i-1}=(f|x^{n-i-1}yx^{i-1}y) =A\frac{(-1)^{i-1}C_n^i -1}{2} =
  \frac{(-1)^{i-1}C_n^i -1}{2}(f|x^{n-1}y). 
\end{equation*}
By the definition \ref{nfzdef} of the dual space,
$\ds^{\vee}=\widetilde{\nfz}$, this
equation is $\z^{\sha}(x^{n-i-1}yx^{i-1}y) = \frac{(-1)^{i-1}C_n^i
  -1}{2}\z^{\sha}(x^{n-1}y)$.  But $\widetilde{\nfz}$ surjects onto
$\nz$, by the map $\z^{\sha}(w)\mapsto\overline{\zeta}(w)$, so this relation is true
also in the new zeta space and we have the desired expression.

\end{proof}

\begin{rmk*}
 Note that the preceding corollary does not work when $i+j=n$ is even; a double zeta is not equal to a rational multiple of a single zeta in even weight in $\widetilde{\nfz}\simeq \ds^\vee$.  This follows from the fact that $F^1_n\ds/ F^2_n\ds= 0$ for even $n$ (theorem \ref{theoreme}).  In \cite{IKZ}, the authors prove in complete generality that $F^d_n\ds /F^{d+1}_n \ds=0$ whenever $d$ and $n$ have opposite parities.
\end{rmk*}

\chapter{The algebra of cell zeta values}

This chapter is an intact article entitled {\it The algebra of
  cell-zeta values}, [BCS],
which is joint work with Francis Brown and
Leila Schneps awaiting publication.  In [BCS], we give an explicit
basis of polygons for the de Rham cohomology 
space, $H^{n-3}(\Mn^\delta)$, and use this to present a new structure
  for the $\Q$  
algebra of multizeta values, $\MZV$, by considering the algebra generated
  by all periods on $\Mn$. 
Here, a {\it period} on $\Mn$ is considered to be the integral of a rational
function over a simplex in $\Mn(\R)$, the real part of moduli space.
In chapter 2, we presented Kontsevich's construction of multizeta
  values as integrals of rational 
functions over simplices, $\delta:=0<t_1<\cdots <t_{n-3}<1$, which are
simplices in $\Mn(\R)$, thus showing that multizetas are indeed periods. 

This work was inspired by the recent
theorem of F. Brown [Br] in which he proves that every period on
$\Mn$ is a $\Q$-linear combination
of multizeta values.    Here, I give a brief and
intuitive introduction to the development of the special periods that
we take as generators of the period algebra, which are called
cell-zeta values.  The 
definitions, structure of the paper and background are given in
the introduction of [BCS]. 

In [Br], product maps on moduli space are introduced,
$$f: \Mn\rightarrow \Mod_{0,s} \times \Mod_{0,r},\ r+s=n+3,$$
which are simply the products of two forgetful maps (see [BCS] section
\ref{prodmapsection}).  He defines two particular product
maps, the simplicial product map, which gives the shuffle relation on
multizetas, and the cubical product map, which gives the stuffle
relation on multizetas.  At the end of section of 7.5 [Br], Brown comments
that these two product maps are extreme cases of a range of
intermediate product formulae.  This paper is a study of the
intermediate product formulae.  These new product formulae yield
relations on cell-zeta values, analogous to, but more general than,
the double shuffle relations on multizeta values, and have the
advantage that they reflect the geometry and symmetry of the $\Mn$. 

Recall the definition of $\Mn(\R)$.  The connected components of
$\Mn(\R)$ are cells and are denoted by the real ordering of the marked
points inside them.   
To any such cell in $\Mn(\R)$ we
may associate a differential form, called a {\it cell
form}, which is the form that has a simple pole along each irreducible boundary divisor which contains a face of the boundary of the associahedron $\delta$.

The cells generate the top dimensional homology group and, by duality between
the homology and the cohomology, the cell forms generate
the top dimensional de Rham
cohomology, $H^{n-3}(\Mn)$.  Based on a theorem of Arnol'd [Ar], we
found that a basis for $H^{n-3}(\Mn)$ is given by 01-forms (proposition
\ref{prop41}).

Our paper answers the following three questions that arose naturally
from studying multizetas as periods.

\begin{ques*} What subspace of the
cohomology is the space of differential forms that give periods,
i.e. that converge on a cell?\end{ques*}  

It is not useful to look at the whole cohomology group since
the integral of any 01-form converges over some cell, but diverges on
others, while certain linear combinations of divergent cell forms will actually
converge on a cell.  
The periods on $\Mn$ are by
definition convergent
integrals of the forms, $\omega \in H^{n-3}(\Mn)$, over some cell
$\gamma$.  By a variable change,
any period can be written as an integral over the standard cell,
$\delta$.  Therefore, to study of periods
on $\Mn$ it is sufficient to study the forms in the cohomology that are
convergent on $\Mn$ and on its set of boundary components $\delta$
which bound the standard cell.  In chapter 4, we prove that this
subspace of convergent differential forms is isomorphic to
$H^{n-3}(\Mn^\delta)$, the 
cohomology of the partially compactified moduli space.  In section
\ref{sec43}
of chapter 3, we give an explicit basis 
for $H^{n-3}(\Mn^\delta)$ thereby answering this first question. 

\begin{ques*}How does one use periods to study
multizeta values? \end{ques*}

With the explicit basis given above, we now have a way to formally represent periods as linear combinations pairs of polygons of the form $(\gamma, \omega)$, where $\gamma$ is a cell and $\omega$ a cell form.
The product of periods which are integrals of a cell form over a cell is given by the pullback formula of a
product map, (proposition \ref{prodmaprel})
$$\int_{\delta_1} \omega_1 \int_{\delta_2} \omega_2 = \int_{\delta_1
  \sh \delta_2} \omega_1\sh \omega_2.$$
Therefore any period can be
represented as a linear combination of pairs of polygons and these polygons form an algebra for the
  shuffle product, 
$$(\delta_1, \omega_1)(\delta_2, \omega_2)=(\delta_1\sh \delta_2,
\omega_1\sh\omega_2).$$
We denote the algebra of
periods or {\it cell numbers} by ${\cal C}$.  By [Br], all the periods
on $\Mn$ are $\Q$ linear combinations of multizeta values.  Therefore,
we have answered the second question, ${\cal C}\simeq \MZV$, on the level
of real numbers.  However, on the combinatorial level of formal
multizetas and formal periods, there is still much work to be done.

We know how to
explicitly express all of the
multizetas as polygon sums by Kontsevich's identity, but we cannot as
of yet explicitly express all of the periods as multizetas (even
though by [Br], we know such an expression exists). 

This leads us to
the question of finding a set of generating relations over $\Q$ for $\cal C$.
We conjecture that the answer to this question is that the algebra of
cell numbers has only the relations coming from variable changes on
periods, algebraic identities on differential forms and product maps.
As usual, because conjectures of this analytic type seem very
difficult to prove, as they would imply important results, such
as the transcendence conjecture on multizeta values, we concentrate
our study on the formal situation in which the only relations are
decreed to be the known relations.  This is the same principle as in
chapter 1 where we defined the formal 
zeta value
algebra which satisfies only shuffle, stuffle and regularization
relations, and leads to the final
main question addressed in [BCS].

\begin{ques*} How can we use the three
known sets of relations on periods to study relations between multiple
zeta values?\end{ques*}

In order to study this, in section \ref{cellalg}, we define the {\it
  formal cell number 
algebra}, ${\cal FC}$, which satisfies exactly the three
sets of period relations outlined above.
Since ${\cal FC}$ surjects onto ${\cal C}$, any identities that we can find
on formal cell numbers are also true for multizeta values, hence the
structure of $\cal FC$ provides a new method for studying multizeta values.
  If the formal cell numbers provide an adequate structure
for multiple zeta values, then we should have the following
commutative diagram:
\begin{equation*}
\xymatrix{
    {\cal FC} \ar@{<.>}[r]^? \ar@{->>}[d]& \FZ \ar@{->>}[d] \\
    {\cal C} \ar@{<->}[r]^\sim & \MZV
  }
\end{equation*}
In the section \ref{explicitbasissection} of [BCS], we discuss the
implications of this 
hypothesis, and give examples of how and why it should be true.  


\begin{center}
{\Large{\bf THE ALGEBRA OF CELL-ZETA VALUES}}\\
\vspace{1cm}
{FRANCIS BROWN, SARAH CARR, LEILA SCHNEPS}\\
\end{center}

\vspace{1cm}

\begin{changemargin}{1cm}{1cm}
\noindent {\it Abstract}. 
\small{ Traditionally, multiple zeta values are viewed as convergent nested
series which can also be expressed as iterated integrals on the
projective line minus three points $\Pro^1\backslash\{0,1,\infty\}$.
They are known to satisfy two sets of quadratic relations known as
the double shuffle relations, which are conjectured to generate all
algebraic relations between them. They were subsequently interpreted
as the periods of the (motivic) fundamental group of
$\Pro^1\backslash \{0,1,\infty\}$.  Recently, Goncharov and Manin
introduced a new version of motivic multiple zeta values, in  which
they are interpreted as periods of mixed Tate motives attached to
the moduli spaces $\Mod_{0,n}$ of genus zero curves with $n$ marked
points.

In this paper, we introduce {\it cell-forms} on $\Mod_{0,\ell+3}$,
which are differential $\ell$-forms diverging along the boundary of
exactly one connected component (cell) of the real moduli space
$\Mod_{0,\ell+3}(\R)$. We give a basis for the space of all forms
convergent on a given cell $X$ in terms of cell-forms and define
{\it cell-zeta values} to be the real numbers obtained by
integrating these forms over $X$. The cell-zeta values satisfy
algebraic relations generalizing the double shuffle relations,
coming from simple geometric operations on the moduli spaces, and
the  algebra of cell-zeta values is in fact equal to the algebra of
multiple zeta values. We conjecture that this new  combinatorial
system of generators and relations gives a complete description of
the algebra of multiple zeta values.}
\end{changemargin}

\begin{center}
\section{Introduction}
\end{center}
Let $n_1,\ldots, n_r\in \N$ and suppose that $n_r\geq 2$. The
multiple zeta values (MZV's)
\begin{equation}\label{MZV}
\zeta(n_1,\ldots, n_r) = \sum_{0<k_1<\ldots<k_r} {1 \over k_1^{n_1}
\ldots k_r^{n_r}}\in \R\ ,\end{equation} were first defined by
Euler, and have recently acquired much importance in their relation
to mixed Tate motives. It is conjectured that the periods of all
mixed Tate motives over $\Z$ are expressible in terms of such
numbers. By a remark due to Kontsevich, every multiple zeta value
can be written as an iterated integral:
\begin{equation}\label{itint}
\int_{0\leq t_1\leq \ldots \leq t_\ell \leq 1} {dt_1 \ldots dt_\ell
\over (\varepsilon_{1}-t_1)\ldots(\varepsilon_\ell-t_\ell) }\ ,
\end{equation} where $\varepsilon_i \in \{0,1\}$, and
$\varepsilon_1=1$ and $\varepsilon_\ell=0$ to ensure convergence,
and $\ell=n_1+\cdots +n_r$. The iterated integral $(\ref{itint})$ is
a period of the motivic fundamental group of
$\Mod_{0,4}=\Pro^1\backslash\{0,1,\infty\}$, whose de Rham
cohomology $H^1(\Mod_{0,4})$ is spanned by the forms ${dt \over t}$
and ${dt \over 1-t}$ \cite{De1,DG}. One proves that the multiple
zeta values satisfy two sets of quadratic relations \cite{Ch1,Ho},
known as the regularised double shuffle relations, and it has been
conjectured that these generate all algebraic relations between
MZV's \cite{Ca2, Wa}.
 This is the  traditional point
of view on multiple zeta values.

On the other hand, by  a general construction due to Beilinson, one
can view the iterated integral $(\ref{itint})$ as a period integral
in the ordinary sense, but this time  of the $\ell$-dimensional
affine scheme
$$\Mod_{0,n} \simeq (\Mod_{0,4})^\ell \backslash \{\hbox{diagonals}\}
= \{(t_1,\ldots, t_\ell): t_i\neq 0,1\ , t_i\neq t_j\}\ ,$$ 
where $n=\ell+3$. This is the moduli space of curves of genus $0$
with $n$ ordered marked points. Indeed, the open domain of
integration $X=\{0< t_1< \ldots < t_\ell < 1\}$\label{stcellDEF}\index{Standard cell, $X=X_{\delta}$} is one of the
connected components of the set of real points $\Mod_{0,n}(\R)$, and
the integrand of $(\ref{itint})$ is a regular algebraic form in
$H^\ell(\Mod_{0,n})$ which converges on $X$.  Thus, the study of
multiple zeta values leads naturally to the study of all periods on
$\Mod_{0,n}$, which was initiated by Goncharov and Manin
\cite{Br,GM}. These periods can be written
\begin{equation}\label{introint}
\int_X \omega\ , \quad \hbox{ where } \omega \in H^\ell(\Mod_{0,n})
\hbox{ has no poles along } \overline{X}\ .\end{equation}
  The
general philosophy of motives and their periods \cite{KZ}
indicates that one should study relations between all such
integrals.  This leads to the following problems:
\begin{enumerate}
  \item Construct a good basis of all regular (logarithmic)
  $\ell$-forms $\omega$ in $H^\ell(\Mod_{0,n})$ whose integral over the cell
  $X$ converges.
  \item Find all relations between the integrals $\int_{X} \omega$
  which arise from natural geometric considerations on the moduli spaces
  $\Mod_{0,n}$.
\end{enumerate}
In this paper, we give an explicit solution to $(1)$, and a family
of relations which conjecturally answers $(2)$. Firstly, we give a
complete  description of the convergent part of the cohomology
$H^\ell(\Mod_{0,n})$ in terms of the combinatorics of polygons. The
corresponding integrals are  much more general  than
$(\ref{itint})$, and the numbers one obtains are called {\it
cell-zeta values}. For $(2)$, it turns out that there are
essentially two types of relations. The first arises from the
dihedral subgroup of automorphisms of $\Mod_{0,n}$ which stabilise
$X$, and the other is a quadratic relation, which we call the {\it
modular shuffle product}, arising from a product of forgetful maps
between moduli spaces. We conjecture that these two simple families
of relations generate the complete set of relations for the periods
of the moduli spaces $\Mod_{0,n}$.

\subsection{Main results}
We give a brief presentation of the main objects introduced in this
paper, and the results obtained using them. \vskip .2cm

There is a stable compactification $\overline\Mod_{0,n}$ of
$\Mod_{0,n}$, such that $\overline\Mod_{0,n}\setminus\Mod_{0,n}$ is
a smooth normal crossing divisor whose irreducible components
correspond bijectively to partitions of the set of $n$ marked points
into two subsets of cardinal $\ge 2$ \cite{DM,Kn}. The real part
$\Mod_{0,n}(\R)$ of $\Mod_{0,n}$ is not connected, but has $n!/2n$
connected components (open cells) corresponding to the different
cyclic orders of the real points $0,t_1,\ldots,t_\ell,1,\infty\in
{\mathbb P}^1(\R)$, up to dihedral permutation \cite{Dev1}. 
Thus, we can identify cells with $n$-sided polygons with edges
labeled by $\{0,t_1,\ldots, t_\ell, 1, \infty\}$.  In the
compactification $\overline{\Mod}_{0,n}(\R)$, the closed cells have
the structure of associahedra or Stasheff polytopes; the boundary of
a given cell is a union of irreducible divisors corresponding to
partitions given by the chords in the  associated polygon. The
standard cell is the cell corresponding to the standard order we
denote $\delta$, given by $0<t_1<\ldots<t_\ell<1$. We write
$\Mod_{0,n}^\delta$ for the union of $\Mod_{0,n}$ with the boundary
divisors of the standard cell. This is a smooth affine scheme
introduced in \cite{Br}.

\subsubsection{Cell-forms.}  A cell-form is a holomorphic differential
$\ell$-form on $\Mod_{0,n}$ with logarithmic singularities along the
boundary components of the stable compactification, having the
property that its singular locus forms the boundary of a single cell
in  the real moduli space $\Mod_{0,n}(\R)$.
\subsubsection{Polygons.} Since a cell of $\Mod_{0,n}(\R)$ is given by
an ordering of $\{0,t_1,\ldots,t_\ell,1,\infty\}$ up to dihedral
permutation, we can identify it as above with an \emph{unoriented}
$n$-sided polygon with edges indexed by the set
$\{0,t_1,\ldots,t_\ell,1,\infty\}$.  Up to sign, the cell-form
diverging on a given cell is obtained by taking the successive
differences of the edges of the polygon (ignoring $\infty$) as
factors in the denominator: \vskip 1.5cm
$\displaystyle{\qquad\qquad\qquad\qquad \qquad\qquad\qquad
\longleftrightarrow \qquad\qquad \pm\,{{dt_1dt_2dt_3}\over
{(t_1-1)(t_3-t_1)(-t_2)}}}$

\vspace{-2cm}
\ \ \ \ \epsfxsize=3cm\epsfbox{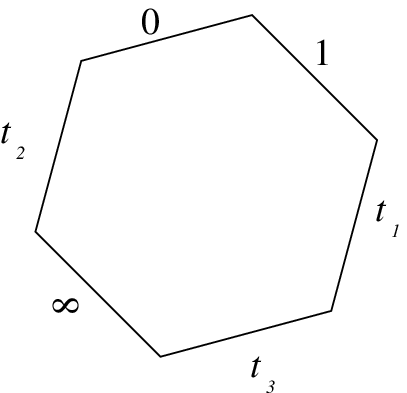}

\vspace{.1cm} \noindent Let ${\cal P}$ denote the $\Q$-vector space
of \emph{oriented} $n$-gons indexed by
$\{0,1,t_1,\ldots,t_\ell,1,\infty\}$. The orientation fixes the sign
of the  corresponding cell form, and this  gives a map
\begin{equation}\label{introrhodef}
\rho:{\cal P}\rightarrow H^\ell(\Mod_{0,n}).\end{equation} In
section \ref{sec41} we prove that this map is surjective and
identify its kernel.
\subsubsection{$01$-cell-forms.}\label{01cellformDEF}\index{01-cell form} These are the
cell-forms corresponding to
polygons in which $0$ appears adjacent to $1$.  In theorem
\ref{thm01cellsspan}, we show that they form a basis of the
cohomology $H^\ell(\Mod_{0,n})$.   In particular, the subspace of
${\cal P}$ of polygons having $0$ adjacent to $1$ is isomorphic to
$H^\ell(\Mod_{0,n})$ via $(\ref{introrhodef})$.
\subsubsection{Insertion forms.} These very particular linear combinations
of $01$ cell-forms are constructed in section \ref{Lyndins}.  We
prove in theorem \ref{convKS} and proposition \ref{insformsspan}
that they form a basis for the cohomology group
$H^\ell(\Mod_{0,n}^\delta)$ of forms with no poles along the
boundary of the standard cell of $\Mod_{0,n}(\R)$. These are
precisely the forms whose integral $(\ref{introint})$ converges.
\subsubsection{Cell-zeta values.}  These are real numbers obtained by
integrating
 insertion forms over the standard cell as in $(\ref{introint})$.
They are a generalization of multiple zeta values to a larger set of
periods on $\Mod_{0,n}$, such as
$$\int_{0<t_1<t_2<t_3<1} {{dt_1dt_2dt_3}\over{(1-t_1)(t_3-t_1)t_2}}.$$
Note that unlike the multiple zeta values, this is not an iterated integral
as in (\ref{itint}).
\subsubsection{Product maps.} Via the pullback, the maps
$f:\Mod_{0,n}\rightarrow \Mod_{0,r}\times \Mod_{0,s}$ obtained by
forgetting disjoint complementary subsets of the marked points
$t_1,\ldots,t_\ell$ yield expressions for products of cell-zeta
values on $\Mod_{0,r}$ and $\Mod_{0,s}$ as linear combinations of
cell-zeta values on $\Mod_{0,n}$:
\begin{equation}\label{introcellprodmap}
\int_{X_1} \omega_1\int_{X_2}\omega_2=\int_{f^{-1}(X_1\times X_2)}
f^*(\omega_1\wedge\omega_2).\end{equation} There is a simple
combinatorial algorithm to compute the multiplication law in terms
of cell-forms. This generalizes the double shuffle multiplication
laws for multiple zeta values, and is  explained in section
\ref{prodmaps}.
\subsubsection{Dihedral relations} These relations between
cell-zeta values are given by \begin{equation}
\label{introdihedralrel}\int_X\omega=\int_X \sigma^*(\omega)\ ,
\end{equation} where $\sigma$ is an automorphism of $\Mod_{0,n}$ which
maps the standard cell to itself: $\sigma(X)=X$, and thus $\sigma$
is a dihedral permutation of the marked points $\{0,1,t_1,\ldots,
t_\ell, \infty\}$.

\subsubsection{The cell-zeta value algebra ${\cal C}$.} The multiplication
laws associated to product maps $(\ref{introcellprodmap})$ make the
space of all cell-zeta values on $\Mod_{0,n}$, $n\ge 5$, into a
$\Q$-algebra which we denote by ${\cal C}$. By Brown's theorem
\cite{Br}, which states essentially that all periods on
$\Mod_{0,n}$ are linear combinations of multiple zeta values,
together with Kontsevitch's expression (\ref{itint}) of multiple
zeta values, we see that ${\cal C}$ is equal to the algebra of
multiple zeta values ${\cal Z}$.
\subsubsection{The formal cell-zeta value algebra ${\cal FC}$.}
By lifting the previous constructions to the level of polygons along
the map $(\ref{introrhodef})$, we define in section \ref{cellalg} an
algebra of  formal cell-zeta values which we denote by ${\cal FC}$.
It is   generated by the {\it insertion words}, which are formal
sums of polygons corresponding to the insertion forms introduced
above, subject to combinatorial versions of the product map
relations $(\ref{introcellprodmap})$ and the dihedral relations
$(\ref{introdihedralrel})$.
  This is analogous to the formal MZV algebra given by the double
shuffle and Hoffmann relations.

\vspace{.3cm} The paper is organised as follows. In $\S2$,  we
introduce cell forms, polygons and define the modular shuffle and
dihedral relations. In  $\S3$, we  define  insertion words
 of polygons which are constructed out of  Lyndon  words, which may be
 of independent combinatorial interest.
These are used to construct the insertion basis of convergent forms
in $\S4$. In $\S\ref{calculations}$, we give complete computations
of  this basis and the  corresponding modular shuffle relations for
$\Mod_{0,n}$, where $n=5,6,7$.

In the remainder of this introduction we sketch the connections
between the formal cell-zeta value algebra and standard results and
conjectures in the theory of multiple zeta values and mixed Tate
motives.

\vspace{0.2cm}
\subsection{Relation to mixed Tate motives and conjectures}

Let  $\MT(\Z)$\label{MTZ}\index{$\MT(\Z)$} denote the category of mixed Tate motives which are
unramified over $\Z$ \cite{DG}.  Let  $\delta$ denote the standard
cyclic structure on $S=\{1,\ldots,n\}$, and let $B_{\delta}$ denote
the divisor which bounds the standard cell $X_{\delta}$. Let
$A_\delta$ denote the set of all remaining divisors on
$\overline{\Mod}_{0,S}\backslash \Mod_{0,S}$, so that
$\Mod_{0,S}^{\delta}=\Mod_{0,S}\cup B_{\delta}$ (\cite{Br}), and
$A_\delta=\overline{\Mod}_{0,S}\backslash \Mod_{0,S}^\delta$. We
write:\index{$M_{\delta},\ B_{\delta},\ A_{\delta}$}
\begin{equation} \label{Mdelta}
M_\delta= H^\ell(\overline{\Mod}_{0,n}\backslash A_\delta, B_\delta
\backslash (B_\delta \cap A_\delta))\ .
\end{equation}
By a result due to Goncharov and Manin  \cite{GM}, $M_\delta$
defines an element in $\MT(\Z)$, and therefore  is equipped with an
increasing weight filtration $W$.  They show that $\gr^W_\ell
M_\delta$ is isomorphic to the de Rham cohomology
$H^\ell(\Mod^\delta_{0,n})$, and that $\gr^W_0 M_\delta$ is
isomorphic to the dual of the relative Betti homology
$H_\ell(\overline{\Mod}_{0,n}, B_\delta )$.

Let $M$ be any element in  $\MT(\Z)$. A framing for $M$ consists of
an integer $n$ and non-zero maps
\begin{equation}
v_0  \in  \Hom( \Q(0), \gr^W_0 M) \quad \hbox{ and } \quad f_n\in
 \Hom( \gr^W_{-2n} M, \Q(n)) \ .
\end{equation}
Two framed motives $(M,v_0,f_n)$ and $(M',v_0',f_n')$ are said to be
equivalent if there is a morphism $\phi:M\rightarrow M'$ such that
$\phi\circ v_0= v_0'$ and $f_n\circ\phi =f_n'$. This generates an
equivalence relation whose equivalence classes are denoted
$[M,v_0,f_n]$. Let $\FMT(\Z)$ denote the set of equivalence classes
of framed mixed Tate motives which are unramified over $\Z$, as
defined in \cite{Go1}. It is a commutative, graded Hopf algebra.

To every convergent cohomology class $\omega \in
H^\ell(\Mod^\delta_{0,n})$, we associate the following framed mixed
Tate motive:\index{Framed mixed Tate motive}\index{${\mathcal{M}}(\Z),\ m(\omega)$}
\begin{equation} \label{momegadef}
m(\omega) = \big[M_\delta, [X_\delta],\omega\big]\ ,
\end{equation}
where $[X_\delta]$ denotes the relative homology class of the
standard cell. This defines a map ${\cal FC} \rightarrow \FMT(\Z)$.
 The maximal period of $m(\omega)$ is exactly the
cell-zeta value
$$\int_{X_\delta} \omega\ .$$

\begin{prop} The dihedral symmetry relation and modular shuffle relations
are motivic. In other words,
\begin{eqnarray}
m(\sigma^*(\omega))&=& m(\omega) \nonumber \ ,\\
m(\omega_1\cdot\omega_2) & = & m(\omega_1)\otimes m(\omega_2)\ , \nonumber
\end{eqnarray}
for every dihedral symmetry $\sigma$ of $X_\delta$, and for every
modular shuffle product $\omega_1\cdot\omega_2$ of convergent forms
$\omega_1,\omega_2$ on $\Mod_{0,r}$, $\Mod_{0,s}$ respectively.
\end{prop}

The motivic nature of our constructions will be clear from the
definitions. We therefore obtain a well-defined map $m$ from the
algebra of formal cell-zeta numbers ${\cal FC}$ to $\FMT(\Z)$.

\begin{conj}
The map $m:\mathcal{FC}\To \FMT(\Z)$  is an isomorphism.
\end{conj}

Since the structure of $\FMT(\Z)$ is known, we are led to more precise
conjectures on the structure of the formal cell-zeta algebra. To
motivate this, let $\LIE=\Q[e_3,e_5,\ldots, ]$\label{LIEDEF}\index{$\LIE$, $F$} denote
the free Lie
algebra generated by one element $e_{2n+1}$ in each odd degree. Set
$$\F = \Q[e_2] \oplus  \LIE.   $$
The underlying graded vector space is generated by, in increasing
weight:
$$e_2 \spc e_3 \spc e_5 \spc e_7 \spc [e_3,e_5] \spc e_9 \spc [e_3,
e_7] \spc [e_3,[e_5,e_3]]\,
, \, e_{11} \spc [e_3,e_9]\, , \, [e_5,e_7] \spc \ldots \ .$$
Let  ${\cal U}\F$  denote the universal enveloping algebra of the
Lie algebra ${\F}$. Then it is known that $\FMT(\Z)$ is dual to
${\cal U}\F$. From the explicit description of $\F$ given above, one
can deduce that the graded dimensions $d_k= \dim_\Q \gr^W_k
\FMT(\Z)$ satisfy Zagier's recurrence relation
\begin{equation} \label{ZagierRec}
d_k=d_{k-2}+d_{k-3}\ ,
\end{equation}
with the initial conditions $d_0=1$, $d_1=0$, $d_2=1$.

\begin{conj} The dimension of the $\Q$-vector space of formal cell-zeta
values on $\Mod_{0,n}$, modulo all linear relations obtained from
the dihedral and modular shuffle relations, is equal to $d_\ell$,
where $n=\ell+3$.  Equivalently, the dual Lie algebra to the co-Lie
algebra obtained by quotienting ${\cal FC}$ by products is
isomorphic to $\F$.
\end{conj}

We verified this  conjecture for $\Mod_{0,n}$ for
$n\leq 9$ by direct calculation (see $\S\ref{calculations}$).  When
$n=9$, the dimension of the convergent cohomology
$H^6(\Mod_{0,9}^\delta)$ is 1089, and after taking into account all
linear relations coming from dihedral and modular shuffle products,
this reduces to a vector space of dimension $d_6=2$.

To compare this picture with the classical picture of multiple zeta
values, let ${\cal FZ}$  denote the formal multi-zeta algebra. This
is the quotient of the free $\Q$-algebra generated by formal symbols
$(\ref{itint})$ modulo the regularised double shuffle relations. It
has been  conjectured that ${\cal FZ}$ is isomorphic to $\FMT(\Z)$,
which  leads to the second main conjecture.

\begin{conj} The formal algebras ${\cal FC}$ and ${\cal FZ}$ are
isomorphic.
\end{conj}
Put more prosaically, this states that
the formal ring of periods of
$\Mod_{0,n}$ modulo dihedral and modular shuffle relations, is
isomorphic to the formal ring of periods of the motivic fundamental
group of $\Mod_{0,4}$ modulo the regularised double shuffle
relations.

By (\ref{itint}), we have a natural linear map ${\cal FZ}\rightarrow
{\cal FC}$. However, at present we cannot show that it is an algebra
homomorphism. Indeed, although it is easy to deduce the regularised
shuffle relation for the image of ${\cal FZ}$ in ${\cal FC}$ from
the dihedral and modular shuffle relations, we are unable to deduce
the regularised stuffle relations.
\begin{rem}
The motivic nature of the regularised double shuffle relations
proved to be somewhat difficult to establish \cite{Go1,Go2, T}. It
is interesting that the motivic nature of the dihedral and modular
shuffle relations we define here is immediate.
\end{rem}

\vspace{.3cm}
\section{The cell-zeta value algebra associated to moduli spaces of curves}

Let $\Mod_{0,n}$\label{Mndefch3}, $n\geq 4$ denote the moduli space of
genus zero 
curves (Riemann
spheres) with $n$ ordered marked points $(z_1,\ldots,z_n)$.  This
space is described by the set of $n$-tuples of distinct points
$(z_1,\ldots,z_n)$ modulo the equivalence relation given by the
action of ${\PSL}_2$.  Because this action is triply transitive,
there is a unique representative of each equivalence class such that
$z_1=0$, $z_{n-1}=1$, $z_n=\infty$.  We define simplicial coordinates
$t_1,\ldots,t_\ell$ on $\Mod_{0,n}$ by setting
\begin{equation} \label{zsimp}
t_1=z_2 \ ,\quad  t_2=z_3 \ , \quad \ldots\ ,\quad  t_\ell=z_{n-2},
\end{equation}
where $\ell=n-3$ is the dimension of $\Mod_{0,n}(\C)$.
This gives the familiar identification
\begin{equation}\label{simplicialisom}
\Mod_{0,n} \cong \{ (t_1,\ldots, t_\ell) \in
(\Pro^1-\{0,1,\infty\})^\ell\mid
t_i\neq t_j \hbox{ for all } i\neq j \}\
.\end{equation}

\vspace{.3cm}
\subsection{Cell forms}

\begin{defn}\label{Cycstrucdef} Let $S=\{1,\ldots,n\}$\index{$S$}.  A {\it cyclic
    structure}\index{Cyclic structure}\index{Dihedral structure} 
  $\gamma$ on $S$ is
a cyclic ordering  of the elements of $S$ or equivalently, an
identification of the elements of $S$ with the edges of an oriented
$n$-gon modulo rotations. A {\it dihedral structure} $\delta$ on $S$
is an identification with the edges of an unoriented $n$-gon modulo
dihedral symmetries.
\end{defn}
We can write a cyclic structure as an ordered $n$-tuple
$\gamma=(\gamma(1), \gamma(2),...,\gamma(n))$ considered up to
cyclic rotations.

\begin{defn}\label{cellform}\index{Cell form, $[s_1,...,s_n]$} Let
  $(z_1,\ldots,z_n)=(0,t_1,\ldots,t_\ell,1,\infty)$
be a representative of a point on $\Mod_{0,n}$ as above. Let
$\gamma$ be a cyclic structure on $S$, and let $\sigma$ be the
unique ordering of $z_1,\ldots, z_n$ compatible with $\gamma$ such
that $\sigma(n)=n$. The \emph{cell-form} corresponding to $\gamma$
is defined to be the differential $\ell$-form
\begin{equation}\label{omegadef}
\omega_{\gamma} = [z_{\sigma(1)},z_{\sigma(2)},\ldots,z_{\sigma(n)}]=
\frac{dt_1\cdots dt_\ell}
{(z_{\sigma(2)}-z_{\sigma(1)}) (z_{\sigma(3)}-z_{\sigma(2)})\cdots
  (z_{\sigma(n-1)} -z_{\sigma(n-2)})}.
\end{equation}
In other words, by writing the terms of $\omega_\gamma =
[z_{\sigma(1)}, ... ,z_{\sigma(n)}]$ clockwise around a polygon, the
denominator of a cell form is just
the product of
successive differences $(z_{\sigma(i)} -z_{\sigma(i-1)})$
 with the two
factors containing $\infty$ simply left out.
\end{defn}

\begin{rem}\label{grlem}
To every dihedral structure there correspond two opposite cyclic
structures. If these are given by $\gamma$ and $\tau$, then we have
\begin{equation} \omega_{\gamma} = (-1)^n
  \omega_{\tau}. \end{equation}
\end{rem}

\begin{example} Let $n=7$, and $S=\{1,\ldots, 7\}$. Consider the cyclic
structure $\gamma$ on $S$ given by the order $1635724$.  The unique
ordering $\sigma$ of $S$ compatible with $\gamma$ and having
$\sigma(n)=n$, is the ordering $2416357$, which can be depicted by
writing the elements $z_{\sigma(1)},\ldots,z_{\sigma(7)}$, or $0$,
$1$, $t_2$, $t_4$, $\infty$, $t_1$, $t_3$ clockwise around a circle:
$$\gamma = (z_{\sigma(1)},\ldots,z_{\sigma(7)})=
(t_1,t_3,0,1,t_2,t_4,\infty).$$
The corresponding cell-form on $\Mod_{0,7}$ is
$$\omega_{\gamma}=[t_1,t_3,0,1,t_2,t_4,\infty]= {dt_1 dt_2 dt_3 dt_4
  \over (t_3-t_1) (-t_3)(t_2-1)(t_4-t_2)}\ .$$
\end{example}

The symmetric group $\Sym(S)$ acts on $\Mod_{0,n}$ by permutation of
the marked points. It therefore acts both on the set of cyclic
structures $\gamma$, and also on the ring of differential forms on
$\Mod_{0,n}$. These actions coincide for cell forms.

\begin{lem} \label{lemsymaction}
For every cyclic structure $\gamma$ on $S$, we have the formula:
\begin{equation}\label{symaction}
\sigma^*(\omega_\gamma) = \omega_{\sigma(\gamma)}\qquad \hbox{ for
all } \sigma \in \Sym(S)\ .\end{equation}
\end{lem}
\begin{proof}Consider the regular $n$-form on
$(\Pro^1)^S_*$ defined by the formula:
\begin{equation}\label{omegalift}
\widetilde{\omega}_\gamma = {dz_1 \wedge \ldots \wedge dz_n \over
(z_{\gamma(1)}-z_{\gamma(2)})\ldots (z_{\gamma(n)}- z_{\gamma(1)})}
\ .\end{equation}
 It is clearly satisfies
 $\sigma^*(\widetilde{\omega}_{\gamma})=
 \widetilde{\omega}_{\sigma(\gamma)}$ for all $\sigma \in D_{\gamma}$.
 A simple calculation shows that $\widetilde{\omega}_\gamma$
is invariant under the action of $\PSL_2$ by M\"obius
transformations. Let $\pi:(\Pro^1)^S_* \rightarrow \Mod_{0,S}$
denote the projection map with fibres isomorphic to  $\PSL_2$. There
is a unique (up to scalar multiple in $\Q^\times$) non-zero
invariant regular 3-form $v$ on $\PSL_2(\C)$ which is defined over
$\Q$. Then, by renormalising $v$ if necessary, we have
$\omega_{\gamma}\wedge v = \widetilde{\omega}_{\gamma}\ .$
 In fact,
 $\omega_{\gamma}$ is the unique
$\ell$-form on $\Mod_{0,S}$ satisfying  this equation. We deduce
that $\sigma^*(\omega_{\gamma})= \omega_{\sigma(\gamma)}$ for all
$\sigma \in D_{\gamma}$.
\end{proof}

Each dihedral structure $\eta$ on $S$ corresponds to a unique
connected component of the real locus $\Mod_{0,n}(\R)$, namely the
component associated to the set of Riemann spheres with real marked
points $(z_1,\ldots,z_n)$ whose real ordering is given by $\eta$. We
denote this component by $X_{S,\eta}$ or $X_{n,\eta}$.  It is an
algebraic manifold with corners with the combinatorial structure of
a Stasheff polytope, so we often refer to it as a
\emph{cell}\label{celldefCH3}\index{Cell, $X_{n,\eta}$}. A
cyclic structure compatible with $\eta$ corresponds to a choice of
orientation of this cell.
 Let $\delta$ once and for all denote the
cyclic order corresponding to the ordering $(1,2,\ldots,n)$. We call
$X_{S,\delta}=X_{n,\delta}$ the \emph{standard cell}\index{Standard cell, $X_{n,\delta}$}.  It is the set
of points on $\Mod_{0,n}$ given by real marked points
$(0,t_1,\ldots,t_\ell,1,\infty)$ in that cyclic order; in simplicial
coordinates it is given by  the standard real simplex
$0<t_1<\ldots<t_\ell<1$.

The distinguishing feature of cell-forms, from which they derive their name,
is given in the following proposition.

\vspace{.2cm}
\begin{prop}\label{CORomegapoles} Let $\eta$ be a dihedral structure
on $S$, and let $\gamma$ be either of the two cyclic substructures
of $\eta$.  Then the cell form $\omega_\gamma$  has  simple poles along the
boundary of the cell $X_{S,\eta}$ and no poles anywhere else.
\end{prop}

\begin{proof}
Let $D\subset \overline{\Mod}_{0,S}\backslash \Mod_{0,S}$ be a
divisor given by a stable partition $S=S_1\coprod S_2$ ({\it i.e.},
such that $|S_i|\ne 1$ for $i=1,2$). In \cite{Br}, the following
notation was introduced:
$$\I_D(i,j) = \I(\{i,j\}\subset S_1) + \I(\{i,j\}\subset S_2)\
,$$\label{funnygonchI}\index{$\I_D(i,j)$}
where $\I(A\subset B)$ is the indicator function which takes the
value $1$ if $A$ is contained in $B$ and $0$ otherwise. Therefore
$\I_D(i,j)\in  \{0,1\}. $  Then we have
\begin{equation} \label{ordcell}
2\,\ord_D (\omega_\gamma) = (\ell-1)- \I_D(\gamma(1),\gamma(2)) -
\I_D(\gamma(2),\gamma(3)) - \ldots- \I_D(\gamma(n),\gamma(1))\ .
\end{equation}
To prove this, observe that $\omega_\gamma=f_\gamma \omega_0$, where
$$f_\gamma= \prod_{i\in \Z/n\Z} {(z_i-z_{i+2}) \over
  (z_{\gamma(i)}-z_{\gamma(i+1)})}\ ,$$ and
$$\omega_0 = {dt_1\ldots dt_\ell \over t_2(t_3-t_1)(t_4-t_2)\ldots
  (t_\ell-t_{\ell-2}) (1-t_{\ell})}$$
is the canonical volume form with no zeros or poles along the
standard cell  defined in \cite{Br}. The proof of (\ref{ordcell})
follows on applying proposition 7.5 from \cite{Br}.

Now, $(\ref{ordcell})$ shows that $\omega_\gamma$ has the worst singularities
when the most possible \\$\I_D(\gamma(i),\gamma(i+1))$ are equal to 1.  This
happens when only two of them are equal to zero, namely
$$S_1=\{\gamma(1),\gamma(2),\ldots,\gamma(k)\}\quad \hbox{ and } \quad
S_2= \{\gamma(k+1),\gamma(k+2),\ldots,\gamma(n)\},\ \ 2\le k\le n-2.$$
In this case, $(\ref{ordcell})$ yields $2\ord_D\omega_\gamma=(\ell-1)-(n-2)=
-2$, so $\ord_D \omega_{\gamma}=-1$. In all other cases we must therefore
have $\ord_D \omega_{\gamma} \geq 0$. Thus the  singular locus of
$\omega_{\gamma}$ is precisely given by the set of divisors bounding
the cell $X_{S,\eta}$.
\end{proof}

\vspace{.3cm}
\subsection{01 cell-forms and a basis of the cohomology of
  $\Mod_{0,n}$}
We first derive some useful identities between certain rational
functions. Let $S=\{1,\ldots, n\}$ and let  $v_1,\ldots,v_n$ denote
coordinates on $\A^n$.
For every cyclic
structure $\gamma$ on $S$, let $\cfl \gamma \cfr=\cfl
v_{\gamma(1)},\ldots,v_{\gamma(n)} \cfr$ denote the rational
function\index{Cell function, $\langle \gamma \rangle$}
\begin{equation}\label{cellfuncdef}
{{1}\over{(v_{\gamma(2)}-v_{\gamma(1)})\cdots
(v_{\gamma(n)}-v_{\gamma(n-1)})(v_{\gamma(1)}-v_{\gamma(n)})}}\in
\Z\Big[v_i, {1\over v_i-v_j}\Big]\ .\end{equation} We refer to such
a function as a cell-function.
 We can
extend its definition linearly to $\Q$-linear combinations of cyclic
structures. Let $X=\{x_1,\ldots, x_n\}$ denote any alphabet on $n$
symbols. Recall that the shuffle product \cite{Re} is defined on
linear combinations of words on $X$ by the inductive formula
\begin{equation}\label{shufflerec}
w \sha e = e\sha w \quad \hbox{ and }\quad   a w \sha a'w' = a(w\sha
a'w') + a'(aw\sha w')\ , \end{equation} where $w,w'$ are any words
in $X$ and $e$ denotes the empty or trivial word.

\begin{defn}\label{defshufprod}  Let $A, B\subset S$  and let
$A\cap B=C=\{c_1,\ldots,c_r\}$.  Let $\gamma_A$ be a cyclic order on
$A$ such that the elements $c_1,\ldots,c_r$ appear in their standard
cyclic order, and let $\gamma_B$ be a cyclic order on $B$ with the
same property.  We write
$\gamma_A=(c_1,A_{1,2},c_2,A_{2,3},\ldots,c_r,A_{r,1})$ and
$\gamma_B=(c_1,B_{1,2},c_2,B_{2,3},\ldots,c_r,B_{r,1})$, where the
$A_{i,i+1}$,  (resp. the $B_{i,i+1}$) together with $C$, form a
partition of $A$ (resp. $B$). We denote the {\it shuffle
  product} of
the two cell-functions $\cfl \gamma_A \cfr$ and $\cfl \gamma_B \cfr
$ with respect to $c_1,\ldots,c_r$ by
$$\cfl \gamma_A \cfr \sha_{c_1,\ldots,c_r}
\cfl\gamma_B\cfr$$
which is defined to be the sum of cell functions
 \begin{equation}\label{formalshufprod}\cfl c_1,A_{1,2}\sha
  B_{1,2},c_2,A_{2,3}\sha B_{2,3},
\ldots,c_r,A_{r,1}\sha B_{r,1}\cfr \ .\end{equation}
\end{defn}

The shuffle product of two cell-functions is related to their actual
product by the following lemma.

\begin{prop}\label{shufprod} Let $A,B \subset S$, such that $|A\cap
B|\geq 2$.  Let $\gamma_A$, $\gamma_B$ be cyclic structures on $A,B$
such that the  cyclic structures on $A\cap B$ induced by  $\gamma_A$
and $\gamma_B$ coincide. If $\gamma_{A\cap B}$ denotes the induced
cyclic structure on $A\cap B$, we have:
\begin{equation}\label{cellfunctionshufflerel}
{{\cfl\gamma_A\cfr \cdot \cfl \gamma_B\cfr }\over{ \cfl
\gamma_{A\cap B}\cfr }}= \cfl \gamma_A\cfr \sha_{\gamma_{A\cap B}}
\cfl \gamma_B\cfr \ .
\end{equation}
\end{prop}

\begin{proof}
Write the cell functions $\cfl \gamma_A \cfr $ and $\cfl \gamma_B
\cfr$ as $\cfl a_{i_1}, P_1, a_{i_2}, P_2,\ldots ,a_{i_r}, P_{r}\cfr
$ and\break $\cfl a_{i_1}, R_1, a_{i_2}, R_2, \ldots, a_{i_r}
,R_{r}\cfr$, where $P_i, R_i$ for $1\leq i\leq r$ are tuples of
elements in $S$. Let $\Delta_{ab}= (b-a)$. We will first prove the
result for $r=2$ and $P_2,R_2=\emptyset$:
\begin{equation}\label{r=2}
\Delta_{ab}\Delta_{ba}\cfl a ,p_1,\ldots ,p_{k_1} ,b\cfr  \cfl a
,r_1,\ldots,
  r_{k_2},b\cfr   =\cfl a ,(p_1,\ldots ,p_{k_1})\sha (r_1,\ldots,
  r_{k_2}),b\cfr  . 
\end{equation}
We prove this case by induction on $k_1+k_2$.

Trivially, for $k_1+k_2=0$ we have
\begin{equation*}
\Delta_{ab}\Delta_{ba}\cfl a, b\cfr  \cfl a, b\cfr   = \cfl a, b\cfr .
\end{equation*}
Now assume the induction hypothesis that
\begin{align*} & \Delta_{ab}\Delta_{ba}\cfl a,
  p_2 ,\ldots, p_{k_1},
  b\cfr  \cfl a ,r_1,\ldots ,r_{k_2}, b\cfr   = \cfl a, \bigl(
  (p_2,\ldots, p_{k_1})\sha 
  (r_1, \ldots
  ,r_{k_2})\bigr),b\cfr
\hbox{ and } \\ & \Delta_{ab}\Delta_{ba}\cfl a ,p_1,\ldots, p_{k_1},
    b\cfr  \cfl a, r_2,\ldots
  ,r_{k_2},b\cfr   = \cfl a ,\bigl( (p_1,\ldots ,p_{k_1})\sha (r_2,\ldots
  ,r_{k_2})\bigr),b\cfr  .\end{align*}
To lighten the notation, let $p_2,\ldots ,p_{k_1}=\underline{p}$ and
$r_2,\ldots, r_{k_2}=\underline{r}$. By the  shuffle recurrence
formula $(\ref{shufflerec})$ and  the induction hypothesis:
\begin{align*}
\cfl a, \bigl( (p_1,\underline{p})\sha (r_1,\underline{r})\bigr) ,
b\cfr   & = \cfl a ,p_1, \bigl( (\underline{p} )\sha
(r_1,\underline{r})\bigr) , b\cfr   +
\cfl a ,r_1 ,\bigl( (p_1, \underline{p}) \sha (\underline{r})\bigr), b\cfr   \\
&= \frac{\Delta_{p_1b}\cfl p_1, \bigl( (\underline{p})
    \sha (r_1, \underline{r})\bigr), b\cfr  }{\Delta_{ab}\Delta_{ap_1}} +
  \frac{ \Delta_{r_1b}\cfl r_1, \bigl( (p_1,\underline{p})
    \sha \underline{r}) \bigr), b\cfr  } {\Delta_{ab}\Delta_{ar_1}}\\
&=\frac{\Delta_{p_1b}\Delta_{bp_1}\Delta_{p_1b}\cfl p_1
,\underline{p}, b\cfr  \cfl
  p_1 , r_1, \underline{r}, b\cfr  }{\Delta_{ab}\Delta_{ap_1} }\\ & \qquad +
\frac{\Delta_{r_1b}\Delta_{br_1}\Delta_{r _1b}\cfl r_1 ,p_1,
  \underline{p},
  b\cfr  \cfl  r_1 , \underline{r}, b\cfr  }{\Delta_{ab}\Delta_{ar_1} }
\end{align*}
Using identities such as $\cfl p_1 ,\underline{p},b\cfr
={\Delta_{ap_1}\Delta_{ba} \over \Delta_{bp_1}}\cfl a,
p_1,\underline{p}, b\cfr$, this is
$$\Big[{\Delta^2_{p_1b}\Delta_{bp_1} \over \Delta_{ab}
\Delta_{ap_1}}\, {\Delta_{ap_1}\Delta_{ba}\over \Delta_{bp_1}}
\,{\Delta_{ba}\Delta_{ar_1} \over \Delta_{bp_1}\Delta_{p_1r_1}} +
{\Delta^2_{r_1 b} \Delta_{br_1}  \over \Delta_{ab}\Delta_{ar_1}}
{\Delta_{ap_1}\Delta_{ba} \over \Delta_{r_1p_1}\Delta_{br_1} }
{\Delta_{ba}\Delta_{ar_1} \over \Delta_{br_1}}\Big]\cfl a , p_1,
\underline{p}, b\cfr \cfl a, r_1, \underline{r}, b\cfr  $$
$$=\Delta_{ab} \Big[{\Delta_{ar_1}\Delta_{bp_1} \over \Delta_{p_1r_1}}
  + {\Delta_{br_1}\Delta_{ap_1}\over \Delta_{r_1p_1}}\Big]\cfl a ,
p_1, \underline{p},  b\cfr  \cfl 
a, r_1, \underline{r}, b\cfr = \Delta_{ab}\Delta_{ba}\cfl a , p_1,
\underline{p},  b\cfr  \cfl a, r_1, \underline{r}, b\cfr .$$
 The last equality is the Pl\"ucker relation
  $\Delta_{ar_1}\Delta_{bp_1}-
  \Delta_{br_1}\Delta_{ap_1}=\Delta_{p_1r_1}\Delta_{ba}$. This proves
  the identity \eqref{r=2}. 
Now,  using the identity
\begin{align*}
\cfl a_{i_1} P_1 a_{i_2} P_2 \ldots a_{i_r} P_r \cfr  =
\Delta_{a_{i_2}a_{i_1}} \cfl a_{i_1} P_1 a_{i_2}\cfr   \times
\Delta_{a_{i_3}a_{i_2}} \cfl a_{i_2} P_2 a_{i_3}\cfr  \times \cdots
\times \Delta_{a_{i_r} a_{i_1}} \cfl a_{i_r} P_r a_{i_1}\cfr  ,
\end{align*}
the general case follows   from \eqref{r=2}.
\end{proof}

\begin{cor} Let $X$ and $Y$ be disjoint sequences of indeterminates and
let $e$ be an indeterminate not appearing in either $X$ or $Y$.  We
have the following identity on cell functions:
\begin{equation} \label{cellfunction1shuffsare0}
\cfl (X,e)\sha_e (Y,e) \cfr= \cfl X\sha Y ,e\cfr =0.
\end{equation}
\end{cor}

\begin{proof} Write $X=x_1,x_2,...,x_n$ and $Y=y_1,y_2,...,y_m$.
By the recurrence formula for the shuffle product and proposition
\ref{shufprod}, we have
\begin{align*} \cfl X\sha Y ,e\cfr   & = \cfl x_1, (x_2,...,x_n\sha
  y_1,...,y_m) , e\cfr +
\cfl y_1, (x_1,...,x_n \sha y_2,...,y_m) , e\cfr   \\
&= \cfl X,e\cfr  \cfl x_1, Y, e\cfr  (e-x_1)(x_1-e) + \cfl y_1, X,
e\cfr   \cfl Y,e\cfr
(y_1-e)(e-y_1) \\
&= \frac{(e-x_1)(x_1-e)}{(x_2-x_1)\cdots (e-x_n)
   (x_1-e)\ (y_1-x_1)(y_2-y_1)\cdots (e-y_m)(x_1-e)} \\
&\qquad + \frac{(y_1-e)(e-y_1)} {
   (x_1-y_1)(x_2-x_1)\cdots (e-x_n) (y_1-e) \ (y_2-y_1)\cdots
   (e-y_m)(y_1-e)} \\
&= \frac{ (-1) + (-1)^2 }{ (x_2-x_1) \cdots (e-x_n)
   \ (y_1-x_1)(y_2-y_1)\cdots (e-y_m)} =0 \ .\\
\end{align*}
\end{proof}
\noindent By specialization, we  can formally extend the definition
of a cell function to the case where some of the terms $v_i$ are
constant, or one of the  $v_i$ is infinite, by setting
$$\cfl v_1,\ldots, v_{i-1}, \infty, v_{i+1},\ldots, v_n\cfr =
\lim_{x\rightarrow \infty} x^2 \cfl v_1,\ldots,
v_{i-1},x,v_{i+1},\ldots, v_n\cfr$$ $$= {1 \over (v_2-v_1)\ldots
(v_{i-1}-v_{i-2})(v_{i+2}-v_{i+1}) \ldots(v_{n}-v_{n-1})(v_1-v_n) }\
.$$ \noindent This is the rational function obtained by omitting all
terms containing $\infty$. By taking the appropriate limit, it is
clear that $(\ref{cellfunctionshufflerel})$ and
$(\ref{cellfunction1shuffsare0})$ are valid in this case too. In the
case where $\{v_1,\ldots, v_n\} = \{0,1,t_1,\ldots, t_\ell,
\infty\}$ we have the formula
\begin{equation}
[v_1,\ldots, v_n]= \cfl v_1,\ldots, v_n \cfr \, dt_1dt_2\ldots
dt_\ell\ .
\end{equation}

\vspace{.1cm}
\begin{defn}\label{01cellfuncDEF} A $01$ cyclic (or dihedral)
  structure is a cyclic (or dihedral) 
structure on $S$ in which the numbers $1$ and $n-1$ are consecutive.
Since $z_1=0$ and $z_{n-1}=1$, a $01$ cyclic (or dihedral) structure
is a set of orderings of the set
$\{z_1,\ldots,z_n\}=\{0,t_1,\ldots,t_\ell,1, \infty\}$, in which the
elements $0$ and $1$ are consecutive.  In these terms, each dihedral
structure can be written as an ordering $(0,1,\pi)$ where $\pi$ is
some ordering of $\{t_1,\ldots,t_\ell,\infty\}$.  To each such
ordering we associate a cell-function $\cfl 0,1,\pi \cfr$, which is
called a $01$ cell-function.
\end{defn}
Since $01$ cell-functions corresponding to different $\pi$ are
clearly different, it follows that there exist exactly $(n-2)!$
distinct $01$ cell-functions $\cfl 0,1,\pi \cfr$.  To these
correspond $(n-2)!$ distinct $01$ cell-forms
$\omega_{(0,1,\pi)}=\cfl 0,1,\pi\cfr \, dt_1\ldots dt_\ell$.

\begin{thm} \label{thm01cellsspan} The set of $01$ cell-forms
$\omega_{(0,1,\pi)}$, where $\pi$ denotes any ordering of $\{t_1,\ldots,t_\ell,
\infty\}$, has cardinal $(n-2)!$ and forms a basis of $H^\ell(\Mod_{0,n},\Q)$.
\end{thm}
\begin{proof} The proof is based on the following well-known result due to
Arnol'd \cite{Ar}.
\begin{thm} 
A basis of $H^\ell(\Mod_{0,n},\Q)$ is given by the classes of the forms
\begin{equation}\label{Omegadefn}
\Omega(\underline{\varepsilon}) := {dt_1 \ldots dt_\ell \over
(t_1-\varepsilon_1)\ldots (t_\ell-\varepsilon_\ell)}\ ,\quad
\varepsilon_{i} \in E_i\ ,\end{equation} where $E_1=\{0,1\}$ and
$E_i =\{0,1,t_1,\ldots, t_{i-1}\}$ for $2\leq i\leq \ell$.
\end{thm}
It suffices to prove that each element
$\Omega(\underline{\varepsilon})$ in $(\ref{Omegadefn})$ can be
written as a linear combination of $01$ cell-forms.  We begin by
expressing a given rational function
${{1}\over{(t_1-\varepsilon_1)\cdots (t_\ell-\epsilon_\ell)}}$ as a
product of cell-functions and then apply proposition \ref{shufprod}.
 To every $t_i$, we associate its {\it type} $\tau(t_i)\in\{0,1\}$
 (which depends on $\varepsilon_1,\ldots, \varepsilon_\ell$)
as follows. If $\varepsilon_i=0$ then $\tau(t_i)=0$; if
$\varepsilon_i=1$, then $\tau(t_i)=1$, but if $\varepsilon_i\ne 0,1$
then $\varepsilon_i=t_j$ for some $j<i$, and the type of $t_i$ is
defined to be equal to the type of $t_j$. Since the indices
decrease,
 the type is
well-defined.

We associate a cell-function $F_i$ to each factor
$(t_i-\varepsilon_i)$ in the denominator of
$\Omega(\underline{\varepsilon})$ as follows:
\begin{equation} F_i=
\begin{cases}\ \ \cfl  0,1,t_i,\infty\cfr  &\hbox{if}\ \varepsilon_i=1\\
-\cfl  0,1,\infty,t_i\cfr  &\hbox{if}\ \varepsilon_i=0\\
\ \ \cfl  0,1,\varepsilon_i,t_i,\infty\cfr  &\hbox{if}\
\varepsilon_i\ne 1\ \hbox{and
the type}\ \tau(t_i)=1\\
-\cfl  0,1,\infty,t_i,\varepsilon_i\cfr  &\hbox{if }\varepsilon_i\ne
0\ \hbox{and the type}\ \tau(t_i)=0\ .
\end{cases}
\end{equation}
We have
$$\Omega(\underline{\varepsilon})=\Delta\prod_{i=1}^\ell F_i\ ,$$
where
$$\Delta=\prod_{j|\varepsilon_j\ne 0,1}(-1)^{\tau(\varepsilon_j)-1}
(\varepsilon_j-\tau(\varepsilon_j))$$
is exactly the factor occurring when multiplying cell-functions as in
proposition
\ref{shufprod}.  This product can be expressed as a shuffle
product, which is a sum of
cell-functions.  Furthermore each one corresponds to a cell beginning
$0,1,\ldots$ since this is the case for all of the $F_i$.
The $01$-cell forms thus span $H^\ell(\Mod_{0,n},\Q)$.
Since there are exactly $(n-2)!$ of them, and
since $\dim H^\ell(\Mod_{0,n},\Q)=(n-2)!$, they must form a basis.
\end{proof}

\vspace{.3cm}
\subsection{Pairs of polygons and multiplication}\label{prodmapsection}

Let $S=\{1,\ldots,n\}$, and let ${\cal P}_S$\index{${\cal P}_S$} denote the $\Q$-vector
space generated by the set of cyclic structures $\gamma$ on $S$,
modulo the relation $\gamma=(-1)^n\overleftarrow{\gamma}$, where
$\overleftarrow{\gamma}$  denotes the cyclic structure with the
opposite orientation to $\gamma$.

\subsubsection{Shuffles of polygons}
Let $T_1,T_2$ denote two subsets of $Z=\{z_1,\ldots,z_n\}$
satisfying:
\begin{eqnarray}\label{Tprodconds1}
T_1\cup T_2 & =& Z \\
\label{Tprodconds2} |T_1\cap T_2| & = & 3\ .
\end{eqnarray}
Let $E=\{z_{i_1},z_{i_2},z_{i_3}\}$ denote the set of three points common to
$T_1$ and $T_2$.  Given two cyclic structures $\gamma_1,\gamma_2$ on $T_1,T_2$
respectively, the restriction $\gamma_1|_E$ gives a cyclic order on $E$.
Let $\varepsilon=1$ if this order is compatible with the standard order
on $\{1,\ldots,n\}$, $\varepsilon=-1$ otherwise.
We define the {\it shuffle} $\gamma_1\sha\gamma_2$ of
$\gamma_1$ and $\gamma_2$ relative to the three points of intersection
of $T_1$ and $T_2$ by the formula
\begin{equation}
\label{Gf}
\varepsilon(\gamma_1\sha \gamma_2) =
\begin{cases}
\displaystyle{\sum_{{\gamma\big|_{T_1}=\gamma_1},\
{\gamma\big|_{T_2}  =
\gamma_2}}} \gamma&\hbox{if}\ \gamma_1\big|_E=\gamma_2\big|_E\\
\displaystyle{ \sum_{{\gamma\big|_{T_1}=\gamma_1},\
    {\gamma\big|_{T_2} =
\overleftarrow{\gamma_2}}}}\!\!\!\!\!\!\!\!\!\! \!\!
(-1)^{|T_2|} \gamma&\hbox{if}\
\gamma_1\big|_E=\overleftarrow{\gamma_2}\big|_E.
\end{cases}
\end{equation}
The formal sum of polygons $\gamma_1\sha \gamma_2$ is well-defined
and non-zero.

Assume $\{z_1,\ldots,z_n\}=\{0,1,\infty,t_1,\ldots,t_\ell\}$ with
$E=\{0,1,\infty\}$.  Assume also that $\gamma_1=(0, A_{1,2}, 1, A_{2,3},
\infty, A_{3,1})$ where $T_1$ is the disjoint union of
$A_{1,2}, A_{2,3}, A_{3,1}$, and $0,1,\infty$, and
$\gamma_2=(0, B_{1,2}, 1, B_{2,3}, \infty, B_{3,1})$, where $T_2$ is
the disjoint union of $B_{1,2}, B_{2,3}, B_{3,1}$, and $0,1,\infty$.
Then $\gamma_1\sha \gamma_2$ is the sum of cyclic structures
$$\gamma= (0, C_{1,2}, 1, C_{2,3}, \infty, C_{3,1})\ ,$$
where each $C_{i,j}$ is a shuffle of the ordered disjoint
sets $A_{i,j}$ and $B_{i,j}$.

\begin{example} \label{exsh1}  Let $T_1=\{0,1,\infty,t_1,t_3\}$ and
$T_2=\{0,1,\infty,t_2,t_4\}$.
If $\gamma_1$ and $\gamma_2$ denote the cyclic orders
$(0,t_1,1,t_3,\infty)$ and $(0,1,t_2,\infty,t_4)$,
then we have
$$\gamma_1\sha\gamma_2=(0,t_1,1,t_2,t_3,\infty,t_4)+
(0,t_1,1,t_3,t_2,\infty,t_4) \ .  $$ We will often write, for
example, $(0,t_1,1,t_2\sha t_3,\infty,t_4)$ for the right-hand side.
\end{example}

\vspace{.2cm}
\subsubsection{Multiplying pairs of polygons: the modular shuffle
  relation}
We will now consider pairs of polygons $(\gamma,\eta)\in {\cal P}_S
\times {\cal P}_S$\label{polypairdef}\index{Polygon pair, $(\gamma, \eta)$}. We can associate a geometric
meaning to a pair 
of polygons  as follows. The left-hand polygon $\gamma$, which we
will write using round parentheses, for example
$(0,t_1,\ldots,t_\ell,1,\infty)$, is associated to the real cell
$X_\gamma$ of the moduli space $\Mod_{0,n}$ associated to the cyclic
structure. The right-hand polygon $\eta$, which we will write using
square parentheses, for example $[0,t_1,\ldots,t_\ell,1,\infty]$, is
associated to the cell-form $\omega_\eta$ associated to the cyclic
structure.  The pair of polygons will be associated to the (possibly
divergent) integral $\int_{X_\gamma}\omega_\eta$. In the following
section we will investigate in detail the map from pairs of polygons
to integrals.

\begin{defn}  Given sets $T_1, T_2$ as above, the
{\it modular shuffle relation} on the vector space ${\cal P}_S\times
{\cal P}_S$ is defined by
\begin{equation}\label{multpairs}
(\gamma_1,\eta_1)\sha (\gamma_2,\eta_2)=(\gamma_1\sha \gamma_2,
\eta_1\sha\eta_2),
\end{equation}
for pairs of polygons $(\gamma_1,\eta_1)\sha
(\gamma_2,\eta_2)$, where $\gamma_i$ and $\eta_i$ are cyclic structures
on $T_i$ for $i=1,2$.

\end{defn}

\begin{example} The following product of two polygon pairs is given by
\begin{align*}
\bigl(
(0,t_1,1,\infty, t_4), [0,\infty,t_1, t_4,1] \bigr) & \bigl( (
0,t_2,1,t_3,\infty), [0,t_3,t_2,\infty,1] \bigr)\\
& = - \bigl( (0,t_1\sha t_2, 1, t_3, \infty,t_4), [0,t_3,t_2,\infty,
t_1,t_4,1] \bigr).
\end{align*}
\end{example}

\vspace{.3cm} Let us now give a geometric interpretation of
(\ref{multpairs}) in terms of integrals of forms on moduli space.
Recall that a \emph{product map}\label{prodmapDEF}\index{Product map} between moduli
spaces was defined 
in\cite{Br} as  follows. Let $T_1,T_2$ denote two subsets of
$Z=\{z_1,\ldots,z_n\}$ satisfying:
\begin{eqnarray}\label{Tprodconds3}
T_1\cup T_2 & =& Z \\
\label{Tprodconds4} |T_1\cap T_2| & = & 3\ .
\end{eqnarray}
Then we can consider the product of forgetful maps:
\begin{equation}
f=f_{T_1}\times f_{T_2} : \Mod_{0,n} \To \Mod_{0,T_1} \times
\Mod_{0,T_2}\ .
\end{equation}
The map $f$ is a birational embedding
because $$\dim \Mod_{0,S} = |S|-3= |T_1|-3+|T_2|-3= \dim
\Mod_{0,T_1}\times \Mod_{0,T_2}\ .$$

If $f$ is a product map as above and
$z_i,z_j,z_k$ are the three common points of $T_1$ and $T_2$,
use an element $\alpha\in \PSL_2$ to map $z_i$ to $0$, $z_j$ to $1$
and $z_k$ to $\infty$.  Let $t_1,\ldots,t_\ell$ denote the images of
$z_1,\ldots,z_n$ (excluding $z_i,z_j,z_k$) under $\alpha$.
Given the indices $i$, $j$ and $k$, the product map is then determined by
specifying a partition of $\{t_1,\ldots,t_\ell\}$ into $S_1$ and $S_2$.
We use the notation $T_i=\{0,1,\infty\}
\cup S_i$ for $i=1,2$.

The multiplication formula (\ref{multpairs}) on pairs of polygons
translates to a multiplication formula for integrals of
cell-forms.

\begin{prop}\label{prodmaprel} Let $S=\{1,\ldots,n\}$, and let $T_1$
  and $T_2$ be subsets 
of $S$ as above, of orders $r+3$ and $s+3$ respectively.
Let $\omega_1$ (resp. $\omega_2$) be a cell-form on
$\Mod_{0,r}$  (resp. on $\Mod_{0,s}$), and let $\gamma_1$ and $\gamma_2$
denote cyclic orderings on $T_1$ and $T_2$.  Then the product rule for
integrals is given by
\begin{equation}\label{temp}\int_{X_{\gamma_1}} \omega_1\int_{X_{\gamma_2}}
\omega_2=\int_{X_{\gamma_1\sha\gamma_2}} \omega_1\sha \omega_2,
\end{equation}
where $\omega_1\sha\omega_2$ converges on the cell $X_\gamma$ for
each term $\gamma$ in $\gamma_1\sha\gamma_2$.

\end{prop}

\begin{proof} The subsets $T_1$ and $T_2$ correspond to a product  map
$$f:\Mod_{0,n}\rightarrow \Mod_{0,r}\times \Mod_{0,s}.$$
The pullback formula gives a multiplication law on the pair of integrals:
\begin{equation}\label{temp2}\int_{X_{\gamma_1}} \omega_1\int_{X_{\gamma_2}}
\omega_2= \int_{X_{\gamma_1}\times X_{\gamma_2}} \omega_1\wedge\omega_2=
\int_{f^{-1}(X_{\gamma_1}\times X_{\gamma_2})} f^*(\omega_1\wedge\omega_2).
\end{equation}
The preimage $f^{-1}(X_{\gamma_1}\times X_{\gamma_2})$ decomposes
into a disjoint union of cells of $\Mod_{0,n}$, which are precisely the cells
given by cyclic orders of $\gamma_1\sha \gamma_2$.
In other words,
$$f^{-1}(X_{\gamma_1}\times X_{\gamma_2})=\sum_{\gamma\in \gamma_1\sha\gamma_2}
X_{\gamma}\ ,$$ where the sum denotes a disjoint union. Now we can
assume without loss of generality that $T_1=\{0,1,\infty,
t_1,\ldots, t_k\}$, $T_2=\{0,1,\infty, t_{k+1},\ldots, t_\ell\}$ and
that $\delta_1,\delta_2$ are the cyclic structures on $T_1,T_2$
corresponding to $\omega_1,\omega_2$, respectively, where $\delta_1,
\delta_2$ restrict to the standard cyclic order  on $0,1,\infty$.
Then, in cell function notation,
$$f^*(\omega_{1}\wedge \omega_{2}) 
= \cfl \delta_1 \cfr \cfl \delta_2 \cfr \,dt_1\ldots dt_\ell= {\cfl
\delta_1 \sha_{\{0,1,\infty\}} \delta_2 \cfr \over \cfl
0,1,\infty\cfr} \,dt_1\ldots dt_\ell=\omega_1\sha \omega_2\ ,$$ by
proposition $\ref{shufprod}$. Since $\omega_1$ and $\omega_2$
 converge on the closed cells $\overline{X}_{\gamma_1}$ and
 $\overline{X}_{\gamma_2}$ respectively, $\omega_1\wedge \omega_2$
 has no poles on the contractible set $\overline{X}_{\gamma_1}\times \overline{X}_{\gamma_2},$
and therefore
 $\omega_1\sha\omega_2=f^*(\omega_1\wedge\omega_2)$ has no poles on the closure of
 $f^{-1}(X_{\gamma_1}\times X_{\gamma_2})$. But $\sum_{\gamma\in \gamma_1\sha\gamma_2}
X_{\gamma}$ is a cellular decomposition of
$f^{-1}(X_{\gamma_1}\times X_{\gamma_2})$, so, in particular,
$\omega_1 \sha \omega_2$ can have no poles along the closure of each
cell $ X_{\gamma},$ where $\gamma\in \gamma_1\sha\gamma_2$.
\end{proof}

\subsubsection{$\Sym(n)$ action on pairs of polygons}
The symmetric group $\Sym(n)$ acts on a pair of polygons by
permuting their labels in the obvious way, and this extends to the
vector space ${\cal P}_S\times {\cal P}_S$ by linearity. If
$\tau:\Mod_{0,n}\rightarrow \Mod_{0,n}$  is an element of $\Sym(n)$,
then the corresponding action on integrals is given by the pullback
formula:
\begin{equation}\label{permaction}
\int_{X_\gamma}\omega_\eta = \int_{\tau(X_\gamma)}
\tau^*(\omega_\eta) = \int_{X_{\tau(\gamma)}} \omega_{\tau(\eta)}\ .
\end{equation}
Note that, unlike for formal pairs of polygons, this formula only
holds for linear combinations of cell-forms which are convergent,
even if each individual cell-form is not convergent over the
integration domain.

Suppose that $\tau$ belongs to the dihedral group which preserves
the dihedral structure underlying a cyclic structure $\gamma$. Let
$\epsilon=1$ if $\tau$ preserves $\gamma$, and $\epsilon=-1$ if
$\tau$ reverses its orientation. We have the following {\it dihedral
relation} between convergent integrals:
\begin{equation}\label{dihrel}
\int_{X_\gamma} \omega_\eta = (-1)^\epsilon\int_{X_\gamma}
\tau^*(\omega_\eta)=(-1)^\epsilon\int_{X_\gamma} \omega_{\tau(\eta)}.
\end{equation}

As above, this formula extends to linear combinations of integrals of
cell-forms as long as the linear combination converges over the integration
domain.

\vspace{.2cm}
\begin{example}
The form corresponding to $\zeta(2,1)$ on $\Mod_{0,6}$ is
$$\frac{dt_1dt_2dt_3}{(1-t_1)(1-t_2)t_3} = [0,1,t_1,t_2,\infty,t_3]
+ [0,1,t_2,t_1,\infty,t_3],$$ which gives  $\zeta(2,1)$ after
integrating over the standard cell.  By applying the rotation
$(1\!,\!2\!,\!3\!,\!4\!,\!5\!,\!6)$, a dihedral rotation of the standard cell, to this
form, one obtains
\begin{align*}[t_1,\infty,t_2,t_3,0,1] +
  [t_1,\infty,t_3,t_2,0,1] & =
[0,1,t_1,\infty, t_2,t_3] +
[0,1,t_1,\infty,t_3,t_2]
\\ & ={{dt_1dt_2dt_3}\over{(1-t_1)t_2t_3}},\end{align*}
which gives $\zeta(3)$ after integrating over the standard cell.
Therefore, we have the following relation on linear combinations of
pairs of polygons:
\begin{equation}
\begin{split}
&\bigl( (0,t_1,t_2,t_3,1,\infty), [0,1,t_1,t_2,\infty,t_3] +
[0,1,t_2,t_1,\infty,t_3] \bigr)\\
&\qquad = \bigl( (0,t_1,t_2,t_3,1,\infty), [0,1,t_1,\infty, t_2,t_3] +
[0,1,t_1,\infty,t_3,t_2]\bigr)
\end{split}
\end{equation}
or
\begin{align*}
\zeta(2,1)=\int_{X_{3,\delta}}
\frac{dt_1dt_2dt_3}{t_3(1-t_2)(1-t_1)} &=
\int_{X_{3,\delta}} \frac{dt_1dt_2dt_3}{t_3t_2(1-t_1)}=\zeta(3). \\
\end{align*}
\end{example}

\vspace{.3cm}
\subsubsection{Standard pairs and the product map relations}\label{prodmaps}
A standard pair of polygons is a pair $(\delta,\eta)$ where the
left-hand polygon is the standard cyclic structure.  Let
$S=\{1,\ldots,n\}$, and $T_1\cup T_2=S$ with $T_1\cap
T_2=\{0,1,\infty\}$  be as above, and let $\gamma_1$ and $\gamma_2$
be cyclic orders on $T_1$ and $T_2$. In the present section we show
how for each such $\gamma_1,\gamma_2$, we can modify the modular
shuffle relation to construct a multiplication law on standard
pairs.

Let $\delta_1$ and $\delta_2$ denote the standard orders on $T_1$
and $T_2$.  Then there is a unique permutation $\tau_i$ mapping
$\delta_i$ to $\gamma_i$ such that $\tau_i(0)=0$, for $i=1,2$.  The
multiplication law, denoted by the symbol $\times$, and called the
{\it product map relation}, is defined by
\begin{equation}\label{pm}
\begin{split}
(\delta_1,\omega_1)\times(\delta_2,\omega_2)&=(\gamma_1,\tau_1(\omega_1))
\sha(\gamma_2,\tau_2(\omega_2))\\
&=(\gamma_1\sha\gamma_2,
\tau_1(\omega_1)\sha\tau_2(\omega_2))\\
&=\sum_{\gamma\in\gamma_1\sha\gamma_2}(\delta,\tau_\gamma^{-1}(
\tau_1(\omega_1)\sha\tau_2(\omega_2))),
\end{split}
\end{equation}
where for each $\gamma\in\gamma_1\sha\gamma_2$, $\tau_\gamma$ is
the unique permutation such that $\tau_\gamma(\delta)=\gamma$ and
$\tau_\gamma(0)=0$.

\begin{example} Let $S=\{0,1,\infty,t_1,t_2,t_3,t_4\}$, $T_1=\{0,1,\infty,
t_1,t_4\}$ and $T_2=\{0,1,\infty,t_2,t_3\}$.  Let the cyclic orders
on $T_1$ and $T_2$ be given by $\gamma_1=(0,t_1,1,\infty,t_4)$ and
$\gamma_2=(0,t_2,1,t_3,\infty)$.  Applying the product map relation
to the pairs of polygons below yields
\begin{equation}
\begin{split}
\bigl((0,t_1,t_4,1,\infty),&[0,1,t_1,\infty,t_4]\bigr)\times
\bigl((0,t_2,t_3,1,\infty),[0,1,t_2,\infty,t_3]\bigr)\\
&=\bigl( (0,t_1,1,\infty, t_4), [0,\infty,t_1, t_4,1] \bigr) \sha\bigl( (
0,t_2,1,t_3,\infty), [0,t_3,t_2,\infty,1] \bigr)\\
& = - \bigl( (0,t_1,t_2, 1, t_3, \infty,t_4), [0,t_3,t_2,\infty,
t_1,t_4,1] \bigr)\\
 &\qquad\qquad\qquad - \bigl( (0,t_2,t_1, 1, t_3, \infty,t_4), [0,t_3,t_2,
\infty, t_1,t_4,1] \bigr)\\
&=\bigl((0,t_1,t_2,t_3,t_4,1,\infty),[0,t_3,\infty,t_1,1,t_2,t_4]+
[0,t_3,\infty,t_2,1,t_1,t_4].
\end{split}
\end{equation}
\end{example}

In terms of integrals, this corresponds to the relation
\begin{equation}
\begin{split}
\zeta(2)^2&=\int_{X_{5,\delta}} {{dt_1dt_4}\over{(1-t_1)t_4}}
\int_{X_{5,\delta}} {{dt_2dt_3}\over{(1-t_2)t_3}}\\
&=\int_{X_{7,\delta}} {{dt_1dt_2dt_3dt_4}\over{t_4(t_4-t_2)(1-t_2)(1-t_1) t_3}}
+{{dt_1dt_2dt_3dt_4}\over{t_4(t_4-t_1)(1-t_1)(1-t_2) t_3}}\\
\end{split}
\end{equation}
We will show in $\S\ref{calculations}$ that the last  two integrals
evaluate to  ${7\over 10}\zeta(2)^2$ and ${3\over 10}\zeta(2)^2$
respectively. \vspace{.3cm}
\subsection{The algebra of cell-zeta values} \label{cellalg}

\begin{defn}\label{cellzetavalueDEFch3}\index{Cell zeta value}\index{${\mathcal C}$}
Let ${\cal C}$ denote the $\Q$-vector space generated by the
integrals $\int_{X_{n,\delta}} \omega$, where $X_{n,\delta}$ denotes
the standard cell of $\Mod_{0,n}$ for $n\ge 5$ and $\omega$ is a
holomorphic $\ell$-form on $\Mod_{0,n}$ with logarithmic
singularities at infinity (thus a linear combination of $01$
cell-forms) which converges on $X_{n,\delta}$.  We call these
numbers {\it cell-zeta values}.  The existence of product map
multiplication laws in proposition \ref{prodmaprel} imply that
${\cal C}$ is in fact a $\Q$-algebra.
\end{defn}

\begin{thm} The $\Q$-algebra ${\cal C}$ of cell-zeta values is isomorphic
to the $\Q$-algebra ${\cal Z}$ of multizeta values.
\end{thm}

\begin{proof} Multizeta values are real numbers which can all be
expressed as integrals $\int_{X_{n,\delta}} \omega$
where $\omega$ is an $\ell$-form of the form
\begin{equation}\label{Kont}
\omega=(-1)^d\prod_{i=1}^\ell {{d\underline{t}}\over{t_i-\epsilon_i}},
\end{equation}
where $\epsilon_1=0$, $\epsilon_i\in \{0,1\}$ for $2\le i\le
\ell-1$, $\epsilon_\ell=1$, and $d$ denotes the number of $i$ such
that $\epsilon_i=1$. Since each such form converges on
$X_{n,\delta}$, the multizeta algebra ${\cal Z}$ is a subalgebra of
${\cal C}$.  The converse is a consequence of the following theorem
due to F. Brown \cite{Br}.
\begin{thm} If $\omega$ is a holomorphic $\ell$-form on $\Mod_{0,n}$ with
logarithmic singularities at infinity and convergent on
$X_{n,\delta}$, then $\int_{X_{n,\delta}} \omega$ is $\Q$-linear
combination of multizeta values.
\end{thm}
Thus, ${\cal C}$ is also a subalgebra of ${\cal Z}$, proving the equality.
\end{proof}

The structure of the multizeta algebra, or rather, of the formal
version of it given by quotienting the algebra of symbols formally
representing integrals of the form (\ref{Kont}) by the main known
relations between these forms (shuffle and stuffle), has been much
studied of late.  The present article provides a different approach
to the study of this algebra, by turning instead to the study of a
formal version of ${\cal C}$.

\begin{defn}\label{formalcell} Let $|S|\geq 5$.
The {\it formal algebra of cell-zeta values}\index{${\mathcal{FC}}$}\index{Formal cell zeta value algebra} ${\cal FC}$ is defined as
follows.  Let ${\cal A}$ be the vector space of formal linear combinations
of standard pairs of polygons in ${\cal P}_S\times {\cal P}_S$
$$\sum_i a_i(\delta,\omega_i)$$
such that the associated $\ell$-form $\sum_i a_i\omega_i$ converges
on the standard cell $X_{n,\delta}$.  Let ${\cal FC}$ denote the
quotient of ${\cal A}$ by the following three families of relations.
\begin{itemize}
\item{\it Product map relations.\ \ }{These relations were defined in section
\ref{prodmapsection}. For every choice of subsets $T_1, T_2$ of
$S=\{1,\ldots,n\}$ such that  $T_1\cup T_2=S$ and $|T_1\cap T_2|=3$,
and every choice of cyclic orders $\gamma_1, \gamma_2$ on $T_1,
T_2$, formula (\ref{pm}) gives a multiplication law expressing the
product of any two standard pairs of polygons of sizes $|T_1|$ and
$|T_2|$ as a linear combination of standard pairs of polygons of
size $n$.}
\item{\it Dihedral relations.\ \ }{For $\sigma$ in the dihedral group
associated to $\delta$, i.e. $\sigma(\delta)=\pm \delta$, there is a
dihedral relation  $(\delta,\omega)=(\sigma(\delta),
\sigma(\omega))$.}
\item{\it Shuffles with respect to one element.\ \ }{The linear combinations
of pairs of polygons $$(\delta,(A,e)\sha (B,e))$$ where $A$ and $B$ are
disjoint of length $n-1$ are zero, as in (\ref{cellfunction1shuffsare0}).}
\end{itemize}

\end{defn}

With the goal of approaching the combinatorial conjectures given in the
introduction, the purpose of the next chapters is to give an explicit
combinatorial description of a set of generators for ${\cal FC}$.  We do
this in two steps.  First we define the notion of a linear
combination of polygons convergent with respect to a chord of the standard
polygon $\delta$, and thence, the notion of a linear combination of
polygon convergent with respect to the standard polygon.  We
exhibit an explicit basis, the basis of {\it Lyndon insertion words and
shuffles} for the subspace of such linear combinations.  In the subsequent
chapter, we deduce from this a set of generators for the formal cell-zeta
value algebra ${\cal FC}$ and also, as a corollary, a basis for the cohomology
$H^\ell(\Mod_{0,n}^\delta)$, where $\Mod_{0,n}^\delta$ denotes the
union of $\Mod_{0,n}$ with the boundary components containing the
boundary of the standard cell.

\vspace{.5cm}
\section{Polygons and convergence}
\vskip .5cm
In the present chapter, we define the notions of bad chord of a polygon
(a generalization of the notion of a divisor on the boundary of the standard
cell of $\Mod_{0,n}$ along which the differential form diverges),
residue of a polygon along a bad chord, convergence of linear combinations of
polygons along bad chords, and finally, convergence of linear combinations
of polygons with respect to the standard polygon $\delta$.  The main
theorem exhibits an explicit basis for the space of linear combinations
of polygons convergent with respect to the standard polygon, consisting
of linear combinations called {\it Lyndon insertion words and shuffles}.
\vspace{.3cm}
\subsection{Bad chords and polygon convergence}
For any finite set $R$, let ${\cal P}_R$ denote the $\Q$ vector space
of {\it polygons on $R$}, i.e. cyclic structures on $R$, identified with
planar polygons with edges indexed by $R$.

Let ${\cal V}$\label{VIDEF}\index{$V_S$} denote the free polynomial shuffle algebra
on the alphabet 
of positive integers, and let $V$ be the quotient of ${\cal V}$ by the
relations $w=0$ if $w$ is a word in which any letter appears more than once
(these relations imply that $w\sha w'=0$ if $w$ and $w'$ are not disjoint).
The Lyndon basis\label{Lyndonbasis}\index{Lyndon basis} for ${\cal V}$ is given by Lyndon
words and shuffles of
Lyndon words.  The elements of this basis which do not map to zero
remain linearly independent in $V$, whose basis consists of Lyndon
words with distinct letters -- such a word is Lyndon if and only if the
smallest character appears on the left -- and shuffles of disjoint
Lyndon words with distinct letters.  Throughout this chapter, we work
in $V$, so that when we refer to a `word',
we automatically mean a word with distinct letters, and shuffles of such
words are zero unless the words are disjoint.  Let $V_S$ be the
subspace of $V$ spanned by the $n!$ words of length $n$ with distinct letters
in the characters of $S=\{1,\ldots,n\}$.  Then the Lyndon basis for $V_S$
is given by Lyndon words of degree $n$ and shuffles of disjoint Lyndon words
the union of whose letters is equal to $S$.

The vector space ${\cal P}_S$ is generated by $n$-polygons with edges indexed
by $S$.  If we consider $(n+1)$-polygons with edges indexed by $S\cup \{d\}$,
we have a natural isomorphism
\begin{equation}\label{basiciso}
V_S \buildrel\sim\over\rightarrow {\cal P}_{S\cup \{d\}}
\end{equation}
given by writing each cyclic structure on $S\cup \{d\}$ as a word
on the letters of $S$ followed by the letter $d$.
Let $I_S\subset {\cal P}_{S\cup\{d\}}$\index{$I_S$} be the set of shuffles of polygons
$(A\sha B,d)$ where $A\cup B=S$ and $A\cap B=\emptyset$. Then under
the isomorphism above, $I_S$ is identified with the subspace of $V_S$
generated by the part of the Lyndon basis consisting of shuffles.  By a
slight abuse of notation, we use the same notation $I_S$ for both the
subspace of ${\cal P}_{S\cup\{d\}}$ and that of $V_S$.

\vspace{.2cm}
\begin{defn}\label{chordDEF}\index{Chord, $\chi(\gamma)$} Let $D=S_1\cup S_2$ denote a stable
  partition of $S$ (partition 
into two disjoint subsets of cardinal $\ge 2$).  Let $\gamma$ be a polygon on
$S$.  We say that the partition $D$ corresponds to a {\it chord of $\gamma$} if
the polygon $\gamma$ admits a chord which cuts $\gamma$ into two pieces
indexed by $S_1$ and $S_2$.  Let a {\it block} of $\gamma$ be a subsequence
of consecutive elements of $\gamma$ for the cyclic order, of length at least
two and at most $n-2$. Thus, a chord divides $\gamma$ into two blocks, and
$\chi(\gamma)$ indexes the set of stable partitions which are {\it compatible}
with $\gamma$, in the sense that they can be realized as chords of
$\gamma$, i.e. in the sense that the subsets $S_1$ and $S_2$ are blocks of
$\gamma$.
\end{defn}

\vspace{.2cm}
\begin{defn} \label{defnconvpoly}
Let $\gamma, \eta$ denote two polygons on $S$. We say
that $\eta$ is \emph{convergent relative to} $\gamma$ if there are no stable
partitions of $S$ compatible with both $\gamma$ and $\eta$:
\begin{equation}\label{convcond}
\chi(\gamma)\cap \chi(\eta) = \emptyset\ .
\end{equation}
In other words, there exists no block of $\gamma$ having the same underlying
set as a block of $\eta$.  If $\eta$ is a polygon on $S$,
then a block of $\eta$ is said to be a {\it consecutive block} if
its underlying set corresponds to a block of the polygon with the
standard cyclic order $\delta$.  The polygon $\eta$ is said to be
{\it convergent} if it has no consecutive blocks at all, i.e., if it
is convergent relative to $\delta$.  Similarly, a polygon $\eta\in
{\cal P}_{S\cup \{d\}}$ is said to be convergent if it has no
chords partitioning $S\cup \{d\}$ into disjoint subsets $S_1\cup S_2$
such that $S_1$ is a consecutive subset of $S=\{1,\ldots,n\}$.
\end{defn}

\vspace{.2cm}
\begin{defn} \label{defnconvword}  We now adapt the definition of convergence
for polygons in ${\cal P}_{S\cup\{d\}}$ to the corresponding words in $V_S$.
A {\it convergent word} in the alphabet $S$ is a word having no
subword which forms a consecutive block.  In other words, if
$w=a_{i_1}a_{i_2}\cdots a_{i_r}$, then $w$ is convergent if it
has no subword $a_{i_j}a_{i_{j+1}}\cdots a_{i_k}$ such that
the underlying set $\{a_{i_j},a_{i_{j+1}},\ldots,a_{i_k}\}=
\{i,i+1,\ldots,i+r\}\subset \{1,\ldots,n\}$.  A convergent word
is in fact the image in $V_S$ of a convergent polygon in
${\cal P}_{S\cup \{d\}}$ under the isomorphism (\ref{basiciso}).
\end{defn}

\begin{example}
When $1\leq n\leq 4$ there are no convergent polygons in ${\cal P}_S$. For
$n=5$, there is only one convergent polygon up to sign, given by
$\gamma=(13524)$.  The other convergent cyclic structure $(14253)$ is just
the cyclic structure $(13524)$ written backwards.
When $n=6$, there are three convergent polygons up to sign:
$$(135264) \ , \quad (152463) \ , \quad (142635) \ .$$
There are 23 convergent polygons for $n=7$. Note that when $n=8$, the
dihedral structure $\eta=(24136857)$ is not convergent even though
no neighbouring numbers are adjacent, because
$\{1,2,3,4\}$ forms a consecutive block for both $\eta$ and $\delta$.
\end{example}

\begin{rem}
The enumeration of permutations satisfying the single condition
that no two adjacent elements in $\gamma$ should be consecutive (the
case $k=2$) is known as the dinner table problem and is a classic
problem in enumerative combinatorics. The more general problem of
convergent words (arbitrary $k$)
seems not to have been studied previously. The
problems coincide for $n\leq 7$, but the counterexample for $n=8$
above shows that the problems are not equivalent for $n\geq 8$.
\end{rem}

\vskip .5cm
\subsection{Residues of polygons along chords}

\vspace{.3cm}
For every stable partition $D$ of $S$ given by $S=S_1\cup S_2$,
we define a residue map on polygons
$$\Res^p_D:{\cal P}_S \To {\cal P}_{S_1\cup\{d\}} \otimes_\Q {\cal
  P}_{S_2\cup\{d\}}$$
as follows.  Let $\eta$ be a polygon in ${\cal P}_S$.  If the partition
$D$ corresponds to a chord of $\eta$, then it cuts $\eta$ into two subpolygons
$\eta_i$ ($i=1,2$) whose edges are indexed by the set $S_i$ and an edge
labelled $d$ corresponding to the chord $D$.  We set\index{$\Res^p_D(\eta)$}
\begin{equation}\label{RespmapDEFch3}
\Res^p_D(\eta)=
\begin{cases}
\eta_1\otimes \eta_2&\hbox{if $D$ is a chord of
  $\eta$}\\
0&\hbox{if $D$ is not a chord of $\eta$}.
\end{cases}
\end{equation}

More generally, we can define the residue for several disjoint chords
simultaneously.
Let $S=S_1\cup \cdots \cup S_{r+1}$ be a partition of
$S$ into $r+1$ disjoint subsets with $r\ge 2$.  For $1\le i\le r$, let
$D_i$ be the partition of $S$ into the two subsets
$(S_1\cup \cdots S_i)\cup (S_{i+1}\cup
\cdots \cup S_{r+1})$.  For any polygon
$\eta\in {\cal P}_S$, we say that $\eta$ admits the chords $D_1,\ldots,D_r$
if there exist $r$ chords of $\eta$, disjoint except possibly for endpoints,
partitioning the edges of $\eta$
into the sets $S_1,\ldots,S_{r+1}$.  If $\eta$ admits the chords
$D_1,\ldots,D_r$, then these chords cut
$\eta$ into $r+1$ subpolygons $\eta_1,\ldots,\eta_{r+1}$.  Let
$T_i$ denote the set indexing the edges of $\eta_i$, so that each
$T_i$ is a union of $S_i$ and elements of the set
$\{d_1,\ldots,d_r\}$ of indices of the chords.  The {\it composed residue map}
$$\Res^p_{D_1,\ldots,D_r}:{\cal P}_S\rightarrow
{\cal P}_{T_1}\otimes \cdots \otimes {\cal P}_{T_r}$$
is defined as follows:
\begin{equation}
\Res^p_{D_1,\ldots,D_r}(\eta)=
\begin{cases}
\eta_1\otimes \cdots \otimes\eta_{r+1}&\mbox{if }\eta\mbox{ admits
}D_1,\ldots,D_r\mbox{ as chords}\\
0&\mbox{if }\eta\mbox{ does not admit }D_1,\ldots,D_r
\end{cases}
\end{equation}

\begin{ex}\label{firstchorddef}
In this example,  $n=12$ and the partition of $S$ given by $D_1$, $D_2$, $D_3$
and $D_4$ is $S_1=\{1,2,3\}$,
$S_2=\{4,10,11,12\}$, $S_3=\{5,9\}$, $S_4=\{6\}$, $S_5=\{7,8\}$.
\vskip .3cm
\epsfxsize=12cm
\centerline{\epsfbox{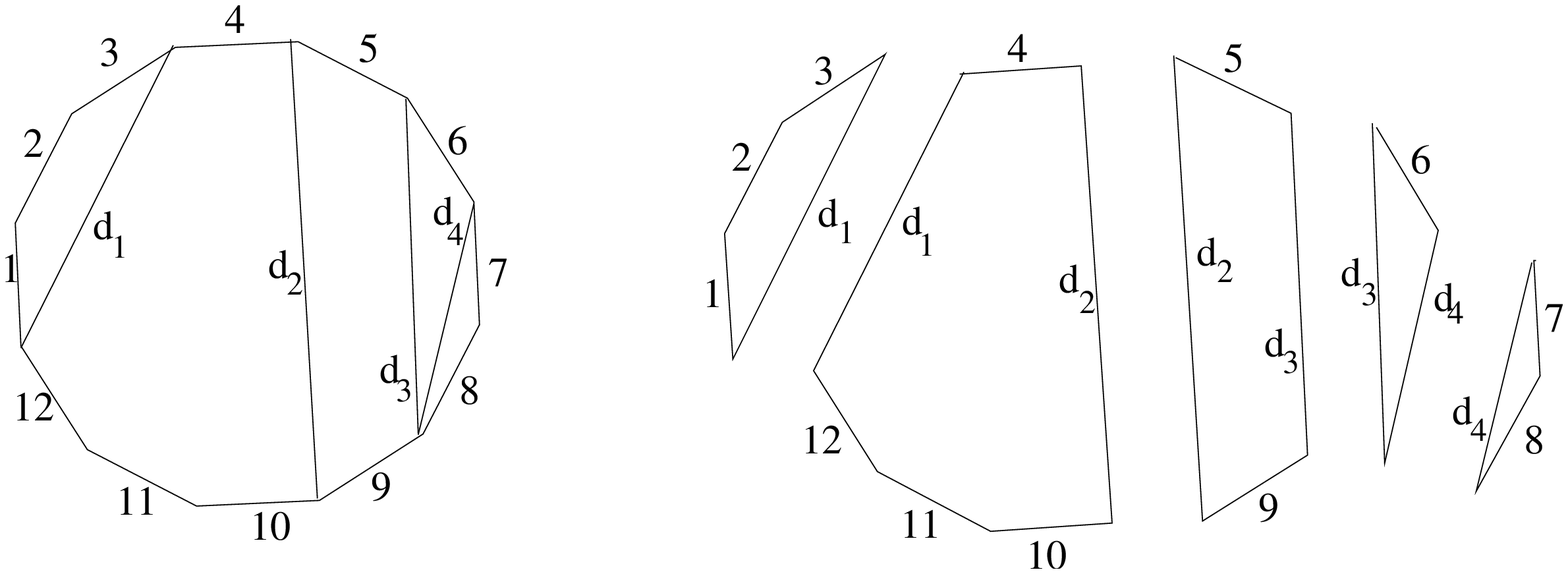}}
\vskip .3cm
We have $T_1=S_1\cup\{d_1\}$, $T_2=S_2\cup \{d_1,d_2\}$,
$T_3=S_3\cup \{d_2,d_3\}$, $T_4=S_4\cup \{d_3,d_4\}$, $T_5=S_5\cup \{d_4\}$.
The composed residue map $\Res^p_{D_1,D_2,D_3,D_4}$ maps the standard
polygon $\delta=(1,2,3,4,5,6,7,8,9,10,11,12)$ to the tensor product of the
five subpolygons shown in the figure.
\end{ex}

The definition of the residue allows us to extend the definition of
convergence of a polygon to linear combinations of polygons.
\vspace{.2cm}
\begin{defn} Let $E$ be a partition of $S\cup \{d\}$ into two subsets,
one of which, $T$, is a consecutive subset of $S$.   Let $\eta
=\sum_i a_i\eta_i$ be a linear combination of polygons. We say that $E$
is a bad chord for $\eta$ if it is a bad chord for any of the $\eta_i$.
The linear combination $\eta$ {\it converges} along $E$ (or along $T$) if the
residue
\begin{equation}\label{conv}
\Res^p_E(\eta)\in I_T\otimes {\cal P}_{S\setminus T\cup \{d\}\cup\{e\}},
\end{equation}
where we recall that $I_T\subset {\cal P}_{T\cup \{e\}}$ is the subspace
spanned by shuffles $(A\sha B,e)$ where $A$ and $B$ are disjoint words the
union of whose letters is equal to $T$.
A linear combination $\eta$ is {\it convergent} if it converges along
all of its bad chords.
\end{defn}

The goal of the following section is to define a set of particular
linear combinations of polygons, the Lyndon insertion words and shuffles,
which are convergent, and show that they are linearly independent.  In
the section after that, we will prove that this set forms a basis for the
convergent subspace of ${\cal P}_{S\cup \{d\}}$.
\vskip .5cm
\subsection{The Lyndon insertion subspace}\label{Lyndins}

\vspace{.3cm} Let a $1n$-word be a word of length $n$ in the distinct
letters of $S=\{1,\ldots,n\}$ in which the letter $1$ appears just to the
left of the letter $n$, and let $W_S\subset V_S\simeq {\cal P}_{S\cup\{d\}}$\index{$W_S$}
denote the subspace generated by these words.   The space $W_S$ is of dimension
$(n-1)!$.  The following lemma will show that $V_S=W_S\oplus I_S$, where
$I_S$ is the subspace of shuffles as before.

\begin{lem}\label{usefullemma}
Fix two elements $a_1$ and $a_2$ of $S=\{1,\ldots,n\}$.  Let
\begin{equation*}
\tau=\sum_i c_i\eta_i,
\end{equation*}
where the $\eta_i$ run over the words of length $n$ in $V_S$ such that
$a_1$ appears just to the left of $a_2$. Then $\tau\in I_S$ if and only if
$c_i=0$ for all $i$.
\end{lem}

\begin{proof} The assumption $\tau\in I_S$ means that we can write $\tau=
\sum_i c_i u_i\sha v_i$. Considering  this in the space ${\cal P}_{S\cup\{d\}}$
isomorphic to $V_S$, it is a sum of cyclic structures $\sum_i c_i(u_i,d)
\sha (v_i,d)$ shuffled with respect to the point $d$.  Choose any
bijection $$\rho:\{1,\ldots,n,d\}\rightarrow \{0,1,\infty,t_1,\ldots,t_{n-2}\}$$
that maps $a_1$ to $0$ and $a_2$ to $1$.  Define a linear map from
${\cal P}_{S\cup \{d\}}$ to $H^{n-2}(\Mod_{0,n+1})$ by first renumbering
the indices $(1,\ldots,n,d)$ of each polygon $\eta\in {\cal P}_{S\cup \{d\}}$
as $(0,1,\infty,t_1,\ldots,t_{n-2})$ via $\rho$, then mapping the
renumbered polygon to the corresponding cell-form (same cyclic order).
By hypothesis,
$\tau=\sum_i c_i \eta_i$ maps to a sum $\omega_\tau=\sum_i c_i\omega_{\eta_i}$
of $01$ cell forms.  Since $\tau$ is a shuffle with respect to one point,
we know by (\ref{cellfunction1shuffsare0}) that $\omega_\tau=0$.
But the $01$ cell-forms $\omega_{\eta_i}$ are linearly
independent by theorem \ref{thm01cellsspan}.  Therefore each $c_i=0$.
\end{proof}

Recall that the shuffles of disjoint Lyndon words form a basis for $I_S$;
we call them {\it Lyndon shuffles}.  A {\it convergent Lyndon shuffle} is a
shuffle of convergent Lyndon words.

\begin{defn}\label{lyndoninsertionshufflesdef}
We will recursively define the set $\cal{L}_S$\index{$\cal{L}_S$} of {\it Lyndon insertion
shuffles}in $I_S$.  If $S=\{1\}$, then ${\cal L}_S=0$.
If $S=\{1,2\}$ then ${\cal L}_S=
\{1\sha  2\}$. In general, if $D$ is any (lexicographically ordered)
alphabet on $m$ letters and $S=\{1,\ldots,m\}$, we define
${\cal L}_D$ to be the image of ${\cal L}_S$
under the bijection $S\rightarrow D$ corresponding to the ordering of $D$.

Assume now that $S=\{1,\ldots,n\}$ with $n>2$, and that we have
constructed all of the sets
${\cal L}_{\{1,\ldots, i\}}$ with $i<n$.  Let us construct
${\cal L}_S$.  The elements of these set are constructed
by taking convergent Lyndon shuffles on a smaller alphabet, and
making insertions into every letter except for the
leftmost letter of each Lyndon word in the shuffle,
according to the following explicit
procedure.  Let $T=\{a_1,\ldots,a_k\}$ be an alphabet with $3\le k\le n$,
with the lexicographical ordering $a_1<\cdots <a_k$,
and choose a convergent Lyndon shuffle $\gamma$ of length $k$ in the
letters of $T$.  Write $\gamma$ as a shuffle of $s>1$ convergent Lyndon words:
$$\gamma=(a_{i_1}\cdots a_{i_{k_1}})\sha  (a_{i_{k_1+1}}\cdots a_{i_{k_2}})
\sha \cdots\sha  (a_{i_{k_{s-1}+1}}\cdots a_{i_{k_s}})$$
where $k_1+\cdots+k_s=k$.  Choose integers
$v_1,\ldots,v_k\ge 1$ such that $\sum_i v_i=n$ and such that for each of the
indices $l=1,i_{k_1+1},\ldots,i_{k_{s-1}+1}$ of the leftmost characters of the
$s$ convergent Lyndon words in $\gamma$, we have $v_l=1$.
For $1\le i\le k$, let $D_i$ denote an alphabet $\{b^i_1,\ldots,b^i_{v_i}\}$.
When $v_i=1$, insert $b^i_1$ into the place of the letter $a_i$ in
$\gamma$; when $v_i>1$, choose any element $V_i$ from ${\cal L}_{D_i}$,
and insert this $V_i$ into the place of the letter $a_i$.

The result is a sum of words in the alphabet
$\cup D_i$.  Note that this alphabet is of cardinal $n$ and equipped
with a natural lexicographical ordering given by the ordering
$D_1,\ldots,D_k$ and the orderings within each alphabet $D_i$.  We
can therefore renumber this alphabet as $1,\ldots,n$.
Since it is a sum of shuffles, the renumbered
element lies in $I_S$, and we call it a {\it Lyndon insertion shuffle}\index{Lyndon insertion shuffle} on $S$.
The original convergent Lyndon shuffle $\gamma$ on $T$ is called the
framing\label{framingDEF}\index{Framing}; 
together with the integers $v_i$, we call this the fixed structure
of the insertion shuffle.
We define ${\cal L}_S$ to be the set of all Lyndon insertion shuffles on $S$.
Note in particular that when $k=n$, so $v_i=1$ for $1\le i\le k$, there are no
non-trivial insertions, and the corresponding elements of ${\cal L}_S$ are
the convergent Lyndon shuffles.
\end{defn}

\begin{example}
We have
$${\cal L}_{\{1,2\}}=\{1\sha  2\}$$
$${\cal L}_{\{1,2,3\}}=\{1\sha  2\sha  3, \ 2\sha  13\}$$
$${\cal L}_{\{1,2,3,4\}}=\{1\sha  2\sha  3\sha  4,\ 13\sha  2\sha  4,\
14\sha  2\sha  3,\ 24\sha  1\sha  3, $$
$$\qquad \qquad \qquad 3\sha  142,\ 13\sha  24,\ 1(3\sha  4)\sha  2\}$$
The last element of ${\cal L}_{\{1,2,3,4\}}$ is obtained by taking
$T=\{1,2,3\}$ and $\gamma=13\sha  2$.  We can only insert in the
place of the character 3 since 1 and 2 are leftmost letters of the
Lyndon words in $13\sha  2$.  As for what can be inserted in the
place of 3, the only possible choices are $k=1$, $v_1=2$,
$D_1=\{b_1,b_2\}$, and $V_1=b_1\sha  b_2$, the unique element of
${\cal L}_{D_1}$.  The natural ordering on the alphabet
$\{T\setminus 3\}\cup D_1$ is given by $(1,2,b_1,b_2)$ since
$b_1\sha  b_2$ is inserted in the place of 3, so we renumber $b_1$
as 3 and $b_2$ as 4, obtaining the new element $1(3\sha  4)\sha  2$.

For $n=5$, ${\cal L}_{\{1,2,3,4,5\}}$ has 34 elements.  Of these, 25
are convergent Lyndon shuffles which we do not list.  The remaining
nine elements are obtained by insertions into the smaller convergent
Lyndon shuffles: they are given by
$$\begin{cases}
2\sha  1(4\sha  35),\ 2\sha  1(3\sha  4\sha  5)&\mbox{insertions into
}2\sha  13\\ 
3\sha  1(4\sha  5)2,\ 4\sha  15(2\sha  3)&\mbox{insertions into }3\sha  142\\
13\sha  2(4\sha  5),\ 1(3\sha  4)\sha  25&\mbox{insertions into }13\sha  24\\
1(3\sha  4)\sha  2\sha  5&\mbox{insertion into }13\sha  2\sha  4\\
1(4\sha  5)\sha  2\sha  3&\mbox{insertion into }14\sha  2\sha  3\\
2(4\sha  5)\sha  1\sha  3&\mbox{insertion into }24\sha  1\sha  3.
\end{cases}$$

\end{example}

\begin{defn}\label{lyndoninsertionwordsdef}
We now define a complementary set, the set ${\cal W}_S$\index{${\mathcal{W}}_S$} of
{\it Lyndon insertion words}\index{Lyndon insertion words}.  Let a {\it special convergent word} $w\in V_S$\index{Special convergent words}\label{specconvw}
denote a convergent word of length $n$ in $S$ such that in the
lexicographical ordering $(1,\ldots,n,d)$, the polygon (cyclic structure)
$\eta=(w,d)$ satisfies $\chi(\delta)\cap \chi(\eta)=\emptyset$; in
other words, the polygon $\eta$ has no chords in common with the standard
polygon.  This condition is a little stronger than asking $w$
to be a convergent word (for instance, $13524$ is a convergent word but
not a special convergent word, since $13524d$ has a bad chord $\{2,3,4,5\}$).
The first elements of ${\cal W}_S$ are given by the special convergent
$1n$-words.  The remaining elements of ${\cal W}_S$ are the
{\it Lyndon insertion words} constructed as follows.  Take a special convergent
word $w'$ in a smaller alphabet $T=\{a_1,\ldots,a_k\}$ with $k<n$ such that
$a_1$ appears just to the left of $a_{k-1}$, and choose positive integers
$v_1,\ldots,v_k$ such that $v_1=v_k=1$ and $\sum_i v_i=n$.  As above, we let
$D_i=\{b^i_1,\ldots,b^i_{v_i}\}$ for $1\le i\le k$, and choose
an element $D_i$ of ${\cal L}_{D_i}$ for each $i$ such that $v_i>1$.  For $i$
such that $v_i=1$, insert $b^i_1$ in the place of $a_i$ in $w'$, and for $i$
such that $v_i>1$ insert $D_i$ in the place of $a_i$.  We obtain a sum of words
$w''$ in the letters $\cup D_i$.  This alphabet has a natural lexicographic
ordering $D_1,\ldots,D_k$ as above, so we can renumber its letters from $1$ to
$n$, which transforms $w''$ into a sum of words $w\in V_S$ called a
{\it Lyndon insertion word}.  Note that by construction, the result is still
a sum of $1n$-words.  The set ${\cal W}_S$ consists of the special convergent
words and the Lyndon insertion words.
\end{defn}

\begin{rem}It follows from lemma \ref{usefullemma} that the intersection
of the subspace $\langle {\cal W}_S\rangle$ in $V_S$ with the subspace
$I_S$ of shuffles is equal to zero.
\end{rem}

\begin{example} We have
$${\cal W}_{\{1,2\}}=\emptyset, \ \
{\cal W}_{\{1,2,3\}}=\emptyset,\ \
{\cal W}_{\{1,2,3,4\}}=\{3142\},$$
$${\cal W}_{\{1,2,3,4,5\}}=\{24153,31524,(3\sha 4)152,(415(2\sha 3)\}$$
The last two elements of ${\cal W}_{\{1,2,3,4,5\}}$ are obtained by
taking $v_1=1,v_2=1,v_3=2,v_4=1$ and $v_1=1,v_2=2,v_3=1,v_4=1$ and
creating the corresponding Lyndon insertion word with respect to $3142$.
\end{example}

\vspace{.2cm}
\begin{thm}\label{linind}
The set ${\cal W}_S\cup {\cal L}_S$ of Lyndon insertion words and shuffles
is linearly independent.
\end{thm}
\begin{proof}
We will prove the result by induction on $n$.  Since ${\cal L}_S\subset
I_S$ and we saw by lemma \ref{usefullemma} that the space generated by
${\cal W}_S$ has zero intersection with $I_S$, we only have to show that
that both ${\cal W}_S$ and ${\cal L}_S$ are linearly independent sets.
We begin with ${\cal L}_S$.  Since
${\cal L}_{\{1,2\}}$ contains a single element, we may assume that $n>2$.

Let $W=A_1\sha \cdots\sha  A_r$ be a Lyndon shuffle, with $r>1$.
We define its {\it fixed
structure} as follows.  Replace every maximal consecutive block (not
contained in any larger consecutive block) in each $A_i$ by a single
letter.  Then $W$ becomes becomes a convergent Lyndon shuffle $W'$ in a
smaller alphabet $T'$ on $k$ letters, which is equipped with an inherited
lexicographical ordering.  If $T=\{1,\ldots,k\}$, then under the
order-respecting bijection $T'\rightarrow T$, $W'$ is mapped to a convergent
Lyndon shuffle $V$ in $T$, called the framing of $W$.
The fixed structure is given by the framing together with
the set of integers $\{v_i\mid 1\le i\le k\}$ defined by $v_i=1$ if that letter
in $T$ does not correspond to a maximal block, and $v_i$ is the length of the
maximal block if it does.  Thus we have $v_1+\cdots+v_k=n$.  We can extend
this definition to the fixed structure of a Lyndon insertion
shuffle, since by definition this is a linear combination of Lyndon
shuffles all having the same fixed structure, and we recover the
framing and fixed structure of the insertion shuffle given in the definition.

\begin{example} If $W$ is the Lyndon shuffle $1546\sha 237$, we
replace the consecutive blocks $23$ and $546$ by letters $b_1$ and $b_2$,
obtaining the convergent shuffle $W'=1b_2\sha b_17$ in the alphabet
$T'=\{1,b_1,b_2,7\}$; renumbering this as $1,2,3,4$ we obtain $V=13\sha 24
\in {\cal L}_{\{1,2,3,4\}}$.  The fixed structure is given by $13\sha 24$ and
integers $v_1=1,v_2=2,v_3=3, v_4=1$.

The Lyndon insertion shuffles $(1, (3 \sha  4 )) \sha  (2, 5 )$ and $(1,3) \sha
(2, (4 \sha  5))$ have the same framing
$13\sha  24$, but since
$(v_1,v_2,v_3,v_4)=(1,1,2,1)$ for the first one and $(1,1,1,2)$ for the
second, they do not have the same fixed structure.  The Lyndon insertion
shuffles $(1, (5) \sha  (3, 4, 6)) \sha  (2, 7) $ and $
(1, (3,5) \sha  (4,6)) \sha  (2,7)$ have the same associated framing
$13\sha  24$ and the same integers $(v_1,v_2,v_3,v_4)=(1,1,4,1)$.
so they have the same fixed structure.
\end{example}
\vskip .2cm
For any fixed structure, given by a convergent Lyndon shuffle $\gamma$
on an alphabet $T$ of length $k$ and associated integers $v_1,\ldots,v_k$
with $v_1+\cdots+v_k=n$, let
$L(\gamma, v_1, ..., v_k)$\label{lyndshwfxst}\index{$L(\gamma, v_1, ..., v_k)$} be the subspace of
$V_S$ spanned by Lyndon shuffles with that fixed structure.  Since
Lyndon shuffles are linearly independent, we have
\begin{align*}
V_{S}=\bigoplus L(\gamma, v_1, ..., v_k)
\end{align*}
Now, as we saw above, a Lyndon insertion shuffle is a linear combination
of Lyndon shuffles all having the same fixed structure, so every
element of ${\cal W}_S\cup{\cal L}_S$ lies in exactly one subspace
$L(\gamma,v_1,\ldots,v_k)$.
Thus, to prove that the elements of ${\cal L}_S$ are
linearly independent, it is only necessary to prove the linear
independence of Lyndon insertion shuffles with the same fixed structure.
If all of the $v_i=1$, then the fixed structure is just a single
convergent Lyndon shuffle on $S$, and these are linearly independent.
So let $(\gamma,v_1,\ldots,v_k)$ be a fixed structure with not all of the
$v_i$ equal to $1$, and let $\omega=\sum_q c_q \omega_q$ be a linear
combination of Lyndon insertion shuffles of fixed structure
$\gamma,v_1,\ldots,v_n$.

Break up the tuple $(1,\ldots,n)$ into $k$ successive tuples
$$B_1=(1,\ldots,v_1),\ B_2=(v_1+1,\ldots,v_1+v_2),\ldots,
B_k=(v_1+\cdots+v_{k-1}+1,\ldots, n).$$

Let $i_1,\ldots,i_m$ be the indices such that $B_{i_1},\ldots,B_{i_m}$ are
the tuples of length greater than $1$.  These tuples correspond to the
insertions in the Lyndon insertion shuffles of type $(\gamma,v_1,\ldots,
v_k)$.  For $1\le j\le m$, let $T_j=\{B_{i_j}\}\cup \{d_j\}$.  This element
$d_j$ is the index of the chord $D_j$ corresponding to the consecutive subset
$B_{i_j}$, which is a chord of the standard polygon and also of every term of
$\omega$.  The chords $D_1,\ldots,D_r$ are disjoint and cut each term of
$\omega$ into $m+1$ subpolygons, $m$ of which are indexed by $T_j$, and the
last one of which is indexed by $T'=
S\setminus \{B_{i_1}\cup \cdots\cup B_{i_m}\}\cup \{d_1,\ldots,
d_m\}$.  Thus we can take the composed residue map\label{cmpresmapDEF}\index{$\Res^p_{D_1,\ldots,D_m}(\omega)$}
$$\Res^p_{D_1,\ldots,D_m}(\omega)\in {\cal P}_{T_1}\otimes
\cdots \otimes {\cal P}_{T_m}\otimes {\cal P}_{T'}.$$
Let us compute this residue.

The alphabet $T'$ is of length $k$ and has a natural ordering corresponding
to a bijection $\{1,\ldots,k\}\rightarrow T'$.  Let $\gamma'$ be the
image of $\gamma$ under this bijection, i.e. the framing.
Let $P^q_1,\ldots,P^q_m$ be the insertions corresponding to the $m$ tuples
$B_{i_1},\ldots,B_{i_m}$ in each term of $\omega=\sum_qc_q\omega_q$. Each
$P^q_j$ lies in ${\cal L}_{B_{i_j}}$.
The image of the composed residue map is then
\begin{equation}\label{compres}
\Res^p_{D_1,\ldots,D_m}(\omega)=\sum_q c_q (P^q_1,d_1)\otimes \cdots\otimes
(P^q_m,d_m)\otimes \gamma'.
\end{equation}
Now assume that $\omega=\sum_q c_q\omega_q=0$. Then
$$\sum_q c_q (P^q_1,d_1)\otimes \cdots\otimes (P^q_m,d_m)\otimes \gamma'=0,$$
and since $\gamma'$ is fixed, we have
$$\sum_q c_q (P^q_1,d_1)\otimes \cdots\otimes (P^q_m,d_m)=0.$$
But for $1\le j\le m$, the $P^q_j$ lie in ${\cal L}_{B_{i_j}}$ and
thus, by the induction hypothesis, the distinct $P^q_j$ for fixed $j$
and varying $q$ are linearly independent.  Since $d_i$ is the largest
element in the lexicographic alphabet $T_i$, the sums $(P^q_j,d_j)$ are
also linearly independent for fixed $j$ and varying $q$, because if
$\sum_q d_q(P^q_j,d_j)=0$ then $\sum_q d_q P^q_j=0$ simply by erasing $d_j$.
The tensor products are therefore also linearly independent, so we must have
$c_q=0$ for all $q$.  This proves that ${\cal L}_S$ is a linearly independent
set.

We now prove that ${\cal W}_S$ is a linearly independent set.  For this,
we construct the framing and fixed structure of a of length $n$ in $V_S$
just as above, by replacing consecutive blocks with single letters, obtaining a word in a smaller alphabet $T'$ and a set of integers corresponding to the
lengths of the consecutive blocks.  For instance, replacing the consecutive
blocks 12 and 354 in the word 12735486 by letters $b_1$ and $b_2$ gives
a convergent word $b_17b_286$ in the alphabet $(b_1,b_2,6,7,8)$; renumbering
this as $(1,2,3,4,5)$ gives the framing as $14253$ and the associated integers
as $v_1=2,v_2=3,v_3=1,v_4=1,v_5=1$.  For every fixed structure of this type,
now given as a convergent word $\gamma$ of length $k<n$ together with integers
$v_1,\ldots,v_k$, we let $W(\gamma,v_1,\ldots,v_k)$ denote the subspace
of $V_S$ generated by words with the fixed structure $(\gamma,v_1,\ldots,
v_k)$.  Since the words of length $n$ form a basis for $V_S$, we again
have $V_S=\oplus W(\gamma,v_1,\ldots,v_k)$.   Therefore, to show that
${\cal W}_S$ is a linearly independent set, we only need to show that the set
of Lyndon insertion words with a given fixed structure is a linearly
independent set.  So assume that we have some linear combination
$\sum_q c_qw_q=0$, where the $w_q$ are all Lyndon insertion words of
given fixed structure $(\gamma,v_1,\ldots,v_k)$.  If $k=n$, then these
insertion words are just words, so they are linearly independent and
$c_q=0$ for all $q$.  So assume that at least one $v_i>1$.  We proceed
very much as above.  Breaking up the tuple $(1,\ldots,n)$ into tuples
$B_1,\ldots,B_k$ as above, and letting $D_1,\ldots,D_m$, $T_j$ and $T'$
denote the same things, we compute the composed residue of $\sum_q
c_qw_q$ and obtain (\ref{compres}).  Then because all of the insertions
$P^q_i$ lie in ${\cal L}_{B_{i_j}}$ and we know that these sets are
linearly independent, we find as above that $c_q=0$ for all $q$.
\end{proof}

\vspace{.3cm}
\subsection{Convergent linear combinations of polygons}

\begin{defn}\label{JSDEF} Let $S=\{1,\ldots,n\}$.\index{$J_S$}
Let $J_S$ be the subspace of ${\cal P}_{S\cup \{d\}}$ spanned
by ${\cal L}_S$ and let $K_S$\index{$K_S$} be the subspace of ${\cal P}_{S\cup \{d\}}$
spanned by  ${\cal W}_S$.
\end{defn}
\vskip .3cm
We prove the main convergence results in two separate theorems, concerning
the subspaces $I_S$ and $W_S$ of $V_S\simeq {\cal P}_{S\cup \{d\}}$
respectively.

\vspace{.2cm}
\begin{thm}\label{convJS}
If $\omega\in I_S\subset {\cal P}_{S\cup \{d\}}$ is convergent, then
$\omega\in J_S$.
\end{thm}
\begin{proof}
One direction of this theorem is easy.  We only need to show that any Lyndon
insertion shuffle is convergent.
If it is a shuffle of convergent Lyndon words,
then there are no consecutive blocks in any of the words.  Therefore
if the letters of any consecutive subset $T$ of $S$ appear as a block
in any term of $\omega$, it must be because they appeared in more than
one of the convergent words which are shuffled together.  So
these letters appear as a shuffle, so the residue lies in
$I_T\otimes {\cal P}_{S\setminus T\cup\{d\}}$.
Now, if we are dealing with a Lyndon
insertion shuffle with non-trivial insertions, then there are
two kinds of bad chords: those corresponding to these insertions, and
those corresponding to consecutive subsets of the insertion sets.  By
definition, the insertions themselves lie in ${\cal L}_T\subset I_T$, and
their expressions are equal to the $I_T$ factors of the residue, so $\omega$
converges along all the bad chords corresponding to insertions.  For the
subchords of these, their letters appear shuffled inside the insertions, so
the previous argument holds.

Write $\omega=\sum_i c_i\omega_i$ where each
$\omega_i=(A^i_1\sha \cdots \sha A^i_{r_i},d)$ is a Lyndon shuffle, $r_i>1$.
Assume that $\omega$ converges along all of
its bad chords.  As above, a consecutive block appearing in any $A^i_j$ is
maximal if the same block does not appear in any other factor
inside a bigger consecutive block.  Factors may appear which contain more than
one consecutive block, but the maximal blocks are disjoint.
\vskip .3cm
We prove the result by induction on the length of the alphabet
$S=\{1,\ldots,n\}$.  The smallest case is $n=3$, since for $n=2$,
the polygons are triangles and have no chords.  For $n=3$, let
$$\omega=a_1(12\sha 3,d)+a_2(13\sha 2,d)+
a_3(1\sha 2\sha 3,d)+a_4(23\sha 1,d).$$
The only non-trivial bad chords are $D=\{1,2\}$, $E=\{2,3\}$. We have
$$\Res^p_D(\omega)=a_1(1,2,e)\otimes (e\sha 3,d)
+a_2(1\sha 2,e)\otimes (e,3,d)$$
$$+a_3(1\sha 2,e)\otimes (e\sha 3,d)+a_4(1\sha 2,e)\otimes (e,3,d).$$
For this to converge means that the left-hand parts of the two right-hand
tensor factors $(e,3,d)$ and $(e\sha 3,d)$ must lie in $I_{\{1,2\}}$.
This implies that $a_1=0$.  For the other residue, we have
$$\Res^p_E(\omega)=a_1(2\sha 3,e)\otimes (1,e,d)
+a_2(2\sha 3,e)\otimes (1,e,d)$$
$$+a_3(2\sha 3,e)\otimes (1\sha e,d)+a_4(2,3,e)\otimes (1\sha e,d).$$
This gives $a_4=0$.
Therefore convergent $\omega$ is a linear combination of $13\sha 2$
and $1\sha 2\sha 3$, which are the basis elements of ${\cal L}_{\{1,2,3\}}$.

\vskip .3cm
The induction hypothesis is that for every alphabet
$S'=\{1,\ldots,i\}$ with $i<n$, if $\omega\in V_{S'}$
is convergent, then $\omega\in K_{S'}$.

Now let $S=\{1,\ldots,n\}$ and assume that $\omega\in V_S$ is convergent.
If no consecutive block appears in any $A^i_j$, then $\omega$ is a linear
combination of convergent Lyndon words, so it is in $J_S$.  Assume some
consecutive blocks do appear, and consider a maximal consecutive block $T$,
which corresponds to a bad chord $E$.  Exactly as in the proof of the
lemma, we decompose $\omega=\gamma_1+\gamma_2$ where $\gamma_k$ is the
sum $\sum_{i\in I_k} c_i\omega_i$, with $I_1$ the set of indices $i$
for which $T$ appears as a block in some $A^i_j$, which by reordering we may
assume to be $A^i_1$, and $I_2$ is the set of indices for which $T$ does not
appear as a block in any $A^i_j$.  As in the lemma, we see immediately that
$\Res^p_E(\gamma_2)\in I_T\otimes {\cal P}_{S\setminus T\cup\{e\}\cup \{d\}}$,
so $\gamma_2$ converges along $E$. Since we are assuming that $\omega$
is convergent, also $\gamma_1$ must converge, so we have
$$\Res^p_E(\gamma_1)\in I_T\otimes {\cal P}_{S\setminus T\cup\{d\}\cup \{e\}}.$$
For each $i\in I_1$, write $A^i_1=B^i_1Y^iC^i_1$, where $Y^i$ consists of the
letters of $T$ in some order, and $B^i_1$ is Lyndon and non-empty.
Then
\begin{equation}\label{small}
\Res^p_E(\gamma_1)=\sum_{i\in I_1} c_i(Y^i,e)\otimes (B^ieC^i\sha
A^i_2\sha\cdots\sha A^i_{r_i},d).
\end{equation}
Putting an equivalence relation on $I_1$ as in the proof of the lemma,
so that $i\sim i'$ if the right-hand factors of (\ref{small}) are
equal, and letting $[i]$ denote the equivalence classes for this
relation, we write the residue as
\begin{equation}\label{smallbis}
\Res^p_E(\gamma_1)=\sum_{[i]\subset I_1} \bigl(\sum_{i\in [i]} c_i(Y^i,e)\bigr)
\otimes (B^{[i]}eC^{[i]}\sha A^{[i]}_2\sha\cdots\sha A^{[i]}_{r_{[i]}},d).
\end{equation}
Since the right-hand factors in the sum over $[i]$ are distinct
Lyndon shuffles, they are linearly independent and therefore we find that
$$(S_{[i]},e)=\sum_{i\in [i]\subset I_1} c_i(Y^i,e)\in I_T$$
for each $[i]\subset I_1$.

We now show that $(S_{[i]},e)$ is not merely in $I_T$, but in $J_T$.
To see this, it is enough to show that $(S_{[i]},e)$ converges on every
subchord of $T$ (consecutive subset inside the set $T$), and apply the
induction hypothesis.  So let $E'$ be a subchord of $E$, corresponding to a
consecutive block $T'$ strictly contained in $T$.

We now decompose the set of indices $I_1$ into two subsets $I_3$ and
$I_4$, where $I_3$ contains the indices $i\in I_1$ such that $T'$ appears
as a consecutive block inside the block $T$ appearing in $A^i_1$,
and $I_4$ contains the indices $i\in I_1$ such that the letters of $T'$ do not
appear consecutively inside the block $T$. Similarly, we partition $I_2$,
the set of indices in the sum $\omega=\sum_i c_i\omega_i$ for
which $T$ does not appear as a block in $A^i_1$,
into two sets $I_5$ and $I_6$, where $I_5$ contains the indices $i\in I_2$
such that $T'$ appears as a block in some $A^i_j$ which we may assume
to be $A^i_1$, and $I_6$ contains the indices $i\in I_2$ of the terms in which
$T'$ does not appear as a block in any $A^i_j$.  We have corresponding
decompositions $\gamma_1=\gamma_3+\gamma_4$, $\gamma_2=\gamma_5+\gamma_6$.
As before, $T'$ must appear as a shuffle in $\gamma_6$, so $\gamma_6$
converges along $E'$.  As for $\gamma_4$, since $T'$ does not appear as a
block or a shuffle, the residue along $E'$ is $0$.  Since by assumption
$\omega$ converges along $E'$, we know that $\gamma_3+\gamma_5$ converges
along $E'$.  Let us show that in fact both $\gamma_3$ and $\gamma_5$
converge along $E'$.

Write $A^i_1=R^iZ^iS^i$ for every $i\in I_3\cup I_5$, where $Z^i$ is a word
in the letters of $T'$.  Note that $R^i$ is Lyndon, and non-empty because
$T'$ cannot appear as a block to the left of any $A^i_j$ by the lemma.
Then for $k=3,5$, we have
\begin{equation}\label{small2}
\Res^p_{E'}(\gamma_k)=\sum_{i\in I_k} c_i(Z^i,e')\otimes (R^ie'S^i\sha
A^i_2\sha\cdots\sha A^i_{r_i}).
\end{equation}
For $k=3,5$, put the equivalence relation on $I_k$ for which
$i\sim  i'$ if the right-hand factors of (\ref{smallbis}) are equal, and
let $\langle i\rangle$ denote the equivalence classes for this relation.
Note that because for $i\in I_3$, $T'$ appears as a block of $T$,
we have $B^i\subset R^i$ and $C^i\subset S^i$, in the sense that in
fact $B^i$ is the left-hand part of $R^i$ and $C^i$ is the right-hand
part of $S^i$.  Therefore in particular, the new equivalence relation
is strictly finer than the old, i.e. the equivalence class $[i]$ breaks
up into a finite union of equivalence classes $\langle i\rangle$.
The residue can now be written
\begin{equation}
\Res^p_{E'}(\gamma_k)=\sum_{\langle i\rangle\subset I_k}
\bigl(\sum_{i\in\langle i\rangle} c_i(Z^i,e')
\bigr) \otimes (R^{\langle i\rangle }e'S^{\langle i\rangle}\sha
A^{\langle i\rangle }_2 \sha\cdots\sha
A^{\langle i\rangle}_{r_{\langle i\rangle}}).
\end{equation}
Then since the right-hand factors for each $k$ are distinct Lyndon shuffles,
they are linearly independent, and furthermore, none of these factors
for $\gamma_3$ can ever occur in $\gamma_5$ for the following reason:
the Lyndon words $R^ie'S^i$ appearing for $k=3$ all have the letters of
$T\setminus T'$ grouped around $e'$, whereas none of the Lyndon words
$R^ie'S^i$ have this property.  Therefore all the right-hand
factors from the residues of $\gamma_3$ and $\gamma_5$ are linearly
independent, so we find that all the left-hand factors
\begin{equation}\label{small3}
\sum_{i\in \langle i\rangle \subset I_k} (Z^i,e')\in I_{T'},
\end{equation}
so that both $\gamma_3$ and $\gamma_5$ converge along $E'$.  In
particular, this means that both $\gamma_1$ and $\gamma_2$ converge
along $E'$.

Now, let us compute the composed residue map $\Res^p_{E,E'}(\gamma_1)$.  First,
for each $i\in I_3$, write $Y^i=U^iZ^iV^i$ where $Z^i$ is a word in the
letters of $T'$, so that $R^i=B^iU^i$, $S^i=U^iC^i$, and
$A^i_1=B^iU^iZ^iV^iC^i$.
Then by (\ref{small2}), we have
$$\Res^p_E(\gamma_1)=\sum_{[i]\in I_3}\bigl(\sum_{i\in [i]} c_i(U^iZ^iV^i,e)
\bigr)\otimes \bigl(B^{[i]}_1eC^{[i]}_1\sha A^{[i]}_2 \sha  \cdots \sha
A^{[i]}_{r_{[i]}},d\bigr)+$$
$$\sum_{[i]\in I_4}\bigl(\sum_{i\in [i]} c_i(Y^i,e)\bigr)\otimes
\bigl(B^{[i]}_1eC^{[i]}_1\sha A^{[i]}_2 \sha  \cdots \sha
A^{[i]}_{r_{[i]}},d\bigr).  $$
The terms for $i\in I_4$ converge along $T'$, so they vanish when taking
the composed residue, and we find
$$\Res^p_{E,E'}(\gamma_1)=
\sum_{[i]\in I_3} \bigl(\sum_{i\in [i]} c_i(Z^i,e')
\otimes (U^ie'V^i,e)\bigr)\otimes
\bigl(B^{[i]}_1eC^{[i]}_1\sha A^{[i]}_2 \sha  \cdots \sha
A^{[i]}_{r_{[i]}},d\bigr).$$
Since for each $[i]\subset I_3$, the right-hand factors are as usual distinct
and linearly independent, this means that for each $[i]\subset I_3$,
$$\Res^p_{E'}(S_{[i]},e)=\sum_{i\in [i]}
c_i(Z^i,e')\otimes (U^ie'V^i,e)\in {\cal P}_{T'\cup \{e'\}}\otimes
{\cal P}_{T\setminus T'\cup\{e'\}\cup \{e\}}.$$
Now, breaking $[i]$ up into separate equivalence classes $\langle i\rangle$,
we have that $U^i$ and $V^i$ are identical for all $i$ in one subclass
$\langle i\rangle$ since $B^i$ and $C^i$ are already identical for
all $i\in [i]$.  So for each $[i]\subset I_3$, we can write
$$\Res^p_{E'}(S_{[i]},e)=\sum_{\langle i\rangle\subset [i]}
\sum_{i\in \langle i\rangle}
c_i(Z^i,e')\otimes (U^{\langle i\rangle}e'V^{\langle i\rangle},e),$$
where the right-hand factors are all distinct words.
Then (\ref{small3}) shows that this sum lies in
$I_{T'}\otimes {\cal P}_{T\setminus T'\cup \{e'\}\cup \{e\}}$, so in fact
$(S_{[i]},e)$ converges along $E'$.  For $[i]\subset I_4$, we already
saw that $\Res^p_{E'}(S_{[i]},e)=0$, so $(S_{[i]},e)$ converges along $E'$
for all $[i]\subset I_1$.  Since this holds for all chords $E'$
corresponding to consecutive subblocks $T'$ of $T$, we see that each
$(S_{[i]},e)$ is convergent along all its bad chords,
and thus, by the induction hypothesis, $(S_{[i]},e)\in J_T$.
Now we can write $\omega=\gamma_1+\gamma_2$ with
$$\gamma_1=\sum_{[i]\in I_1}  c_{[i]} B^{[i]}(S_{[i]})C^{[i]}\sha
A^{[i]}_2\sha \cdots \sha A^{[i]}_{r_i}$$
with $S_{[i]}\in J_T$.  This means that the maximal block $T$, which appeared
only in $\gamma_1$, has been replaced by an insertion in the sense of the
definition of Lyndon insertion shuffles.

To conclude the proof of the theorem, we successively replace each of the
maximal blocks in $\omega$ by insertion terms in the same way.  Insertions
are by definition convergent and contain no blocks, so as we
proceed to substitute insertions for the maximal blocks one by one,
blocks which were previously not maximal may become maximal;
however the order in which the blocks are substituted by insertions
is of no importance as long as only maximal blocks are treated at each step.
The final result displays $\omega$ as a linear combination of convergent
Lyndon shuffles and Lyndon insertion shuffles, so $\omega\in J_S$.
\end{proof}

\begin{thm}\label{convKS}
Let $\eta\in W_S\subset {\cal P}_{S\cup \{d\}}$.
Then $\eta$ is convergent if and only if $\eta\in K_S=\langle
{\cal W}_S\rangle$.
\end{thm}
\begin{proof} The proof that $\omega\in K_S$ is convergent is exactly
as at the beginning of the proof of the previous theorem.  So
let $\omega\in W_S$, write
$$\omega=\sum_i a_i\eta_i$$
where each $\eta_i$ is a $1n$-polygon (a $1n$-word concatenated with $d$),
and assume $\omega$ is convergent.
The only possible bad chords for $\omega$ are the consecutive blocks
appearing in the $\eta_i$.  Let $T$ be a subset of $S$ corresponding
to a maximal consecutive block.
\vskip .3cm
\begin{lem}\label{lemleft}No maximal consecutive block
having non-trivial intersection with $\{1,n\}$ can appear in
any of the $1n$-words $\eta_i$ of $\omega$.
\end{lem}
\begin{proof}  If $T$ is a maximal block containing both $1$ and $n$,
then $T=\{1,\ldots,n\}$ which does not correspond to a chord.
Let $T$ be a maximal consecutive block appearing in $\omega$ which
contains $1$ but not $n$, say $T=\{1,\ldots,m\}$.
If $T$ appears as a consecutive block in some $\eta_i$,
we may write $\eta_i= (K^i,Z^i,1,n,H^i,d)$ where $Z^i$ is an ordering of
$\{2,\ldots,m\}$.
Then
$$\Res^p_E(\sum_i a_i\eta_i)=\sum_i a_i(Z^i,1,e)\otimes (K^i,e,n,H^i,d).$$
The assumption that $\omega$ converges along $E$ means that
this residue lies in $I_T\otimes {\cal P}_{S\setminus T\cup\{e,d\}}$.
So for constant words $K,H$ (i.e. constant right-hand tensor factor),
we must have
\begin{equation}
\sum_{i\mid K^i=K,H^i=H} a_i(Z^i,1,e)\in I_T,
\end{equation}
in other words, a sum of words $\sum_i a_i (Z^i1)$ must be a shuffle.
But this is impossible by a Lyndon basis argument. Using a backwards Lyndon
basis in which all Lyndon words are as usual but written right to left,
the words ending in 1 generate the degree 1 part of the algebra and
are linearly independent from the shuffles, which generate the part of
degree $\ge 2$.  So we must have $a_i=0$ for all $i$.

Now let $T=\{m,\ldots,n\}$.  We write $\eta_i=(K^i,1,n,Z^i,H^i,d)$
where $(n,Z^i)$ is an ordering of $T$, and we have
$$\Res^p_E(\sum_i a_i\eta_i)=\sum_i a_i(n,Z^i,e)\otimes (K^i,1,e,H^i,d).$$
Convergence implies that
\begin{equation}
\sum_{i\mid K^i=K,H^i=H} a_i(n,Z^i,e)\in I_T,
\end{equation}
Using a Lyndon basis in which the lexicographical ordering is the backwards
order $n<\cdots<1$, the $nZ^i$ are all Lyndon words, so as above,
they cannot sum to a shuffle.
\end{proof}

Now we can complete the proof of the theorem.  It runs almost exactly as
the proof of the previous theorem.  Let $\omega=\sum_i a_i\eta_i$ be
a sum of $1n$-words which converges and consider a maximal consecutive
block $T\subset \{2,\ldots,n-1\}$.  Let $I_1$ be the set of indices
$i$ such that $\eta_i$ contains the block $T$ and $I_2$ the other indices.
For $i\in I_1$, write $\eta_i=(K^i,Z^i,H^i,d)$ where $Z^i$ is an ordering of
$T$.  Then
$$\Res^p_T(\omega)=\sum_{i\in I_1} a_i(Z^i,e)\otimes (K^i,e,H^i,d).$$
Let $i\sim i'$ be the equivalence relation on $I_1$ given by
$K^i=K^{i'}$ and $H^i=H^{i'}$.  Then
$$\Res^p_T(\omega)=\sum_{[i]\in I_1} \bigl(\sum_{i\in [i]}
a_i(Z^i,e)\bigr)\otimes (K^{[i]},e,H^{[i]},d),$$
so by the convergence assumption, we have
$$S_{[i]}=\sum_{i\in [i]} a_i(Z^i,e)\in I_T$$
for each $[i]\subset I_1$.  Therefore we can write $\omega$ with the
insertion $S_{[i]}$ as
$$\omega=\sum_{[i]\subset I_1} a_i(K^{[i]},S_{[i]},H^{[i]},d)+
\sum_{i\in I_2} a_i\eta_i,$$
and the maximal block $T$ no longer appears in $\omega$.
We prove that $S_{[i]}\in J_T$  exactly as in the proof of the previous
theorem: considering a maximal consecutive block $T'\subset T$ occurring
in a factor of $S_{[i]}$, one shows that $S_{[i]}$ converges along $T'$
if and only if $\omega$ converges along $T'$.  Since $\omega$ does
converge by assumption, $S_{[i]}$ also converges, and since this holds
for all consecutive blocks $T'\subset T$, $S_{[i]}$ converges on all
its subdivisors and therefore $S_{[i]}\in J_S=\langle {\cal L}_S\rangle$.
Finally, one deals with the disjoint maximal blocks appearing in $\omega$ one
at a time until no blocks at all remain.

\end{proof}

\vspace{.3cm}
\section{Explicit generators for ${\cal FC}$ and
  $H^\ell(\Mod_{0,n}^\delta)$}\label{explicitbasissection}

In this chapter, we show that the map from polygons to cell-forms is
surjective, and compute its kernel.  From this and the previous
chapter, we will conclude that the pairs $(\delta,\omega)$, where
$\omega$ runs through the set ${\cal W}_S$ of Lyndon insertion words
for $n\ge 5$ form a generating set for the formal cell-zeta algebra
${\cal FC}$. In the final section, we show that the images of the
elements of ${\cal W}_S$ in the cohomology $H^\ell(\Mod_{0,n})$
yield an explicit basis for the convergent cohomology
$H^\ell(\Mod_{0,n}^\delta)$, and discuss its dimension.

\subsection{From polygons to cell-forms}\label{sec41}
Let $S=\{1,\ldots,n\}$.  The bijection $\rho:S\cup \{d\}\rightarrow
\{0,t_1,\ldots,t_{\ell+1},1,\infty\}$ given by associating the
elements $1,\ldots,n,d$ to $0,t_1,\ldots,t_{\ell+1}, 1,\infty$
respectively, induces a  map $f$ from polygons to cell-forms:
$$\eta=(\sigma(1),\ldots,\sigma(n),d)\buildrel{f}\over\rightarrow
\omega_\eta=[\rho(\sigma(1)),\ldots,\rho(\sigma(n)),\infty].$$ The
map $f$ extends by linearity to a map from ${\cal P}_{S\cup \{d\}}$
to the cohomology group \\$H^{n-2}(\Mod_{0,n+1})$.  The purpose of
this section is to prove that $f$ is a surjection, and to determine
its kernel.

Recall that $I_S\subset {\cal P}_{S\cup \{d\}}$ denotes the subvector space of
${\cal P}_{S\cup \{d\}}$ spanned by the {\it shuffles with respect to the
element $d$}, namely by the linear combinations of polygons
$$(S_1 \sha S_2, d)$$
for all partitions $S_1\coprod S_2$ of $S$.

\begin{prop}\label{prop41} Let $S=\{1,\ldots,n\}$.  Then the cell-form map
$$f:{\cal P}_{S\cup \{d\}} \To H^{n-2}(\Mod_{0,n+1})$$
is surjective with kernel equal to the subspace $I_S$.
\end{prop}

\begin{proof} The surjectivity is an immediate consequence of the
fact that $01$ cell-forms form a basis of $H^{n-2}(\Mod_{0,n+1})$
(theorem \ref{thm01cellsspan}), since all such cell-forms are the
images under $f$ of polygons having the edge labelled $1$ next to
the one labelled $n$.

Now, $I_S$ lies in the kernel of $f$ by the corollary to proposition
\ref{shufprod}. So it only remains to show that the kernel of $f$ is
equal to $I_S$.  But this is a consequence of counting the
dimensions of both sides.  By theorem \ref{thm01cellsspan}, we know
that the dimension of $H^{n-2}(\Mod_{0,n+1})$ is equal to $(n-1)!$.
As for the dimension of ${\cal P}_{S\cup \{d\}}/I_S$, recall from
the beginning of chapter 3 that ${\cal P}_{S\cup \{d\}}\simeq V_S$,
which can be identified with the graded $n$ part of the quotient of
the polynomial algebra on $S$ by the relation $w=0$ for all words
$w$ containing repeated letters.  Thus $V_S$ is the vector space
spanned by words on $n$ distinct letters, so it is of dimension
$n!$. But instead of taking a basis of words, we can take the Lyndon
basis of Lyndon words (words with distinct characters whose smallest
character is on the left) and shuffles of Lyndon words.  The
subspace $I_S$ is exactly generated by the shuffles, so the
dimension of the quotient is given by the number of Lyndon words on
$S$, namely $(n-1)!$. Therefore ${\cal P}_{S\cup \{d\}}/I_S\simeq
H^{n-2}(\Mod_{0,n+1})$.
\end{proof}

\begin{rem}
The above proof has an interesting consequence.  Since the map from
polygons to differential forms does not depend on the role of $d$,
the kernel cannot depend on $d$, and any other element of $S\cup
\{d\}$ could play the same role.  Therefore $I_S$, which is defined
as the space generated by shuffles with respect to the element $d$,
is equal to the space generated by shuffles of elements of $S\cup
\{d\}$ with respect to any element of $S$; it is simply the subspace
generated by {\it shuffles with respect to one element} of $S\cup
\{d\}$.
\end{rem}

\begin{cor}
Let $W_S\subset {\cal P}_{S\cup \{d\}}$ be the subset of polygons
corresponding to $1n$-words (concatenated with $d$).  Then
$$f:W_S\simeq H^{n-2}(\Mod_{0,n+1}).$$
\end{cor}

\begin{proof} The proof follows from the fact that ${\cal P}_{S\cup \{d\}}=
W_S\oplus I_S$.
\end{proof}

\subsection{Generators for ${\cal FC}$}

By definition, ${\cal FC}$ is generated by all linear combinations
of pairs of polygons $\sum_i a_i (\delta,\omega_i)$ whose associated
differential form converges on the standard cell, but modulo the
relation (among others) that shuffles are equal to zero.  In other
words, since ${\cal P}_{S\cup \{d\}}=W_S\oplus I_S$, we can redefine
${\cal FC}$ to be generated by linear combinations $\sum_i
a_i(\delta,\omega_i)$ such that $\sum_i a_i\omega_i\in W_S$ and such
that the associated differential form converges on the standard
cell.

The following proposition states that the notion of the residue of a
polygon and the residue of the corresponding cell-form coincide. In
order to state it, we must recall that one can define the map
$$\rho: \cal{P}_S \To \Omega^\ell(\Mod_{0,S})\ ,$$
from polygons labelled by $S$ to cell forms in a coordinate-free way
(one can do this directly from equation $(\ref{omegalift})$). In
$\S1$, this map was defined in explicit coordinates by fixing any
three marked points at $0,1$ and $\infty$. This essence of lemma
$\ref{lemsymaction}$ is that $\rho$ is independent of the choice of
three marked points, and is thus  coordinate-free.

\begin{prop} \label{propresformula}  Let $S=\{1,\ldots,n\}$ and let
$D$ be a stable partition $S_1\cup S_2$ of $S$ corresponding to a boundary
divisor of $\Mod_{0,n}$, with $|S_1|=r$ and $|S_2|=s$.  Let $\rho$ denote
the usual map from polygons to cell-forms.  Then the following
diagram is commutative:
$$\xymatrix{{\cal P}_S\ar[r]^\rho\ar[d]_{\Res^p_D} &H^\ell(\Mod_{0,n})
\ar[d]^{\Res_D}\\
{\cal P}_{S_1\cup \{d\}}\otimes  {\cal P}_{S_2\cup \{d\}}\ar[r]^{\!\!\!
\!\!\!\!\!\!\!\!\!\!\!\!\!\!\!\!\!\!\rho \otimes\rho}&\
H^{r-2}(\Mod_{0,r+1})\otimes H^{s-2}(\Mod_{0,s+1}).  }$$
In other words, the usual residue of differential forms
corresponds to the combinatorial residue of polygons.
\end{prop}
\begin{proof}
 Let $\eta\in {\cal P}_S$ be a polygon, and let $\omega_\eta$
be the associated cell-form.  If $D$ is not compatible with
$\omega_\eta$, then $\omega_\eta$ has no pole on $D$ by proposition
\ref{CORomegapoles}, so $\Res_D(\omega)=0$.

  We shall
work in explicit coordinates, bearing in mind that this does not
affect the answer, by the remarks above.
%
 Therefore assume that $\eta$ is the polygon
numbered with the standard cyclic order on $\{1,\ldots,n\}$, and
that $D$ is compatible with $\eta$.  The corresponding cell-form is
given in simplicial coordinates by $[0,t_1,\ldots,t_\ell,1,\infty]$.
By applying a cyclic rotation, we can assume that $D$ corresponds to
the partition
$$S_1 = \{1,2,3,\ldots, k+1\} \ \hbox{ and } \ \
S_2=\{k+2,\ldots, n-1, n\}$$ for some $1\leq k\leq \ell$.
 In simplicial coordinates, $D$ corresponds to the  blow-up of the
cycle $0=t_1=\cdots=t_k$. We compute the residue of $\omega_\eta$
along $D$ by applying the variable change $t_1  = x_1\ldots x_\ell,
\ldots,  t_{\ell-1} = x_{\ell-1}x_\ell,  t_\ell = x_\ell$ to the
form $\omega_\eta=[0,t_1,\ldots,t_\ell,1,\infty]$. The standard cell
$X_\eta$ is given by $\{0<x_1,\ldots,x_\ell<1\}$. In these
coordinates, the divisor $D$ is given by $\{x_k=0\}$, and the form
$\omega_\eta$ becomes
\begin{equation} \label{omegacubicalform}
\omega_{\eta} = {dx_1\ldots dx_\ell \over x_1(1-x_1)\ldots
x_\ell(1-x_\ell)}.  \end{equation} The residue of $\omega_\eta$
along $x_k=0$ is given by
\begin{equation}\label{proofres}
{dx_1\ldots dx_{k-1} \over x_1(1-x_1)\ldots x_{k-1}(1-x_{k-1}) }
\otimes {dx_{k+1}\ldots dx_{\ell} \over x_{k+1}(1-x_{k+1})\ldots
x_{\ell}(1-x_{\ell})}\ .\end{equation} Changing back to simplicial
coordinates via
  $x_1=a_1/a_2,\ldots,x_{k-2}=a_{k-2}/a_{k-1}$,
$x_{k-1}=a_{k-1}$,  and $x_\ell=b_\ell$,
$x_{\ell-1}=b_{\ell-1}/b_\ell, \ldots,x_{k+1}=b_k/b_{k+1}$ defines
simplicial coordinates on $D\cong \Mod_{0,r+1}\times \Mod_{0,s+1}$.
The standard cells induced by $\eta$ are
$(0,a_1,\ldots,a_{k-1},1,\infty)$ on $\Mod_{0,r+1}$ and
$(0,b_k,\ldots, b_{\ell},1,\infty)$ on $\Mod_{0,s+1}$. If we compute
$(\ref{proofres})$ in these new coordinates, it gives precisely
$$ [0,a_1,\ldots,a_{k-1},1,\infty]\otimes [0,b_k,\ldots,b_\ell,1,\infty]\ ,$$
which is the tensor product of the cell forms corresponding to the
standard cyclic orders $\eta_1,\eta_2$ on $S_1\cup\{d\}$ and
$S_2\cup\{d\}$ induced by $\eta$. Therefore $\rho(\Res^p_D
\eta)=\Res_D \omega_\eta$.

To conclude the proof of the proposition, it is enough to notice
that applying $\sigma\in \Sym(n)$ to the formula
$\hbox{Res}_D\omega_\eta=\omega_{\eta_1}\otimes \omega_{\eta_2}$
yields
$$\hbox{Res}_{\sigma(D)}\sigma^*(\omega_\eta)=
\hbox{Res}_{\sigma(D)}\omega_{\sigma(\eta)}=\sigma^*(\omega_{\eta_1})
\otimes \sigma^*(\omega_{\eta_2})=\omega_{\sigma(\eta_1)}\otimes
\omega_{\sigma(\eta_2)}.$$
Here, 
$\sigma(\eta_i)$ is the  cyclic order induced by $\sigma(\eta)$ on
the set $\sigma(S_1)\cup \{\sigma(d)\}$, where $\sigma(d)$
corresponds to the partition $S=\sigma(S_1)\cup \sigma(S_2)$. Thus
$\rho(\Res^p_{\sigma(D)} \sigma(\eta))=\Res_{\sigma(D)}
\omega_{\sigma(\eta)}$ for all $\sigma \in \Sym(n)$, which proves
that $\rho(\Res^p_D \gamma)=\Res_D \omega_\gamma$ for all cyclic
structures $\gamma\in\cal{P}_S$, and all divisors $D$.
\end{proof}

\begin{cor}\label{neatcoro} A linear combination $\eta=\sum_i a_i\eta_i\in W_S
\subset {\cal P}_{S\cup \{d\}}$ converges with respect to the standard polygon
if and only if its associated form $\omega_\eta$ converges on the
standard cell.  
\end{cor}

\begin{proof}
We first show that \begin{equation}\label{rescond}\Res^p_D(\eta)\in
I_{S_1}\otimes {\cal P}_{S_2\cup \{d\}}+ {\cal P}_{S_1\cup
\{d\}}\otimes I_{S_2}\ ,\end{equation} if and only if $\omega_\eta$
converges along the corresponding divisor $D$ in the boundary of the
standard cell. If $(\ref{rescond})$ holds,
then by proposition \ref{prop41} together with the previous
proposition, $\Res_D(\omega_\eta)=0$. Conversely, if
$\Res_D(\omega_\eta)=0$ for a divisor $D$ in the boundary of the
standard cell, then by the previous proposition, $\Res^p_D(\eta)\in$
Ker$(\rho\otimes\rho)$, which is exactly equal to $I_{S_1}\otimes
{\cal P}_{S_2\cup \{d\}}+ {\cal P}_{S_1\cup \{d\}}\otimes I_{S_2}$.

We now show that $(\ref{rescond})$ is equivalent to the convergence
of $\eta$. But since $\eta\in W_S$, the argument of lemma
$\ref{lemleft}$ implies that $(\ref{rescond})$ holds automatically
for any $D$ which intersects $\{1,n\}$ non-trivially. If $D$
intersects $\{1,n\}$ trivially, then we can assume that $\{1,n\}
\subset S_2$. In that case, the fact that $W_{S_2}\cap I_{S_2}=0$
(lemma \ref{usefullemma}) implies that $(\ref{rescond})$ is
equivalent to the apparently stronger condition
$$\Res^p_D(\eta)\in
I_{S_1}\otimes {\cal P}_{S_2\cup \{d\}}\ ,$$ and thus $\eta$
converges along $S_1$ in the sense of definition (\ref{conv}). This
holds for all divisors $D$ and thus completes the proof of the
corollary.
\end{proof}

\begin{cor} The Lyndon insertion words of ${\cal W}_S$ form a generating
set for ${\cal FC}$.  Furthermore, ${\cal FC}$ is defined by subjecting
this generating set to only two sets of relations:
\begin{itemize}
\item{dihedral relations}
\item{product map relations}
\end{itemize}
\end{cor}

\vspace{.3cm}
\subsection{The insertion basis for $H^\ell(\Mod_{0,n}^\delta)$}\label{sec43}

\begin{defn} Let an {\it insertion form} be the sum of $01$-cell forms
obtained by renumbering the Lyndon insertion words of ${\cal W}_S$ via
$(1,\ldots,n,d)\rightarrow (0,t_1,\ldots,t_{\ell+1},1,\infty)$.
\end{defn}

\begin{prop}\label{insformsspan} The insertion forms form a basis for
$H^{n-2}(\Mod_{0,n+1}^\delta)$.
\end{prop}

This is an immediate corollary of all the preceding results.

It is interesting to attempt to determine the dimension of the
spaces $H^\ell(\Mod_{0,n}^\delta)$.  The most important numbers
needed to compute these are the numbers $c_0(n)$\index{$c_0(n)$}\label{csub0n} of special
convergent words (convergent 01 cell-forms) on $\Mod_{0,n}$.  We
have $c_0(4)=0$, $c_0(5)=1$, $c_0(6)=2$, $c_0(7)=11$, $c_0(8)=64$,
$c_0(9)=461$.

\begin{prop}
Set $I_1=1$, and let $I_r$ denote the cardinal of the set
${\cal L}_{\{1,\ldots,r\}}$ for $r\ge 2$.  The dimensions
dim$\,H^\ell(\Mod_{0,n}^\delta)$ are given by
\begin{equation}\label{diminsertionformula}
d_n = \sum^n_{r=5} \ \ \ \sum_{i_1+\cdots + i_{r-3}= n-3} I_{i_1}\ldots
I_{i_r} c_0(r)\ , \end{equation} where the inner sum is over all
partitions of $(n-3)$ into $(r-3)$ strictly positive integers.
\end{prop}

We have $I_1=I_2=1$, $I_3=2$, $I_4=7$.  The formula gives
\begin{equation*}
\begin{cases}
d_5=I_1^2c_0(5)=1\ ,\\
d_6=I_1I_2c_0(5)+I_2I_1c_0(5)+I_1^3c_0(6)=1+1+2=4\ ,\\
d_7=I_1I_3c_0(5)+I_2^2c_0(5)+I_3I_1c_0(5)+I_1^2I_2c_0(6)+I_1I_2I_1c_0(6)
+I_2I_1^2c_0(6)+c_0(7)\\
\ \ =5c_0(5)+3c_0(6)+c_0(7)=5+6+11=22\ .
\end{cases}
\end{equation*}

These expressions give the dimensions as sums of positive terms. A
very different  formula for dim$\,H^\ell(\Mod_{0,n}^\delta)$ is
given in the appendix using point-counting methods.

\vspace{.3cm}
\subsection{The insertion basis for $\Mod_{0,n}$, $5\le n\le 9$}
 \label{calculations} In this section we list the insertion bases
in low weights.  In the case $\Mod_{0,5}$, there is a single
convergent cell form: \begin{equation}\label{insbasis5}
\omega=[0,1,t_1,\infty,t_2].
\end{equation}
The corresponding period integral is the cell-zeta value:
$$\zeta(\omega) = \int_{(0,t_1,t_2,1,\infty)} [0,1,t_1,\infty,t_2] =
\int_{0\leq t_1\leq t_2\leq 1}  {dt_1 dt_2 \over (1-t_1) t_2 } =
\zeta(2)\ .$$ Here we use the notation of round brackets for cells
in the moduli space $\Mod_{0,n}$ introduced in section
\ref{prodmaps}: the cell $(0,t_1,t_2,1,\infty)$ is the same as the
cell $X_{5,\delta}$ corresponding to the standard dihedral order on
the set $\{0,t_1,t_2,1,\infty\}$. Since $C_0(5)$ is 1-dimensional,
the space of periods in weight $2$, namely the weight 2 graded part
${\cal C}_2$ of the algebra of cell-zeta values ${\cal C}$ of
section \ref{cellalg}, is just the 1-dimensional space spanned by
$\int_{X_{5,\delta}} \omega=\zeta(2)$.

\subsubsection{The case $\Mod_{0,6}$}
The space $C(6)$ is four-dimensional, generated by two
$01$-convergent cell-forms (the first row in the table below) and
two forms (the second row in the table below) which come from
inserting $\cal{L}_{1,2}=\{1\sha2\}$ and $\cal{L}_{2,3}=\{2\sha 3\}$
into the unique convergent $01$ cell form on $\Mod_{0,5}$
(\ref{insbasis5}). The position of the point $\infty$ plays a
special role. It gives rise to another grading, corresponding to the
two columns in the table below, since $\infty$ can only occur in two
positions. \vspace{0.1in}
\begin{center}
\begin{tabular}{|c|c|c|}
  \hline
$C_0(6)$ & $\omega_{1,1}=[0,1,t_2,\infty,t_1, t_3]$ &
$\omega_{1,2}=[0,1,t_1,t_3, \infty, t_2]$ \\
\hline
$C_1(6)$ & $\omega_{2,1}= [0,1,t_1, \infty, t_2\sha t_3]$ &
$\omega_{2,2}=[0,1,t_1\sha t_2,\infty, t_3]$ \\
  \hline
\end{tabular}
\end{center}
\vspace{0.1in} We therefore have four generators in weight 3. There
are no product relations on $\Mod_{0,6}$, so in order to compute the space
of cell-zeta values, we need only compute the action of the dihedral
group on the four differential forms.  In particular, the order 6 cyclic
generator $0\mapsto t_1\mapsto t_2\mapsto t_3\mapsto 1\mapsto \infty\mapsto 0$
sends
$$\omega_{1,1}\mapsto -\omega_{2,1}-\omega_{2,2},\ \
\omega_{1,2}\mapsto \omega_{1,1},\ \ \omega_{2,1}\mapsto
-\omega_{1,2}-\omega_{2,1},\ \ \omega_{2,2}\mapsto \omega_{2,1}.$$
Thus, letting $X$ denote the standard cell $X_{6,\delta}=
(0,t_1,t_2,t_3,1,\infty)$, we have $\int_X \omega_{1,1}=\int_X
\omega_{1,2}$, $\int_X \omega_{2,1} =\int_X \omega_{2,2}$ and
$2\int_X \omega_{2,2}=\int_X \omega_{1,2}$, so in fact the periods
form a single orbit under the action of the cyclic group of order 6
on $H^\ell(\Modf_{0,S})$. We deduce that the space of periods of
weight 3 is of dimension $1$, generated for instance by $\int
\omega_{2,1}$. Since $\omega_{2,1}$ is the standard form for
$\zeta(3)$, we have
$$ \begin{array}{ccccc}
\zeta(0,1,t_2,\infty,t_1, t_3)& = & \displaystyle{\int_{X} {dt_1
dt_2 dt_3 \over (1-t_2)(t_1-t_3)t_3}} & = &2\, \zeta(3)\ , \nonumber
\vspace{0.04in} \\ \vspace{0.04in} \zeta(0,1,t_1,t_3,\infty, t_2) &=
&\displaystyle{\int_{X} {dt_1 dt_2 dt_3 \over
(1-t_1)(t_1-t_3)t_2}} &=& 2\, \zeta(3)\ , \nonumber \\
\zeta(0,1,t_1,\infty, t_2\sha t_3) &=& \displaystyle{\int_{X} {dt_1
dt_2 dt_3 \over
(1-t_1)t_2t_3}} &=  &\zeta(3)\ , \nonumber \vspace{0.04in}\\
\zeta(0,1,t_1\sha t_2,\infty,t_3) &=& \displaystyle{\int_{X} {dt_1
dt_2 dt_3 \over (1-t_1)(1-t_2)t_3}}& = & \zeta(3)\ ,\nonumber
\end{array}
$$
Note that $\omega_{2,2}$ is the  standard form usually associated to
$\zeta(2,1)$, so that we have recovered the well-known identity
$\zeta(2,1)=\zeta(3)$, which is normally obtained using stuffle,
shuffle and Hoffmann relations on multizetas.

\vspace{.3cm}
\subsubsection{The case $\Mod_{0,7}$} The insertion basis is listed in the
following table. It consists of 22 forms, eleven of which lie in
$C_0(7)$, six of which come from making one insertion into a
convergent $01$ cell-form from $C_0(6)$ (using
$\cal{L}_{1,2}=\{1\sha 2\}$ and $\cal{L}_{2,3}=\{2\sha 3\}$), and
five of which come from making two insertions into the unique
convergent $01$ cell-form from $C_0(5)$ (which also uses
$\cal{L}_{1,2,3}=\{1\sha 2\sha3, 2\sha13 \}$ and
$\cal{L}_{2,3,4}=\{2\sha 3\sha4, 3\sha24\}$).
\begin{center}
\begin{tabular}{|c|c|c|c|}
  \hline
$C_0(7)$ & $[0,1,t_2,\infty,t_3, t_1,t_4]$ & $[0,1,t_1,t_3, \infty,
t_2,t_4]$  & $[0,1,t_1,t_4,t_2,\infty,t_3]$ \\
 &$[0,1,t_2,\infty,t_4, t_1,t_3]$ &  $[0,1,t_1,t_3, \infty, t_4,t_2]$&
 $[0,1,t_2,t_4,t_1,\infty,t_3]$ \\
& $[0,1,t_3,\infty,t_1, t_4,t_2]$ &  $[0,1,t_2,t_4, \infty, t_1,t_3]$&
$[0,1,t_3,t_1,t_4,\infty,t_2]$ \\
& &  $[0,1,t_3,t_1, \infty, t_2,t_4]$ & \\
& &  $[0,1,t_3,t_1, \infty, t_4,t_2]$& \\
 \hline
$C_1(7)$ &  $[0,1,t_2, \infty,t_1, t_3\sha t_4]$ &
$[0,1,t_1,t_4,\infty, t_2\sha t_3]$  & $[0,1,t_1\sha t_2,t_4,\infty,t_3]$\\
 &$[0,1,t_3,\infty,t_1\sha t_2,t_4]$ & $[0,1,t_2\sha t_3, \infty,t_1,
 t_4]$ & $[0,1,t_1,t_3\sha t_4, \infty, t_2]$  \\
  \hline
$C_2(7)$ & $[0,1,t_1,\infty, t_3\sha(t_2,t_4)]$ &
$[0,1,t_1\sha t_2,\infty, t_3\sha t_4]$
& $[0,1,t_2\sha(t_1,t_3), \infty, t_4]$ \\
& $[0,1,t_1, \infty, t_2 \sha t_3 \sha t_4]$ & & $[0,1, t_1\sha
t_2\sha t_3,
\infty, t_4]$\\
 \hline
\end{tabular}
\end{center}

\vspace{0.1in} The standard multizeta forms can be decomposed into sums
of insertion forms as follows:
\begin{equation}
\begin{split}
&{{dt_1dt_2dt_3dt_4}\over{(1-t_1)t_2t_3t_4}}\qquad\
\,=[0,1,t_1,\infty,t_2\sha t_3\sha
t_4]\\
&{{dt_1dt_2dt_3dt_4}\over{(1-t_1)(1-t_2)t_3t_4}}=[0,1,t_1\sha t_2,\infty,
t_3\sha t_4]\\
&{{dt_1dt_2dt_3dt_4}\over{(1-t_1)t_2(1-t_3)t_4}}=[0,1,t_1,t_3,\infty,t_2,t_4]+
[0,1,t_1,t_3,\infty,t_4,t_2]+\\
&\qquad\qquad \qquad\qquad \qquad\qquad
[0,1,t_3,t_1,\infty,t_2,t_4]+ [0,1,t_3,t_1,\infty,t_4,t_2]\\
&{{dt_1dt_2dt_3dt_4}\over{(1-t_1)(1-t_2)(1-t_3)t_4}}=[0,1,t_1\sha
t_2\sha t_3, \infty,t_4]
\end{split}
\end{equation}
In general, the standard multizeta form having factors $(1-t_{i_1}),\ldots,
(1-t_{i_r})$ (with $i_1=1$) and $t_{j_1},\ldots,t_{j_s}$ (with $j_s=n$)
in the denominator is equal to the shuffle form:
$$[0,1,t_{i_1}\sha\cdots\sha t_{i_r},\infty,t_{j_1}\sha \cdots\sha
t_{j_s}],$$ so to decompose it into insertion forms it is simply
necessary to decompose the shuffles $t_{i_1}\sha\cdots\sha t_{i_r}$
and $t_{j_1}\sha \cdots\sha t_{j_s}$ into linear combinations of
Lyndon insertion shuffles.

Computer computation confirms that the space of periods on $\Mod_{0,7}$ is of
dimension $1$ and is generated by $\zeta(2)^2$. Indeed,
up to dihedral equivalence, there are six product maps on
$\Mod_{0,7}$, given by
\begin{equation}
\begin{cases}
(0,t_1,t_2,t_3,t_4,1,\infty)\mapsto (0,t_1,t_2,1,\infty)\times (0,t_3,t_4,1,
\infty)\\
(0,t_1,t_2,1,t_3,t_4,\infty)\mapsto
(0,t_1,t_2,1,\infty)\times(0,1,t_3,t_4,\infty)\\
(0,t_1,t_2,1,t_3,\infty,t_4)\mapsto
(0,t_1,t_2,1,\infty)\times(0,1,t_3,\infty,t_4)\\
(0,t_1,t_2,1,t_3,\infty,t_4)\mapsto
(0,t_1,1,t_3,\infty)\times(0,t_2,1,\infty,t_4)\\
(0,t_1,t_2,t_3,1,t_4,\infty)\mapsto
(0,t_1,t_2,1,\infty)\times(0,t_3,1,t_4,\infty)\\
(0,t_1,t_2,1,t_3,t_4,\infty)\mapsto
(0,t_1,1,t_3,\infty)\times(0,t_2,1,t_4,\infty)
\end{cases}
\end{equation}
Following the algorithm from section \ref{prodmaps}, we have six associated
relations between the integrals of the 22 cell-forms.   Then, explicitly
computing the dihedral action on the forms yields a further set of
linear equations, and it is a simple matter to solve the entire
system of equations to recover the 1-dimensional solution.  It also
provides the value of each integral of an insertion form as a rational
multiple of any given one; for instance all the values can be computed
as rational multiples of $\zeta(2)^2$.  In particular, we easily recover
the usual identities
$$\zeta(4)={{2}\over{5}}\zeta(2)^2,\ \
\zeta(3,1)={{1}\over{10}}\zeta(2)^2,\ \
\zeta(2,2)={{3}\over{10}}\zeta(2)^2,\ \
\zeta(2,1,1)={{2}\over{5}}\zeta(2)^2.$$

\vspace{.3cm}
\subsubsection{The cases $\Mod_{0,8}$ and $\Mod_{0,9}$}

There are 64 convergent $01$ cell-forms in on $\Mod_{0,8}$, and the
dimension of $H^5(\Mod^{\delta}_{0,8})$ is 144.  The remaining 80
forms are obtained by Lyndon insertion shuffles as follows:
\begin{itemize}
\item{44 forms obtained by making the four insertions:
$$(t_1\sha t_2,t_3,t_4,t_5),(t_1,t_2\sha t_3,t_4,t_5),
(t_1,t_2,t_3\sha t_4,t_5), (t_1,t_2,t_3,t_4\sha t_5)$$  into the
eleven $01$ cell-forms of $\Mod_{0,7}$}
\item{12 forms obtained by the six insertion possibilities:
\begin{align*}
&(t_1\sha t_2\sha t_3,t_4,t_5), (t_2\sha t_1t_3,t_4,t_5),
(t_1,t_2\sha t_3\sha t_4,t_5), (t_1,t_3\sha t_2t_4,t_5),\\&
(t_1,t_2,t_3\sha t_4\sha t_5), (t_1,t_2,t_4\sha t_3t_5) 
\end{align*} into the
two $01$ cell-forms of $\Mod_{0,6}$}
\item{6 forms obtained by the three insertion possibilities:
$$(t_1\sha t_2,t_3\sha t_4,t_5), (t_1\sha t_2,t_3,t_4\sha t_5),
(t_1,t_2\sha t_3,t_4\sha t_5)$$
into the two $01$ cell-forms of $\Mod_{0,6}$}
\item{4 forms obtained by the four insertions:
$$(t_1\sha t_2\sha t_3,t_4\sha t_5),
(t_2\sha t_1t_3,t_4\sha t_5),
(t_1\sha t_2, t_3\sha t_4\sha t_5),
(t_1\sha t_2, t_4\sha t_3 t_5)$$
into the single $01$ cell-form of $\Mod_{0,5}$}
\item{14 forms obtained by the fourteen insertions:
\begin{align*}
& (t_1t_3\sha t_2t_4,t_5),
(t_3\sha t_1t_4t_2,t_5),
(t_1t_3\sha t_2\sha t_4,t_5),
(t_1t_4\sha t_2\sha t_3,t_5),
(t_2t_4\sha t_1\sha t_3,t_5),\\
& (t_2\sha t_1(t_3\sha t_4),t_5),
(t_1\sha t_2\sha t_3\sha t_4,t_5),
(t_1,t_2t_4\sha t_3t_5),
(t_1,t_4\sha t_2t_5t_3),
(t_1,t_2t_4\sha t_3\sha t_5),\\
&(t_1,t_2t_5\sha t_3\sha t_4),
(t_1,t_3t_5\sha t_2\sha t_4),
(t_1,t_3\sha t_2(t_4\sha t_5)),
(t_1,t_2\sha t_3\sha t_4\sha t_5)\end{align*}
into the single $01$ cell-form of $\Mod_{0,5}$.}
\end{itemize}

The case of $\Mod_{0,9}$ is too large to give explicitly.  There are
461 convergent $01$ cell-forms, and
dim$\,H^6(\Mod^{\delta}_{0,9})=1089$. An interesting phenomenon
occurs first in the case $\Mod_{0,9}$; namely, this is the first
value of $n$ for which convergent (but not $01$) cell-forms do not
generate the cohomology.  The 1463 convergent cell-forms for
$\Mod_{0,9}$ generate a subspace of dimension 1088.

For $5\le n\le 9$, computer computations have confirmed the
main conjecture, namely:  {\it for $n\le 9$, the weight $n-3$ part
${\cal FC}_{n-3}$ of the formal cell-zeta algebra ${\cal FC}$ is
of dimension $d_{n-3}$, where $d_n$ is given by the Zagier formula
$d_n=d_{n-2}+d_{n-3}$ with $d_0=1$, $d_1=0$, $d_2=1$.}

\chapter{Cohomology of $(\Mod^\gamma_{0,n})$ }

\begin{defn}\label{derhamc}The $k$th de Rham cohomology group of a smooth
  manifold, $X$, is defined to be the group of closed
  differential $k$ forms  on $X$ (those whose exterior derivative is
  0) modulo exact ones (those which are the exterior derivative of a
  $k-1$ form).
\end{defn}

In the first section of this chapter, we use the theory of spectral
sequences of a fibration to review the proof of the following
well-known dimension result on cohomology groups of genus 0 moduli
space, a complete proof of which is difficult (or impossible) to find in the literature.

\begin{thm}\label{dimension} For $n\geq 3$, the dimension of
$H^{n-3}(\Mn,\Q)$ is $(n-2)!$ and the dimension of
  $H^k(\Mn,\Q)$ is 0 whenever $k>n-3$.\end{thm}

This result 
 was used and reproved (though not as explicitly as we do in the following sections) in a well-known theorem by Arnol'd.  For each cohomology
 group, $H^{k}(\Mn,\Q)$, Arnol'd
explicitly exhibits a set, $B^k_n$ of differential forms whose 
classes form a basis of $H^k(\Mn,\Q)$ and
which has the astonishing property that the ring, $A$, generated by $B^1_n$
 contains $B^k_n$ for all $k$ and $A$ is isomorphic to
 $H^*(\Mn,\Q)$.  For the remainder of the text, we will denote $H^{k}(\Mn,\Q)$ simply by
 $H^k(\Mn)$\index{$H^k(\Mn)$}.

\begin{defn}
We denote by $\omega_{t_i,t_j}$ the differential forms defined by :
\begin{align} \omega_{t_i,t_j} &= \frac{dt_j -dt_i}{t_i-t_j},\ 1\leq
  i < j\leq n-3 \\
\omega_{0,t_j} &= \frac{dt_j}{t_j} \\ 
\omega_{1,t_j} &= \frac{dt_j}{1-t_j}.\end{align}
We call the ring generated by these forms {\bf Arnol'd's Ring}, $A$.\index{Arnol'd's Ring, $A$}
\end{defn}
\begin{thm}[Arnol'd]\label{arnoldring}  Let $i_1,...,i_k$ be distinct integers in the interval $[1,n]$.
The elements of Arnol'd's ring of the form,
$$\bigwedge_{l=1}^k \omega_{z,t_{i_l}},\ z=t_j,\ 0\hbox{ or }1$$ (where $j<i_l$)
form a basis of $H^k(\Mn)$.  In particular,
a basis of
$H^{n-3}(\Mn)$ is given by
$${ \bigwedge_{j=1}^{n-3} \frac{dt_j}{t_j-z}\ :\ z=0,1\ \mathrm{or}\ t_i,
i <j }.$$ 
\end{thm}

Because Arnol'd's theorem is a key ingredient in our work, in section 4.1 of this chapter, we recall a self-contained proof of the theorem \ref{dimension}.
Here, we recall some definitions and properties of divisors on $\M_{0,n}$ used throughout the rest of the text.

\begin{defn}
 Let $Z$ be the set denoting marked points on $\Mn$, $Z=\{z_1,...,z_n\}$ and let $\rho$ be the set of all partitions of $Z$, in which each set in
 the partition has cardinality greater than or equal to 2.
\end{defn}

\begin{defn}\label{newd_adef}\index{$d_A$}
Let $\{z_{i_1},...,z_{i_k}\}=K\subset Z$ be a subset of the marked
points.  Then the divisor which is obtained as the exceptional divisor by blowing up along $z_{i_1}=\cdots =
z_{i_k}$ in $\Pro^{n-3}_1$ is denoted by $d_K$.
\end{defn}

Recall from page \pageref{d_Adef} that $d_A=d_{Z\setminus A}$ in $\M_{0,n}$.

\begin{defn}\label{m0ngamma}  We denote by $D$\index{$D$} the disjoint union, $\sqcup_{i\in \rho} \{d_i\}$ where each $\{d_i\}$ is a singleton whose single element is the (irreducible) boundary divisor in $\M_{0,n}\setminus \Mn$ defined by the partition $i$ as in the definition \ref{newd_adef}.  Likewise, if $\gamma\sqcup \gamma^c$ is a partition of $\rho$, we denote by $D_{\gamma^c}:= \sqcup_{i\in \gamma^c} \{d_i\}$.  We denote by $\Mn^\gamma:=\M_{0,n}\setminus D_{\gamma^c}$ and call $\Mn^\gamma$ a {\rm partial compactification}\index{Partial compactification}\index{$\Mn^\gamma$} of $\Mn$.
\end{defn}

So we have $\Mn\subset \Mn^\gamma \subset \M_{0,n}$.

We remark here that the results outlined in this chapter are combinatorial ones obtained by considering an irreducible boundary component as the set of its defining marked points, therefore by a slight abuse of notation, and when no ambiguity can arise, we will denote $D_\gamma$ simply by $\gamma$.


In the second section of this chapter we show an analog of Arnol'd's
theorem for partial compactifications.  In particular, we show that
the cohomology rings of the partial compactifications are the subrings of
Arnold's ring converging on the partial compactification.

In the third section, we will recall Brown's proof that  $\Mn^\delta$
is an affine variety when $D_\delta$ is the set of divisors each of which contains a face of the boundary of an
associahedron in $\Mn(\R)$ (as in chapter 3),
and deduce that any subset, $D_\gamma\subseteq D_\delta$  of boundary
divisors also has
the property that $\Mn^\gamma$ is affine (for example when
$|\gamma|=1, 2$).

In the fourth section, for some of the families $D_\gamma$ of boundary divisors
from section 4.3, we display explicit bases of the top dimensional
cohomology groups of $\Mn^\gamma$.
For $\gamma=\delta$, the union of the boundary divisors
of the standard cell, this computation was done in chapter 3 where we defined the basis of {\it insertion forms}.  In chapter 4, we
generalize the method of insertion forms.

In the last section, we study a combinatorial presentation of the
Picard group of 
divisors on $\M_{0,n}$ based on work of S. Keel
and A. Gibney.  
We extend their techniques of
calculating a basis to calculating an explicit expression of any 
boundary divisor in terms of these bases by using polygon techniques. 

\section{Spectral sequences of a fibration}
The goal of this section is to recall the proof theorem \ref{dimension} by induction, using only the Leray theorem of the
cohomology of a spectral sequence.

{\it Proof (of theorem \ref{dimension})}.
We begin the proof by justifying the base case, $n=3$, in which case
$\Mod_{0,3}$ is a point.  The dimension
of $H^{0}(\Mod_{0,3})$ is thus $1!=1$ and
$H^k(\Mod_{0,3})=0$ for all $k>0$.

So we have two induction hypotheses,
$\dim(H^{n-4}(\Mod_{0,n-1}))=(n-3)!$ and for $k>n-4$,
$H^k(\Mod_{0,n-1})=0$.  

Before detailing the proof, let us recall some notation and some
results on spectral sequences.

Let \begin{equation}\label{filt} \emptyset = X_{-1} \subset X_{0}
  \subset \cdots
  \subset X_{k} = X \end{equation}
be a filtration of the topological space $X$.  And let
$$C_{q}(X_{-1}) \subset \cdots \subset C_q(X_k)$$ be the formal
groups\label{CqXdef}\index{$C_q(X)$} of
$\Z$ linear combinations of $q$ dimensional oriented simplices in the subspaces
$(X_i)$.  Such filtrations exist for nice topological spaces, such as
CW complexes.  There exist linear maps, $\partial_q :C_q(X_i) \rightarrow
C_{q-1}(X_i)$ which send a simplex to its boundary in $X_i$.  Since
the simplices are oriented, $\partial_{q-1} \circ \partial_q =0$.

Let $$E_0^{i,q-i}  := C_q(X_i)/C_q(X_{i-1}).$$\label{Eblahdef}  Then
$\partial_q$ 
induces an exact sequence, $(E_0^{i, q-i}, d_0^{i,q-i})$.

Let $Z_r^{i,q-i}\subset E_0^{i,q-i}$\index{$Z_r^{i,q-i}$} be the subgroup of elements,
$\alpha$, such that the coset of $\alpha$ contains a representative
$a$ such that $\partial_q (a) \in C_{q-1}(X_{i-r}).$  From this
definition, we see that $Z_0^{i,q-i} = E_0^{i,q-i}$.  

By the filtration, we have $Z_k^{i,q-i} \subset Z_{k-1}^{i,q-1}$.
If $r$ is
sufficiently large, we obtain a stable group, $Z_\infty^{i,q-i}$ whose
elements $\alpha$ contain a coset $a$ such that $\partial_q (a) = 0$.

Let $B_r^{i,q-i}\subset E_0^{i,q-i}$\label{Bblahdef}\index{$B_r^{i,q-i}$} be the subset of elements
$\alpha$ whose coset contains an element $a$ such that there exists an
element $b\in C_{q+1}(X_{i+r-1})$ such that $a=\partial_{q+1}(b)$.

By the filtration, we have that $B_{k}^{i,q-i} \subset B_{k+1}^{i,q-i}$.
Furthermore, if $r$ is large enough, we obtain a stable group,
$B_\infty^{i,q-i}$ which contains all of the elements $\alpha =
\overline{a}$ such that $a$ is the boundary of some simplex in
$C_q(X)$.  Since $\partial_{q-1}\circ\partial_q=0$, we have that
$B_\infty^{i,q-i}\subset Z_\infty^{i,q-i}$.
  
Let $E_r^{i,q-i}:=Z_r^{i,q-i}/B_r^{i,q-1}$\index{$E_r^{i,q-i}, E_r$}.  The differential
$d_0^{i,q-i}$ induces a complex, $$d_r^{i,q-i}: E_r^{i,q-i}
\rightarrow E_{r}^{i-r,q-i+r-1}.$$\index{$d_r^{i,q-i},\ d_r$}  Let $$E_r:= \bigoplus_{i,q}
E_r^{i,q-i},$$ which is a complex for the differential $d_r:=
\bigoplus d_r^{i,q-i}:= E_r\rightarrow E_{r}$.

\begin{defn} The complex $(E_r, d_r)$ is a {\it spectral sequence}\index{Spectral sequence} for
  the given
  filtration of $X$.\end{defn}

\begin{thm}\label{Fomenko} $E_{r+1}$ is the homology group of $E_r$
  with respect to
  the differential $d_r$, in particular, $$E_{r+1}^{p,q} \simeq
  \ker(d_r^{p,q}) /
  \mathrm{Im}(d_{r}^{p-r,q+r-1}).$$ \end{thm}

The proof of this theorem can be found in many textbooks, such as [FFG].

In our studies, we are concerned with the stabilizing groups,
$$E_\infty^{i,q-i} = Z_\infty^{i,q-i}/B_\infty^{i,q-i}.$$

Now let us consider the fibration 
\begin{align}\label{fibration} \Mn & \rightarrow \Mod_{0,n-1} \\
(0,t_1,...,t_{n-3},1,\infty)& \mapsto (0,t_1,...,t_{n-4}, 1,\infty),
\end{align} with
fiber equal to $\Pro^1(\C) \setminus\{0,t_1,...,t_{n-4},1,\infty\}$
over $(0,t_1,...,t_{n-4}, 1,\infty)$.

We write this fibration in the classical notation as\index{$F,\ E,\ B$}
\begin{equation}
F \hookrightarrow E \rightarrow B
\end{equation}
where $F$ is the fiber, the projective line minus $n-1$ points, $B$ the
base, $\Mod_{0,n-1}$, and $E$ is $\Mn$.

\begin{lem}\label{hoF} The homology groups of $F$ as $\R$ vector
  spaces are given
  by
\begin{equation} H_q(F,\R) \simeq \begin{cases} \R & q=0 \\ \R^{n-2} & q=1 \\
0 & q > 1 \end{cases}
\end{equation}
\end{lem}
This simple lemma may be deduced by using a long-exact Mayer-Vietoris sequence.

Any
differentiable manifold has the homotopy type of a CW complex \cite{Mi}. 
In particular, $\Mn$ is a CW complex and so there exists a filtration
 as in \eqref{filt} on $\Mn$,
$$\emptyset = X_{-1}\subseteq X_0 \subseteq \cdots \subseteq X_k
=\Mn,$$ where $X_i$ denotes the $i$th skeleton of $\Mn$.

We have a bundle of groups over $\Mod_{0,n-1}$ given by the family,
$\{H_1(F_{b})\ : \ b\in B\}$,  and the
associated family
of homomorphisms, $$h[\lambda]: H_1(F_{b_0}) \rightarrow
H_1(F_{b_1}),$$\label{hlambdadef}\index{$h[\lambda]$} for all paths $\lambda$ from $b_0$
to $b_1$ on $B$. 
The homomorphism, $h[\lambda]$ comes from lifting the path $\lambda$ to
$\Mn$.  The choice of a lift of $b_0$ to $F_{b_0}\subset
\Mn$ determines the lift of $\lambda$ uniquely.  Therefore the
endpoint of this lift, a lift of $b_1$, is uniquely defined, giving
a map from $F_{b_0}$ to $F_{b_1}$.  This map induces a map on the
fundamental groups and hence passes to the $H_1$.  The $h[\lambda]$ satisfy
\begin{align} 
& h[\mathrm{Id}] = \mathrm{Id} \\
& h[\lambda] = h[\lambda']
\ \mathrm{whenever}\  \lambda\mathrm{\ is\ homotopic\ to\ }\lambda' \\
& \lambda: b_0\rightarrow b_1, \mu: b_1\rightarrow b_2,\ \mathrm{then}
\ h[\mu\circ
  \lambda] = h[\mu]\circ h[\lambda]. \end{align}

\begin{defn}\label{simple} A fibration is simple if for all
  $\lambda_1,\lambda_2
  :b_0\rightarrow 
  b_1$, $h[\lambda_1]=h[\lambda_2]$.  \end{defn}

\begin{claim}\label{claimsimple} The fibration \eqref{fibration} is
  simple. \end{claim}
\begin{proof} Given two points $b_0\ne b_1$ on $B$, we may fix a path
$\lambda_0$ from $b_0$ to $b_1$, and every path on $B$
  from $b_0$ to $b_1$ is homotopic to a loop starting at $b_0$
  composed with $\lambda_0$. Thus, by definition \ref{simple}
  we only need to show simplicity for loops $\lambda$ on $B$ based at
  a point $b$, in 
  other words that $(h[\lambda] :H_1(F_b) \rightarrow H_1(F_b) ) =
  \mathrm{Id}$. 
  We saw above that loops on $B$, and in fact homotopy classes of
  loops on $B$, 
act on $\pi_1(F)$; in other words there is a group action of $\pi_1(B)$ on $\pi_1(F)$.  This action can be explicitly computed as follows.

Recall the definition of the Artin braid group, $ B_{n-1}$, generated by the fundamental braids $\sigma_{i}, i=1,...,n-2$, subject to the relations $\sigma_i\sigma_{i+1}\sigma_{i}=\sigma_{i+1}\sigma_{i}\sigma_{i+1}$ and $\sigma_{i}\sigma_j=\sigma_j\sigma_i$ for all $|i-j|\geq 2$.  The full mapping class group $\Gamma_{0,[n-1]}$ is defined to be the quotient of $B_{n-1}$ by the two relations, $(\sigma_1\cdots \sigma_{n-2})^{n-1}=1$ and $\sigma_1\cdots \sigma_{n-2}\cdot \sigma_{n-2} \cdots \sigma_1 =1$.  There is a surjection $B_{n-1}\twoheadrightarrow \Sym_{n-1}$ given by mapping $\sigma_i\mapsto (i,i+1)$,
which factors through $\Gamma_{0,[n-1]}$.  

Let the free group, $\pi_1(F)$, be generated by the loops, $x_1,...,x_{n-1}$ around the marked points of $F$, whose product equals 1.  The group $\Gamma_{0,[n-1]}$ acts on $\pi_1(F)$ via
\begin{equation}\label{action}\sigma_i(x_j) = \begin{cases} x_j & j<i,\ j>i+1\\
                   x_{j+1} & j=i\\
 x_{j}^{-1}x_{j-1} x_{j} & j=i+1.
                  \end{cases}\end{equation}

The fundamental group of $B=\Mod_{0,n-1}$, known as the pure mapping class group, $\Gamma_{0,n-1}$, is the kernel of the surjection, $\Gamma_{0,[n-1]}\twoheadrightarrow \Sym_{n-1}$.  It is generated by the elements $x_{i,j}=\sigma_{j-1}\cdots \sigma_{i+1}\sigma_{i}^2 \sigma_{i+1}^{-1}\cdots \sigma_{j-1}^{-1}$.  It acts on $\pi_1(F)$ by restriction of the action \eqref{action}, and it is easy to see that each generator, $x_{i,j}$ maps each $x_k$ to a conjugate of $x_k$.  Thus, $\pi_1(B)$ passes to the trivial action on
  $\pi_1(F)^{ab}=H_1(F)$.
  \end{proof}

\begin{thm}[Leray] Given a simple fibration,
$$F\hookrightarrow E \rightarrow B,$$
  there exists a cohomology spectral sequence, $\{E_r^{p,q}, d_r\}$ such that
  $$E_2^{p,q} \simeq H^p(B,H^q(F))$$ and converging to $H_*(E)$ in
  other words, $$\bigoplus_{r+s=n} E_\infty^{r,s} \simeq H^n(E).$$
\end{thm}

An introductory proof of this famous theorem can be found for example in [FFG].
This theorem is the major ingredient in our proof of the dimension
result.  

\begin{lem}\label{tensor}
For the fibration \eqref{fibration} we have $E_2^{p,q} \simeq H^p(B)\otimes
H^q(F),$ and if $q>1$ then $E_2^{p,q}=0$. 
\end{lem} 

\begin{proof} $H_q(F)$ is a finite dimensional real vector space, $\R^k$,
  so its dual, $H^q(F) \simeq \R^{k}$.
  By the Leray theorem,  
\begin{align}
E_2^{p,q}\simeq H^p(B,H^q(F)) \simeq H^p(B,\R^{k}) \simeq H^p(B)\otimes \R^{k}
 \simeq H^p(B)\otimes H^q(F). 
\end{align}
This holds because the action of $\pi_1(B)$ on $\R^k\simeq H_q(F)$ is
trivial as we saw
in the proof of claim \ref{claimsimple}.
If $q>1$ then $E_2^{p,q}=0$ by lemma \ref{hoF}.
\end{proof}

\begin{lem}\label{box} If $E_2^{p,q} \neq 0$, then $0\leq p \leq n-4$
  and $0\leq q \leq 1$.
Therefore, $E_2^{p,q} \neq 0$ implies that $p+ q\leq n-3$.
\end{lem}

\begin{proof}
By \ref{tensor}, $E_2^{p,q}=H^p(B)\otimes H^q(F)$; the left hand
factor is 0 whenever $p>\mathrm{dim}(B)=n-4$ by the induction
hypothesis and the right hand one is
0 whenever $q>1$ by lemma \ref{hoF}.
\end{proof}

We can now conclude the induction proof of the vanishing statement of
theorem \ref{dimension}.
\begin{cor}  $H^k(E)=0$ whenever $k>n-3$.  
\end{cor}
\begin{proof}  If $k>n-3$, then $E_2^{p,q}=0$ for  $p+q=k$.  Recall by
  theorem \ref{Fomenko}
  that $E_{r+1}^{p,q}$ is the homology group of $E_r^{p,q}$.
Therefore,   
if $E_2^{p,q}=0$, then $E_\infty^{p,q} =0$.
 So by
  Leray, $\bigoplus_{p+q=k}
  E_\infty^{p,q}=H^k(E)=0$.
\end{proof}

\begin{lem}\label{stable}  For the given fibration,
  $\Mn\rightarrow \Mod_{0,n-1}$, we have
$E_2^{n-4,1} = E_\infty^{n-4,1}$.
\end{lem}
\begin{proof} To prove this, we show that $E_2^{n-4,1} =E_3^{n-4,1}$,
  therefore by theorem \ref{Fomenko}, the sequence stabilizes at
  $E_2^{n-4,1} =E_\infty^{n-4}.$  

We have
  $$E_3^{n-4,1}\simeq \ker(d_2^{n-4,1}) /
  \mathrm{Im}(d_2^{n-2,0}).$$  The kernel of $$d_2^{n-4,1}: E_2^{n-4,1}
  \rightarrow E_2^{n-6,2}$$ is all of $E_2^{n-4,1}$ since the image,
  $E_2^{n-6,2}=0$ by \ref{box} for $q=2>1$.

Likewise, the image
  of $d_2^{n-2,0}: E_2^{n-2,0}\rightarrow E_2^{n-4,1}$ is 0 since
  $E_2^{n-2,0}=0 $ by \ref{box} for $p=n-2>n-3$. 

This proves the lemma.
\end{proof}

The previous sequence of lemmas and claims allows us to deduce the
following key
proposition of this section. 

\begin{prop}\label{mainspec} The cohomology group, 
$$H^{n-3}(E) \simeq H^{n-4}(B) \otimes H^1(F).$$
\end{prop}

\begin{proof}  By the
  Leray theorem, we have that
\begin{align} H^{n-3}(E) & =\bigoplus_{p+q=n-3}
  E_\infty^{p,q}\\
&= E_\infty^{n-4,1} \oplus E_\infty^{n-2,0} \ \hbox{(by lemma \ref{box})}\\
&= E_\infty^{n-4,1}\  \hbox{(also by lemma
    \ref{box})} \\
&=  H^{n-4}(B)\otimes H^1(F). \end{align}
 
The last equality follows from the previous lemma \ref{stable} and
  lemma \ref{tensor}, since $$E_2^{n-4,1} =
  E_\infty^{n-4,1}=H^{n-4}(B)\otimes H^1(F).$$ 
\end{proof}

By the induction hypothesis, $\mathrm{dim}(H^{n-4}(\Mod_{0,n-1}))
  =(n-3)!$ and by 
  lemma \ref{hoF}, $\mathrm{dim}(H^1(F)) = n-2$.  Therefore as a corollary to
  proposition \ref{mainspec}, we obtain the claim made in the main
  theorem \ref{dimension},
$$\mathrm{dim}(H^{n-3}(\Mn)) = (n-2)(n-3)! = (n-2)!.$$ 

\section{Cohomology of partial compactifications, $\Mn^\gamma$}

In this section, we prove an analog of Arnol'd's theorem for
the cohomology of
the subspaces 
$\Mn^\gamma \subset \M_{0,n}$, as defined in definition
\eqref{m0ngamma}, in the case where $\Mn^\gamma$ is an
affine variety.

Firstly, it is shown that for certain sets of divisors,
$\gamma$, those such that $\Mn^\gamma$ is an affine variety, we have a
natural injection
$$H^{n-3}(\Mn^\gamma)\hookrightarrow
H^{n-3}(\Mn).$$  Then we give a theorem that
shows how to
explicitly calculate $H^{n-3}(\Mn^\gamma)$ as a subspace of
Arnol'd's ring of theorem \ref{arnoldring} of differential $n-3$ forms.

\begin{prop}\label{subspaceprop} Let $\gamma$ be such that
  $\Mn^\gamma$ is an affine variety.  Then the top dimensional cohomology
  $H^{n-3}(\Mn^\gamma)$ is isomorphic to a subspace of
  $H^{n-3}(\Mn)$.
\end{prop}

\begin{proof}
The heart of justifying this proposition is the following
important result of Grothendieck.

\begin{thm}\cite{Gr1} Let $X$ be an affine algebraic scheme over $\C$, assume
  that $X$ is regular (i.e. ``non-singular'').
  Then the complex cohomology, $H^{\bullet}(X,\C)$ can be calculated as the
  cohomology of the algebraic de Rham complex, (i.e. the complex of
  differential forms on $X$ which are ``rational and everywhere defined'').
\end{thm}

The Deligne-Mumford compactification, $\M_{0,n}$, is a smooth
manifold and the set of divisors we remove, $D\setminus \gamma$, is a
closed subset of $\M_{0,n}$, so $\Mn^\gamma$ is a smooth manifold.

A $k$ form, $\omega$, on $X=\Mn^\gamma$ will be denoted {\it algebraic}\label{fungrtypedef}\index{Algebraic ($k$ form)} 
if it is rational and
 everywhere defined, in other 
  words, it is global and holomorphic on $X$ and there are
rational functions, 
$f_{i_i,...,i_k}(t_1,...,t_{n-3})$ such that \begin{equation}\label{rat} \omega=\sum f_{i_1,...,i_k}\
  dt_{i_1}\wedge \cdots \wedge dt_{i_k}.\end{equation} 
Such a form must be meromorphic on 
$\M_{0,n}$ because it will have poles of finite order on the boundary
of $X$ which is given by blowing up at coalescing marked points.

What Grothendieck's theorem says is that the cohomology group of
classes of algebraic forms, in which two
elements are in the same class if they differ by 
an exact form, $d\alpha$ where $\alpha$ is algebraic, is
isomorphic to the usual de Rham cohomology group.
Therefore, in the following arguments, we may assume that a cohomology
class in $H^{n-3}(X)$ is an equivalence class of algebraic $n-3$ forms.

Let $$\Phi: \Omega^{k}(\Mn^\gamma)\rightarrow
\Omega^{k}(\Mn)$$ denote the restriction map applied to an algebraic $k$ form.  Let $d\alpha$ denote an exact $k$ form on
$\Mn^\gamma$.  Then in particular $\alpha$ is algebraic,
and its restriction, $\Phi(\alpha)$ is of Grothendieck type on $\Mn$.
Thus $\Phi(d\alpha) = d\Phi(\alpha).$  Thus $\Phi$ descends to 
 a $\Q$-linear map on cohomology,
$$\phi:H^{n-3}(\Mn^\gamma) \rightarrow H^{n-3}(\Mn)$$ which sends a class
in the cohomology,
$\overline{\omega}$, to its restriction on $\Mn$.

We will now justify
that this map is injective.  Let $\omega$ be an $n-3$ form such
that $\overline{\omega}\neq 0$ is in the
kernel of $\phi$. 
Then the restriction of $\Phi(\omega)$ is an exact form, $d\alpha$,
on $\Mn$, for $\alpha$  
a meromorphic $n-4$ form on $\M_{0,n}$, holomorphic on $\Mn$.  We claim
$\alpha$ must also be holomorphic on $\Mn^\gamma$.  
For if it weren't,
we can suppose that $\alpha$ has a pole on $\Mn^\gamma$ (in particular on some
boundary divisor $\gamma_i$ in $\gamma$) of order $m>0$.  Then since
$\alpha$ has 
the form \eqref{rat}, $d\alpha=\Phi(\omega)$ would have a pole of order
greater than or equal to $m$ on $\gamma$.  But $\Phi(\omega)$ has the
exact expression, \eqref{rat}, as $\omega$, which by assumption is
holomorphic on $\Mn^\gamma$.  Thus $\alpha$ is holomorphic on
$\Mn^\gamma$ and $\omega$ is exact.
This proves injectivity.

Since $\phi$ is injective, we can consider the injection map as an
inclusion $$H^{n-3}(\Mn^\gamma)\hookrightarrow
H^{n-3} (\Mn).$$ 

\end{proof}

\begin{prop}\label{mngammabasis}
Assume that $\Mn^\gamma$ is an affine variety.
A basis for $H^{n-3}(\Mn^\gamma)$ is given by the classes of 
the $n-3$ forms in the basis of Arnol'd's ring from theorem \ref{arnoldring}
which do not have a pole on $D_\gamma$.  We call such forms ``convergent on $\gamma$'' or ``holomorphic on $\gamma$''.

\end{prop}

\begin{proof} 
Let $A_{i}$ be the sub-vector space of Arnol'd's ring $A$ generated by
$i$ forms.   By Arnol'd's theorem, $A_{n-3}\simeq H^{n-3}(\Mn)$.  
Let $A^\gamma$ be the subspace of $A_{n-3}$ of differential forms
  convergent on $\gamma$.

We have a map $$\rho: A^\gamma \rightarrow
H^{n-3}(\Mn^\gamma),$$ given by associating a form to its cohomology
class.  This map is injective because as we saw above, if
$\omega_1-\omega_2 = d\alpha$, an exact form on $\Mn^\gamma$, then
$d\alpha$ is an exact form on $\Mn$.  
This shows that $\rho$ is injective.

By Grothendieck's theorem, each cohomology class in
  $H^{n-3}(\Mn^\gamma)$ contains a representative which is algebraic, holomorphic on $\gamma$ and thus an element of $A^\gamma$.  We can therefore
  further conclude that 
$\rho$ is surjective.

Hence, $H^{n-3}(\Mn^\gamma) \simeq A^\gamma$ as vector spaces, so a
basis for $A^\gamma$ yields a basis for $H^{n-3}(\Mn^\gamma)$.

\end{proof}

\section{Some affine subvarietes of $\M_{0,n}$}

In this section, we prove that certain
partial compactifications of $\Mn$ contained in $\M_{0,n}$ are affine varieties, that is,
we justify that the addition of some subsets of
divisors to $\Mn$ gives an affine space.  The partial compactifications
we refer to are according to definition
\ref{m0ngamma}.   We first recall some important definitions and
properties of divisors.

\begin{defn}\index{Prime divisor} A prime divisor on $\M_{0,n}$ is an irreducible
  subvariety of $\M_{0,n}$ of codimension 1.
\end{defn}

\begin{defn}\label{defdiv}\index{Weil divisor} A Weil divisor on $\M_{0,n}$ is a formal
  finite linear 
  combination over $\Q$ 
of prime divisors. \end{defn}

In this thesis, we refer to irreducible boundary divisors in $\M_{0,n}\setminus \Mn$ simply as divisors\index{Divisor} as in definition \ref{m0ngamma}.

Every divisor, $d_K$ contains a face of the boundary of
some associahedron $(z_{i_1},...,z_{i_n})$ in $\M_{0,n}(\R)$ where the
elements of $K$ are in a consecutive block in any order.  We can
picture the divisor as a chord along that associahedron as in chapter
3, page \pageref{firstchorddef}.  For example, we can picture the divisor $d_K, K=\{t_1,t_3\}$ in
$\M_{0,6}$ as the chord in the polygon in figure 8.

\begin{defn}\label{intersectingdivisors}\index{Intersection-divisor}  Let $d_J$ and $d_K$ be two divisors satisfying the following three conditions:
 firstly $d_K$ and $d_J$ each contain a face of the boundary of a single associahedron, secondly as chords of the polygon representing that associahedron, $d_J$ and $d_K$ cross
 inside the polygon, and finally $2\leq |J\cup K| \leq n-2$.  We call $d_{J\cup K}$ an
  {\bf intersection-divisor} of $d_J$ and $d_K$.  \end{defn} An intersection-divisor corresponds to the
  chord adjoining adjacent
  endpoints as in figure 8.  An intersection-divisor is not the intersection of divisors, because the intersection of divisors is of codimension $\geq 2$.

\begin{center}
 \input{pillowgon3Penn.pstex_t}
\end{center}

Note that according to this definition, two equivalent divisors may
have up to four associated intersection-divisors, depending on whether one chooses $J$ or
$Z\setminus J$, $K$ or $Z\setminus K$ as the labelling for the divisors.

We have the following result by F. Brown allowing us to
deduce the cohomology of certain partially compactified moduli spaces.

\begin{thm}\cite{Br}\label{affinedelta}  Let $\delta$ be set of boundary divisors
  which contain faces on the boundary of the associahedron
  $(z_{i_1},...,z_{i_n})$, 
\begin{equation}\label{deltadivs} \delta = 
  \{d_K : K=\{z_{i_j},...,z_{i_{j+k}} \}\},\end{equation}
where $K$ is a consecutive
  set of marked points along the associahedron.

Then the partially compactified moduli space, $\Mn^\delta$, is an
affine variety.
\end{thm}

\begin{proof}[Sketch of proof]

Without loss of generality, we may assume that $(i_1,...,i_n)=(1,...,n)$.
The proof for arbitrary dihedral orderings can be repeated by
replacing $j$ with $i_j$ everywhere.

In F. Brown's thesis, he considers the ring of
$\PSL_2$-invariant
regular functions on $\Mn$, $\{ u_{ij}, \frac{1}{u_{ij}}\ |\ i,j,i+1,
j+1\hbox{ distinct modulo }n \}$, defined by the cross ratio
$$u_{ij}: (z_1,...,z_n)\mapsto \frac{ (z_i-z_{j+1}) ( z_{i+1}-z_j)}{
    (z_i -z_j) (z_{i+1}- z_{j+1})} . $$

These functions can be labeled by chords on the polygon whose vertices are
labeled in the standard cyclic order.  The function $u_{ij}$
corresponds (naturally) to the chord between the vertices $i$ and
$j$.  This ring of functions has the following defining relation.

For any four
distinct vertices, $\{i,j,k,l\}$ with the imposed dihedral order of
the polygon, define the sets,
\begin{align}
&A = \{\{p,q\} : i\leq p< j, k\leq q <l\}\\
&B = \{\{p,q\} : j\leq p <k, l\leq q <i \}.
\end{align}
These are pairs of chords which ``cross completely'', namely every
chord in $A$ intersects every chord in $B$ and vice versa.  
Let $u_A=\Pi_{a\in A} u_a$ and $u_B=\Pi_{b\in B} u_b$.  Then,
\begin{equation}\label{K} u_A+u_B =1,\end{equation}
for any two sets of completely crossing chords.

Let $J$ be the ideal generated by the relations \eqref{K}.  It is then
shown that
$$\Mn^\delta = \Spec (\Z[u_{ij}]/J).$$

\end{proof}

In order to extend F. Brown's theorem to more general partial
compactifications, we use the following classical algebro-geometric
construction.
\begin{prop}\label{affinediv}
Let $X$ be an affine variety, $X= {\mathrm{Spec}} (R)$,  and let $D$
be a Cartier divisor on $X$. Then
$X\setminus D$ is an affine variety.
\end{prop}

\begin{proof}

A Cartier divisor, $D$, is associated to a line bundle ${\cal L}(D)$
which is an invertible sheaf
over $X$.  Since $X$ is affine, every invertible sheaf is
ample, so there exists an $n>0$ such that ${\cal L}(nD) = ({\cal
  L}(D))^{\otimes n}$ is very ample.  
Then, $({\cal
  L}(D))^{\otimes n}$ is given by a section $f \in \Gamma(X, {\cal
  O}_X(nD))$ where $f$ vanishes exactly on $D$ and is non-zero on
  $X\setminus D$. 
We have then that $X\setminus 
  D={\mathrm{Spec}}(R[\frac{1}{f}])$ is affine.
\end{proof}

From definitions \ref{m0ngamma} and \ref{defdiv}, the boundary divisors, $D_\gamma$,
are Weil divisors.  Since $\M_{0,n}$ is smooth, Weil and Cartier divisors coincide, and hence we may consider $D_\gamma$ as a Cartier divisor.

\begin{cor}\label{affinepart}
The partial compactifications, $\Mn^\gamma$, 
are affine varieties for the following sets of boundary divisors $\gamma\subset
D$:
\begin{enumerate}

\item Any $\gamma$ containing only one boundary divisor, $|\gamma|=1$,
\item Any $\gamma$ containing any 2 boundary divisors,
\item Any $\gamma$ containing 3 boundary divisors such that if two
  divisors intersect (as chords), then the third divisor is an intersection-divisor of the two.
\end{enumerate}
\end{cor}

\begin{proof}

All of these
sets, $\gamma$, contain divisors which contain a face on the boundary of an
associahedron in $\M_{0,n}$, so we apply proposition \ref{affinediv}.  

For part (1), let $\gamma = \{d_K \}$
where $K=\{z_{i_1},...,z_{i_k}\}$ for $k<n-1$.  Then $d_K$ contains the face of the
boundary of any associahedron enumerated by $\Delta=(z_{i_1},...,z_{i_k}, ... )$.  Let
$\delta$ denote the set of boundary divisors which contain a face of $\Delta$ as in
\eqref{deltadivs}, so 
$\Mn^\delta$ is affine.  By successive removal of all the divisors,
we recursively obtain affine varieties, the
final one being $\Mn^\gamma$.

For part
(2), let $\gamma = \{ d_P, d_Q\}$ where $P=\{z_{i_1},...,z_{i_p}\}$,
$Q=\{z_{j_1},...,z_{j_q}\}$ and without loss of generality $P\cap Q =
\{z_{i_{k+1}}, ..., z_{i_p} \} =\{z_{j_1},...,z_{j_{p-k}} \}$ ($P\cap
Q$ may be empty).  Then
$d_P, d_Q$ are divisors containing the face of the boundary of any
associahedron enumerated by $$\Delta =
(z_{i_1},...,z_{i_p},z_{j_{p-k+1}},...,z_{j_q},...).$$
As in part (1), by recursive removal of divisors from $\Mn^\delta$, we
have that $\Mn^\gamma$ is affine.

Finally for part (3),  assume first
that $\gamma$ contains any two divisors which
intersect (as chords) and a third which is an intersection-divisor of the two.  This means that
we can find sets, $P$ and $Q$ as in part (2) such that $\gamma =\{d_P,
d_Q, d_{P\cup Q}\}$.  Then all of these divisors are
on the boundary of the associahedron,
$(z_{i_1},...,z_{i_p},z_{j_{p-k+1}},...,z_{j_q},...)$.
By applying the proposition, we have that $\Mn^\gamma$ is affine. 
If $\gamma$ contains three divisors defined by $R$, $P$, $Q$ which are disjoint subjets of $Z$, we can construct $d_P,
d_Q$ and $d_R$ as in part (2).

All of the varieties, $\Mn^\gamma$ for $\gamma$ from the cases (1)-(3), are
therefore affine.

\end{proof}

The previous corollary can be extended to many other
partial compactifications, but we only treat these three in detail here.

\section{Explicit bases for $H^{n-3}(\Mn^\gamma)$}

In this section, I generalize a result of the previous chapter,
namely I use the methods introduced there to calculate the
top dimensional cohomology of the subspaces $\Mod_{0,n}^\gamma \subset
\M_{0,n}$ for certain small subsets $\gamma\subset \delta$, where $\delta$
denotes as usual the union of the divisors which contain the face of the boundary
of an associahedron.
We let $\ell=n-3$ and $Z=\{z_1,...,z_n\}=\{0,1,\infty,
t_1,...,t_\ell\}$\label{ellzdef}\index{$\ell$}\index{$Z$}.

Let us recall some definitions and results from chapter 3  
that will be useful
throughout this and subsequent sections.

\begin{defn}\label{Lynbas4}\cite{Re}\index{Lyndon basis}
Let $\{x_1,...x_n\}$ be a set of non-commutative variables with a
lexicographic ordering and
let $P_n$ be the $\Q$-vector space 
generated by
monomials of degree $n$ such that every variable appears exactly
once.  The
Lyndon basis for $P_n$ is given by the set $\{A_1\sha \cdots \sha A_k\}$ where
the $A_i$ form a partition of the variables and the first letter of every
$A_i$ is the smallest letter appearing in $A_i$ for the imposed
lexicographic ordering.  We say that $A_1\sha \cdots \sha A_k$ is
a Lyndon shuffle\index{Lyndon shuffle} of degree $k$.
\end{defn}

The Lyndon basis is an alternative basis to the standard basis of
permutations of the $n$ variables.  There are $(n-1)!$ degree 1 Lyndon
elements, since these are all monomials which start with the smallest
letter.  
Let $I_n\subset P_{n}$ denote vector subspace of Lyndon shuffles of degree $\geq 2$, whose dimension is
$n!-(n-1)!=(n-1)(n-1)!$.

\begin{defns}
Given a divisor, $d_K$, $K=\{t_{i_1},...,
t_{i_r}\}$, we define ${\cal{P}}_{d_K}$ to be $\Q$-vector space of polygons
with sides
decorated by the marked points in $K$.   The
subspace $I_{d_K}\subset {\cal{P}}_{d_K}$ is generated by shuffle sums with
respect to one point (as in definition \ref{PnIn}).
\end{defns}

We may often denote the vector spaces, ${\mathcal{P}}_{d_{K}}$ and $I_{d_{K}}$ simply by ${\mathcal{P}}_K$ and $I_K$.  We will often also note ${\mathcal{P}}_{d_{K\cup\{e\}}}$ and $I_{d_{K\cup \{e\}}}$ by ${\mathcal{P}}_{d_K \cup \{e\}}$ and by $I_{d_K\cup \{e\}}$ Note that $I_n$, where $n$ is an integer, is different from $I_K$, where $K$ is a set. 

Let $\pi$\label{polyformmap}\index{$\pi$} be the map that sends a polygon to its
associated cell form. 
The $\Res_d$ map sends a cell form to its residue along a divisor $d$
while the $\Res_d^p$ map sends a polygon to the tensor product of the
polygons cut the chord, $e$, as in definitions \eqref{proofres} and
\eqref{RespmapDEFch3}.  
The $\Res_d$ and $\Res_d^p$ maps are related by the identity,
$$\Res_d(\pi (\omega^p)) =\pi(\Res_d^p(\omega^p)),$$ for any polygon
$\omega^p$. 

Recall corollary \ref{neatcoro} that identifies the kernel of the
residue map on a divisor $d$,
$$ \mathrm{ker}(\Res_d) = \pi^{-1}(I_{d\cup \{e\} }\otimes
  {\mathcal{P}}_{Z\setminus d \cup \{e\}}).$$

\begin{thm}\label{allcohoms} Let
$\gamma =
\{\gamma_1,...,\gamma_k\}$ be a set of boundary divisors in $\M_{0,n}$ such
that $\Mn^\gamma$ is affine.  Then, the $\Q$-vector space $H^\ell(\Mod_{0,n}^\gamma)$
  coincides with the intersection of vector
  spaces, $$ \bigcap_{i=1}^k \pi((\Res_{\gamma_i}^{p})^ {-1} (
I_{\gamma_i \cup \{e\}} \otimes {\cal P}_{Z\setminus \gamma_i \cup \{
  e\}})).$$

Furthermore, a basis for $H^\ell(\Mod_{0,n}^\gamma)$ can easily be
deduced from a Lyndon basis of the polygons in $I_{\gamma_i \cup
  \{d\}} \otimes {\cal P}_{Z\setminus \gamma_i \cup \{d\}}$ using
insertion forms.
\end{thm}

\begin{proof}
From theorem \ref{subspaceprop}, we have an injection $$H^\ell(\Mn^\gamma)
\hookrightarrow H^\ell(\Mn).$$  By theorem \ref{thm01cellsspan} a basis for
$H^\ell(\Mn)$ is given by 01-forms.

By applying proposition \ref{mngammabasis}, we obtain a basis for
$H^\ell(\Mn^\gamma)$ by taking the subspace of 01-forms which converge
on $\gamma$.
A form is convergent on $\Mn^\gamma$ if and only if it is convergent
on all of the divisors, $\gamma_i\in \gamma$, since by the hypothesis,
it is convergent on the interior, $\Mn$.

A cell form, $\omega$, is convergent on $\gamma_i$ if and only if its
residue on 
$\gamma_i$ is 0, in other words if and only if
\begin{equation}\label{kerres} \omega \in \ker(\Res_{\gamma_i}).\end{equation}

We rely on two important combinatorial properties of 01-cyclic structures.
Not only do 01-forms form a basis for the cohomology, but also 01-polygons
form a basis for the $\Q$ vector space which is freely generated by
01-cyclic structures.  Therefore, each $\ell$ form,
$\omega$, has a unique lifting, $\omega^p$ to a linear combination of
01-polygons.  By
propostion \ref{neatcoro}, the condition 
\eqref{kerres} can be restated as
\begin{equation}\label{convomega} \omega^p \in
  (\Res_{\gamma_i}^p)^{-1}(I_{\gamma_i \cup \{e\}} \otimes 
      {\cal P}_{Z\setminus \gamma_i \cup \{e\}}).\end{equation}  If a form is
      convergent on all $\gamma_i$ it must be in the intersection of
      the spaces spanned by the spaces \eqref{convomega}.
\end{proof}

In the examples that follow, we exploit this theorem and the methods
of chapter 3 of insertion forms to calculate bases of cohomologies
for some natural $\Mn^\gamma$.  Recall from definition \ref{lyndoninsertionshufflesdef} that an insertion form is a
cell form coming from a linear combination of polygons such that the
polygon residue map maps them to $I_d\otimes {\cal P}_{d'}$ for some
divisors $d$ and $d'$. 

\subsubsection{Case 1: $|\gamma|=1$}

Firstly, we treat the smallest and most natural case of a partial
  compactification, namely that obtained by removing all boundary
  divisors except one
  from $\M_{0,n}$.  It was shown in corollary \ref{affinepart} that if
  $|\gamma|=1$, 
  $\Mn^\gamma$ is an affine space.

Let $\gamma=\{ d_R\}$
 for $R = \{z_{i_1},...,z_{i_r}\}$ and let $\omega$ be a differential
 $\ell$-form 
written in the 01-basis, where $\ell=n-3$ as in chapter 3.  In writing 01-cell forms, it is useful to
 choose an appropriate equivalence class representative modulo
 $\PSL_2$.  So without loss of generality, we may assume that
 $R=\{t_{i_1},..., t_{i_r}\}$, where one of the $t_{i_j}$ may be
 $\infty$. 

From theorem \ref{allcohoms}, $\omega$
converges if and only if $$\Res_\gamma^p(\omega^p) \in I_{\gamma \cup
  \{e\}}\otimes {\cal 
  P}_{Z\setminus \gamma \cup \{e\}}.$$  The 01-gons that have 0
residue along this divisor are those that don't contain
the block $ t_{i_1},...,t_{i_r}$; let this set of 01-gons be
denoted ${\cal W}^p_{\gamma_0}$\label{tiredofmakinglabels}\index{${\cal W}^p_{\gamma_0},\ {\mathcal{W}}^p_{\gamma_{\sh}}$}.

To calculate the dimension of the cohomology, we
count the number of fixed structures containing this block.  There are
$(n-1-r)!$ such fixed structures and $r!$ ways of ordering the
elements in the block. So the number of 01-cell forms that map
identically to 0 by $\Res_\gamma^p$ is $(n-2)!-(n-r-1)!r!$.  The
projection from these polygons to 01-forms are
in the basis of
$H^\ell(\Mod_{0,n}^\gamma)$ along with the insertion forms.  

The insertion forms for the divisor, $\gamma$ are linear combinations
of 01-forms which map to  $I_{\gamma \cup \{d\}}\otimes {\cal
  P}_{\gamma \cup \{d\}}$ and don't map identically to 0.  These forms
are the images of $\pi$ of formal sums of $n$-gons,
$$P=[0,1, Z_1, R_1\sh R_2, Z_2],$$
where $Z_1 \cup Z_2= \{t_1,...,t_\ell\}\setminus
\{t_{i_1},...,t_{i_r}\} $ and
$R_1\cup R_2 =\{t_{i_1},...,t_{i_r}\}$.  Let ${\cal
  W}^p_{\gamma_{\sh}}$ be the set of such polygons.  The image of $\pi({\cal
  W}^p_{\gamma_{\sh}})$ forms a linearly independent set of
$\ell$-forms in the cohomology by an argument used in chapter 3, theorem \ref{linind}.
The 01-forms form a basis, therefore, we only need to worry about
linear dependance for any fixed $Z_1$ and $Z_2$.  But the Lyndon
shuffles form a basis for the polynomial algebra, so for any fixed
$Z_1$ and $Z_2$ the forms $[0,1,Z_1,R_1\sh R_2,Z_2]$ are linearly
independent.
 
To count
the dimension of polygons
that map to
$  I_{\gamma \cup \{e\}}\otimes {\cal
  P}_{\gamma \cup \{e\}}$ is simple.
There are $(r-1)!$ degree 1
Lyndon generators in the shuffle algebra, therefore there are $r!-
(r-1)! = (r-1)(r-1)!$
shuffle generators in $I_{\gamma\cup \{e\}}$.  There are $(n-r-1)!$ fixed
structures, so
we conclude that \begin{align}\dim(H^\ell(\Mod_{0,n}^\gamma)) & =
  (n-2)!-(n-r-1)!r! +  
(n-r-1)!(r-1)!(r-1)\\ 
&= (n-2)!-(n-r-1)!(r-1)! .\end{align}

A basis for
  $H^{\ell}(\Mn^\gamma)$ is given by $\pi({\cal W}^p_{\gamma_0}
  \sqcup {\cal W}^p_{\gamma_{\sh}})$.

\begin{exes}  
(1) $n=6$, $\gamma$ contains the boundary divisor 
corresponding to $t_1=t_2$.

Then $r=2$ so we conclude that the dimension of
$H^\ell(\Mod_{0,n}^\gamma)$ is 18.  A basis for the cohomology is
given by the 12 $01$-forms, $[0,1,\{t_1,t_2,t_3,\infty\}]$
such that $t_1$ is not next to $t_2$, together with the 6 shuffle sums
\begin{align*}
&[0,1,t_1\sh t_2,t_3,\infty], [0,1,t_1\sh t_2,\infty,t_3],
[0,1,t_3,t_1\sh t_2,\infty],\\ &[0,1,\infty,t_1\sh t_2,t_3],
[0,1,t_3,\infty,t_1\sh t_2], [0,1,\infty,t_3,t_1\sh t_2].\end{align*}

(2) $n=6$, $r=3$, $\gamma$ consists of the boundary divisor
corresponding to $t_1=t_2=t_3$.  

The dimension is 20, and the basis
elements are given by the 6 forms,
$[0,1,t_i,\infty,t_j,t_k]$, the
6 forms $[0,1,t_i,t_j,\infty,t_k]$, the 4 Lyndon shuffles,
$$[0,1,t_1\sh (t_2,t_3),\infty], [0,1,(t_1,t_2)\sh t_3,\infty],
[0,1,(t_1,t_3)\sh t_2,\infty], [0,1,t_1\sh t_2\sh t_3,\infty],$$
and finally the 4 Lyndon shuffles,
$$[0,1,\infty,t_1\sh (t_2,t_3)], [0,1,\infty,(t_1,t_2)\sh t_3],
[0,1,\infty,(t_1,t_3)\sh t_2], [0,1,\infty,t_1\sh t_2\sh t_3].$$

\end{exes}

\subsubsection{Case 2: $|\gamma|=2$ and the divisors are disjoint}
In this case, we are considering two divisors that do not cross as chords of any polygon.  
Let these divisors be given by the equalities of the
marked points in the sets $R=\{z_{i_1},...,z_{i_r}\},
S=\{z_{j_s},...,z_{j_s}\}$, $R \cap S =\emptyset.$
(Recall that by corollary \ref{affinepart} we know that $\Mn^\gamma$
is affine.)

In this case a basis for the cohomology is given by sets of 01-forms
whose associated polygon either maps identically to zero or to
$I_{\gamma_i \cup \{e\}} \otimes {\cal P}_{Z\setminus \gamma_i \cup
  \{e\}}$ for the corresponding $\Res^p_{\gamma_i}$ maps, $i=1,2$.
As in case 1, the forms mapping identically to zero are all forms
whose associated polygon contains no consecutive block of $R$ or of
$S$.  The other forms are insertions of Lyndon shuffles of degree two
or higher of $R$ (resp. $S$, resp. both) into 01-forms on $Z\setminus
R$ (resp. $Z\setminus S$, resp. $Z\setminus (R\cup S))$.

Here, we count the
dimension and give a small example of an explicit basis.  In the
following formula we count the dimension of $\Mn^\gamma$ by 
methods similar to case 1.  The first line counts the 01-polygons
whose polygon residue is identically 0 for both
$d_R$ and $d_S$ (of lengths $r$ and $s$), where the last term
counts their overlap.  The
second (resp. third) line counts the insertions that land in $I_{R\cup
  \{d\}}$ (resp. $I_{S\cup \{d\}}$) for
$\Res_R^p$ (resp. $\Res_S^p$) and map to 0 for $\Res_S^p$
(resp. $\Res_R^p$). The last line counts the number
of terms that land in $I_{R\cup \{d\}}$ and $I_{S\cup \{d\}}$ for the
respective residue maps.
\begin{align}
&(n-2)! -(n-r-1)!r! -(n-s-1)!s! +(n-r-s)!r!s! \\ &+ (n-r-1)!(r-1)!(r-1)
-(n-r-s)!(r-1)!(r-1)s! \\ &+ (n-s-1)!(s-1)!(s-1)- (n-r-s)!(s-1)!(s-1)r!
\\ &+(n-r-s)!(r-1)!(r-1)(s-1)!(s-1).
\end{align}
\begin{ex} (1) Let $n=6$, $R=\{t_1,t_2\}$ and
  $S=\{t_3,\infty\}$. The dimension of the cohomology is 14 by the
  formula. There 
  are 8 01-cyclic structures such that $t_1$ is not next to $t_2$ and
  $t_3$ is not next to $\infty$,
\begin{align*}
& [0,1,t_1,t_3,t_2,\infty], [0,1,t_1,\infty,t_2,t_3],
  [0,1,t_2,t_3,t_1,\infty], [0,1,t_2,\infty,t_1,t_3], \\
& [0,1,t_3,t_1,\infty,t_2], [0,1,t_3,t_2,\infty,t_1],
  [0,1,\infty,t_1,t_3,t_3], [0,1,\infty,t_2,t_3,t_1].
\end{align*}
Then we add the 6 insertion elements to form the basis, 
\begin{align*}
&[0,1,\infty , t_1\sh
  t_2, t_3], [0,1,t_3, t_1\sh t_2,\infty], [0,1,t_1, t_3\sh
  \infty, t_2], [0,1,t_2, t_3\sh\infty,t_1] \\
&[0,1,t_3\sh \infty, t_1\sh t_2], [0,1,t_1\sh t_2, t_3\sh \infty].
\end{align*}

\end{ex}

\subsubsection{Case 3: $| \gamma| =3$ and contains two divisors that
  intersect as chords and their intersection-divisor}

Let $d_1$ and $d_2$ be any divisors that intersect (as chords) as in definition
\ref{intersectingdivisors} and consider them now as chords of a
polygon.  A
chord between adjacent enpoints of $d_1$ and $d_2$ 
may represent an intersection-divisor if it cuts the polygon
into two sections, each with at least two edges.  For $n\geq 5$ and
for any $d_1,d_2$ that intersect (as chords),  there exists
at least one well defined intersection-divisor, since the four
possible intersection chords form partitions of the edges of the
polygon.  Therefore we can can find sets $R, S\subset Z$
such that $d_1=d_R$, $d_2=d_S$, and $3\leq |R\cup S|\leq n-2$ so $d_{R\cup
  S}$ is a well defined intersection-divisor.  Let $R$ and $S$ be such
sets and let $\gamma =\{d_R, d_S, d_{R\cup S}\}$.

The ideas used in the description of the
cohomology of $\Mn^\gamma$ are similar to the ones used in chapter 3
for finding
the cohomology of $\Mod_{0,n}^\delta$.  In fact, they provide a
sort of base case for studying the origin of insertion forms, since
this cohomology space consists of forms which converge on many
divisors at the same time, some of which overlap.  The method of
constructing this space consists simply of finding elements of the
vector space of polygons decorated with the marked points in
$\Mod_{0,n}$ and categorizing the polygons according to their image by
the residue maps along the divisors in $\gamma$.  We construct vector
spaces of polygons that map to 0 or to
$I\otimes \cal{P}$ for the $\Res^p$ maps.  According to corollary
\ref{neatcoro}, this classification gives all of the differential
forms convergent on the divisors in $\gamma$ and on $\Mod_{0,n}$.

To construct the combinatorial polygon sets that describe the
cohomology, we consider the subsets of marked points that define the
boundary components we are looking for.

 Let $R\subset Z$ be the set $\{ z_{i_1},..., z_{i_r}\}$
and let $S\subset Z$ be $ \{z_{j_1}, ...,
z_{j_s} \}$. Since the intersection of $R$ and $S$ is supposed to be
non-empty, we may assume that $z_{i_1}=z_{j_1},\ ...,\ 
z_{i_k}=z_{j_k}$.

Let the lexicographic order on $R$ be $z_{i_{k+1}}< \cdots <
z_{i_r}< z_{i_1}<z_{i_2}<\cdots <z_{i_k}$, and on $S$ be
$z_{j_1}< \cdots < z_{j_s}$.  In this way, we have that the
elements of $R$ are less than those of $S$.

For any set, $P$, with a given lexicographic order, define
${\cal L}^i(P)$\label{lyndononepointDEF}\index{${\cal L}^i(P)$}
to be the set of Lyndon shuffles on $P$ of degree $i$
for the shuffle product.

In the following paragraphs, we construct the
sets, ${\cal W }_0, {\cal{W}}_{d_R}, {\cal{W}}_{d_S},
  {\cal{W}}_{d_{R\cup S}}, {\cal{W}}_{d_{R\cup S},S}$\label{defdiv3},\index{${\cal W }_0, {\cal{W}}_{d_R}, {\cal{W}}_{d_S},
  {\cal{W}}_{d_{R\cup S}}, {\cal{W}}_{d_{R\cup S},S}$} that describe the cohomology of $\Mn^\gamma$.

Let ${\cal W}_0^p$ be the set of 01 $n$-gons, $\{ \omega^p\}$, such that
$\Res_{d_R}^p(\omega^p)=0, \Res^p_{d_S}(\omega^p)=0$ and $\Res^p_{d_{R\cup
  S}}(\omega^p)=0$, in other words all generating 01 polygons with no
chord, $e$, 
that cuts $\omega$ into a polygon in ${\cal P}_{K\cup \{e\}} \otimes
{\cal P}_{\overline{K}\cup \{e\}}$
where $K=R$, $S$ or $R\cup S$ and $\overline{K}=\{z_1,...,z_n\} \setminus K$.

The number of elements in ${\cal W}_0^p$ is given by the following
 formula,
\begin{align}
|{\cal W}_0^p | & = (n-2)! - \Bigl( (n-1-s-r+k)!(r+s-k)! \\
& \qquad +r!\bigl( (n-1-r)! - (n-1-s-r+k)!(s-k+1)!\bigr) \\
& \qquad + s!\bigl( (n-1-s)! - (n-1-s-r+k)!(r-k+1)!\bigr) \Bigr) .
\end{align}
The first term subtracted off counts all of the elements which do not map to 0
for $\Res_{d_{R\cup S}}^p$, i.e. those polygons
which can be cut by a chord in such a way that one side contains only elements
labelled by $R\cup S$.  The second (resp. third) term counts the polygons which
do not map to 0 for the $\Res_{d_R}^p$ (resp. $\Res_{d_S}^p$) map and
subtracts off the 
intersection with the first term.

\begin{ex}\label{gamma3ex}  Let $n=6$, $R=\{t_1, t_2\},\
  S=\{t_2,t_3\}$.  Then ${\cal W}_0^p$ contains the 4 cell forms,
$$[0,1,t_1,t_3,\infty,t_2], [0,1,t_3,t_1, \infty,t_2],
  [0,1,t_2,\infty,t_1,t_3], [0,1,t_2,\infty,t_3,t_1].$$
\end{ex} 

Let ${\cal W}_{d_R}^p$ be the set of 01 $n$-gons,
$\{ \omega^p \}$, such that 
\begin{align*} \Res^p_{d_r}(\omega^p)& \in I_{R\cup\{e\}}
\otimes {\cal P}_{Z\setminus R\cup \{e\}} \\
& \neq 0,\end{align*}
and such that
$\Res^p_{d_S}(\omega^p)=0$ and $\Res^p_{d_{R\cup S}}(\omega^p)=0$.
This set contains all shuffle sums of 
elements of ${\cal L}^i(R),\ i\geq 2$, inserted into polygons decorated
by $\{z_1,...,z_n,e\}\setminus \{z_{i_1},...,z_{i_r}\}$ and such that
the elements of $S\cup\{e\} \setminus R$ are never consecutive.  The
  cardinality of ${\cal W}_{d_R}^p$ is
$(r-1)!(r-1)((n-r-1)!-(s-k+1)!(n-r-s+k-1)!)$.

Let ${\cal W}_{d_S}^p$
be defined similarly, so it has cardinality is
$(s-1)!(s-1)((n-s-1)!-(r-k+1)!(n-r-s+k-1)!)$.

\begin{ex}  Let $n$, $R$ and $S$ be as in example \ref{gamma3ex}.
  Then ${\cal W}_{d_R}^p$ is the set containing the two elements,
  $$[0,1,(t_1\sh t_2), \infty,t_3], [0,1,t_3,\infty, (t_1\sh t_2)],$$
and ${\cal W}_{d_S}^p$ contains the two elements,
$$[0,1,(t_3\sh t_2), \infty,t_1], [0,1,t_1,\infty, (t_3\sh t_2)].$$
\end{ex}

We define ${\cal W}_{d_{R\cup S}}^p$ 
to be the set of insertions of elements $ (A_1\sh \cdots \sh A_J)\in {\cal
  L}^i(R\cup S),\ i\geq
2$, such that the shuffle factors, $A_k$, do not contain blocks that
equal $R$ or $S$ into the fixed 01 structures on $Z\cup \{e\}\setminus
(R\cup S)$ in the place of $e$.  One could say
that these are the
standard type of shuffle elements that only have one level of
insertion.  These elements map to $I_{R\cup S\cup \{e\}} \otimes {\cal
  P}_{Z\cup \{e\}\setminus(R\cup S)}$ for $\Res_{d_{R\cup S}}^p$ and
can map either to 0 or to $I\otimes {\cal P}$ for $\Res_{d_R}^p,
\Res_{d_S}^p$. 
There are $(r+s-k-1)!(r+s-k-1)$ degree two or higher Lyndon shuffle
elements on $R\cup S$. Any shuffle term with a consecutive block of
$R$ begins with
$z_{i_{k+1}}$ by the lexicographic order we imposed on 
$R\cup S$, so the
number of shuffles that have such a consecutive block in a factor is
$(r-1)!(s-k)!(s-k)$.  The number of shuffles where $S$ is a
consecutive block in one of the factors is $((r-k)!(r-k)-(r-k)!)s! +
(r-k)!(s-1)!$.  The first term in this expression counts the
number of shuffle sums where $S$ is in a successive block and in which
there are other terms from $R\setminus S$ in that factor.  By the
lexicographic order, the elements of $S$ are never the first letters
of those factors, therefore for each fixed structure, we have $s!$ such
shuffles.  The second term counts the number of shuffles in which 
the letters in $S$ form their own word.  By the lexicographic
ordering, we have $(r-k)!(s-1)!$ such shuffles.  Then we may insert all such
elements into the $(n-2-(r+s-k)+1)!$ fixed 01-structures on $Z\cup
\{e\} \setminus (R\cup S)$.  The cardinality of 
${\cal W}_{d_{R\cup S}}^p$ is therefore
\begin{align*}& (n-1-r-s+k)!\bigl((r+s-k-1)!(r+s-k-1)-(r-1)!(s-k)!(s-k)
\\ & \qquad -(((r-k)!(r-k)-(r-k)!)s! +
(r-k)!(s-1)!)\bigr).\end{align*}

\begin{ex} Let $n$, $R$ and $S$ be as in example \ref{gamma3ex}.  Then
  there are the four elements in ${\cal W}_{d_{R\cup S}}^p$,
$$[0,1,\infty,(t_1\sh t_2\sh t_3)], [0,1,\infty, (t_1 t_3\sh t_2)],
  [0,1,(t_1\sh t_2\sh t_3),\infty], [0,1,(t_1 t_3\sh t_2),\infty].$$
\end{ex}

The final set in the combinatorial description of this cohomology
follows closely the methods of chapter 3, where we
construct a second level of insertion sets by inserting
elements in ${\cal L}^{\geq 2}(S)$ into the place of
$e$ in those elements of ${\cal L}^{\geq 2}(R\cup \{e\}\setminus S)$
in which a consecutive block of $e$ is not equal to a shuffle factor.  In order
to respect the given 
lexicographic order on $R\cup S$, $e$ is greater than all elements
in $R\setminus S$.  Then, to obtain an element in the cohomology,
insert these into the place of $e$ in 01-gons decorated by $Z\cup
\{e\} \setminus (R\cup S)$.
Let ${\cal W}^p_{d_{R\cup S}, S}$ be the
set obtained by this insertion process.  The cardinality of ${\cal
  W}^p_{d_{R\cup S}, S}$ is $(n-1-r-s+k)!(s-1)!(s-1)\bigl( (r-k)!(r-k)
- (r-k)!\bigr)$, where the first factor is the number of 01 fixed
structures on $Z\cup \{e\} \setminus (R\cup S)$, the second, the
number of degree two Lyndon shuffles on $S$ and the third, the number
of fixed structures which are Lyndon shuffles on $R\cup \{e\}
\setminus S$ such that $e$ is not its own shuffle factor.

\begin{exes}  In the extended example, there are no level two insertion
  elements in ${\cal W}^p_{d_{R\cup S}, S}$.

Take $n=7$, $R=\{t_1,t_2,t_3\}$, $S=\{t_3,t_4\}$.  Then there are four
level two insertion elements in ${\cal W}^p_{d_{R\cup S}, S}$ given by
\begin{align*} & [0,1,\infty, ((t_1,(t_3\sh t_4))\sh t_2)],
  [0,1,\infty, (t_1\sh 
(t_2,(t_3\sh t_4)))], \\ & [0,1,((t_1,(t_3\sh t_4))\sh t_2),\infty],
[0,1,(t_1\sh (t_2,(t_3\sh t_4))),\infty].\end{align*}  
\end{exes}

In the above examples, we constructed sets of polygons.  To pass to elements
of the cohomology, let ${\cal W}_i$\label{cellformstuff}\index{${\cal W }_0^p, {\cal{W}}_{d_R}^p, {\cal{W}}_{d_S}^p,
  {\cal{W}}_{d_{R\cup S}}^p, {\cal{W}}_{d_{R\cup S},S}^p$} be the image
of ${\cal W}_i^p$ for 
the map from
polygons to cell forms for $i=0, d_R, d_S, d_{R\cup S}$ and
$d_{R\cup S}, S$.

\begin{prop}
The cell forms in
the disjoint union,
\begin{equation}\label{Wgamma} {\cal W}_\gamma = {\cal W}_{0}\cup {\cal W}_{d_R} \cup {\cal W}_{d_S} \cup
{\cal W}_{d_{R\cup S}} \cup {\cal W}_{d_{R\cup S}, S}\end{equation}
form a basis for the $\Q$ vector space of differential $\ell$ forms,
$H^\ell(\Mod_{0,n}^\gamma)$ and its dimension is given by
\begin{align*}
&(n-2)! + (n-1-(r+s-k))! \bigl( (r-1)!(s-k)! + (s-1)!(r-k)!
-(r+s-k-1)! \bigr)\\
& \qquad -(n-1-r)!(r-1)! -(n-1-s)!(s-1)!.
\end{align*}
\end{prop}

\begin{proof}
Because
01-polygons are in bijection with the 01-cell form basis for
$H^\ell(\Mn)$ by theorem \ref{thm01cellsspan}, we may and will prove
this theorem
combinatorially on the polygons ${\cal W}^p_\gamma$.  

 First, we justify that the given sets are indeed disjoint.  We can
write the elements of the sets as sums of 01 $n$-gons of the form,
$$\omega = \sum_\sigma [0,1,\sigma(z_1,...,z_{n-2})]$$
where $\sigma \in \Sym_n$.
The elements of ${\cal W}_0^p$ are sums of length one and contain no
permutations of the marked
points where the elements of $R$, $S$ or $R\cup S$ are consecutive.
The other sets contain sums of 01 $n$-gons whose terms
all contain consecutive elements from at least one of $R$, $S$ or
$R\cup S$, so ${\cal W}_0^p$ is disjoint from the union of the other
four sets. Similarly, by construction, ${\cal W}_{d_R}^p$
(resp. ${\cal W}_{d_S}^p$) contains only
sums of 01 $n$-gons whose terms contain consecutive elements of $R$
(resp. $S$),
but neither consecutive elements from $S$ (resp. $R$) nor $R\cup S$.
Therefore, ${\cal W}_{d_R}^p$ and ${\cal W}_{d_S}^p$ are disjoint from the
other sets.  Finally, the sets ${\cal W}_{d_{R\cup S}}^p$ and
${\cal W}_{d_{R\cup S}, S}^p$ are disjoint from each other, since they
are sums whose terms are
insertions of Lyndon shuffles on $r+s-k$ elements into 01
$n$-gons. Since Lyndon shuffles form a basis for the shuffle algebra
on $n$ distinct letters, they are distinct sums.  The Lyndon shuffles in
${\cal W}_{d_{R\cup S}}^p$ do not contain shuffle factors with
consecutive elements in either $R$ or $S$, but those in
${\cal W}_{d_{R\cup S}, S}^p$ contain only sums of Lyndon shuffles whose
shuffle factors contain consecutive sequences in $S$.  Therefore
${\cal W}_{d_{R\cup S}}^p$ and ${\cal W}_{d_{R\cup S}, S}^p$ are also
disjoint sets.  The map from 01-polygons to cell forms is injective,
so the corresponding sets of cell forms, ${\cal W}_i$, are
disjoint as well.

By theorem \ref{allcohoms}, we need to justify that ${\cal
  W}_\gamma^p$ is a basis for
\begin{equation}\label{intersect3}
\bigcap_{K= R, S, R\cup S} \pi\bigl( (\Res^p_{d_K})^{-1}(I_K \otimes
  {\cal P}_{Z\cup \{e\} \setminus K})\bigr). 
\end{equation}

Recall the definition of a framing for a divisor,
$d_K$ on page \pageref{framingDEF}.  If
$\Res_K^p(\omega^p) \in {\cal P}_K\otimes f$, such that $f$ is
a 01-polygon, we define the framing
of $\omega^p$ with respect to
$\overline{K}$ to be the right hand factor, $f$,  of its
image.  Depending on the context, the framing may also be an element of
a basis for ${\cal P}_{\overline{K}}$ which is the sum of 01-gons.  Since the
right hand factors are assumed to be basis elements for ${\cal
  P}_{\overline{K}}$,
$\Res_K^p(\omega^p)=0$ if and only if $\Res_K^p(\omega^p)=0$ for each
framing on the letters $Z\cup\{e\} \setminus K$.

Note: Polygons decorated by a set, $K\cup \{e\}$, are
  isomorphic to the noncommutative polynomial algebra in $K$, where
  $e$ can be just
  considered as an ``marker'', i.e. the side that follows it
in the clockwise ordering corresponds to the first letter of the
  corresponding monomial.

\begin{claim}\label{linind3}
The sets, \eqref{Wgamma}, are linearly independent.
\end{claim}
\begin{proof}
To prove linear independence, we recall theorem
\ref{thm01cellsspan}.  Since the 01-forms form a basis for $H^\ell(\Mn)$,
and ${\cal W}_0$ contains 01-forms, we only need to show linear
independance for the other four sets.  Let $\omega^p$
be a linear combination of forms from ${\cal W}^p_\gamma \setminus {\cal
  W}^p_0$.  The vector space of 01-gons can be written as the direct sum,
$V_0 \oplus V_{R\cup S} \oplus V_S \oplus V_R$, where $V_0$ is
generated by 01-gons with no consecutive blocks of $R, S$ or $R\cup
S$; $V_R$ (resp. $V_S$) is the vector space generated by 01-gons
with consecutive blocks of $R$ (resp. $S$), but no consecutive blocks
of $R\cup S$ and no consectutive blocks of $S$ (resp. $R$); $V_{R\cup
  S}$ is generated by 01-forms with consecutive blocks of $R\cup S$.

By the hypothesis, $\omega^p \in V_{R\cup S} \oplus V_S \oplus V_R$, so
write
$$ \omega^p = \sum a^R_i v^R_i + \sum a^S_i v^S_i + \sum a^{R,S}_i
v^{R,S}_i,$$
where the $v^K_i\in {\cal W}_K^p, K=d_R, d_S$ and the last sum has terms
from ${\cal W}_{d_{R\cup S}}^p$ and ${\cal W}^p_{d_{R\cup S},S}$.  We only need
to show linear independence for each individual sum, so we assume that
$\sum a^R_i v^R_i = \sum a^S_i v^S_i = \sum a^{R,S}_i
v^{R,S}_i=0$ and show that this implies that all of the coefficients,
$a_{\bullet}^{\bullet}$, must be 0.  

Each $v^R_i$ is a sum
of 01-gons with a framing (the same in each term) by $Z\cup \{e\}\setminus
R$.  Then we can rewrite the first sum separating out terms with the same
framing $f_i$, $$\sum_{f_i} \sum_j a^R_{f_i,j} v^R_{f_i, j}=0,$$
so that
\begin{align*}
\Res_{d_R}^p(\sum_{f_i} \sum_j a^R_{f_i,j} v^R_{f_i, j}) & =
\sum_{f_i} \sum_j a^R_{f_i,j} \Res_{d_R}^p(v^R_{f_i, j}) \\
& = \sum_{f_i} \sum_j a^R_{f_i,j} (P^R_{f_i,j} \otimes f_{i,j}) \\
& =0,
\end{align*}
where $P^R\in {\cal P}_{R\cup \{e\}}, f_{i,j}\in {\cal
  P}_{Z\cup \{e\} \setminus R}$.  The fixed structures, $f_{i,j}$, are
  01-polygons and are linearly independent, thus the sum only equals 0
  if $\sum_j a^R_{f_i,j} P^R_{f_i,j}=0$ for each $f_i$.  But for
  each $j$, $P^R_{f_i,j}$ is an element of the Lyndon basis for the
  polynomial algebra in $R$.
Hence these are also linearly 
  independent and
  $\sum a^R_i v^R_i =0$ if and only if all $a_i^R$ are zero, proving
  the claim for ${\cal W}_{d_R}^p$.

 The proof for ${\cal W}_{d_S}^p$ is identical.

Next, we look at the sum, $\sum a^{R,S}_i
v^{R,S}_i$ and as in the previous two cases, we can write this term as 
$\sum_{f_i} \sum_j a^{R,S}_{f_i, j} v^{R,S}_{f_i,j}$, where here the
$f_i$ are fixed structures on $Z\cup \{e\} \setminus (R\cup S)$.  Then
we only need to prove linear independence for each $f_i$, $\sum_j
a^{R,S}_{f_i, j} P^{R,S}_{f_i,j}$.  If the polynomials $P^{R,S}$ come
from ${\cal W}^p_{d_{R\cup S}}$, they are linearly independent for the
reasons above.  
So we only need to look at level two insertion elements,
$P^{R,S}_{f_i,j} \in {\cal W}_{d_{R\cup S}, S}^p$.  Without loss of
generality, assume that all of the terms, $P^{R,S}_{f_i,j}$ are in
${\cal W}_{d_{R\cup S}, S}^p$.  Now as above, we can break the sum of
the $P^{R,S}$ into fixed structures, $g_i$, on $R\setminus S$,
and apply the residue map,
\begin{align*}
\Res_{d_S}^p(\sum_{g_i}\sum_j
a^{R,S}_{f_i,g_i, j} P^{R,S}_{f_i,g_i,j}) & = \sum_{g_i}\sum_j
a^{R,S}_{f_i,g_i, j} \Res(P^{R,S}_{f_i,g_i,j})\\
&=\sum_{g_i}\sum_j
a^{R,S}_{f_i,g_i, j} P^{S}_{f_i,g_i,j} \otimes g_i\\
&=0.
\end{align*}
  By definition the $g_i$
are Lyndon shuffles in $R\setminus S$ and are therefore linearly
independent.  And we proceed as before: the $P^S$ are Lyndon shuffles
in $S$, so the sum is zero only if all $a^{R,S}_{f_i,j}$ are zero.
\end{proof}

\begin{claim}
The set ${\cal W}_\gamma$ spans the set of forms convergent on $\gamma$.
\end{claim}
\begin{proof}

From the previous claim, we may extend the set of linearly independent
elements in ${\cal W}_\gamma^p$ to a basis, ${\cal B}$, of 01-polygons.  
Then we show
that if an element written in this basis is in the intersection
\eqref{intersect3}, then the coefficient on the basis elements ${\cal
  B}\setminus {\cal W}_\gamma^p$ must be 0.  Therefore the set, $\pi({\cal
  W}_\gamma^p)={\cal W}_\gamma$, spans the space of convergent cell
forms on $\gamma$.

The standard basis of 01-gons is $B=\{[0,1, \sigma(Z\setminus
  \{0,1\})] ; \sigma \in \Sym_{n-2} \}$.  As in claim \ref{linind3},
  we may write the space of 
  01-gons as $$V=V_{R\cup S} \oplus V_0 \oplus V_S \oplus V_R.$$
  We now construct an alternative basis to
  $B$.  The polygons in ${\cal W}^p_0$ span $V_0$ since they are elements
  of $B$.  Furthermore, they are in the intersection of the preimage
  of the three residue maps by definition \ref{defdiv3}.  Let ${\cal
  W}^p_0={\cal B}_0$.

As
a basis for $V_S$, we take instead of permutations of $S$, the 
Lyndon basis for $S$ and insert into 
the framings
given by $Z\cup
\{e\}\setminus S$.  Let $V_S=V_{S^1} \oplus V_{S^{\geq 2}}$.  A basis
  for $V_{S^{\geq 2}}$ is given by ${\cal W}_{d_S}^p$ and that for
  $V_{S^1}$ is given by all insertions of degree 1 Lyndon elements in
  $S$.  As before, ${\cal W}^p_{d_S}$ is in the preimage of the three
  residue maps.  Let ${\cal
  W}^p_{d_S}={\cal B}_{S^{\geq 2}}$ and let the basis for
  $V_{S^1}={\cal B}_{S^1}$.

By the same argument, we construct ${\cal B}_{R^1}$ and ${\cal
    B}_{R^{\geq 2}}$ as bases for $V_{R^1}$ and $V_{R^{\geq 2}}$.

We take as a basis for $V_{R\cup S}$, insertions of Lyndon shuffles and
write $V_{R\cup S}$ as
the direct sum of the two vector spaces,
$V_{(R\cup S)^1} \oplus V_{(R\cup S)^{\geq 2}}$ with respective bases,
${\cal
  B}_{(R\cup S)^1}$ and ${\cal B}_{(R\cup S)^{\geq 2}}$ as previously.

We can write an alternative to basis for $V_{(R\cup S)^{\geq 2}}$ by
taking inserting Lyndon shuffles of $S$ into consecutive blocks of $S$
which appear in a shuffle factor, call this basis ${\cal B}_{R\cup S}'$.  

\begin{ex}\label{stupidex} Consider the subset of marked points in $\M_{0,9}$,
     $$Z=\{0,1,\infty, t_1, t_2, t_3, t_4, t_5,t_6\}$$ and let
     $R=\{t_1,t_2,t_3\},\ S= \{t_3,t_4,t_5\}$.  In the usual basis for
     ${\cal L}^{\geq 2}(R\cup S)$, we have the elements,
\begin{align*} B_1= \{ (t_1, &t_3,t_4,t_5\sh t_2), (t_1,t_3,t_5,t_4\sh
     t_2) (t_1,t_4,t_3,t_5\sh t_2), (t_1,t_4,t_5,t_3\sh
     t_2), \\ &(t_1,t_5,t_3,t_4\sh t_2), (t_1,t_5,t_3,t_4\sh
     t_2)\}.\end{align*}
We take the alternative basis in which the elements, 
\begin{align*} B_2= \{ (t_1, &t_3,t_4,t_5\sh t_2), (t_1,t_3,t_5,t_4\sh
     t_2) (t_1,(t_3\sh t_4,t_5)\sh t_2), (t_1,(t_3,t_4 \sh t_5)\sh
     t_2), \\ &(t_1, (t_3,t_5 \sh t_4) \sh t_2), (t_1,(t_3\sh t_4 \sh t_5)\sh
     t_2)\}\end{align*}
appear as insertions.  To construct an element of ${\cal B}_{R\cup
     S}'$, we insert
      elements of $B_2$ into $e$ in framings of 01-polygons on $\{0,1,\infty,
     t_6, e\}.$ 
\end{ex}

The subspace spanned by ${\cal B}_{R\cup S}'$ can be written as 
$$W_{R^1} \oplus W_{S^1} \oplus W_{S^{\geq 2}} \oplus W_{(R,S)^{\geq
    2}}.$$

The space $W_{R^1}$ is spanned by the elements which are
    insertions of shuffles in which $R$ appears as a block in one
    factor, likewise for $W_{S^1}$.  $W_{S^{\geq 2}}$ is spanned by
    insertions of shuffles in $S$ of degree $\geq 2$ into one factor
    as in example \ref{stupidex}.  And $W_{(R,S)^{\geq 2}}$ is spanned
    by those elements in which neither $S$ nor $R$ appears as a block
    in any shuffle factor of $R\cup S$.  The sets ${\cal W}_{d_{R\cup
    S},S}$ and ${\cal W}_{d_{R\cup
    S}}$ are subsets of $B_{R\cup S}'$ and are respectively bases for
    $W_{S^{\geq 2}}$ and $W_{(R,S)^{\geq 2}}$.  As before let ${\cal
    W}_{d_{R\cup
    S},S}= {\cal B}_{(R\cup S), S^{\geq 2}}$ and ${\cal W}_{d_{R\cup
    S}}= {\cal B}_{(R\cup S)^{\geq 2}}$.  We let ${\cal B}_{(R\cup
    S), R^1}$ and
    ${\cal B}_{(R\cup S), S^1}$ be the bases for $W_{R^1}$
    and $W_{S^1}$ 
    respectively.  Then ${\cal B}_{(R\cup S), S^{\geq 2}}$, ${\cal B}_{(R\cup
    S)^{\geq 2}}$, ${\cal B}_{(R\cup S), R^1}$ and
    ${\cal B}_{(R\cup S), S^1}$ form a partition of ${\cal B}_{R\cup
    S}'$

So now we have that \begin{align*}
{\cal B}= {\cal B}_0\cup & {\cal B}_{S^{\geq 2}} \cup {\cal
  B}_{S^1} \cup {\cal B}_{R^{\geq 2}} \cup {\cal
  B}_{R^1}\cup {\cal B}_{(R\cup
    S)^1} \cup {\cal B}_{(R\cup
    S)^{\geq 2}} \\ & \cup {\cal B}_{(R\cup S), S^{\geq 2}}  \cup {\cal
  B}_{(R\cup S), R^1} \cup {\cal B}_{(R\cup
  S), S^1}\end{align*} is a basis for the 01-gons where $${\cal W}_\gamma =
  {\cal B}_0\cup {\cal B}_{S^{\geq 2}} \cup {\cal B}_{R^{\geq 2}} \cup
  {\cal B}_{(R\cup S), S^{\geq 2}} \cup {\cal B}_{(R\cup
    S)^{\geq 2}}.$$  We can now justify that if $\omega^p$ is in
  the intersection, \eqref{intersect3}, then $\omega^p$ is in
  the space spanned by ${\cal W}_\gamma$.

The elements in the bases for the subspaces, $V_0, V_S, V_R$ and
$V_{R\cup S}$ all have unique framings for a well-chosen basis for
${\cal P}_{\overline K}$, $K=R, S, R\cup S$, namely the basis coming
from the construction of the ${\cal B}$ sets.
Let $\omega^p$ be in the intersection
\eqref{intersect3}, we can write $\omega^p$ in the basis ${\cal B}$ as
$$\omega^p = \omega_{000} + \omega_{001} +\omega_{010}+\omega_{011}+
\omega_{100}+ \omega_{101}+ \omega_{110}+\omega_{111},$$
where $\omega_{000}$ are terms that are in the kernel of all three
residue maps, $\omega_{001}$ is in the kernel of $\Res^p_{d_R}$ and
$\Res^p_{d_S}$, but $0\neq \Res^p_{d_{R\cup 
    S}} (\omega_{001}) \in I_{R\cup S}\otimes {\cal
  P}_{\overline{R\cup S}}$, and so on.
This decomposition is unique. 

First, we verify that the term $\omega_{000} \in {\cal W}_\gamma^p$.
Since $\Res^p_{D_{R\cup S}}(\omega_{000})=0$, the coefficient on the
elements ${\cal B}_K\ (K= (R\cup S)^\bullet,\bullet)$
must be 0, since for each framing, the elements of ${\cal B}_K$ form a
basis for their image in $I_{R\cup S}$.  Since
$\Res_{d_S}^p(\omega_{000})=0$ and  $\Res_{d_R}^p(\omega_{000})=0$,
then the coefficient on $\omega_{000}$ on ${\cal B}_K$, $K=R^\bullet,
S^\bullet$ is also 0 for the same reason.  Therefore $\omega_{000} \in
\langle {\cal B}_0\rangle$.

Similarly, $\omega_{010}\in \langle {\cal B}_{S^{\geq 2}}
\rangle$ and 
$\omega_{100}\in \langle {\cal B}_{R^{\geq 2}}\rangle$.

The term $\omega_{110}$ must be identically 0.  For the same reasons
as above, it must lie in the space, $\langle {\cal B}_{S^{\geq 2}} \cup {\cal
  B}_{S^1} \cup {\cal B}_{R^{\geq 2}} \cup {\cal
  B}_{R^1} \rangle$.  The framing for blocks of $R$ consists of
permutations of $Z\cup\{e\} \setminus R$ where the elements of
$S\cup \{e\} \setminus R$ are not in a consecutive block since these
elements are in the spaces generated by the ${\cal B}_{(R\cup
  S)^\bullet, \bullet}$ bases since $R\cap S$ is non-empty.  Therefore
if an element maps not to 0,
and to $I_{R}\cup {\cal P}_{\overline{R}}$ for $\Res_{d_{R}}^p$, it
must map to 0 for $\Res_{d_{S}}^p$ map.

Now, we look at the last four terms,
$\omega_{R,S}= \omega_{001}+\omega_{101}+\omega_{011}+ \omega_{111}$.
Since $$0\neq \Res_{d_{R\cup S}}^p(\omega_{R,S}) \in I_{R\cup
  S}\otimes {\cal P}_{\overline{R\cup S}},$$ then the coefficients on
$\omega_{R,S}$ on the elements of ${\cal B}_{(R\cup S)^1}$ are 0.

We have that $\omega_{111} + \omega_{101}$ maps by $\Res_{d_R}^p$ to 
${\cal P}_{R}\otimes {\cal P}_{\overline{R}}$ which we can write as
  $$({\cal L}^1(R) \otimes {\cal P}_{\overline{R}}) \oplus (I_R \otimes
  {\cal P}_{\overline{R}}).$$  The
  residue map,
\begin{align*}
\Res_{d_R}^p : \langle {\cal B}_{(R\cup S),R^1} \rangle & \oplus
    \langle {\cal B}_{(R\cup
    S)^{\geq 2}} \cup {\cal B}_{(R\cup S), S^{\geq 2}} \cup {\cal B}_{(R\cup
  S), S^1} \rangle \\ & \twoheadrightarrow ({\cal L}^1(R)  \otimes {\cal
    P}_{\overline{R}} )
    \oplus (I_R \otimes {\cal P}_{\overline{R}})
\end{align*} 
is equal to the direct sum of the residue maps,
\begin{align*} {\cal R}^1 \oplus {\cal R}^{\geq 2} : \bigl( \langle
    & {\cal B}_{(R\cup
    S),R^1}  \twoheadrightarrow  {\cal
    L}^1(R) \otimes 
     {\cal P}_{\overline{R}} \bigr) \oplus \\ & \bigl(\langle {\cal B}_{(R\cup
    S)^{\geq 2}} \cup {\cal B}_{(R\cup S), S^{\geq 2}} \cup {\cal B}_{(R\cup
  S), S^1}  \rangle \twoheadrightarrow I_R\otimes {\cal
    P}_{\overline{R}}\bigr). \end{align*}  We have such a
    decomposition of the
    residue map because the polygons in ${\cal B}_{(R\cup
    S),R^1}$ (for any framing) only have consecutive blocks in $R$ which are
    of degree 1, so they map to ${\cal L}^1(R)\otimes {\cal P}$.
    Furthermore, for any framing
    in which $R$ is not a consecutive block, it is a degree two or
    higher shuffle, and therefore maps to
    $I_R\otimes{\cal P}$.  Since we are searching for elements that
    map to $I_R$ as a left-hand factor, $\langle B_{(R\cup S), R^1}
    \rangle$ cannot be in the intersection \eqref{intersect3} and therefore 
   $\omega_{101} +\omega_{111} \in \langle {\cal B}_{(R\cup
    S)^{\geq 2}} \cup {\cal B}_{(R\cup S), S^{\geq 2}} \cup {\cal B}_{(R\cup
  S), S^1}  \rangle$.

We may repeat the same proof as above, substituting $S$ for $R$, which
    shows that $\langle B_{(R\cup S), S^1}
    \rangle$ cannot be in the intersection \eqref{intersect3} and therefore 
   $\omega_{011} +\omega_{111} \in \langle {\cal B}_{(R\cup
    S)^{\geq 2}} \cup {\cal B}_{(R\cup S), S^{\geq 2}}\rangle$.

We have now shown that if $\omega^p$ is in the intersection
\eqref{intersect3}, then $\omega^p$ is in the space spanned by ${\cal
  W}_\gamma^p$.  The map from 01-polygons to $H^\ell(\Mn)$ is
bijective, therefore by theorem \ref{allcohoms}, ${\cal
  W}_\gamma$ spans $H^\ell(\Mn^\gamma)$. 
\end{proof}

We have proven that ${\cal W}_\gamma$ is a set of 01-forms which are
linearly independent and span $H^\ell(\Mn^\gamma)$ and therefore form a basis.

\end{proof}

\section{The non-adjacent bases of $Pic(\M_{0,n})$}

The following result emerged from the search for
sets of divisors, $\gamma$,
that satisfy the criterion that
$\Mn^\gamma$ be affine.  If $D$ is an ample divisor, then $\M_{0,n}\setminus
D =\Mn^\gamma$ is an affine space.
Given some ``natural set'', $\gamma$ (we considered for example sets
$\gamma$ which are in the support of a multizeta form),
we searched for explicit
divisors having support equal to $\gamma^c$, the complement of
$\gamma$.  We then attempted to prove, using a methods of
A. Gibney and S. Keel, that these are ample in the Picard group.
As we will outline in this section, their methods our similar to ours in that
they describe the Picard group as generated by polygon divisors. 
We have not yet succeeded in proving ampleness for the desired sets,
$\gamma$, however the search led
to a new presentation of $Pic(\M_{0,n})$ with a very simple form which
we will prove in
this section.

This section may stand alone from the rest of the text, hence we recall some definitions for the ease of the reader.

\begin{defn}\label{Divxdef}\index{$Div(X)$} Let $X$ be a smooth manifold, and let
  $Div(X)$ be the
  group formally generated by Weil divisors on $X$.  The Picard group,
  $Pic(X)$, is the quotient of $Div(X)$ by the principal divisors.\end{defn}

We have the following
characterization/definition of the Picard group of Weil divisors on
$\M_{0,n}$.

\begin{thm}\cite{Ke}\label{Keelsbasis} The Picard group,
  $Pic(\M_{0,n})$\index{$Pic(\M_{0,n})$}, is isomorphic 
  to
  $Div(\M_{0,n})/\sim$, where $\sim$ denotes numerical equivalence of
  divisors.
\end{thm}

Any simple closed loop on a stable curve in the Deligne-Mumford stable compactification of $\Mn$ partitions the points of $Z$ into two subsets
as in figure 9.
Pinching this loop to a point yields a nodal topological surface.  The stable curves of this type are obtained by
putting all possible complex structures on this topological surface.
A single boundary component parametrizes these stable curves for a
given pinched loop.
We
denote by $d_A$
the boundary divisor in which the loop pinches the subset
$A\subset Z$, hence $d_A=d_{Z\setminus A}$.  We denote the set of
irreducible boundary divisors on $\M_{0,n}$ by $D^n$.  This set has
cardinality $2^{n-1}-1-n$.

\begin{center}
 \scalebox{.8}{\input{pinchPenn.pstex_t}}
\end{center}
\begin{ex}

The set of boundary divisors, $D^4$, on $\M_{0,4}$ contains the
three divisors, $d_{z_1,z_2}, d_{z_1,z_3}$ and $d_{z_1,z_4}$.

The set of boundary divisors, $D^5$, on $\M_{0,5}$ contains the 10
divisors, $d_A$, where $A\subset Z$ has cardinality 2. 
\end{ex}

\begin{thm}\cite{Ke}
A presentation of 
the Picard group, $Pic(\M_{0,n})$, is given by taking 
the classes, $\delta_A$ of the boundary divisors, $d_A\in D^n$ as generators, subject to the
following relations:  for any four distinct elements,
$z_i,z_j,z_k,z_l$ in $Z$,
$$\sum_{\textstyle{{z_i,z_j \in A}\atop{z_k,z_l\notin A}}} \delta_A =
\sum_{\textstyle{{z_i,z_k \in A}\atop{z_j,z_l\notin A}}} \delta_A =
\sum_{\textstyle{{z_i,z_l \in A}\atop{z_j,z_k\notin A}}} \delta_A.$$
\end{thm}


The following theorem specifies a basis
 for $Pic(\M_{0,n})$ and also yields an expression for its dimension.

\begin{thm}\label{Keel} [Gi, 2008]
Let $[z_{i_1},...,z_{i_n}]$ denote a cyclic ordering of the marked
points, considered as labelling the consecutive edges of an $n$-gon.  Then a basis for $Pic(\M_{0,n})$ is given by the
divisors defined by
nonempty subsets of marked points on the $n$-gon which do not form an
adjacent set of vertices on the $n$-gon.  We call this set of divisors
the {\bf non-adjacent basis}.\end{thm}

{\bf Remark}.
The following combinatorial formula for the dimension follows
immediately from counting the elements of the non-adjacent basis:
\begin{equation}
\mathrm{dim}(Pic(\M_{0,n}))=2^{n-1}-1-n- {n\choose
  2}+n=2^{n-1}-1-{n\choose 2}.
\end{equation}
This dimension was found by S. Keel in \cite{Ke} as the dimension of the first Chow group of $\M_{0,n}$.

%

\begin{ex}
Consider the standard ordering,
$(z_1,z_2,z_3,z_4,z_5)$.  Then the non-adjacent basis for
$Pic(\M_{0,5})$ for this ordering is
given by the five divisors $$d_{\{z_1,z_3\}}, d_{\{ z_1,z_4\}}, d_{\{z_2,z_4\}},
d_{\{z_2,z_5\}}, d_{\{z_3,z_5\}}.$$
\end{ex}

\section{A new presentation of $Pic(\M_{0,n})$}

In this section, we give a simple expression of each boundary divisor
in $Pic(\M_{0,n})$ in
terms of any non-adjacent basis.
This yields a new and very simple presentation for $Pic(\M_{0,n})$
with a minimal set of relations.

Before stating the theorem, we introduce some notation.
Given an $n$-gon decorated by marked points in the cyclic order, $(z_{i_1},...,z_{i_n})$, a
divisor in the basis of 
the Picard group can be described as an ordered list of disjoint subsets,
$B_1,...,B_N$, where each $B_i$ is a set of adjacent points on the
$n$-gon but no pair $(B_J,B_{J+1})$ is a set of adjacent points,
and the divisor is given by the blowup at the equality of the marked points in
$\cup_{1}^N B_i$.  Each pair, $B_{J}, B_{J+1}$ mod $N$, defines
a non-empty gap between them which we denote $G_J$.\label{adjacentsetdef}\index{$B_i,\ G_i$}  Specifically, let 
$B_J=\{z_{i_j},...,z_{i_{j+k}} \},
B_{J+1}=\{z_{i_p},...,z_{i_{p+q}}\}$.  Then
$G_J=\{z_{i_{j+k+1}},...,z_{i_{p-1}}\}$. In this way, we can write a
basis divisor as $(B_1,G_1,...,B_N,G_N)$. 

\begin{thm} Let $\delta$ denote a dihedral ordering $(z_{i_1},\ldots,
z_{i_n})$ on the points $z_1,\ldots,z_n$.  Then
$Pic(\M_{0,n})$ is generated by the set of boundary divisors of
$\M_{0,n}$ (denoted by subsets of $\{z_1,\ldots,z_n\}$ of cardinality
between $2$ and $n-2$), subject to the relations

\begin{equation}\label{sarahsthm}\delta_I=\sum_{J\in {\cal J}}
  \delta_J - \sum_{K\in {\cal K}} 
\delta_K,
\end{equation}
where $I$ denotes a consecutive subset of points for the ordering
$\delta$, ${\cal J}$ denotes the set of non-adjacent subsets 
$$J=B_1\cup \cdots \cup B_j$$ 
of $\{z_1,\ldots,z_n\}$ such that $I$ is equal to a ``segment'' of even
length, $$
B_i,G_i,\ldots,B_k,G_k\hbox{ or }G_i,B_{i+1},G_{i+1},\ldots,G_{k},B_{k+1}$$
of $(B_1,G_1,\ldots,B_N,G_N)$, and ${\cal K}$ denotes the set of
non-adjacent subsets $K=B_1\cup \cdots\cup B_j$ such that
$I$ is equal to a ``segment'' of odd length,
$$B_i,G_i,\ldots,B_k,G_k,B_{k+1}\hbox{ or }
G_i,B_{i+1},G_{i+1},\ldots,B_k,G_k$$
of $(B_1,G_1,\ldots,B_N,G_N)$.
\end{thm} 

The beauty of the theorem is more easily seen by rephrasing it as: the coefficients of any divisor in the basis of the Picard group given by a cyclic ordering can be calculated by the parity of the defining blocks of the divisor.  
The precise statement of the theorem does not do justice to its simplicity, as illustrated in the following example.

\begin{ex}
We have the following expression for the divisor, $\delta_{1,2,3}$, in the basis of $Pic(\M_{0,6})$ given by the cyclic ordering $(1,2,3,4,5,6) =(z_1,z_2,z_3,z_4,z_5,z_6)$:
$$\delta_{1,2,3} = -\delta_{1,3} + \delta_{1,4} +\delta_{3,6}-\delta_{4,6}+\delta_{1,2,4}-\delta_{1,3,5}+\delta_{1,4,5}.$$
\end{ex}

\begin{proof}
We will do this proof by induction.  

Let $\Delta_\delta$ be set of divisors which is a basis for
$Pic(\M_{0,n})$ with respect to the cyclic order $\delta$ by theorem
\ref{Keel}.  We denote by 
$\delta_{B_1\cdots B_N} = (B_1,G_1,\ldots,B_N,G_N)$ an element of
$\Delta_\delta$.  Let $I$ be a consecutive subset for the cyclic order
$\delta$, and $\delta_I$ the corresponding boundary divisor in the
Picard group.  Then we
may restate the theorem as follows.  One can express $\delta_I$ as a
linear combination of elements of $\Delta_\delta$:

\begin{equation}\label{deltaI}
\delta_I = \sum C_{B_1,...,B_{J_k}}^I (B_1,G_1...,B_{J_k}, G_{J_k}),
\end{equation}\index{$C_{B_1,...,B_{J_k}}^I$} and the coefficients are given by
\begin{equation}\label{piccoeffs}
C^I_{B_1,...,B_N} = \begin{cases} 
1 & I=\bigcup_{p=1}^j B_{i+p} \cup
  G_{i+p}\\
-1 & I=\bigl(\bigcup_{p=1}^j B_{i+p} \cup
  G_{i+p} \bigr) \cup B_{i+j+1}\\
0 &\hbox{otherwise},
\end{cases}
\end{equation}
where for $1<p<j$, $i+p$ is taken modulo $n$. 

The following theorem gives the base case for the induction.

\begin{thm}\label{Gib}\cite{Gi}
The coefficient, $C_{B_1,B_2}^I$ in the Picard basis with respect to
$\gamma$ of the basis divisor $(B_1,G_1,B_2,G_2)$ is given by
\begin{equation}\label{angela}
C_{B_1,B_2}^I =
\begin{cases} 1 & I=B_i\cup G_j \hbox{ for any } i,j \\
-1 & I=B_i\hbox{ or }I=G_j\hbox{ for any } i,j\\
0 &\hbox{otherwise.}
\end{cases}
\end{equation}
Furthermore, by recursion on $N$, this formula allows one to calculate
$C^I_{B_1,...,B_N}$ for any basis element $(B_1,...,G_N)\in \Delta_\delta$.
\end{thm}

To calculate the coefficient recursively for $N=3$, we use the following
artful technique due to A. Gibney.  Let $B_1B_2=B_1\cup B_2$,
$G_1G_2=G_1 \cup G_2$ and $G_3B_1=G_3 \cup B_1$.  Consider another basis
of the Picard group
containing $(B_1B_2, G_1G_2, B_3, G_3)$.  Then the coefficient of
$\delta_I$ on $(B_1B_2, G_1G_2, B_3, G_3)$, $C^I_{B_1B_2, B_3}$, 
is equal to the coefficient of 
\begin{equation}\label{bd}
\sum C_{B_1,...,B_{J_k}}^I
(B_1,G_1...,B_{J_k}, G_{J_k})\hbox{ on }(B_1B_2, G_1G_2, B_3, G_3),
\end{equation}
 by the
expression \eqref{deltaI} in this new basis.  By theorem \ref{Gib}, the
the only non-zero terms in the expression \eqref{bd} are
the four terms in which the basis element in the basis with respect to
$\delta$ can be written as a union of the sets, $B_1B_2, G_1G_2, B_3$
and $G_3$:
$$ C^I_{B_1,B_2,B_3}, C_{B_1,B_2}^I, C_{G_1, G_2}^I\hbox{ and }
C^I_{G_3B_1, B2}.$$  
Hence we have
\begin{align}\label{N3}
C^I_{B_1B_2,B_3} & = C^I_{B_1,B_2,B_3} C^{B_1,B_2,B_3}_{B_1B_2,B_3} +
C^I_{B_1,B_2} C^{B_1,B_2}_{B_1B_2,B_3} + C^I_{G_1,G_2}
C^{G_1,G_2}_{B_1B_2,B_3} \notag \\ & \qquad + C^I_{G_3B_1,B_2}
C^{G_3B_1,B_2}_{B_1B_2,B_3} \notag \\
&= C^I_{B_1,B_2,B_3} - C^I_{B_1,B_2} -C^I_{G_1,G_2} + C^I_{G_3B_1,B_2}
\notag \\
C^I_{B_1,B_2,B_3} & = C^I_{B_1B_2,B_3} + C^I_{B_1,B_2} + C^I_{G_1,G_2}
- C^I_{G_3B_1,B_2}.
\end{align}

\begin{ex}
In $\M_{0,6}$, take a basis for the Picard group defined by the
standard cyclic order, $(z_1,z_2,z_3,z_4,z_5,z_6)$.  In this example, we
write the divisor, $\delta_I$, $I=\{z_2,z_3,z_4\}$, in this basis.  By
theorem \ref{Gib}
we can calculate most 
of the coefficients directly, since all but one of the basis elements
can be written as the partition into four sets, $(B_1, G_1, B_2, G_2)$.
\begin{align}
\delta_I = \delta_{z_1,z_4} -\delta_{z_1,z_5} -\delta_{z_2,z_4}
+\delta_{z_2,z_5} + \delta_{z_1,z_3,z_4} 
+\delta_{z_1,z_4,z_6}  +C^I_{\{z_1\},
  \{z_3\}, \{z_5\}} \delta_{z_1,z_3,z_5}.
\end{align}
We can now apply the recursion step in \eqref{N3}. The first term,
$C^I_{\{z_1,z_3\}, \{z_4\}}=0$ by theorem \ref{Gib}.  Likewise,
$C^I_{\{z_1\},\{z_3\}} = 0$, $ C^I_{\{z_2 \},\{z_4\}} = -1$ and
$C^I_{\{z_1,z_6\},\{z_3\}} = 0$, so $C^I_{\{z_1\},
  \{z_3\}, \{z_5\}} = -1$.
\end{ex}

We can generalize this recursive procedure to prove
\eqref{piccoeffs},
which is equivalent to the formula \eqref{sarahsthm} in the statement
of the theorem.

It will be useful to consider a visual interpretation
of a basis divisor $(B_1,...,G_N)$, which pictures the divisor as an
$N$-gon with the sets $B_i, G_j$ on the vertices as in figure \ref{pillowgon2}.

\begin{figure}
\begin{center}
 \input{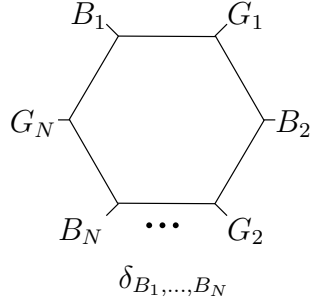}
\caption{Polygon representing a basis divisor}
\label{pillowgon2}
\end{center}
\end{figure}

We have an equivalent restatement of theorem using the pictorial representation
 in figure \ref{pillowgon2} of a basis divisor: The coefficient of
 $\delta_I$ on the 
basis divisor
$\delta_{B_1,...,B_{N}}$ is $(-1)^m$ if $I$ is the union of $m$ of
  the vertices on the polygon representation of the basis divisor
  $\delta_{B_1,...,B_{N}}$.

Using this pictorial interpretation, we prove \eqref{piccoeffs}
 by induction on $N$, the number of blocks of consecutive elements on
 $\delta$ that define the basis divisor.

Statement \ref{sarahsthm} is true for
$N=2$ by theorem \ref{Gib}.  

Assume statement \ref{sarahsthm} is true for $N-1$.  Let
$(B_1,..,G_N)$ be a basis element in $\Delta_\delta$.

We denote by $B_1\cdots B_{N-1} = B_1\cup \cdots \cup B_{N-1}$, $G_1\cdots
G_{N-1} = G_1\cup \cdots \cup G_{N-1}$, $G_NB_1=G_N\cup B_1$. Let
$\delta'$ denote the dihedral ordering of $\{z_1,\ldots,z_n\}$ given
by $$\delta'=(B_1,\ldots, B_{N-1},G_1,\ldots, G_{N-1},B_N,G_N),$$ where the
ordering on the points inside each set $B_i$, $G_i$ is that inherited
from $\delta$.  By theorem \ref{Keel}, the ordering $\delta'$ determines
a basis $\Delta_{\delta'}$ of
$Pic(\M_{0,n})$, and the divisor $\delta_{B_1\cdots B_{N_1},B_N}$,
which
we denote by $d=(B_1\cdots B_{N-1}, G_1\cdots G_{N-1}, B_N,
  G_N)$ is in the set $\Delta_{\delta'}$.
 
By expression \eqref{bd}, the coefficient of
$\delta_I$ on $d$ in the 
$\Delta'_{\delta'}$ basis is
equal to the coefficient of 
\begin{equation*}
\sum C_{B_1,...,B_{J_k}}^I
(B_1,G_1...,B_{J_k}, G_{J_k})\hbox{ on } d.
\end{equation*}

Just as in expression \eqref{N3} for $N=3$, the only divisors in 
  $\Delta_{\delta}$ which have a non-zero coefficient on $d$ 
  are 
\begin{equation}
\begin{cases}
\delta_{B_1\cup \cdots \cup B_N}=(B_1,G_1,\ldots,B_N,G_N)\\
\delta_{B_1\cup\cdots\cup B_{N-1}}=(B_1,G_1,\ldots,B_{N-1},G_{N-1}B_NG_N)\\ 
\delta_{G_1\cup \cdots \cup G_{N-1}}=(G_1,B_2,...,G_{N-1},B_{N}G_NB_1)\\
\delta_{G_N\cup B_1\cup B_2\cup\cdots\cup B_{N-1}}=(G_NB_1,G_1,
 B_2,...,G_{N-2},B_{N-1},G_{N-1}B_N).
\end{cases}
\end{equation}
By \eqref{angela}, these four coefficients are respectively
$1,-1, -1$ and $1$.  Thus, we obtain  
\begin{equation}\label{divcoeffs} C^I_{B_1,...,B_N}  = C^I_{B_1\cdots
    B_{N-1}, B_N} 
+C^I_{B_1,...,B_{N-1}} + C^I_{G_1,...,G_{N-1}} -C^I_{G_NB_1, B_2,...,
  B_{N-1}}
\end{equation}
which we rewrite more concisely as
\begin{equation}C =T_1+T_2 +T_3-T_4.\end{equation}
Each of the basis
elements, $\delta_{B_1\cdots B_{N-1}, B_N}, \delta_{B_1,...,B_{N-1}},
\delta_{G_1, ..., G_{N-1}}, \delta_{G_NB_1,...,B_N}$ is a divisor
defined by less than $N$ blocks, so by the induction hypothesis,
$T_1=(-1)^{m_1}$ if $I$ is the union of $m_1$ vertices in polygon 1
in figure \ref{pillowgon4}, $T_2=(-1)^{m_2}$ if $I$ is the union of
$m_2$ vertices in polygon 2.  Likewise, polygon 3 gives $T_3$ and
polygon 4 gives $T_4$.

\begin{figure}
\begin{center}
\scalebox{0.8}{
 \input{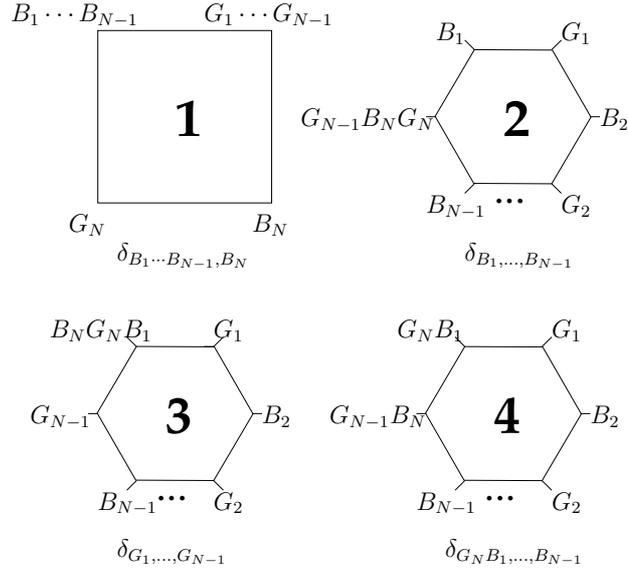}}
\caption{Polygons corresponding to $T_1, T_2, T_3, T_4$}
\label{pillowgon4}
\end{center}
\end{figure}

To prove \eqref{piccoeffs}, we need the induction step in each of the
following cases.

{\it Case 1.} $I$ is not the union of any collection of the sets
 $B_i, G_j$.  Then
by the expression \eqref{divcoeffs} and the induction hypothesis, the
 coefficient of $\delta_I$ on 
$(B_1,...,G_N)$ is 0.

{\it Case 2.} $I$ is the union of an even number of consecutive
subsets:    
$$I=\bigcup_{p=1}^j B_{i+p} \cup
  G_{i+p}\ {\mathrm{ or }}\ I = \bigcup_{p=1}^j G_{i+p} \cup
  B_{i+p+1}$$ and these 
  sets are not ``on the boundary'', in other words
  $i\geq 1$ and $i+j\leq N-2$.  By theorem \ref{Gib}, $T_1=0$.  By the
  induction
  hypothesis, $T_2=T_3=T_4=1$.  Then $\delta^I_{B_1,...,B_N}=
  1+1-1=1$ verifying statement \ref{piccoeffs} of the theorem.  This
  case covers all divisors $\delta_I$ such that $I$ is the union of an
  even number
  of subsets from $\{ B_2,G_2, ..., B_{N-2}, G_{N-2}\}$ or
  $\{G_2,B_3,..., G_{N-2}, B_{N-1}\}$.

{\it Case 3.} $I$ is the union of an odd number of non-boundary
consecutive subsets:
$$I=\bigl( \bigcup_{p=1}^j B_{i+p} \cup
  G_{i+p}\bigr) \cup B_{i+j+1} \ {\mathrm{ or }}\ I = \bigl(
  \bigcup_{p=1}^j G_{i+p} \cup 
  B_{i+p+1}\bigr) \cup G_{i+j+1},$$
$i\geq 1, i+j+1\leq N-2$.  Then by the same arguments as in case 2,
  $T_1=0, T_2=T_3=T_4 = -1$ so that $\delta^I_{B_1,...,B_N}=
  -1-1+1=-1$ verifying statement \eqref{piccoeffs} of the theorem.
  This case covers all $\delta_I$ where $I$ is the union of an odd
  number of subsets from $\{B_2,G_2, ...,G_{N-3},B_{N-2}\}$
  or $\{G_2,B_3,...,B_{N-2}, G_{N-2}\}$.

For the following 12 boundary cases, we calculate $T_1$ by theorem \ref{Gib}
and $T_2, T_3, T_4$ are gotten from the induction hypothesis.  The
results of the calculation are summarized in the following table and
can be deduced from the parity of loops around the vertices in
the polygons in
figure \ref{pillowgon4}.  Recall that the last
column, $C=T_1+T_2+T_3-T_4$ denotes $C^I_{B_1, ...,B_N}$.  To prove
the theorem, it 
suffices to calculate that $C=1$ when $I$ is the union of an even
number of subsets, $B_i, G_j$, and that $C=-1$ when $I$ is the
union of an odd number of such subsets.

 \begin{table}[ht]
\begin{tabular}{c | l |  c c c c | c}
Case & $I$ & $T_1$ & $T_2$ & $T_3$ & $T_4$ & $C$\\ [0.5ex]
\hline
\hline
4 & $G_{i} \cup B_{i+1}\cup \cdots \cup B_N \cup G_N\cup B_1, \ 2\leq
i\leq N-1$ & 0 & 1 & 1 & 1 & 1\\
\hline
5 & $B_{i}\cup G_{i} \cup \cdots \cup G_N \cup B_1, \ 2 \leq i\leq N-1$
& 0 & -1  & -1& -1 & -1 \\
\hline
6 & $G_N$ or $B_N$ & -1 & 0 & 0 & 0 & -1\\
\hline 
7 & $B_N\cup G_N$ & 1 & 0 & 0 & 0 & 1\\
\hline 
8 & $G_{i} \cup B_{i+1}\cup \cdots \cup B_N \cup G_N,\ 1\leq i \leq
N-1$ & 0 & -1 & 0 & 0 & -1 \\
\hline
9  & $B_{i}\cup G_{i} \cup \cdots \cup B_N \cup G_N,\ 2\leq i \leq
N-1$ & 0 & 1 & 0 & 0 & 1 \\
\hline
10 & $B_i\cup \cdots \cup B_N, \ 2\leq i \leq N-1 $& 0 & 0 & 0 & 1 & -1\\
\hline
11 & $ G_i\cup \cdots \cup B_N,\ 1\leq i\leq N-1 $&0 & 0 & 0 &
-1 & 1 \\ 
\hline
12 & $G_i\cup  \cdots \cup G_{N-1},\ 1\leq i \leq {N-1} $ & 0
& 0 & -1 &0& -1 \\
\hline
13 & $B_i \cup G_{i} \cup \cdots \cup G_{N-1},  \ 2\leq i \leq {N-1} $& 0
& 0 & 1 & 0 & 1 \\
\hline
14 & $ G_i\cup \cdots \cup B_{N-1}, \ 2 \leq i \leq {N-2} $& 0
& 1 & 1 & 1 & 1\\
\hline 
15 & $B_i \cup \cdots \cup B_{N-1}, 2\leq i \leq {N-1} $ & 0 & -1 & -1 & -1
& -1 \\

\hline\hline
\end{tabular}
\end{table}

The fifteen cases above cover all possible consecutive subsets $I$ and
all verify the statement of the theorem.

\end{proof}

\chapter{Index of Notations and Definitions by Chapter}\notag

In this index, the number following the definition indicates the page.
\begin{multicols}{2}{
\section*{Chapter 1}
$\zeta(k_1,...k_d)$: \pageref{zetadef}\\
$\MZV$: \pageref{zetadef}\\
Depth, weight of a multizeta value: \pageref{zetadef}\\
Convergent: \pageref{convergentdef}\\
$\cdot$\ : \pageref{cdot}\\
$*$, $st(\underline{a}, \underline{b})$: \pageref{stuffledefn}\\
$\sh$, $sh(\omega_1, \omega_2)$: \pageref{shdefintro}\\
Double shuffle: \pageref{doubleshufrels}\\
Hoffman's relation: \pageref{Hoffman}\\
$\FZ$, $\zeta^F(w)$: \pageref{FZ}\\
$\nfz$, $\mathfrak{z}(w)$: \pageref{NFZdef}\\
$\nz$, $\overline{\zeta}(w)$: \pageref{nznumbersdef}\\
$\Qxy,\ \Qyi$: \pageref{qxyqyi}\\
$\piy$: \pageref{piydefintro}\\
$\Delta_{\sh},\ \Delta_*$: \pageref{coprods}\\
$\ds$: \pageref{dsdefintro}\\
Primitive: \pageref{dsdefintro}\\
The Poisson bracket, $\{f,g\}$: \pageref{Pbdef}\\
$F_n^i(\ds)$: \pageref{Fndef}\\
$\grt$: \pageref{grtdef}\\
$\widetilde{\nfz}$, $\z^{\sha}(w)$, $\z^*(v)$: \pageref{nfzdef}\\
$H^{n-3}(\Mn)$: \pageref{cohomnot}, \pageref{cohomclaim}\\
$\Mn$: \pageref{Mndef}\\
$\ell$: \pageref{elldef}\\
The fat diagonal, $\Delta$: \pageref{fatdelta}\\
$\M_{0,n}$: \pageref{Mnbar}\\
$d_A$: \pageref{d_Adef}\\
Cell, $(z_{i_1},...,z_{i_n})$: \pageref{celldef}\\
Period: \pageref{perioddefintro}\\
$\mathcal{C}$: \pageref{perioddefintro}\\
$Z$: \pageref{Zdef}\\
$\delta$, the standard cell: \pageref{standardcelldelta}\\
Cell form, $[z_{i_1},..., z_{i_n}]$: \pageref{cellformdefintro}\\
$\mathcal{P}_Z$, $I_Z$: \pageref{PnIn}\\
Partial compactification, $\Mn^\delta$: \pageref{partcompdef}\\
$\mathcal{FC}$: \pageref{FCDEF}\\
Chord: \pageref{chorddefCh1}\\
Residue map, $\Res_{d}^p$: \pageref{polyresDEF}\\

\section*{Chapter 2}
Weight, depth of a multizeta: \pageref{depdivCh2}\\
$\mathfrak{f}$: \pageref{freedef}\\
$(f|w)$: \pageref{f|wdef}\\
Weight, $w(f)$ in $\Qxy$: \pageref{depthQxy}\\
Depth, $d(f)$ in $\Qxy$:\pageref{depthQxy}\\
$V_n$: \pageref{depthQxy}\\\
$\Lxy$, and bracket, $[f,g]$: \pageref{freeliedef}\\
$\mathbb{L}_n^i[x,y]$: \pageref{freeliegradeddef}\\
Weight, depth in $\Qyi$: \pageref{qyigradeddef}\\
$\Lyi$: \pageref{lyidef}\\
$\piy(f)$: \pageref{piy}\\
The Poisson bracket, $\{f,g\}$: \pageref{pbrakdef}\\
Lyndon word: \pageref{Lyndonworddef}\\
Lyndon-Lie word, $[\omega]$: \pageref{lyndonlieworddef}\\
$C_n^i$: \pageref{Cni}\\
$F^i\ds$: \pageref{Fids}\\
$P$, $\overline{P}$: \pageref{Biglem}\\
$Q$: \pageref{a_n}\\
$\Lambda$, $\Lambda_D$, $\Lambda_A$: \pageref{TDEF}\\
$D_j$, $D_{j,z}$: \pageref{D_jz}\\
$A_k$: \pageref{A_kdef}\\
$\Phi_{KZ}$: \pageref{Phikz}\\
$\Phi_{\sha}$: \pageref{Phisha}\\
$DM$, $DM_\gamma$: \pageref{DMDEF}\\
$M$: \pageref{MDEF}\\
$N$: \pageref{NDEF}\\

\section*{Chapter 3}

The standard cell, $X =X_{\delta}$: \pageref{stcellDEF}, \pageref{MTZ}\\
01-cell form: \pageref{01cellformDEF}\\
Mixed Tate motive, $\MT(\Z)$: \pageref{MTZ}\\
$X_{\delta}$, $M_{\delta}$, $A_{\delta}$ and $B_{\delta}$: \pageref{MTZ}\\
Framed mixed Tate motive, ${\mathcal{M}}(\Z)$: \pageref{momegadef}\\
$m(\omega)$: \pageref{momegadef}\\  
$\LIE$, $\F$: \pageref{LIEDEF}\\
$\Mod_{0,n}$: \pageref{Mndefch3}\\
$S$: \pageref{Cycstrucdef}\\
Cyclic structure, dihedral structure: \pageref{Cycstrucdef}\\
Cell-form, $[s_1,...,s_n]$: \pageref{cellform}\\
Cell, $(\gamma_1,..., \gamma_n)$, $X_{n,\gamma}=X_{\gamma}$:
\pageref{celldefCH3}\\
Standard cell, $X_{S,\delta}=X_{n,\delta}$: \pageref{celldefCH3}\\
$\I_D(i,j)$: \pageref{funnygonchI}\\
Cell-function, $\langle \gamma \rangle$: \pageref{cellfuncdef}\\
01 cell-function: \pageref{01cellfuncDEF}\\
${\mathcal{P}}_S$: \pageref{prodmapsection}\\
Pairs of polygons, $(\gamma,\eta)$: \pageref{polypairdef}\\
Product map: \pageref{prodmapDEF}\\
Cell-zeta value, $\mathcal{C}$: \pageref{cellzetavalueDEFch3}\\
Formal cell-zeta values, $\mathcal{FC}$: \pageref{formalcell} \\
$V_S$, $I_S$: \pageref{VIDEF}\\
Lyndon basis: \pageref{Lyndonbasis}\\
Chord, $\chi(\gamma)$: \pageref{chordDEF}\\
The polygon residue map, $\Res_p^D(\eta)$: \pageref{RespmapDEFch3}\\
$W_S$: \pageref{Lyndins}\\
Lyndon insertion shuffles, ${\mathcal{L}}_S$:
\pageref{lyndoninsertionshufflesdef}\\
Framing: \pageref{framingDEF}\\
Lyndon insertion words, ${\mathcal{W}}_S$: \pageref{lyndoninsertionwordsdef}\\
Special convergent words: \pageref{specconvw}\\
$L(\gamma, v_1, ..., v_k)$: \pageref{lyndshwfxst}\\
Composed residue map, $\Res^p_{D_1,\ldots,D_m}$:
\pageref{cmpresmapDEF}\\
$J_S$, $K_S$: \pageref{JSDEF}
$c_0(n)$: \pageref{csub0n}\\

\section*{Chapter 4}
Arnol'd's ring, $A$: \pageref{arnoldring}\\
$H^k(\Mn)$: \pageref{derhamc}\\
$d_K$: \pageref{newd_adef}\\
$D$: \pageref{m0ngamma}\\
$\Mn^\gamma$: \pageref{m0ngamma}\\
$C_q(X)$: \pageref{CqXdef}\\
$Z_r^{i,q-i}$: \pageref{Eblahdef}\\
$B_r^{i,q-i}$: \pageref{Bblahdef}\\
$E_r$, $E_r^{i,q-i}$: \pageref{Bblahdef}\\
$d_r$, $d_r^{i,q-i}$: \pageref{Bblahdef}\\
Spectral sequence: \pageref{Bblahdef}\\
$F$, $E$, $B$: \pageref{fibration}\\
$h[\lambda]$: \pageref{hlambdadef}\\
Algebraic ($k$ form): \pageref{fungrtypedef}\\
Prime divisor: \pageref{defdiv}\\
Weil divisor: \pageref{defdiv}\\
Divisor: \pageref{defdiv}\\
Intersection-divisor: \pageref{intersectingdivisors}\\
$\ell$, $Z$: \pageref{ellzdef}\\
Lyndon basis, Lyndon shuffle: \pageref{Lynbas4}\\
$\pi$: \pageref{polyformmap} \\
${\mathcal{W}}_{\gamma_0}^p$, ${\mathcal{W}}_{\gamma_{\sha}}^p$:
\pageref{tiredofmakinglabels}\\ 
${\cal L}^i(P)$: \pageref{lyndononepointDEF}\\
${\cal W }_0^p, {\cal{W}}_{d_R}^p, {\cal{W}}_{d_S}^p,
  {\cal{W}}_{d_{R\cup S}}^p, {\cal{W}}_{d_{R\cup S},S}^p$:
  \pageref{defdiv3}\\
${\cal W }_0, {\cal{W}}_{d_R}, {\cal{W}}_{d_S},
  {\cal{W}}_{d_{R\cup S}}, {\cal{W}}_{d_{R\cup S},S}$:
  \pageref{cellformstuff}\\
$Div(X)$: \pageref{Divxdef}\\
$Pic(\M_{0,n})$: \pageref{Keelsbasis}\\
$B_i$, $G_i$: \pageref{adjacentsetdef}\\
$C_{B_1,...,B_{J_k}}^I$: \pageref{deltaI}

}\end{multicols}
\addcontentsline{toc}{chapter}{Alphabetical Index}
\printindex

\end{document}